\UseRawInputEncoding
\documentclass[12pt,reqno]{amsart}
\usepackage[english]{babel}
\usepackage{amsmath,amsthm,amssymb,mathrsfs}
\usepackage{array,cite,enumitem,xcolor}
\usepackage[left=2.9cm,right=2.9cm,top=2.8cm,bottom=2.8cm]{geometry}
\usepackage[colorlinks=true,urlcolor=blue,
citecolor=red,linkcolor=blue,linktocpage,pdfpagelabels,
bookmarksnumbered,bookmarksopen,backref=page]{hyperref}

\renewcommand*{\backrefalt}[4]{%
  \ifcase #1\relax
  \else\space #2%
  \fi}

\numberwithin{equation}{section}

\theoremstyle{plain}
\newtheorem{theorem}{Theorem}[section]
\newtheorem{lemma}[theorem]{Lemma}
\newtheorem{corollary}[theorem]{Corollary}
\newtheorem{proposition}[theorem]{Proposition}
\theoremstyle{definition}

\theoremstyle{remark}
\newtheorem{remark}[theorem]{Remark}

\newcommand{\N}{\mathbb N}
\newcommand{\R}{\mathbb R}

\newcommand{\norm}[1]{\left\|#1\right\|}
\newcommand{\cK}{\mathcal K}
\newcommand{\cC}{\mathcal C}
\newcommand{\cS}{\mathcal S}
\title[Spectral-edge continuation, crossover, and stability]
{Spectral-Edge Continuation, Dimensional Crossover, and Endpoint Stability for Partially Confined Nonlinear Schr\"odinger Equations}
\author[X.~X.~Zhong]{Xuexiu Zhong}

\address{South China Research Center for Applied Mathematics and Interdisciplinary
Studies \& School of Mathematical Sciences, South China Normal University,
Guangzhou 510631, P.~R. China}
\email{zhongxuexiu1989@163.com}
\thanks{Xuexiu Zhong was supported by the NSFC (No.~12271184), the Young
Top-notch Talent Project of Guangdong Province (2024TQ08A725), and the Guangzhou
Basic and Applied Basic Research Foundation (2024A04J10001).}
\makeatletter
\@namedef{subjclassname@2020}{%
  \textup{2020} Mathematics Subject Classification}
\makeatother

\date{}

\begin{document}
\begin{abstract}
We study the global geometry and endpoint dynamics of positive solutions of
$$
 -\Delta u+|y|^2u+\lambda u=g(u),
 \qquad (y,z)\in\mathbb R^d\times\mathbb R^m,
 \qquad m=N-d,
$$
when the harmonic potential confines only the $y$-variables.  Coercivity is lost at the lower spectral threshold $\lambda=-d$.  Under subcritical asymptotic conditions on $g$ at zero and infinity, a zero-mass Liouville condition, and a strict Hamiltonian sign condition, we construct a connected component of centered positive solutions whose frequency projection is the full interval $(-d,+\infty)$.

The continuum joins two asymptotically different geometries.  As
$\lambda+d\to0^+$, an anisotropic rescaling selects the lowest confined harmonic mode and an $m$-dimensional scalar ground state in the free variables. As $\lambda\to+\infty$, an isotropic blow-up yields the $N$-dimensional scalar ground state.  Thus the same component undergoes a dimensional crossover from effective dimension $m$ to dimension $N$.  These endpoint profiles and their
mass laws hold for every sequence of centered positive solutions, rather than only along the distinguished component.

The proof combines Gaussian--exponential localization on compact frequency strips, a Lyapunov--Schmidt reduction at the lower edge, compact spectral perturbation at high frequency, and Leray--Schauder continuation.  At both ends we prove uniqueness and differentiability of the centered branch, determine the Morse indices and symmetry kernels of the real and imaginary linearizations in the full energy space, and obtain sharp derivatives of the mass.  Under the standard Hamiltonian-flow hypotheses, the resulting slope criterion classifies orbital stability and instability away from the two mass-critical powers.  Finally, the mass map on the global component, combined with endpoint control on the full centered solution set, yields existence, nonexistence, and multiplicity for normalized standing waves and explains the simultaneous occurrence of the critical exponents $2+4/m$ and $2+4/N$.
\end{abstract}

\maketitle

\section{Introduction}
Let $1\le d\le N-1$, set $m=N-d$, and write
$x=(y,z)\in\mathbb{R}^d\times\mathbb{R}^m$.  We study the global set of positive
solutions of
\begin{equation}\label{E:fixed-lambda}
        -\Delta u+|y|^2u+\lambda u=g(u),\qquad 0<u\in X,
\end{equation}
where
\begin{equation}\label{E:Xdef}
 X:=\left\{u\in H^1(\mathbb{R}^N):\int_{\mathbb{R}^N}|y|^2u^2\,\mathrm dx<\infty\right\},
 \qquad
 \|u\|_X^2:=\int_{\mathbb{R}^N}(|\nabla u|^2+|y|^2u^2+u^2)\,\mathrm dx .
\end{equation}
Our goals are to continue positive states throughout the natural frequency
range, identify their profiles and linearized spectra at both ends, and relate
the two endpoint geometries to orbital stability.  The resulting
fixed-frequency theory is then applied to the normalized problem
\begin{equation}\label{E:normalized}
 -\Delta u+|y|^2u+\lambda u=g(u),\qquad 0<u\in X,
 \qquad \int_{\mathbb{R}^N}u^2\,\mathrm dx=a>0,
\end{equation}
where $\lambda$ is the Lagrange multiplier and the constraint is encoded by the mass map on the positive-solution set of \eqref{E:fixed-lambda}.

Equation \eqref{E:fixed-lambda} is the stationary equation for standing waves in a partially harmonic trap; see \cite{AntonelliCarlesDrumondSilva,BBJV}.  The first $d$ directions are confined, whereas the remaining $m=N-d$ directions are translation invariant.  From the dynamical viewpoint, long-time resonant effects generated by partial harmonic trapping were analyzed in \cite{HaniThomann}.  In three dimensions the cubic choices $(d,m)=(2,1)$ and $(1,2)$ model, respectively, cigar- and disk-shaped geometries; see Remark~\ref{R:physical-crossover}.  Rigorous dimension-reduction results under strong anisotropic confinement are available both for effective nonlinear Schr\"odinger equations \cite{BenAbdallahMehatsSchmeiserWeishaeupl} and at the many-body ground-state level \cite{LiebSeiringerYngvason,SchneeYngvason}.
This anisotropy is already visible in
$$
 H:=-\Delta+|y|^2=(-\Delta_y+|y|^2)-\Delta_z,
 \qquad \sigma(H)=\sigma_c(H)=[d,+\infty)
$$
\cite[Proposition~2.1]{CWZZ}.  Positive finite-energy solutions require $\lambda\ge-d$ \cite[Theorem~1.1(i)]{CWZZ}, while the quadratic form of $H+\lambda$ is coercive precisely for $\lambda>-d$.  At $\lambda=-d$, the zero-mass Liouville results of \cite[Theorem~1.1(ii)--(iii)]{CWZZ} exclude solutions under the hypotheses recalled in Corollary~\ref{C:normalized}.  The threshold modes are generated by
\begin{equation}\label{E:phi0}
 \phi_0(y):=\pi^{-d/4}e^{-|y|^2/2},\qquad
 (-\Delta_y+|y|^2)\phi_0=d\phi_0,\qquad \|\phi_0\|_2=1 .
\end{equation}
Thus the lower endpoint is the bottom of continuous spectrum, rather than an isolated eigenvalue of the full operator.

The broader variational and dynamical background includes orbital stability by
constrained minimization \cite{CazenaveLions}, sharp interpolation estimates
and mass-critical thresholds \cite{Weinstein}, and the abstract
slope--index theory of Grillakis, Shatah and Strauss
\cite{GrillakisShatahStraussI,GrillakisShatahStraussII}.  External potentials
break part of the translation symmetry and require the neutral modes to be
identified before that theory can be applied; compare \cite{Stuart}.  In the
partially confined setting, scattering, constrained variational methods,
orbital stability, uniqueness, and multiplicity have been studied in
\cite{AntonelliCarlesDrumondSilva,BBJV,Jia,JinHong,WeiWu,JeanjeanSong,SongYan,ShanShuaiWu}.
Related endpoint questions for three-dimensional pure-power ground states are discussed in the preprint of Song and Hajaiej \cite{SongHajaiej}. Related endpoint questions for three-dimensional pure-power ground states are discussed in the preprint of Song and Hajaiej \cite{SongHajaiej}. Here the primary object is instead the global geometry of a connected component in $(\lambda,u)$-space, together with sequence-independent endpoint control of the full centered positive-solution set.  The global-branch strategy is related to \cite{BartschZhongZou,JZZ} and, at a foundational level, to Rabinowitz's global bifurcation theorem \cite{Rabinowitz}.  More recent continuation theory on unbounded cylinders also highlights loss of compactness as a central global alternative; see \cite{ChenWalshWheeler}.  In the present problem three features must be treated simultaneously: translations remain possible in $z$; the lower endpoint is a continuous-spectrum threshold; and the two endpoints require different scalings.  Projection onto the lowest harmonic mode near $\lambda=-d$ produces an $m$-dimensional equation, whereas
the potential disappears in the high-frequency blow-up and the limit is $N$-dimensional.

We use the threshold obstruction, zero-mass Liouville criteria, and block-symmetry theorem of \cite{CWZZ}.  Building on the global-branch strategy of \cite{BartschZhongZou,JZZ}, we establish the following results:
\begin{enumerate}[label=(\roman*)]
\item a compact-strip localization theorem combining Gaussian confinement in $y$ with a Hamiltonian exclusion of nonzero plateaus in $z$;
\item sequence-independent asymptotics at both endpoints of the full centered positive-solution set, under anisotropic and isotropic rescalings, respectively;
\item a uniform complementary-mode resolvent theorem, a Lyapunov--Schmidt construction, and a Schur-complement proof of centered nondegeneracy and Morse index one at the lower spectral edge;
\item a connected component $\mathcal C$ with $P_1(\mathcal C)=(-d,+\infty)$, anchored by local degree $-1$;
\item high-frequency uniqueness, uniform centered invertibility, and a compact spectral-perturbation argument which determines how the confined and free translation modes split;
\item differentiable mass laws at both ends and, under explicit dynamical hypotheses, an orbital-stability classification based on the full real--imaginary linearized spectrum;
\item existence, nonexistence, and multiplicity statements for \eqref{E:normalized}, with existence read from $\mathcal{C}$ and nonexistence from the full centered solution set $\mathcal{S}$.
\end{enumerate}
Control of the full centered solution set $\mathcal{S}$ is needed for the nonexistence conclusions, since continuation identifies only one connected component.

We briefly indicate the proof.  Centering the block-radial solutions supplied by \cite[Theorem~1.3]{CWZZ} removes translation in $z$.  Radial compactness handles $m\geq2$, while monotonicity prevents loss of mass when $m=1$.  The harmonic potential and the strict sign $\frac12g(s)s-G(s)>0$ then give, on compact frequency strips,
$$
 u(y,z)\le C e^{-\theta|y|^2/2-\sigma|z|}.
$$
This estimate yields compactness in a Gaussian--exponentially weighted centered space, the common setting for degree and continuation.

Write $\mu=\lambda+d$.  At the lower edge,
$$
\begin{cases}
 \mu^{-1/(\alpha-2)}u(y,\zeta/\sqrt\mu)
 \longrightarrow \phi_0(y)V_\alpha(\zeta)~\hbox{in}~H^1(\mathbb{R}^N),\\
 -\Delta V_\alpha+V_\alpha=c_\alpha V_\alpha^{\alpha-1}~\hbox{in}~\mathbb{R}^m,\\
 0<V_\alpha\in H_{\mathrm{rad}}^{1}(\mathbb{R}^m),
 \end{cases}
$$
where $c_\alpha:=\mu_0\int_{\mathbb{R}^d}\phi_0^\alpha\mathrm{d}y$.  The zero-mass Liouville condition first forces the unscaled amplitude to vanish.  At high frequency, isotropic rescaling yields the unique positive radial $H^1$-solution of $-\Delta W_\beta+W_\beta=\mu_\infty W_\beta^{\beta-1}$ in $\mathbb{R}^N$.  Uniformly over the full centered solution set,
$$
 \|u\|_2^2\sim
 \begin{cases}
 (\lambda+d)^{\frac2{\alpha-2}-\frac m2}\|V_\alpha\|_{L^2(\mathbb{R}^m)}^{2},
     &\lambda+d\to0^+,\\
 \lambda^{\frac2{\beta-2}-\frac N2}\|W_\beta\|_{L^2(\mathbb{R}^N)}^{2},
     &\lambda\to+\infty.
 \end{cases}
$$
These laws exhibit the two critical exponents $2+4/m$ and $2+4/N$, namely the mass-critical exponents associated with the two limiting NLS scalings; compare \cite{Weinstein}.

The lower-edge product profile is next corrected by Lyapunov--Schmidt
reduction.  A normalized-difference argument gives local uniqueness, and a
second reduction for the linearization gives centered nondegeneracy and Morse
index one.  Matching the nonzero weighted-space and energy-space spectra yields
local degree $-1$.  Compactness on every bounded frequency strip then prevents
the component from terminating in the interior, so its projection is
$(-d,+\infty)$.  At high frequency, the universal profile convergence and the
nondegeneracy of $W_\beta$ give uniqueness and a uniform inverse in the
centered form space.  A separate spectral-compactness lemma preserves the
negative and near-zero spectral clusters; the exact $z$-translations remain
neutral, whereas the broken $y$-translations move to the positive side.  The
differentiated Lyapunov--Schmidt equations and the high-frequency inverse then
give the two mass slopes and the endpoint stability classification.  Finally,
continuity of the mass on $\mathcal{C}$ gives existence and multiplicity of positive normalized solutions for \eqref{E:normalized}, whereas nonexistence follows from universal endpoint
limits on $\mathcal{S}$ and compact-strip bounds.

Section~\ref{sec:assumption-results} states the hypotheses and the main
results.  Sections~\ref{sec:preliminaries}--\ref{sec:Asymptotics-high}
establish compactness and the universal endpoint limits.
Section~\ref{sec:LUTD} constructs the local lower-edge branch and computes its
degree, and Section~\ref{sec:global-conti} performs the global continuation.
Section~\ref{sec:endpoint-stability} proves high-frequency uniqueness, the
endpoint spectral and mass-slope results, and orbital stability or instability.
Section~\ref{sec:normalized-solution} treats prescribed mass as an application,
including exact multiplicity in the two diagonal endpoint cases.

\section{Analytical framework and main fixed-frequency results}\label{sec:assumption-results}

\subsection{Structural assumptions on the nonlinearity}

Let
$$2_k^*:=\begin{cases}
        \frac{2k}{k-2},& k\geq3,\\
        +\infty, & k=1,2.
        \end{cases}
        \qquad\hbox{and}\qquad
        G(s):=\int_0^s g(t)\,\mathrm{d}t.
$$
We impose the following assumptions on the nonlinearity.

\begin{enumerate}[label=$\mathbf{(G_{\arabic*})}$]
    \item $g\in C^1([0,+\infty),\R)$ satisfies $g(0)=g'(0)=0$ and $g(s)>0$ for $s>0$. \label{gcon1}
    \item \label{gcon2} There exist constants
$$2<\alpha<2_m^*,\qquad 2<\beta<2_N^*,\qquad \mu_0>0,
        \qquad \mu_\infty>0,$$
such that
\begin{equation}\label{E:gsmalllarge}
        \lim_{s\to0^+}\frac{g(s)}{s^{\alpha-1}}=\mu_0,
        \qquad
        \lim_{s\to+\infty}\frac{g(s)}{s^{\beta-1}}=\mu_\infty,
\end{equation}
and the same limits hold after differentiating, namely
\begin{equation}\label{E:gprimelimits}
        \lim_{s\to0^+}\frac{g'(s)}{s^{\alpha-2}}=\mu_0(\alpha-1),
        \qquad
        \lim_{s\to+\infty}\frac{g'(s)}{s^{\beta-2}}=\mu_\infty(\beta-1).
\end{equation}
    \item \label{gcon3} The zero-mass equation
\begin{equation}\label{E:zero-mass}
        (H-d)u=g(u),
\end{equation}
has no nontrivial block-radial bounded nonnegative classical solution satisfying the following locally uniform growth condition: for every compact $K\subset\mathbb R^m$, there are $C_K>0$ and $\eta_K\in[0,\frac12)$ such that
$$
 \sup_{z\in K}\bigl(|u(y,z)|+|\nabla_yu(y,z)|+|\Delta_yu(y,z)|\bigr)
 \le C_Ke^{\eta_K|y|^2},\qquad y\in\mathbb R^d.
$$
    \item \label{gcon4} The Hamiltonian density has the strict sign
    \begin{equation}\label{E:positive-Hamiltonian-density}
        \frac12g(s)s-G(s)>0\qquad\hbox{for every }s>0.
    \end{equation}
\end{enumerate}
In \ref{gcon3} no integrability in $z$ is assumed; the locally uniform growth condition is satisfied by the limits in Lemma~\ref{lem:polynomial-growth-zero-mass-limit}.  Assumptions
\ref{gcon1}--\ref{gcon2} determine the endpoint scalings and suffice for the high-frequency limit and the local construction.  Assumption~\ref{gcon3} is used only to exclude a nonzero lower-edge limit (Lemma \ref{L:lower-vanishing}), whereas \ref{gcon4} is used only in the compact-strip Hamiltonian argument (Proposition~\ref{prop:uniform-z-tail}) and its compactness consequences.  Thus the two assumptions serve different roles in the argument.

The elliptic and continuation results below require only the hypotheses just
listed.  For the dynamical conclusions we state separately the following
standard Hamiltonian-flow assumption.  Let $X_{\mathbb C}$ be the
complexification of $X$ and set
\begin{equation}\label{E:gauge-covariant-nonlinearity}
 \mathfrak g(\psi):=
 \begin{cases}
  g(|\psi|)\psi/|\psi|,&\psi\ne0,\\
  0,&\psi=0.
 \end{cases}
\end{equation}
We shall say that \ref{condition-D} holds if
\begin{enumerate}[label=$\mathbf{(D)}$]
\item \label{condition-D} the Cauchy problem
\begin{equation}\label{E:dynamical-assumption-flow}
 i\partial_t\psi=H\psi-\mathfrak g(\psi),\qquad \psi(0)\in X_{\mathbb C},
\end{equation}
is locally well posed in $X_{\mathbb C}$ with continuous dependence and the
usual continuation criterion, conserves the mass $M$ and energy $E$ defined by
\begin{equation}\label{E:mass-energy-dynamics}
 M(\psi):=\int_{\mathbb{R}^N}|\psi|^2\mathrm{d}x,~
 E(\psi):=\frac{1}{2}\int_{\mathbb{R}^N}\bigl(|\nabla\psi|^2+|y|^2|\psi|^2\bigr)\mathrm{d}x
           -\int_{\mathbb{R}^N} G(|\psi|)\mathrm{d}x,
\end{equation}
has $C^2$ action functionals near the standing-wave orbits considered
below, and satisfies the flow regularity required by the instability
theorem in \cite{GrillakisShatahStraussII}.
\end{enumerate}
This assumption is deliberately separated from \ref{gcon1}--\ref{gcon4}: the latter are elliptic hypotheses and, for a general $C^1$ function $g$, do not by themselves encode a complete Cauchy
theory.

\begin{remark}[On the dynamical assumption]\label{R:dynamical-assumption}
For pure powers, and more generally whenever the associated gauge-covariant
Nemytskii map satisfies the standard energy-subcritical Strichartz estimates,
\ref{condition-D} follows from the usual partially confined Cauchy theory; see
\cite{AntonelliCarlesDrumondSilva,BBJV}.  All existence, uniqueness,
continuation, endpoint-asymptotic, and spectral statements of this paper are
independent of \ref{condition-D}.  Only the final passage from the slope and spectral
information to orbital stability or instability uses it.
\end{remark}

Under assumptions \ref{gcon1} and \ref{gcon2}, there exists $C>0$ such that
\begin{equation}\label{E:growth}
        |g'(s)|\leq C\bigl(1+s^{\beta-2}\bigr),\qquad s\geq 0.
\end{equation}

\begin{remark}\label{remark:20260616-1112}
By \cite[Theorem~3.3]{CWZZ}, assumption \ref{gcon3} is automatic when
$g(s)\ge cs^q$ for all $s>0$, with $c>0$, provided either $m=N-d\le2$ and
$q>1$, or $m>2$ and
$$
        1<q\leq \frac{m}{m-2}.
$$
Further criteria under integrability assumptions appear in
\cite[Theorem~1.1(ii)--(iii)]{CWZZ}.
\end{remark}

\begin{remark}[Model nonlinearities]
The hypotheses include both pure-power and genuinely double-power models. The pure-power choice $g(s)=\kappa s^{p-1}$ corresponds to $\alpha=\beta=p$; the preceding criterion applies with $q=p-1$. For example, let
$$g(s)=\kappa_0s^{\alpha-1}+\kappa_\infty s^{\beta-1},\qquad \kappa_0,\kappa_\infty>0,$$
where $2<\alpha<\beta<2_N^*$. Then \ref{gcon1}--\ref{gcon2} hold with endpoint coefficients
$\mu_0=\kappa_0$ and $\mu_\infty=\kappa_\infty$, while
$$
 \frac{1}{2}g(s)s-G(s)
 =\kappa_0\Bigl(\frac{1}{2}-\frac1\alpha\Bigr)s^\alpha
  +\kappa_\infty\Bigl(\frac{1}{2}-\frac1\beta\Bigr)s^\beta>0
 \qquad(s>0),
$$
so \ref{gcon4} also holds.  Assumption~\ref{gcon3} follows from the preceding remark whenever $m\le2$, or whenever $m>2$ and $\alpha-1\le m/(m-2)$.  Thus the assumptions allow the small- and large-amplitude powers, and hence the two endpoint scalings, to be genuinely different.
\end{remark}

\subsection{Symmetry, centering, and the compactness class}

We first verify the decay required by the symmetry theorem quoted below.

\begin{lemma}[Decay of finite-energy solutions]\label{L:individual-decay}
Assume \ref{gcon1}--\ref{gcon2} and $\lambda>-d$.  Every positive classical solution $u\in X$ of \eqref{E:fixed-lambda} belongs to
$L^\infty(\mathbb{R}^N)\cap C_0(\mathbb{R}^N)$.
\end{lemma}

\begin{proof}
Set $\lambda_-:=\max\{-\lambda,0\}$.  By \ref{gcon1}--\ref{gcon2},
$$
 -\Delta u\le a(x)u,\qquad
 0\leq a(x):=\lambda_-+\frac{g(u(x))}{u(x)}
 \leq C(1+u^{\beta-2}).
$$
Since $u\in H^1(\mathbb{R}^N)$ and $\beta<2_N^*$, the norms $\|a\|_{L^q(B_2(x_0))}$ are bounded independently of $x_0$ for some $q>\max\{N/2,1\}$.  The standard local estimate for nonnegative subsolutions therefore yields
$$\|u\|_{L^\infty(B_1(x_0))}\leq C_u\|u\|_{L^2(B_2(x_0))},\qquad x_0\in\mathbb{R}^N,$$
with $C_u$ independent of $x_0$; see also the subsolution arguments in \cite[proofs of Theorems~2.3 and~3.2]{BrezisLieb}. Since $u\in L^2$, this proves both boundedness and uniform decay at infinity.
\end{proof}

We shall use the following theorem from \cite[Theorem~1.3]{CWZZ}.

\noindent\textbf{Theorem A.}~
Assume $\lambda>-d$ and $g\in C^1([0,+\infty),[0,+\infty))$ with
$g'(0)=0$.  Every positive solution $u\in C^2(\mathbb{R}^N)\cap X$ of \eqref{E:fixed-lambda} satisfying $\lim_{|x|\to\infty}u(x)=0$ has the following properties:
\begin{enumerate}[label=(\roman*)]
    \item it is radially symmetric and decreasing with respect to the first $d$ confined variables $y=(x_1,\ldots,x_d)$;
    \item it is radially symmetric and decreasing with respect to the remaining $m=N-d$ free variables $z=(x_{d+1},\ldots,x_N)$, up to a translation in the $z$-hyperplane.
\end{enumerate}

The potential makes bounded subsets of $X$ tight in $y$, since
\begin{equation}\label{eq:chebixuefu}
\int_{|y|>R} u^2 \,\mathrm{d}x \leq R^{-2} \int_{\mathbb{R}^N} |y|^2 u^2 \,\mathrm{d}x.
\end{equation}
Translations in $z$ remain.  By Lemma~\ref{L:individual-decay} and Theorem~A, we may center every positive solution at $z=0$ and introduce
\begin{equation}\label{E:Xr}
X_r:=\{u\in X:u(y,z)=U(|y|,|z|)\}.
\end{equation}
The compactness mechanism differs for $m\ge2$ and $m=1$; see
Lemma~\ref{L:embedding}.  If $m\ge2$, radial compactness in the free variables prevents translations, and it suffices to use the positivity cone
\begin{equation}\label{E:cone}
        \mathcal{K}:=\{u\in X_r:u\ge0\}.
\end{equation}
If $m=1$, this radial compactness is lost.  In this case we use the closed monotone cone
\begin{equation}\label{E:cone-m-one}
        \mathcal{K}:=\{u\in X_r:u\ge0,
        \ u \text{ is nonincreasing in } |y| \text{ and in } |z|\}.
\end{equation}

Thus $\mathcal{K}$ is dimension-dependent only in the preliminary compactness argument.  Since $g$ need not be monotone, the resolvent need not preserve the cone when $m=1$. Accordingly, after proving \eqref{eq:20260616-1731-global}, we use in every dimension
the same Gaussian--exponentially weighted centered space, on which the positive-part resolvent is completely continuous.

\subsection{The centered positive-solution set}

Let
\begin{equation}\label{E:solution-set}
        \cS:=\{(\lambda,u)\in(-d,+\infty)\times\cK:
        u>0\text{ solves }\eqref{E:fixed-lambda}\}.
\end{equation}

Given a lower frequency
$\lambda_0>-d$, choose $\theta\in(0,1)$ and $\tau>0$ so that
$$
 \lambda_0+\theta d>0,
 \qquad
 \tau<\sqrt{\lambda_0+\theta d}.
$$
Writing $s^+:=\max\{s,0\}$ and $g_+(s):=g(s^+)$, we use the centered Banach space
$$
\begin{aligned}
 \mathcal X_{\theta,\tau}
 &:=\left\{u\in X_r\cap C_0(\mathbb{R}^N):
 e^{\frac\theta2|y|^2+\tau|z|}u\in L^\infty(\mathbb{R}^N)\right\},
 \\
 \|u\|_{\mathcal X_{\theta,\tau}}
 &:=\|u\|_X+
 \left\|e^{\frac\theta2|y|^2+\tau|z|}u\right\|_\infty,
\end{aligned}
$$
and, for $\lambda\ge\lambda_0$, the positive-part resolvent map
$$
 T_\lambda^+(u):=(H+\lambda)^{-1}g_+(u).
$$
Here $(H+\lambda)^{-1}:X'\to X$ is the form resolvent of
Lemma~\ref{L:form-resolvent}.  Membership of solutions in
$\mathcal X_{\theta,\tau}$ will follow from localization, rather than being assumed.

\begin{theorem}[Uniform localization and compactness]
\label{T:compactness}
Assume \ref{gcon1} and \ref{gcon2}, and let
$$
    [\Lambda_1,\Lambda_2]\subset\subset(-d,+\infty).
$$
Then every solution $(\lambda,u)\in\cS$ with
$\lambda\in[\Lambda_1,\Lambda_2]$ satisfies
\begin{equation}\label{E:Linftybound}
    \norm{u}_{L^\infty(\R^N)}
    \le C(\Lambda_1,\Lambda_2).
\end{equation}

Moreover, for every $m\ge1$, every $\theta\in(0,1)$, and every
$\tau>0$ satisfying
$$
    \Lambda_1+\theta d>0,
    \qquad
    \tau<\sqrt{\Lambda_1+\theta d},
$$
the positive-part fixed-point map $T_\lambda^+$ is completely
continuous on bounded subsets of the weighted centered space
$\mathcal X_{\theta,\tau}$, uniformly for
$\lambda\in[\Lambda_1,\Lambda_2]$. Every nonzero fixed point of
$T_\lambda^+$ is a centered positive solution of
\eqref{E:fixed-lambda}.

If, in addition, \ref{gcon4} holds, then the set
$$
    \mathcal U_{\Lambda_1}^{\Lambda_2}
    :=
    \bigl\{u:(\lambda,u)\in\cS,\
    \lambda\in[\Lambda_1,\Lambda_2]\bigr\}
$$
is compact in $X_r$.  Every centered positive solution
in the strip belongs to $\mathcal X_{\theta,\tau}$; consequently,
the nonzero fixed points of $T_\lambda^+$ are precisely the centered positive solutions in the strip.
\end{theorem}

The limiting equations are as follows.  Define
\begin{equation}\label{E:c-alpha}
        c_\alpha:=\mu_0\int_{\R^d}\phi_0(y)^\alpha\mathrm{d}y.
\end{equation}
Let $V_\alpha$ be the unique positive radial solution of
\begin{equation}\label{E:Valpha}
        -\Delta_z V+V=c_\alpha V^{\alpha-1},
        \qquad V\in H^1(\R^m),\quad V>0,
\end{equation}
and let $W_\beta$ be the unique positive radial solution of
\begin{equation}\label{E:Wbeta}
        -\Delta W+W=\mu_\infty W^{\beta-1},
        \qquad W\in H^1(\R^N),\quad W>0.
\end{equation}
Existence of these scalar ground states is covered by the classical scalar-field
theory of \cite{BerestyckiLions}, while uniqueness follows from \cite{Kwong}; their
nondegeneracy is a separate spectral result, for which we use
\cite[Lemma~2.1]{ChangGustafsonNakanishiTsai}.

\begin{theorem}[Universal endpoint asymptotics and dimensional crossover]\label{T:asymptotics}
Assume \ref{gcon1} and \ref{gcon2}, and let $(\lambda_n,u_n)\in\cS$.

\smallskip
\noindent\textbf{(i) The lower spectral edge.}
Assume in addition \ref{gcon3}, and suppose
$\mu_n:=\lambda_n+d\to0^+$.  Define
\begin{equation}\label{E:lower-rescale}
        \widetilde{u}_n(y,\zeta):=
        \mu_n^{-1/(\alpha-2)}u_n\bigl(y,\zeta/\sqrt{\mu_n}\bigr),
        \qquad (y,\zeta)\in\R^d\times\R^m.
\end{equation}
Then the whole sequence satisfies
\begin{equation}\label{eq:20260618-1027}
\widetilde{u}_n\to \phi_0(y)V_\alpha(\zeta)\qquad\hbox{strongly in}~H^1(\mathbb{R}^d\times\mathbb{R}^m).
\end{equation}
Moreover, if we write
$$\widetilde{u}_n=\phi_0(y)v_n(\zeta)+r_n,\qquad\int_{\mathbb{R}^d}r_n(y,\zeta)\phi_0(y)\,\mathrm{d}y=0,
$$
then
\begin{equation}\label{eq:20260618-1030}
v_n\to V_\alpha\qquad\hbox{in}~
W^{2,q}_{\mathrm{loc}}(\mathbb{R}^m)\cap H^1(\mathbb{R}^m)
\quad\hbox{for every }1<q<+\infty,
\end{equation}
and
\begin{equation}\label{eq:20260618-1031}
r_n\to0\qquad\hbox{strongly in}~H^1(\mathbb{R}^N).
\end{equation}
Consequently,
\begin{equation}\label{E:lower-mass}
        \norm{u_n}_2^2
        =\mu_n^{\frac{2}{\alpha-2}-\frac m2}
        \left(\norm{V_\alpha}_{L^2(\R^m)}^2+o(1)\right).
\end{equation}
In particular,
\begin{equation}\label{E:lower-limit}
        \norm{u_n}_2^2\to
        \begin{cases}
        0, & \alpha<2+4/m,\\
        \norm{V_\alpha}_{2}^{2}, & \alpha=2+4/m,\\
        +\infty, & \alpha>2+4/m.
        \end{cases}
\end{equation}

\smallskip
\noindent\textbf{(ii) The high-frequency end.}
Suppose $\lambda_n\to+\infty$.  Define
\begin{equation}\label{E:high-rescale}
        \widehat{u}_n(\xi):=\lambda_n^{-1/(\beta-2)}
        u_n(\xi/\sqrt{\lambda_n}),
        \qquad \xi\in\R^N.
\end{equation}
Then the whole sequence satisfies
\begin{equation}\label{E:high-convergence}
        \widehat{u}_n\to W_\beta\qquad \hbox{in}~C_{\mathrm{loc}}^{2}(\mathbb{R}^N)\cap H^1(\mathbb{R}^N)\cap L^q(\mathbb{R}^N)~\hbox{for every}~1\leq q<+\infty.
\end{equation}
 Consequently,
\begin{equation}\label{E:high-mass}
        \norm{u_n}_2^2
        =\lambda_n^{\frac{2}{\beta-2}-\frac N2}
        \left(\norm{W_\beta}_{L^2(\R^N)}^2+o(1)\right),
\end{equation}
and hence
\begin{equation}\label{E:high-limit}
        \norm{u_n}_2^2\to
        \begin{cases}
        +\infty, & \beta<2+4/N,\\
        \norm{W_\beta}_{2}^{2}, & \beta=2+4/N,\\
        0, & \beta>2+4/N.
        \end{cases}
\end{equation}
\end{theorem}

\begin{theorem}[Local branch and spectral information at $\lambda=-d$]\label{T:local-degree}
Assume \ref{gcon1} and \ref{gcon2}.  Then a centered positive branch exists for all
$0<\lambda+d$ sufficiently small, and its rescaled profile converges to $\phi_0V_\alpha$ as in \eqref{E:lower-rescale}--\eqref{eq:20260618-1027}.
On its interval of construction, the Lyapunov--Schmidt branch is continuous
in $\lambda$ and $C^1$ in the rescaled variables.  If, in addition, \ref{gcon3} holds, there exists $\mu_*>0$ such that, for every $0<\lambda+d<\mu_*$, equation \eqref{E:fixed-lambda} has exactly one centered positive solution, equivalently exactly one nonzero solution in $\mathcal{K}$.  Denote it by $u_\lambda$.  The centered linearized operator
$$
 \mathcal{L}_\lambda=H+\lambda-g'(u_\lambda)\quad\hbox{on }X_r
$$
has trivial kernel and Morse index one.

Moreover, fix any $\lambda_0\in(-d,-d+\mu_*)$, choose $\theta\in(0,1)$ with
$\lambda_0+\theta d>0$, and take $0<\tau<\sqrt{\lambda_0+\theta d}$.  In the
Gaussian--exponentially weighted centered space $\mathcal X_{\theta,\tau}$ defined in
\eqref{E:weighted-space}, $u_{\lambda_0}$ is an isolated fixed point of the positive-part
map $T_{\lambda_0}^+$.  Writing $\deg_{LS}$ for the Leray--Schauder degree, for every sufficiently small isolating neighborhood
$\mathcal O_{\lambda_0}\subset\mathcal X_{\theta,\tau}$,
\begin{equation}\label{E:weighted-degree-local}
 \deg_{LS}(I-T_{\lambda_0}^+,\mathcal O_{\lambda_0},0)=-1\ne0.
\end{equation}
This formulation and degree value apply for every $m\ge1$.
\end{theorem}

\begin{theorem}[Global spectral-edge continuum]\label{T:global}
Assume \ref{gcon1}, \ref{gcon2}, \ref{gcon3}, and \ref{gcon4}.  Let $P_1:(-d,+\infty)\times X_r\to(-d,+\infty)$ be the projection
$P_1(\lambda,u)=\lambda$.  There exists a connected component $\mathcal{C}\subset\mathcal{S}$ such that
\begin{equation}\label{E:projection}
        P_1(\mathcal{C})=(-d,+\infty).
\end{equation}
The component $\mathcal{C}$ contains the local branch $\{(\lambda,u_\lambda):-d<\lambda<-d+\mu_*\}$ from Theorem \ref{T:local-degree}.
In view of Theorem~\ref{T:asymptotics}, it connects, after the corresponding rescalings, the $m$-dimensional lower-edge profile $\phi_0V_\alpha$ to the $N$-dimensional high-frequency profile $W_\beta$.
\end{theorem}

\begin{theorem}[Endpoint branches, mass slopes, and orbital stability]
\label{T:endpoint-stability}
Assume \ref{gcon1}--\ref{gcon4}.  There are constants
$\mu_{\mathrm s}>0$ and $\Lambda_{\mathrm s}>0$ with the following
properties.

\smallskip
\noindent\textbf{(i) Lower edge.}
For $0<\mu=\lambda+d<\mu_{\mathrm s}$, the centered positive solution
$u_\lambda$ of Theorem~\ref{T:local-degree} is a $C^1$ branch and
\begin{equation}\label{E:main-lower-mass-slope}
 \frac{\mathrm d}{\mathrm d\lambda}\|u_\lambda\|_2^2
 =\mu^{\kappa_--1}
   \left(\kappa_-\|V_\alpha\|_{L^2(\mathbb R^m)}^2+o(1)\right),
 \qquad
 \kappa_-:=\frac2{\alpha-2}-\frac m2 .
\end{equation}

\smallskip
\noindent\textbf{(ii) High-frequency end.}
For every $\lambda>\Lambda_{\mathrm s}$, equation
\eqref{E:fixed-lambda} has exactly one centered positive solution, denoted
again by $u_\lambda$.  This family is $C^1$ in $\lambda$, and
\begin{equation}\label{E:main-high-mass-slope}
 \frac{\mathrm d}{\mathrm d\lambda}\|u_\lambda\|_2^2
 =\lambda^{\kappa_+-1}
   \left(\kappa_+\|W_\beta\|_{L^2(\mathbb R^N)}^2+o(1)\right),
 \qquad
 \kappa_+:=\frac2{\beta-2}-\frac N2 .
\end{equation}
At either end the real linearized operator has one negative eigenvalue and,
on the full energy space,
\begin{equation}\label{E:main-full-kernels}
 \ker\bigl(H+\lambda-g'(u_\lambda)\bigr)
 =\operatorname{span}\{\partial_{z_1}u_\lambda,\ldots,
                         \partial_{z_m}u_\lambda\}.
\end{equation}
The imaginary linearized operator $H+\lambda-\frac{g(u_\lambda)}{u_\lambda}$ is nonnegative and has kernel $\operatorname{span}\{u_\lambda\}$.

If, in addition, \ref{condition-D} holds, consider the gauge-covariant nonlinear
Schr\"odinger flow
\begin{equation}\label{E:main-dynamical-NLS}
 i\partial_t\psi=H\psi-\mathfrak g(\psi),
 \qquad \mathfrak g\ \hbox{as in }\eqref{E:gauge-covariant-nonlinearity}.
\end{equation}
Then $e^{i\lambda t}u_\lambda$ is orbitally stable in the complex energy
space, modulo phase and translations in $z$, in the following
endpoint ranges:
\begin{align*}
 &\lambda+d\ \hbox{small and } \alpha<2+\frac4m,
 &&\lambda\ \hbox{large and } \beta<2+\frac4N.
\end{align*}
It is orbitally unstable in the corresponding supercritical ranges
$\alpha>2+4/m$ and $\beta>2+4/N$.  At
$\alpha=2+4/m$ or $\beta=2+4/N$, the leading slope vanishes and no
stability conclusion is asserted without a higher-order expansion.
\end{theorem}

\begin{remark}[Physical interpretation of the dimensional crossover]
\label{R:physical-crossover}
Theorem~\ref{T:asymptotics} describes two distinct endpoint geometries.
With $\mu=\lambda+d$, the lower-edge profile is
$$
 u(y,z)\sim\mu^{1/(\alpha-2)}
 \phi_0(y)V_\alpha(\sqrt\mu\,z),
$$
where the asymptotic notation refers to the rescaled convergence in that
theorem.  The confined variables remain on the oscillator scale, while
the characteristic scale in the free variables is $\mu^{-1/2}$;
the complementary harmonic modes vanish after rescaling.
For the three-dimensional cubic model, $(d,m)=(2,1)$ and $(1,2)$ thus
give the quasi-one- and quasi-two-dimensional regimes of cigar- and
disk-shaped traps, respectively; see
\cite{GorlitzEtAl,SalasnichParolaReatto}.
Related dimension-reduction results in the strong-confinement regime
appear in
\cite{BenAbdallahMehatsSchmeiserWeishaeupl,LiebSeiringerYngvason,SchneeYngvason}.

At high frequency the profile is instead
$$
 u(x)\sim\lambda^{1/(\beta-2)}W_\beta(\sqrt\lambda\,x),
$$
so all spatial directions have the scale $\lambda^{-1/2}$ and the limiting
profile is $N$-dimensional.  The two endpoint geometries account for the
different mass-critical exponents $2+4/m$ and $2+4/N$ in
Theorems~\ref{T:asymptotics} and~\ref{T:endpoint-stability};
the stability conclusions additionally require \ref{condition-D}.
\end{remark}

\section{Spectral and compactness preliminaries}\label{sec:preliminaries}

\subsection{Spectrum and coercivity}
Let
$$H_y:=-\Delta_y+|y|^2,\qquad H=H_y-\Delta_z.$$
The harmonic oscillator satisfies
\begin{equation}\label{E:spectrum-Hy}
\sigma(H_y)=\{d+2k:k\in\N\cup\{0\}\}.
\end{equation}
Since $\sigma(-\Delta_z)=\sigma_c(-\Delta_z)=[0,+\infty)$, the tensor-sum spectral theorem \cite[Corollary, Page 301]{Reed-Simon-1980} gives
\begin{equation}\label{E:spectrum-H}
\sigma(H)=\sigma_c(H)=[d,+\infty).
\end{equation}
Therefore, for $\lambda>-d$,
\begin{equation}\label{E:coercive1}
\int_{\mathbb{R}^N}\bigl(|\nabla u|^2+|y|^2u^2+\lambda u^2\bigr)\mathrm{d}x\geq (d+\lambda)\norm{u}_2^2.
\end{equation}
Moreover,
\begin{equation}\label{E:coercive2}
\begin{aligned}
&\min\left\{\frac{d+1}{d+\lambda},1\right\}\int_{\mathbb{R}^N}\bigl(|\nabla u|^2+|y|^2u^2+\lambda u^2\bigr)\mathrm{d}x\\
\leq& \norm{u}_X^2
\leq \max\left\{1, \frac{d+1}{d+\lambda}\right\}\int_{\mathbb{R}^N}\bigl(|\nabla u|^2+|y|^2u^2+\lambda u^2\bigr)\mathrm{d}x.
\end{aligned}
\end{equation}

The preceding form estimates also give the parameter dependence of the linear resolvent that will be used repeatedly below.  We record it once in the natural $X$--$X'$ framework.

\begin{lemma}[Uniform resolvent bounds and norm continuity]
\label{L:form-resolvent}
Let $J:X\to X'$ be the canonical embedding defined by
\begin{equation}\label{E:canonical-embedding}
 \langle Ju,v\rangle_{X',X}:=
 \int_{\mathbb{R}^N}uv\,\mathrm{d}x,
 \qquad u,v\in X.
\end{equation}
For every $\lambda>-d$, let $A_\lambda:X\to X'$ be the form realization of $H+\lambda$, namely
\begin{equation}\label{E:form-realization}
 \langle A_\lambda u,v\rangle_{X',X}
 :=
 \int_{\mathbb{R}^N}
 \bigl(\nabla u\cdot\nabla v+|y|^2uv+\lambda uv\bigr)\,\mathrm dx.
\end{equation}
Then $A_\lambda$ is a topological isomorphism from $X$ onto $X'$, and its inverse
$$
 R_\lambda:=A_\lambda^{-1}=(H+\lambda)^{-1}
 \in\mathcal L(X',X)
$$
satisfies
\begin{equation}\label{E:form-resolvent-bound}
 \|R_\lambda\|_{\mathcal L(X',X)}
 \leq
 \max\left\{1,\frac{d+1}{d+\lambda}\right\}.
\end{equation}
Moreover, for every $\lambda,\nu>-d$, the first resolvent identity in this form setting is
\begin{equation}\label{E:first-form-resolvent-identity}
 R_\lambda-R_\nu
 =(\nu-\lambda)R_\lambda J R_\nu.
\end{equation}
Consequently, for every compact interval $I\subset\subset(-d,+\infty)$, there exists $C_I>0$ such that
\begin{equation}\label{E:uniform-form-resolvent-continuity}
 \sup_{\lambda\in I}\|R_\lambda\|_{\mathcal L(X',X)}\le C_I,
 \qquad
 \|R_\lambda-R_\nu\|_{\mathcal L(X',X)}
 \leq C_I|\lambda-\nu|
 \quad\hbox{for all }\lambda,\nu\in I.
\end{equation}
In particular, $\lambda\mapsto R_\lambda$ is locally Lipschitz continuous from $(-d,+\infty)$ into $\mathcal{L}(X',X)$.
\end{lemma}

\begin{proof}
The canonical embedding $J$ is bounded.  By \eqref{E:coercive2},
\begin{equation}\label{E:form-coercivity-constant}
 \langle A_\lambda u,u\rangle
 \ge c_\lambda\|u\|_X^2,
 \qquad c_\lambda:=\min\left\{1,\frac{d+\lambda}{d+1}\right\}>0.
\end{equation}
Lax--Milgram theorem gives the isomorphism and $\|R_\lambda\|\le c_\lambda^{-1}$.  Since $A_\nu-A_\lambda=(\nu-\lambda)J$, the resolvent identity gives \eqref{E:first-form-resolvent-identity}.  Taking norms and using the positive minimum of $c_\lambda$ on $I$ proves the remaining assertions.
\end{proof}

\smallskip
\noindent\emph{Spectral-edge caveat.}
The uniformity in Lemma~\ref{L:form-resolvent} requires the frequency set to stay away from $-d$.  Indeed, the coercivity constant in \eqref{E:form-coercivity-constant} degenerates as $\lambda+d\to0^+$.  This is why the lower spectral edge is treated by projection and rescaling rather than by a uniform bound for $R_\lambda$.

We shall use the standard interior $W^{2,p}$ estimate
\cite[Theorem~9.11]{GilbargTrudinger}, the Sobolev embeddings
\cite[Theorem~7.26]{GilbargTrudinger}, and, when the right-hand side is
locally H\"older continuous, the interior Schauder estimate
\cite[Theorem~6.2]{GilbargTrudinger}.

\subsection{Compact embeddings in the symmetric cone}

\begin{lemma}[Compact embedding]\label{L:embedding}
For every $2<q<2_N^*$, the embedding $X\cap\mathcal{K}\hookrightarrow L^q(\mathbb{R}^N)$ is compact.  Equivalently, every $X$-bounded sequence in $\mathcal K$ has a subsequence converging strongly in all such $L^q$ spaces.
\end{lemma}

\begin{proof}
After taking a subsequence, $u_n\rightharpoonup u$ in $X$, $u_n\to u$ a.e., and $u_n\to u$ locally in $L^q$ for $2\le q<2_N^*$.  For $w_n:=u_n-u$, Lions' lemma \cite[Lemma~1.21]{Willem} reduces the conclusion to
\begin{equation}\label{eq:20260620-1201}
\lim_{n\to\infty}\sup_{\xi\in\mathbb{R}^N}\int_{B_1(\xi)} |w_n|^2\,\mathrm{d}x=0,
\end{equation}
which implies $u_n\to u$ in $L^q(\mathbb{R}^N)$ for every $2<q<2_N^*$.

If \eqref{eq:20260620-1201} fails, there are $\delta>0$ and
$\xi_n=(a_n,b_n)$ with $|\xi_n|\to\infty$ such that the corresponding ball integrals are at least $\delta$.  The uniform $y$-moment bound forces $(a_n)$ to remain bounded (see \eqref{eq:chebixuefu}), so $|b_n|\to\infty$.  Since $u\in L^2$, after decreasing the
constant we obtain
\begin{equation}\label{eq:20260620-1224}
 \int_{B_1(a_n,b_n)}u_n^2\,\mathrm dx\ge c_0>0.
\end{equation}

If $m\geq2$, radiality in $z$ produces $k_n\geq c|b_n|^{m-1}\to\infty$ disjoint rotates of the ball in \eqref{eq:20260620-1224}, each carrying mass at least $c_0$.  This contradicts the uniform $L^2$ bound.

If $m=1$, take $b_n\to+\infty$ and choose $R_y$ with $B_1(a_n)\subset B_{R_y}^d$.  The slab $B_{R_y}^d\times(b_n-1,b_n+1)$ carries mass at least $c_0$.
Monotonicity in $|z|$ gives the same lower bound on every slab
$B_{R_y}^d\times[2(j-1),2j]$, $1\le j\le\lfloor(b_n-1)/2\rfloor$, again a contradiction.  This proves \eqref{eq:20260620-1201}.
\end{proof}

\begin{remark}
For $m=1$ this is the mechanism of \cite[Proposition~2.1]{JeanjeanSong}, here also covering $d=1$.  For $m\ge2$, radiality in $z$ suffices and no monotonicity is required.
\end{remark}

\begin{lemma}\label{lemma:20260622-0829}
Assume \ref{gcon1} and \ref{gcon2}.  Let $\{u_n\}\subset\mathcal K$ be bounded in $X$ and suppose, after passing to a subsequence, that $u_n\rightharpoonup u$ in $X$.  Then $g(u_n)\to g(u)$ in $X'$.
\end{lemma}
\begin{proof}
Under the assumptions \ref{gcon1} and \ref{gcon2}, for every $\varepsilon>0$ there is $C_\varepsilon>0$ such that
$$|g'(s)|\leq \varepsilon+C_\varepsilon |s|^{\beta-2},\qquad s\geq 0.$$
By the mean value theorem,
$$|g(a)-g(b)|
        \leq \varepsilon |a-b|+C_\varepsilon (|a|^{\beta-2}+|b|^{\beta-2})|a-b|,
        \qquad a,b\geq0.$$
Thus, for every $\varphi\in X$ with $\norm{\varphi}_X\le1$,
$$
\begin{aligned}
\left|\int_{\mathbb{R}^N}(g(u_n)-g(u))\varphi\,\mathrm{d}x\right|
&\leq C\varepsilon \norm{u_n-u}_2\norm{\varphi}_2 \\
&\quad + C_\varepsilon
\left\|(|u_n|^{\beta-2}+|u|^{\beta-2})(u_n-u)\right\|_{\beta'}\norm{\varphi}_\beta\\
&\leq C\varepsilon +C_\varepsilon \left\|(|u_n|^{\beta-2}+|u|^{\beta-2})(u_n-u)\right\|_{\beta'}.
\end{aligned}
$$
By Lemma \ref{L:embedding}, $u_n\to u$ in $L^\beta(\mathbb{R}^N)$; hence H\"older's inequality gives
$$\left\|(|u_n|^{\beta-2}+|u|^{\beta-2})(u_n-u)\right\|_{\beta'}\to 0.$$
By the arbitrariness of $\varepsilon>0$, we obtain
$$\lim_{n\to\infty}\|g(u_n)-g(u)\|_{X'}=\lim_{n\to \infty}\sup_{\|\varphi\|_X\leq 1}\left|\int_{\mathbb{R}^N}(g(u_n)-g(u))\varphi\,\mathrm{d}x\right|=0.$$
Hence, $g(u_n)\to g(u)$ in $X'$.
\end{proof}

\subsection{\texorpdfstring{$C^1$}{C1} limits of the rescaled nonlinearities}
\begin{lemma}[Uniform rescaled $C^1$ limits]\label{lemma:20260618-1129}
Assume \ref{gcon1} and \ref{gcon2}.  For $a>0$ define
$$
 F_a^{(0)}(s):=\frac{g(as)}{a^{\alpha-1}},
 \qquad
 F_a^{(\infty)}(s):=\frac{g(as)}{a^{\beta-1}},
 \qquad s\ge0.
$$
When the scaling regime is unambiguous, we suppress the superscript and write simply $F_a$.
Then, for every $C>0$, the following assertions hold.
\begin{enumerate}[label=(\roman*)]
\item As $a\to0^+$,
\begin{align}
 &\sup_{0<s\leq C}
 \frac{|F_a^{(0)}(s)-\mu_0s^{\alpha-1}|}{s^{\alpha-1}}
 \longrightarrow0, \label{eq:20260622-0819}\\
 &\sup_{0<s\leq C}
 \frac{|(F_a^{(0)})'(s)-\mu_0(\alpha-1)s^{\alpha-2}|}{s^{\alpha-2}}
 \longrightarrow0. \label{E:rescaled-derivative-small}
\end{align}
In particular, for some $a_*>0$ and $C_1>0$,
\begin{equation}\label{eq:20260622-1306}
 0\leq F_a^{(0)}(s)\leq C_1s^{\alpha-1},\qquad
 |(F_a^{(0)})'(s)|\leq C_1s^{\alpha-2}
\end{equation}
for $0<a<a_*$ and $0\le s\le C$.
\item As $a\to+\infty$,
\begin{align}
 &\sup_{0<s\le C}
 \frac{|F_a^{(\infty)}(s)-\mu_\infty s^{\beta-1}|}{s}
 \longrightarrow0,
 &&\label{eq:20260618-1130}\\
 &(F_a^{(\infty)})'(s)\longrightarrow
 \mu_\infty(\beta-1)s^{\beta-2},
 &&\text{uniformly for }s\in[0,C]. \label{E:rescaled-derivative-large}
\end{align}
In particular, $F_a^{(\infty)}(s)\to\mu_\infty s^{\beta-1}$ uniformly
on $[0,C]$.
\end{enumerate}
The same conclusions hold for the positive-part extensions on $[-C,C]$,
with each limiting power $s^\gamma$ replaced by $(s^+)^\gamma$ and
with the error norms displayed above unchanged.  On $s\le0$, both
the extended nonlinearities and their derivatives vanish.
\end{lemma}
\begin{proof}
For $0<s\leq C$,
$$
 \frac{F_a^{(0)}(s)}{s^{\alpha-1}}
 =\frac{g(as)}{(as)^{\alpha-1}},
 \qquad
 \frac{(F_a^{(0)})'(s)}{s^{\alpha-2}}
 =\frac{g'(as)}{(as)^{\alpha-2}}.
$$
Since $0<as\le aC$, the two limits at the origin in
\eqref{E:gsmalllarge}--\eqref{E:gprimelimits} are uniform on this shrinking interval. This proves \eqref{eq:20260622-0819} and
\eqref{E:rescaled-derivative-small}; the bounds \eqref{eq:20260622-1306} follow at once.

For the large-scale limits, fix $\varepsilon>0$.  By
\eqref{E:gprimelimits}, there is $R>0$ such that
$$
 |g'(t)-\mu_\infty(\beta-1)t^{\beta-2}|
 \le\varepsilon t^{\beta-2},\qquad t\ge R.
$$
Splitting at $as=R$ and using continuity of $g'$ on $[0,R]$ gives
$$
 e_a(C):=\sup_{0\leq s\leq C}
 |(F_a^{(\infty)})'(s)-\mu_\infty(\beta-1)s^{\beta-2}|
 \leq\max\{\varepsilon C^{\beta-2},C_Ra^{2-\beta}\}.
$$
Letting $a\to\infty$ and then $\varepsilon\downarrow0$ proves
\eqref{E:rescaled-derivative-large}.  Since $F_a^{(\infty)}(0)=0$,
for $0<s\le C$ we also have
$$
 \frac{|F_a^{(\infty)}(s)-\mu_\infty s^{\beta-1}|}{s}
 \le\frac1s\int_0^s
 |(F_a^{(\infty)})'(t)-\mu_\infty(\beta-1)t^{\beta-2}|
 \,\mathrm{d}t
 \leq e_a(C),
$$
which proves \eqref{eq:20260618-1130}.  Finally, $g_+\in C^1(\mathbb{R})$ by \ref{gcon1}, and the positive-part assertions follow immediately.
\end{proof}

\subsection{Reusable comparison and compactness principles}

The localization arguments below use the same positive-part comparison in
several different scalings.  We record the mechanism once, including the
explicit barriers that will be needed later.

Let $\Omega\subset\mathbb{R}^N$ be a Lipschitz domain, possibly unbounded,
and let $\mathscr L=-\Delta+W(x)$, where
$W\in L^\infty_{\rm loc}(\Omega)$ is real and bounded below.
Writing $W^\pm=\max\{\pm W,0\}$, define the Dirichlet form domain by
$$
 \mathscr{V}_0(\Omega):=
 \left\{\varphi\in H_0^1(\Omega):
 \int_\Omega W^+\varphi^2\mathrm{d}x<\infty\right\}.
$$
If $W\ge-b$ with $b\ge0$, then
$\int_\Omega W^-\varphi^2\le b\|\varphi\|_2^2$ for every
$\varphi\in H_0^1(\Omega)$, so the negative potential term is already
integrable.  We shall use the coercivity condition
\begin{equation}\label{E:common-barrier-coercivity}
 \int_\Omega\bigl(|\nabla\varphi|^2+W\varphi^2\bigr)\mathrm{d}x
 \ge c_0\|\varphi\|_{L^2(\Omega)}^2,
 \qquad \varphi\in\mathscr{V}_0(\Omega),
\end{equation}
for some $c_0>0$.

The barriers used below satisfy the following elementary identities and inequalities.
\begin{itemize}
\item For $0<\theta<1$,
\begin{equation}\label{E:common-Gaussian-identity}
 (-\Delta_y+|y|^2)e^{-\theta|y|^2/2}
 =\bigl[\theta d+(1-\theta^2)|y|^2\bigr]e^{-\theta|y|^2/2}.
\end{equation}
\item
Let $0<\theta<1$, $\tau>0$, and
$\langle z\rangle=(1+|z|^2)^{1/2}$.  For
\begin{equation}\label{E:common-product-barrier}
 \Phi_{\theta,\tau}(y,z)
 :=e^{-\theta|y|^2/2-\tau\langle z\rangle}
\end{equation}
 the relations $|\nabla_z\langle z\rangle|\le1$ and
$\Delta_z\langle z\rangle\geq0$ imply
\begin{equation}\label{E:common-product-barrier-identity}
 (H+\lambda-\eta)\Phi_{\theta,\tau}
 \ge\bigl[\lambda+\theta d-\eta-\tau^2
 +(1-\theta^2)|y|^2\bigr]\Phi_{\theta,\tau}.
\end{equation}
\item For $a,R>0$ and $0<\sigma<\sqrt a$,
\begin{equation}\label{eq:20260910-0954}
 (-\Delta+a)e^{-\sigma(|x|-R)}
 =\left(a-\sigma^2+\frac{(N-1)\sigma}{|x|}\right)
 e^{-\sigma(|x|-R)}\ge0,\qquad |x|>R.
\end{equation}
\end{itemize}

\begin{lemma}[Exterior comparison and explicit barriers]
\label{L:common-exterior-barrier}
With $\Omega$, $\mathscr L$, and $\mathscr{V}_0(\Omega)$ as above, assume \eqref{E:common-barrier-coercivity}. Let $v,\Phi\ge0$ have locally finite form energy up to the boundary:
$$
 \int_{\Omega\cap B_L}
 \bigl(|\nabla s|^2+(1+W^+)s^2\bigr)\mathrm{d}x<\infty
 \qquad (s=v,\Phi;\ L>0).
$$
Suppose that, for some $C\ge0$,
$$
 \mathscr L(v-C\Phi)\le0\quad\hbox{in }\mathcal D'(\Omega),
 \qquad v\le C\Phi\quad\hbox{on }\partial\Omega.
$$
If, in addition,
$$
 h:=(v-C\Phi)^+\in\mathscr{V}_0(\Omega),
$$
then $v\le C\Phi$ in $\Omega$.  Boundary inequalities are understood in
the trace sense: they give zero trace for $h$, but on an unbounded domain
do not by themselves imply its global form-domain membership.
In particular, the differential inequality holds when
$\mathscr Lv\le0\le\mathscr L\Phi$.
For a signed or complex-valued function, apply the statement to its modulus
when the corresponding distributional inequality and form-domain
condition hold.

In parts (i) and (iii) below, the condition
$(v-C\Phi)^+\in\mathscr{V}_0(\Omega)$ need not be assumed:
bounded nonnegative subsolutions suffice, with the same local regularity
and boundary interpretation.  This uses the positive lower bound for
$W$ on the respective comparison domains; the final part of the proof
provides admissible tests on bounded truncations.

The following consequences will be used repeatedly.
\begin{enumerate}[label=(\roman*)]
\item Fix $0<\theta<1$ and let $\lambda$ and $K$ range in fixed bounded sets.
Every bounded nonnegative subsolution $u$ satisfying
$$
 (H+\lambda-K)u\le0\quad\hbox{in }\mathbb R^N,
 \qquad 0\le u\le M,
$$
obeys
\begin{equation}\label{E:common-Gaussian-conclusion}
 u(y,z)\le C_\theta M e^{-\theta|y|^2/2},
\end{equation}
where $C_\theta$ depends only on $\theta$, $d$, and the prescribed
parameter bounds.
\item Fix $0<\theta<1$, $\tau>0$, and $\eta\in\mathbb R$.
For $R>0$, set $\Omega_R:=\mathbb{R}^d\times\{|z|>R\}$. Define
$$\Psi_R(y,z):=e^{-\theta|y|^2/2-\tau(|z|-R)}.$$
If $\delta_\eta:=\lambda+\theta d-\eta-\tau^2>0$, then
$\Psi_R$ is a supersolution of $H+\lambda-\eta$ on $\Omega_R$.
In particular, if $0\leq v\in X$ and $B\geq0$ satisfy
$$
 (H+\lambda-\eta)v\leq0\quad\hbox{in }\Omega_R,
 \qquad
 v(y,z)\leq B e^{-\theta|y|^2/2}\quad\hbox{on }|z|=R,
$$
then $v\le B\Psi_R$ on $\Omega_R$.  With $\Phi_{\theta,\tau}$ instead of
$\Psi_R$, the same boundary data are dominated by
$B e^{\tau\sqrt{1+R^2}}\Phi_{\theta,\tau}$.

For the inhomogeneous resolvent, assume
$\delta_0:=\lambda+\theta d-\tau^2>0$ and
$A_f:=\|e^{\theta|y|^2/2+\tau|z|}f\|_\infty<\infty$.
Then $f\in L^2$, $\lambda>-d$, and comparison on the whole space yields
\begin{equation}\label{E:common-weighted-resolvent-estimate}
 |(H+\lambda)^{-1}f|
 \le \frac{e^\tau}{\delta_0}
 \|e^{\theta|y|^2/2+\tau|z|}f\|_\infty
 e^{-\theta|y|^2/2-\tau\langle z\rangle}.
\end{equation}
The constant is uniform for $\lambda\in[\lambda_0,\Lambda]$ whenever
$\lambda_0+\theta d-\tau^2>0$.

\item Let $a,R>0$ and $0<\sigma<\sqrt a$.  If $v$ is a bounded nonnegative
subsolution of $-\Delta v+av\le0$ on $\{|x|>R\}$ and
$v\le C$ on $|x|=R$ for some $C\ge0$, then
\begin{equation}\label{E:common-Euclidean-barrier}
 v(x)\le C e^{-\sigma(|x|-R)},\qquad |x|\ge R.
\end{equation}
\end{enumerate}
\end{lemma}

\begin{proof}
Set $q=v-C\Phi$ and $h=q^+\in\mathscr{V}_0(\Omega)$.
The standard positive-part truncation of the weak inequality
$\mathscr{L}q\le0$ gives
$$
 \mathscr{L}h\le0\quad\hbox{in }\mathcal D'(\Omega).
$$
Since $h\in\mathscr{V}_0(\Omega)$, the weak form is continuous in the
form norm.  Nonnegative approximation in this norm therefore permits
$h$ itself as a test.  Using \eqref{E:common-barrier-coercivity}, we obtain
$$
 0\geq\int_\Omega\bigl(|\nabla h|^2+Wh^2\bigr)\mathrm{d}x
 \ge c_0\|h\|_{L^2(\Omega)}^2.
$$
Thus $h=0$, proving the comparison under the global form-domain assumption.

For (i), choose $R_y>0$ sufficiently large that, uniformly over the
prescribed parameter sets,
$$
 |y|^2+\lambda-K\ge1,
 \qquad \theta d+\lambda-K+(1-\theta^2)|y|^2\ge0
 \quad\hbox{for }|y|\ge R_y.
$$
With $\mathscr{L}=H+\lambda-K$ and $\Omega=\{|y|>R_y\}$,
\eqref{E:common-Gaussian-identity} gives
$$
 \begin{cases}
 \mathscr{L}u\le0\le\mathscr{L}(e^{-\theta|y|^2/2})
     &\text{in }\Omega,\\
 u\le M=M e^{\theta R_y^2/2}e^{-\theta|y|^2/2}
     &\text{on }\partial\Omega.
 \end{cases}
$$
The bounded-subsolution comparison established at the end of the proof
therefore gives
$$
 u(y,z)\le M e^{\theta R_y^2/2}e^{-\theta|y|^2/2}
 \qquad (|y|\ge R_y).
$$
The same bound holds for $|y|\le R_y$ by $u\le M$, proving
\eqref{E:common-Gaussian-conclusion} with
$C_\theta=e^{\theta R_y^2/2}$.

For (ii), on $\Omega_R$, direct computation gives
$$
 (H+\lambda-\eta)\Psi_R
 =\left[\delta_\eta+(1-\theta^2)|y|^2
       +\frac{(m-1)\tau}{|z|}\right]\Psi_R\geq0.
$$
With $\mathscr{L}=H+\lambda-\eta$, the hypotheses give
$$
 \begin{cases}
 \mathscr{L}v\le0\le\mathscr{L}\Psi_R &\text{in }\Omega_R,\\
 v(y,z)\le B e^{-\theta|y|^2/2}=B\Psi_R(y,z)
     &\text{on }\partial\Omega_R.
 \end{cases}
$$
Since $v\in X$ and $\Psi_R$ has finite energy on $\Omega_R$, their
positive difference $(v-B\Psi_R)^+$ has finite form energy; its zero
trace places it in $\mathscr{V}_0(\Omega_R)$.
Extension by zero and the oscillator inequality give the coercivity
constant
$$d+\lambda-\eta=\delta_\eta+(1-\theta)d+\tau^2>0.$$
Thus the energy comparison applies and gives $v\le B\Psi_R$.
Since $r\mapsto\sqrt{1+r^2}-r$ is decreasing,
$$\Psi_R(y,z)
 =e^{\tau(R+\langle z\rangle-|z|)}\Phi_{\theta,\tau}(y,z)
 \le e^{\tau\sqrt{1+R^2}}\Phi_{\theta,\tau}(y,z)
 \qquad (|z|\geq R).$$

To prove \eqref{E:common-weighted-resolvent-estimate}, put
$w=(H+\lambda)^{-1}f\in X=\mathscr{V}_0(\mathbb{R}^N)$.  Kato's distributional inequality
\cite[Lemma~A, p.~138]{Kato1972} gives
$$
 (H+\lambda)|w|\le |f|
 \le A_f e^{\tau(\langle z\rangle-|z|)}\Phi_{\theta,\tau}
 \le A_f e^\tau\Phi_{\theta,\tau}.
$$
Here $f,\Delta w\in L^1_{\rm loc}$, as required for this application.
By \eqref{E:common-product-barrier-identity} with $\eta=0$,
$$
 (H+\lambda)\Phi_{\theta,\tau}
 \ge\delta_0\Phi_{\theta,\tau}\quad\hbox{in }\mathbb R^N.
$$
Hence
$$
 w_*:=\frac{A_f e^\tau}{\delta_0}\Phi_{\theta,\tau}\in X,
 \qquad (H+\lambda)w_*\ge |f|,
$$
and consequently
$$
 (H+\lambda)(|w|-w_*)\le0
 \quad\hbox{in }\mathcal D'(\mathbb R^N).
$$
The Sobolev chain rule gives $h=(|w|-w_*)^+\in H^1(\mathbb R^N)$.
Since $0\le h\le |w|$,
$$
 \int_{\mathbb R^N}|y|^2h^2\mathrm{d}x
 \leq\int_{\mathbb R^N}|y|^2|w|^2\mathrm{d}x<\infty.
$$
Thus $h\in X=\mathscr{V}_0(\mathbb R^N)$ is an admissible test in
the form sense.  Testing the difference inequality gives
$$
 0\geq\int_{\mathbb R^N}
 \bigl(|\nabla h|^2+(|y|^2+\lambda)h^2\bigr)\mathrm{d}x
 \geq(d+\lambda)\|h\|_2^2.
$$

Since $d+\lambda>0$, this yields $h=0$ and proves the asserted bound.
For the uniform version, replace $\delta_0$ in the denominator by
$\lambda_0+\theta d-\tau^2$.

For (iii), take $\mathscr{L}=-\Delta+a$, $\Omega=\{|x|>R\}$, and
$\Phi(x)=e^{-\sigma(|x|-R)}$.  By \eqref{eq:20260910-0954},
$$
 \begin{cases}
 \mathscr{L}v\le0\le\mathscr{L}\Phi &\text{in }\Omega,\\
 v\le C=C\Phi &\text{on }\partial\Omega.
 \end{cases}
$$
If $(v-C\Phi)^+\in\mathscr{V}_0(\Omega)$, the energy comparison already
proves \eqref{E:common-Euclidean-barrier}.  The following argument
removes this assumption for bounded $v$.

It remains to justify (i) and (iii) for bounded subsolutions without
assuming $(v-C\Phi)^+\in\mathscr{V}_0(\Omega)$.  If
$(v-C\Phi)^+\notin\mathscr{V}_0(\Omega)$, the preceding global energy
test is unavailable, even though $v\le C\Phi$ on $\partial\Omega$.
We instead construct admissible tests on bounded truncations.
On the comparison domain, $W\ge c>0$: one may take $c=1$ in (i)
after the above choice of $R_y$, and $c=a$ in (iii).
Write $M_v=\|v\|_{L^\infty(\Omega)}$ and choose $0<\gamma<\sqrt c$.
Then
$$B_\gamma(x):=\sum_{j=1}^N\cosh(\gamma x_j),
 \qquad \mathscr{L} B_\gamma=(W-\gamma^2)B_\gamma\geq0,$$
and
$$
 \inf_{|x|=L}B_\gamma(x)
 \ge\cosh\!\left(\frac{\gamma L}{\sqrt N}\right)
 \longrightarrow\infty.
$$
Fix $\varepsilon>0$.  Choose $L_\varepsilon$ larger than the boundary
radius ($R_y$ in (i), $R$ in (iii)) and sufficiently large that
$\varepsilon\cosh(\gamma L/\sqrt N)\ge M_v$ for every
$L\ge L_\varepsilon$. On $\Omega_L:=\Omega\cap B_L$, set
$$
 q_\varepsilon:=v-C\Phi-\varepsilon B_\gamma,
 \qquad h_{\varepsilon,L}:=(q_\varepsilon)^+.
$$
Then
$$
 \mathscr{L}q_\varepsilon
 =\mathscr{L}(v-C\Phi)-\varepsilon\mathscr{L}B_\gamma\le0
 \quad\hbox{in }\mathcal D'(\Omega_L).
$$
The original boundary condition and the choice of $L$ give
$$
 \begin{cases}
 q_\varepsilon\le-\varepsilon B_\gamma\le0
     &\text{on }\partial\Omega\cap B_L,\\
 q_\varepsilon\le M_v-\varepsilon\cosh(\gamma L/\sqrt N)\le0
     &\text{on }\Omega\cap\partial B_L.
 \end{cases}
$$
Thus $h_{\varepsilon,L}\in H_0^1(\Omega_L)$ by local $H^1$ regularity and the zero trace.  Moreover, $0\le h_{\varepsilon,L}\leq v$ and hence
$$\int_{\Omega_L}W^+h_{\varepsilon,L}^2\mathrm{d}x
 \leq\int_{\Omega_L}W^+v^2\mathrm{d}x<\infty.$$
Consequently $h_{\varepsilon,L}\in\mathscr{V}_0(\Omega_L)$ is an admissible test for $\mathscr{L}q_\varepsilon\le0$.  Testing by $h_{\varepsilon,L}$ gives
\begin{align*}
 0&\geq\int_{\Omega_L}
 \bigl(\nabla q_\varepsilon\cdot\nabla h_{\varepsilon,L}
       +Wq_\varepsilon h_{\varepsilon,L}\bigr)\mathrm{d}x\\
 &=\int_{\Omega_L}
 \bigl(|\nabla h_{\varepsilon,L}|^2+Wh_{\varepsilon,L}^2\bigr)
 \mathrm{d}x
 \geq c\|h_{\varepsilon,L}\|_{L^2(\Omega_L)}^2.
\end{align*}
Therefore $h_{\varepsilon,L}=0$ for every sufficiently large $L$. Letting $L\to\infty$, we obtain $v\le C\Phi+\varepsilon B_\gamma$ in $\Omega$; then $\varepsilon\downarrow0$ gives $v\le C\Phi$. This completes the bounded-subsolution comparisons in (i) and (iii) without any global form-domain assumption on $(v-C\Phi)^+$.
\end{proof}

\begin{lemma}[Projected compactness under a singular free scaling]
\label{L:projected-Q-compactness}
Let
$$
 X_y:=\{w\in H^1(\mathbb{R}^d): |y|w\in L^2(\mathbb{R}^d)\}.
$$
Let $\varepsilon_n\to0^+$ and let $\Omega_n\subset\mathbb{R}^\ell$ exhaust
$\mathbb{R}^\ell$.  Suppose that
$$
 U_n\in L^2_{\rm loc}(\Omega_n;X_y)
       \cap H^1_{\rm loc}(\Omega_n;L^2(\mathbb{R}^d))
$$
is a weak solution of
\begin{equation}\label{E:abstract-Q-equation}
 (H_y-d)U_n+\varepsilon_n
 \bigl(-\Delta_\xi U_n+B_n(\xi)\cdot\nabla_\xi U_n+c_n(\xi)U_n\bigr)
 =\varepsilon_n h_n
\end{equation}
on $\mathbb{R}^d\times\Omega_n$, where
$h_n\in L^2_{\rm loc}(\Omega_n;L^2(\mathbb{R}^d))$.
Assume that $B_n\in W^{1,\infty}_{\rm loc}(\Omega_n;\mathbb{R}^\ell)$ and $c_n\in L^\infty_{\rm loc}(\Omega_n)$, with their local norms uniformly bounded in $n$, and that, for every $K\subset\subset\mathbb{R}^\ell$,
\begin{equation}\label{E:abstract-Q-local-bounds}
 \sup_n\int_K\int_{\mathbb{R}^d}(|U_n|^2+|h_n|^2)\,\mathrm{d}y\,\mathrm{d}\xi
 <+\infty.
\end{equation}
If $P$ denotes the $\phi_0$-projection in the $y$-variable and $Q=I-P$, then
\begin{equation}\label{E:abstract-Q-conclusion}
 \int_K\int_{\mathbb{R}^d}|QU_n|^2\,\mathrm{d}y\,\mathrm{d}\xi
 \leq C_K\varepsilon_n
\end{equation}
for every $K\subset\subset\mathbb{R}^\ell$ and all sufficiently large $n$.
\end{lemma}

\begin{proof}
Put $R_n=QU_n$.  Since $Q$ acts only in $y$ and is bounded on $L^2_y$
and $X_y$, it commutes with the free-variable terms, and $R_n$ satisfies
the projected equation.  Fix $K\subset\subset\mathbb R^\ell$ and choose
$\eta\in C_c^\infty(\mathbb R^\ell)$ equal to one on $K$, with
$K_1:=\operatorname{supp}\eta\subset\subset\Omega_n$ for large $n$.

The test $\eta^2R_n$ is legitimate in the local form space
$L^2(\Omega_n;X_y)\cap H_0^1(\Omega_n;L^2_y)$.
Indeed, truncate in $y$ and then mollify in both variables, keeping the
$\xi$-support in a fixed compact subset of $\Omega_n$.  This approximates
$\eta^2R_n$ by compactly supported smooth functions in both the
$L^2_\xi(X_y)$ and $H^1_\xi(L^2_y)$ norms; the weighted $y$-term converges
by truncation.  All terms in the weak equation are continuous in these
norms.  Thus the smooth tests pass to $\eta^2R_n$.

All integrals below are over $\mathbb R^d\times\Omega_n$.  Testing gives
\begin{align}
 &\int\eta^2\bigl(|\nabla_yR_n|^2+(|y|^2-d)R_n^2\bigr)
 +\varepsilon_n\int\eta^2|\nabla_\xi R_n|^2 \notag\\
 &\quad+2\varepsilon_n\int\eta R_n\nabla_\xi R_n\cdot\nabla\eta
 +\varepsilon_n\int\eta^2(B_n\cdot\nabla_\xi R_n)R_n
 +\varepsilon_n\int c_n\eta^2R_n^2
 =\varepsilon_n\int\eta^2(Qh_n)R_n.
 \label{E:abstract-Q-tested-identity}
\end{align}
The oscillator gap bounds the first term below by $2\int\eta^2R_n^2$.
For the cutoff and drift terms, respectively,
\begin{align*}
 2\left|\int\eta R_n\nabla_\xi R_n\cdot\nabla\eta\right|
 &\le\tfrac12\int\eta^2|\nabla_\xi R_n|^2
       +2\int|\nabla\eta|^2R_n^2,\\
 \int\eta^2(B_n\cdot\nabla_\xi R_n)R_n
 &=-\tfrac12\int\operatorname{div}_\xi(\eta^2B_n)R_n^2.
\end{align*}
Young's inequality bounds the source by
$\frac12\int\eta^2R_n^2+\frac12\varepsilon_n^2\int\eta^2|Qh_n|^2$.
The local bounds on $B_n,c_n$ therefore yield
$$
 \frac32\int\eta^2R_n^2
 +\frac{\varepsilon_n}{2}\int\eta^2|\nabla_\xi R_n|^2
 \le C_\eta\varepsilon_n\int_{\mathbb R^d\times K_1}|R_n|^2
 +\frac{\varepsilon_n^2}{2}\int_{\mathbb R^d\times K_1}|Qh_n|^2.
$$
Since $Q$ is an $L^2_y$ contraction, \eqref{E:abstract-Q-local-bounds}
bounds the right-hand side by $C_\eta\varepsilon_n$.  As $\eta=1$ on $K$,
this proves \eqref{E:abstract-Q-conclusion}.
\end{proof}

\begin{lemma}[Envelope upgrade from local to global convergence]
\label{L:envelope-global-compactness}
Let $a_1>a\ge0$ and $b_1>b\ge0$.  Suppose that $w_n\to w$ locally in measure on $\mathbb{R}^d\times\mathbb{R}^m$ and
$$
 |w_n(y,z)|+|w(y,z)|
 \leq C e^{-a_1|y|^2-b_1|z|}.
$$
Then $w_n\to w$ in $L^q(\mathbb{R}^N)$ for every $1\le q<+\infty$.  If the local convergence is uniform, then also
$$
 \|e^{a|y|^2+b|z|}(w_n-w)\|_\infty\to0.
$$
If $w_n\to w$ also in $X$, then the convergence holds in the corresponding weighted $X$--supremum norm.
\end{lemma}

\begin{proof}
The common envelope and local convergence in measure give the $L^q$ conclusion
by dominated convergence, first on a fixed ball and then globally.  Under local
uniform convergence, the assertion holds uniformly on that ball; outside it,
multiplication by the weaker weight leaves the uniform tail
$2C e^{-(a_1-a)|y|^2-(b_1-b)|z|}$.  This proves the weighted supremum
convergence.
\end{proof}

\begin{lemma}[Monotone-cylinder Hamiltonian exclusion]
\label{L:common-cylinder-Hamiltonian}
Let $k\ge0$, with the convention that $\mathbb{R}^0$ consists of one point. Let $\mathcal V\in C(\mathbb{R}^k;\mathbb{R})$ and
$f\in C([0,+\infty))$.  Suppose that $V\ge0$ is a bounded classical solution of
\begin{equation}\label{E:common-cylinder-equation}
 -\Delta_YV-V_{\tau\tau}+\mathcal V(Y)V=f(V)
 \qquad\hbox{in }\mathbb{R}^k\times\mathbb{R},
\end{equation}
and $V_\tau\le0$.  Assume that its endpoint limits $V_\pm$ exist and that
$$
 V(\cdot,\tau+\cdot)\longrightarrow V_\pm(\cdot)
 \quad\hbox{in }C^1_{\rm loc}(\mathbb{R}^k\times\mathbb{R})
 \quad\hbox{as }\tau\to\pm\infty.
$$
Suppose, moreover, that there is $\Gamma\in L^2(\mathbb{R}^k)\cap L^\infty(\mathbb{R}^k)$ such that, for all $(Y,\tau)$,
\begin{equation}\label{E:common-cylinder-envelope}
 |V(Y,\tau)|+|\nabla_YV(Y,\tau)|+|V_\tau(Y,\tau)|\le\Gamma(Y),
\end{equation}
and
\begin{equation}\label{E:common-cylinder-integrability}
 |\mathcal V|\Gamma^2\in L^1(\mathbb{R}^k),
 \qquad
 Y\longmapsto\sup_{0\le s\le\Gamma(Y)}|F(s)|
       \in L^1(\mathbb{R}^k),
 \qquad F(s):=\int_0^sf(t)\,\mathrm dt.
\end{equation}
Define
$$
 J(w):=\frac{1}{2}\int_{\mathbb{R}^k}
 \bigl(|\nabla_Yw|^2+\mathcal V(Y)w^2\bigr)\mathrm{d}Y
 -\int_{\mathbb{R}^k}F(w)\mathrm{d}Y.
$$
If $V_+\equiv0$, $V_-\not\equiv0$, and $J(V_-)>0$, then no such $V$ exists.
\end{lemma}

\begin{proof}
The envelope assumptions make
$$
 \mathscr{H}(\tau):=\tfrac12\|V_\tau(\cdot,\tau)\|_2^2-J(V(\cdot,\tau))
$$
finite.  For $k=0$, multiplication of the ODE by $V_\tau$ makes
$\mathscr{H}$ constant.  For $k\ge1$, choose a smooth cutoff $\rho_R$
equal to one on $B_R^k$ and zero outside $B_{2R}^k$, with
$|\nabla\rho_R|\le C/R$, and define
$$
 \mathscr{H}_R(\tau)
 :=\int\rho_R^2\left[\tfrac12|V_\tau|^2
 -\tfrac12|\nabla_YV|^2-\tfrac12\mathcal V V^2+F(V)\right].
$$
Testing the equation by $\rho_R^2V_\tau$ gives
\begin{equation}\label{E:common-cylinder-truncated-Hamiltonian}
 \mathscr{H}_R'(\tau)
 =2\int\rho_RV_\tau\nabla_YV\cdot\nabla\rho_R.
\end{equation}
The error is supported in $A_R=B_{2R}^k\setminus B_R^k$, so
\begin{equation}\label{E:common-cylinder-cutoff-error}
 |\mathscr{H}_R'(\tau)|\le\frac CR\int_{A_R}\Gamma^2.
\end{equation}
Integrating over a fixed interval $[\tau_1,\tau_2]$ and letting $R\to\infty$
shows that $\mathscr{H}(\tau_2)=\mathscr{H}(\tau_1)$; the passage
$\mathscr{H}_R\to\mathscr{H}$ follows from
\eqref{E:common-cylinder-envelope}--\eqref{E:common-cylinder-integrability}.
Finally, the assumed $C^1_{\rm loc}$ convergence of the translates
$V(Y,\tau+s)$ includes
$$V_\tau(Y,\tau)\to0,\qquad (V,\nabla_YV)(Y,\tau)\to(V_\pm,\nabla_YV_\pm)(Y),$$
since $V_\pm$ is independent of $s$.  Dominated convergence, with the
majorants $\Gamma^2$, $|\mathcal{V}|\Gamma^2$, and $\sup_{0\le s\le\Gamma}|F(s)|$, gives
$$\|V_\tau(\cdot,\tau)\|_2\longrightarrow0,\qquad J(V(\cdot,\tau))\longrightarrow J(V_\pm)
 \quad\hbox{as }\tau\to\pm\infty.$$
Constancy of $\mathscr{H}$ forces $J(V_-)=J(V_+)=0$, contradicting $J(V_-)>0$.
\end{proof}

\section{Uniform localization and compactness on frequency strips}\label{sec:priori-estimates}

\subsection{Uniform amplitude bounds}
\begin{proposition}[Uniform $L^\infty$ estimate]\label{prop:L-infty}
Under assumptions \ref{gcon1} and \ref{gcon2}, for any $M>-d$, there exists
$C_M>0$ such that every centered positive solution $u\in\mathcal{K}$ of
\eqref{E:fixed-lambda} with $\lambda\in(-d,M]$ satisfies
$$\|u\|_\infty=u(0,0)\le C_M.$$
\end{proposition}
\begin{proof}
Suppose, for contradiction, that there are $\lambda_n\in(-d,M]$ and positive
solutions $u_n\in\mathcal K$ such that
$M_n:=\|u_n\|_\infty\to+\infty$.  By \textbf{Theorem~A} and our centering convention, each
$u_n$ is nonincreasing in $|y|$ and $|z|$; hence its maximum is attained at the
origin.  Set
$$r_n:=M_n^{-\frac{\beta-2}{2}}, \quad v_n(\xi):=M_n^{-1} u_n\left(r_n \xi\right), \quad \xi=(Y, Z) \in \mathbb{R}^d\times \mathbb{R}^m .$$
Then  $0 \leq v_n \leq 1, v_n(0)=1$ and
\begin{equation}\label{eq:20260616-1316}
-\Delta v_n=-r_n^4|Y|^2 v_n-r_n^2 \lambda_n v_n+\frac{g\left(M_n v_n\right)}{M_n^{\beta-1}}~\hbox{in}~\mathbb{R}^N.
\end{equation}
Under the assumptions \ref{gcon1} and \ref{gcon2}, the right-hand side of \eqref{eq:20260616-1316} is locally bounded in $\mathbb{R}^N$.  Moreover, Lemma \ref{lemma:20260618-1129} gives
$$\frac{g(M_ns)}{M_n^{\beta-1}}=F_{M_n}^{(\infty)}(s)\to \mu_\infty s^{\beta-1}\qquad\hbox{uniformly for }s\in[0,1].$$
Since $r_n^2\lambda_n\to0$ and $r_n^4|Y|^2\to0$ locally uniformly,
interior $W^{2,p}$ estimates first give, for $p>N$, compactness in
$C^{1,\gamma}_{\rm loc}$.  In addition,
Lemma~\ref{lemma:20260618-1129} gives uniform convergence of both $F_{M_n}^{(\infty)}$
and $\big(F_{M_n}^{(\infty)}\big)'$ on $[0,1]$.  Hence the nonlinear right-hand sides converge in
$C^{0,\gamma'}_{\rm loc}$ for every $\gamma'<\gamma$, and the interior Schauder
estimate upgrades the convergence, after extraction, to
$v_n\to v$ in $C^{2,\gamma'}_{\rm loc}(\mathbb{R}^N)$.  Here $0\le v\le1$, $v(0)=1$, and
\begin{equation}\label{eq:20260616-1323}
-\Delta v=\mu_\infty v^{\beta-1} \quad \hbox{in}~\mathbb{R}^N.
\end{equation}
No such limit exists in the stated exponent range: for $N=1,2$ this follows
from \cite[Theorem~8.4]{QuittnerSouplet}, while for $N\ge3$ it follows from the
Gidas--Spruck theorem \cite[Theorem~1.1]{GidasSpruck}, because
$1<\beta-1<2_N^*-1$.  This contradiction proves the asserted estimate.
\end{proof}

\subsection{Gaussian localization in the confined variables}
For any $-d<\Lambda_1\leq \Lambda_2<+\infty$, define the solution set
\begin{equation}\label{eq:20260616-1333}
\mathcal{U}_{\Lambda_1}^{\Lambda_2}:=\{u:(\lambda,u)\in\cS,
        \ \lambda\in[\Lambda_1,\Lambda_2]\}.
\end{equation}
For convenience, we also adopt the notation $\mathcal{U}_{-d}^{\Lambda_2}$
for the family corresponding to $\lambda\in(-d,\Lambda_2]$.

By Proposition \ref{prop:L-infty}, for
$-d\leq\Lambda_1\leq\Lambda_2<+\infty$ with $\Lambda_2>-d$, there exists
$C_{\Lambda_2}>0$ such that
\begin{equation}\label{eq:20260616-1352}
\max_{x\in \mathbb{R}^N}u(x)\leq C_{\Lambda_2},\quad \forall u\in \mathcal{U}_{\Lambda_1}^{\Lambda_2}.
\end{equation}
Under assumptions \ref{gcon1} and \ref{gcon2}, the quotient $g(s)/s$ extends
continuously to $s=0$ with value $0$.  Hence the same uniform $L^\infty$ bound gives
\begin{equation}\label{eq:20260616-1354}
K_{\Lambda_2}:=
\sup_{u\in \mathcal{U}_{-d}^{\Lambda_2}}
\sup_{x\in \mathbb{R}^N}\frac{g(u(x))}{u(x)}<+\infty,
\end{equation}
where the quotient is understood as $0$ at points at which $u=0$.  In particular,
$K_{\Lambda_2}$ depends on the upper frequency bound $\Lambda_2$, but not on a positive
distance of $\lambda$ from $-d$.

\begin{proposition}[Uniform sub-Gaussian decay in the confined variables]
\label{prop:subGaussian-y-decay}
Let $-d<\Lambda_2<+\infty$ and define $\mathcal{U}_{-d}^{\Lambda_2}$ as above. Under the assumptions \ref{gcon1} and \ref{gcon2},
for every $\theta\in(0,1)$, there exists a constant $C_\theta>0$, depending only on
$\theta, \Lambda_2$ and $d$, such that
$$
u(y,z)\le C_\theta e^{-\frac{\theta}{2}|y|^2},
\qquad (y,z)\in\mathbb{R}^d\times\mathbb{R}^m,\quad u\in \mathcal{U}_{-d}^{\Lambda_2} .
$$
Equivalently,
$$\sup_{u\in \mathcal{U}_{-d}^{\Lambda_2}}\sup_{(y,z)\in\mathbb{R}^N}
e^{\frac{\theta}{2}|y|^2}u(y,z)<+\infty
\qquad\hbox{for every }\theta\in(0,1).$$
In particular, for every fixed $R>0$,
$$u(y,z)\leq u(y,0)\leq C_\theta e^{-\frac{\theta}{2}|y|^2}
\qquad\hbox{uniformly for } ~u\in \mathcal{U}_{-d}^{\Lambda_2}~\hbox{and}~ |z|=R .$$
\end{proposition}

\begin{proof}
For $u\in\mathcal U_{-d}^{\Lambda_2}$, the uniform amplitude bound and
\eqref{eq:20260616-1354} give
$$
 (H+\lambda-K_{\Lambda_2})u\le0.
$$
Apply Lemma~\ref{L:common-exterior-barrier}(i), with
$M=C_{\Lambda_2}$.  Its comparison radius and constant are uniform for
$\lambda\in(-d,\Lambda_2]$, and \eqref{E:common-Gaussian-conclusion} gives the
asserted estimate on all of $\mathbb{R}^N$.
\end{proof}

\subsection{Uniform tails and exponential localization in the free variables}
We first prove uniform decay of the family $\mathcal U_{\Lambda_1}^{\Lambda_2}$ as $|z|\to+\infty$, where $-d<\Lambda_1\le\Lambda_2<+\infty$.  Recall that $\mathcal U_{\Lambda_1}^{\Lambda_2}\subset\mathcal U_{-d}^{\Lambda_2}$.
\begin{proposition}[Uniform decay to zero]
\label{prop:uniform-z-tail}
For any $-d<\Lambda_1\leq \Lambda_2<+\infty$, under the assumptions
\ref{gcon1}, \ref{gcon2}, and \ref{gcon4}, we have
\begin{equation}\label{eq:20260616-1939}
\lim_{R\rightarrow+\infty}
\sup_{u\in\mathcal{U}_{\Lambda_1}^{\Lambda_2}}\sup_{\substack{y\in\mathbb{R}^d\\ |z|\ge R}}
u(y,z)=\lim_{R\rightarrow+\infty}
\sup_{\substack{u\in\mathcal{U}_{\Lambda_1}^{\Lambda_2}\\ |z|=R}}u(0,z)=0.
\end{equation}
\end{proposition}

\begin{proof}
Write $u(y,z)=U(|y|,|z|)$.  \textbf{Theorem~A} gives
$$\sup_{y\in\mathbb{R}^d,\,|z|\geq R}u(y,z)=U(0,R),$$
so it suffices to control $U(0,R)$.  If the assertion failed, there would
exist $u_n\in\mathcal U_{\Lambda_1}^{\Lambda_2}$, with parameters
$\lambda_n\to\lambda$, and $R_n\to\infty$ such that
$U_n(0,R_n)\ge\varepsilon_0>0$.  Put $\mu_1=d+\Lambda_1$ and choose
$0<\varepsilon<\varepsilon_0$ so that
$$
 g(s)\le\frac{\mu_1}{2}s\qquad(0\le s\le\varepsilon).
$$
By Lemma~\ref{L:individual-decay}, there is a first $r_n\ge R_n$ with
$$
 U_n(0,r_n)=\varepsilon,\qquad
 U_n(0,t)\le\varepsilon\quad(t\ge r_n).
$$
Define $V_n(y,\tau)=U_n(|y|,r_n+\tau)$.  Then
$$
 -\Delta_yV_n-V_{n,\tau\tau}
 -\frac{m-1}{r_n+\tau}V_{n,\tau}
 +|y|^2V_n+\lambda_nV_n=g(V_n).
$$
The uniform amplitude bound and interior estimates give, after extraction,
$V_n\to V$ in $C^2_{\rm loc}$, where
\begin{equation}\label{eq:20260616-1642}
 -\Delta_yV-V_{\tau\tau}+|y|^2V+\lambda V=g(V).
\end{equation}
Moreover,
$$
 V(0,0)=\varepsilon,\qquad V_\tau\le0,\qquad
 0\le V(y,\tau)\le\varepsilon\quad(\tau\ge0),
$$
and Proposition~\ref{prop:subGaussian-y-decay} gives
\begin{equation}\label{eq:20260616-1644}
 0\le V(y,\tau)\le C_\theta e^{-\theta|y|^2/2}.
\end{equation}

The endpoint limits $V_\pm(y)=\lim_{\tau\to\pm\infty}V(y,\tau)$ exist.
Translation in $\tau$, followed by local elliptic compactness, shows that
both belong to $X_y$ and solve
$$-\Delta_yV_\pm+|y|^2V_\pm+\lambda V_\pm=g(V_\pm)~\hbox{in}~\mathbb{R}^d. $$
Since $0\leq V_+\leq \varepsilon$, testing this equation and using
$\inf\sigma(H_y)=d$ gives
$$\mu_1\|V_+\|_2^2
 \leq\int_{\mathbb{R}^d} g(V_+)V_+ \mathrm{d}y
 \leq\frac{\mu_1}{2}\|V_+\|_2^2,
$$
hence $V_+\equiv 0$.

Set
$$J_\lambda(w):=\frac{1}{2}\int_{\mathbb{R}^d}
 (|\nabla_yw|^2+|y|^2w^2+\lambda w^2)\mathrm{d}y
 -\int_{\mathbb{R}^d}G(w)\mathrm{d}y.$$
The Gaussian bound and the interior gradient estimate give, for
$0<\theta_0<\theta$,
$$
 |\nabla_{y,\tau}V(y,\tau)|
 \le C(1+|y|^2)e^{-\theta_0|y|^2/2},
$$
Because $g'(0)=0$ and $V$ is uniformly bounded, $|G(V)|\le CV^2$.
Thus one may take
$\Gamma(y)=C(1+|y|^2)e^{-\theta_0|y|^2/2}$ in
\eqref{E:common-cylinder-envelope}; then
$\bigl||y|^2+\lambda\bigr|\Gamma^2$ and
$y\mapsto\sup_{0\le s\le\Gamma(y)}|G(s)|$ are integrable.  Hence the precise
hypotheses of Lemma~\ref{L:common-cylinder-Hamiltonian} hold with
$\mathcal V(y)=|y|^2+\lambda$ and $f=g$.  Since
$V_-(0)\ge\varepsilon$, the stationary equation and
\eqref{E:positive-Hamiltonian-density} yield
$$
 J_\lambda(V_-)=\int_{\mathbb{R}^d}
 \left(\frac{1}{2}g(V_-)V_--G(V_-)\right)>0.
$$
This contradicts Lemma~\ref{L:common-cylinder-Hamiltonian} and proves
\eqref{eq:20260616-1939}.
\end{proof}

Proposition~\ref{prop:uniform-z-tail} yields uniform exponential decay.
\begin{proposition}[Uniform exponential decay in the free variables]
\label{prop:uniform-z-exponential-decay}
Suppose that the assumptions of Proposition~\ref{prop:uniform-z-tail} hold.  Let
$\mu:=d+\Lambda_1>0$ and choose $\theta\in(0,1)$ with
$\Lambda_1+\theta d>0$.  For every
$\sigma\in(0,\sqrt{\Lambda_1+\theta d})$, there exist $R_0,C>0$ such that
$$u(y,z)\leq C e^{-\frac{\theta}{2}|y|^2}e^{-\sigma(|z|-R_0)}
\qquad\hbox{for}~ |z|\geq R_0,\quad u\in\mathcal{U}_{\Lambda_1}^{\Lambda_2}.
$$
 In particular,
\begin{equation}\label{eq:20260616-1731-global}
u(y,z)\le C e^{-\frac{\theta}{2}|y|^2}e^{-\sigma |z|}
\qquad\hbox{for all}~(y,z)\in\mathbb{R}^N,\quad u\in \mathcal{U}_{\Lambda_1}^{\Lambda_2},
\end{equation}
after changing the constant $C>0$.
\end{proposition}

\begin{proof}
Choose $\eta>0$ so that
$$
 \Lambda_1+\theta d-\eta-\sigma^2>0,
 \qquad d+\Lambda_1-\eta>0.
$$
Since $g(s)=o(s)$ at zero, Proposition~\ref{prop:uniform-z-tail} gives
$R_0>0$, uniformly over the strip, such that
$$
 (H+\lambda-\eta)u\le0\qquad\hbox{on }\{|z|>R_0\}.
$$
Proposition~\ref{prop:subGaussian-y-decay} supplies the Gaussian boundary
bound on $|z|=R_0$.  Lemma~\ref{L:common-exterior-barrier}(ii), applied with
the radial barrier $e^{-\theta|y|^2/2-\sigma(|z|-R_0)}$, now gives the
exterior estimate.  Increasing the constant on $|z|\le R_0$ proves
\eqref{eq:20260616-1731-global}.
\end{proof}

\subsection{The unified weighted fixed-point setting}
\label{sec:fixed-point-framework}
The product decay estimate places every centered positive solution
on a compact frequency strip in the weighted space announced in
Section~\ref{sec:assumption-results}.  We now establish the mapping
properties needed for degree and continuation.

For $s\in\mathbb{R}$, set
\begin{equation}\label{E:positive-part-extension}
 s^+:=\max\{s,0\},\qquad g_+(s):=g(s^+).
\end{equation}
Since $g(0)=g'(0)=0$, the function $g_+$ belongs to $C^1(\mathbb{R})$.

Fix $\lambda_0>-d$.  Choose $\theta\in(0,1)$ such that
$\lambda_0+\theta d>0$, and take
$$
 0<\tau<\sqrt{\lambda_0+\theta d}.
$$
Define the Gaussian--exponentially weighted centered space
\begin{equation}\label{E:weighted-space}
\begin{aligned}
\mathcal{X}_{\theta,\tau}
&:=\left\{u\in X_r\cap C_0(\mathbb{R}^N):
 e^{\frac{\theta}{2}|y|^2+\tau|z|}u\in L^\infty(\mathbb{R}^N)\right\},\\
\|u\|_{\mathcal X_{\theta,\tau}}
&:=\|u\|_X+
 \left\|e^{\frac{\theta}{2}|y|^2+\tau|z|}u\right\|_\infty .
\end{aligned}
\end{equation}
Because $X_r$ consists of functions radial about $y=0$ and $z=0$, the centering is built into the ambient space.  The space $\mathcal X_{\theta,\tau}$ is Banach: a Cauchy sequence converges in $X$ and in the weighted supremum norm, and the two limits agree locally.
Define
\begin{equation}\label{E:positive-part-map-weighted}
 T_\lambda^+(u):=(H+\lambda)^{-1}g_+(u).
\end{equation}
Here $(H+\lambda)^{-1}:X'\to X$ is the form resolvent from
Lemma~\ref{L:form-resolvent}; for $u\in\mathcal X_{\theta,\tau}$ the source $g_+(u)$ belongs in fact to $L^2(\mathbb{R}^N)\subset X'$.

If, in addition, \ref{gcon4} holds, then the choice
\eqref{E:weighted-space} does not discard any centered positive solution on a compact frequency strip.  Indeed, Proposition \ref{prop:uniform-z-exponential-decay}, applied with
$\Lambda_1=\lambda_0$ and $\sigma=\tau$, gives, for every $\Lambda>\lambda_0$,
\begin{equation}\label{E:no-loss-weighted-space}
 \sup_{\substack{(\lambda,u)\in\cS\\
 \lambda\in[\lambda_0,\Lambda]}}
 \left\|e^{\frac{\theta}{2}|y|^2+\tau|z|}u\right\|_\infty<+\infty.
\end{equation}
Together with the uniform $X$-bound proved in Corollary \ref{cor:20260616-2057} below, this shows that the whole centered
positive solution set on $[\lambda_0,\Lambda]$ is bounded in the fixed weighted norm. By \textbf{Theorem~A} every positive finite-energy solution can be translated in the free variables to a centered one, so no positive solution is lost modulo the unavoidable translation
symmetry.

\begin{lemma}[Uniform weighted compactness for all $m\ge1$]
\label{L:weighted-compact-extension}
Assume \ref{gcon1} and \ref{gcon2}.  Let $m\ge1$, and let
$\lambda_0,\theta,\tau$ be as above.  Then $T_\lambda^+$ is a completely continuous map on bounded subsets of $\mathcal X_{\theta,\tau}$, uniformly for $\lambda$ in every compact
interval $[\lambda_0,\Lambda]$.  For each fixed $\lambda$ it is $C^1$ on $\mathcal{X}_{\theta,\tau}$, with
$$
 DT_\lambda^+(u)h=(H+\lambda)^{-1}\bigl(g_+'(u)h\bigr).
$$
Every nonzero fixed point is a centered positive solution of
\eqref{E:fixed-lambda} and hence belongs to $\mathcal{K}$.  If, in addition, \ref{gcon4} holds, then all centered positive solutions in the strip belong to $\mathcal{X}_{\theta,\tau}$ and are therefore captured by this fixed-point formulation.
\end{lemma}

\begin{proof}
It is convenient to separate the proof into the weighted resolvent estimate, compactness, and $C^1$-regularity.

\smallskip
\noindent\emph{Step 1: a uniform weighted resolvent estimate.}
Let $B\subset\mathcal X_{\theta,\tau}$ be bounded.  Choose
\begin{equation}\label{E:weighted-intermediate-parameters}
\begin{gathered}
 \theta<\theta_1<\min\{1,(\alpha-1)\theta\},\qquad
 \tau<\tau_1<
 \min\left\{(\alpha-1)\tau,\sqrt{\lambda_0+\theta_1d}\right\}.
\end{gathered}
\end{equation}
Set
$$
 A(y,z):=\frac{\theta}{2}|y|^2+\tau|z|,
 \qquad
 A_1(y,z):=\frac{\theta_1}{2}|y|^2+\tau_1|z|,
$$
and
$$
 \mathcal{F}_{\theta_1,\tau_1}
:=\{f\in L^2:f\text{ is block-radial},\ e^{A_1}f\in L^\infty\},\qquad
 \|f\|_{\mathcal{F}_{\theta_1,\tau_1}}
 :=\|f\|_2+\|e^{A_1}f\|_\infty .
$$
We claim that, uniformly for $\lambda\in[\lambda_0,\Lambda]$,
\begin{equation}\label{E:weighted-linear-resolvent}
 \|(H+\lambda)^{-1}f\|_{\mathcal X_{\theta,\tau}}
 \le C\|f\|_{\mathcal{F}_{\theta_1,\tau_1}}.
\end{equation}
Indeed, Lemma~\ref{L:common-exterior-barrier}(ii), with
$(\theta,\tau)$ there replaced by $(\theta_1,\tau_1)$, and Kato's
inequality give
\begin{equation}\label{E:weighted-linear-pointwise}
 |(H+\lambda)^{-1}f|
 \le C\|e^{A_1}f\|_\infty
 e^{-\theta_1|y|^2/2-\tau_1\langle z\rangle}.
\end{equation}
The form-resolvent bound of Lemma~\ref{L:form-resolvent} gives the $X$-part
of \eqref{E:weighted-linear-resolvent}; multiplying
\eqref{E:weighted-linear-pointwise} by $e^A$ gives its weighted supremum part.

\smallskip
\noindent\emph{Step 2: invariance of bounded sets and complete continuity.} Since $B$ is bounded in $\mathcal X_{\theta,\tau}$, there exists $M_B>0$ such that
$$
 |u(y,z)|\le M_Be^{-A(y,z)},\qquad u\in B.
$$
By \ref{gcon2} and the definition of $g_+$, after increasing a constant depending only on $B$,
$$
 |g_+(s)|\le C_B|s|^{\alpha-1}
 \qquad\hbox{for }|s|\le M_B.
$$
Hence \eqref{E:weighted-intermediate-parameters} implies
\begin{equation}\label{E:double-weight-source}
 0\leq g_+(u(y,z))
 \leq C_Be^{-(\alpha-1)A(y,z)}
 \leq C_Be^{-A_1(y,z)},
 \qquad u\in B.
\end{equation}
In particular, $g_+(B)$ is bounded in $\mathcal{F}_{\theta_1,\tau_1}$. Applying \eqref{E:weighted-linear-resolvent} gives
\begin{equation}\label{E:weighted-resolvent-self-map-bound}
 \sup_{\substack{u\in B\\ \lambda\in[\lambda_0,\Lambda]}}
 \|T_\lambda^+(u)\|_{\mathcal{X}_{\theta,\tau}}<+\infty.
\end{equation}
Moreover, because the source is nonnegative, positivity of the resolvent together with \eqref{E:weighted-linear-pointwise} gives the strict gain
\begin{equation}\label{E:weighted-resolvent-gain}
 0\leq T_\lambda^+(u)(y,z)
 \leq C_B
 e^{-\frac{\theta_1}{2}|y|^2-\tau_1\langle z\rangle},
 \qquad
 u\in B,\quad\lambda\in[\lambda_0,\Lambda].
\end{equation}

We now prove relative compactness.  Let $u_n\in B$ and
$\lambda_n\in[\lambda_0,\Lambda]$.  Passing to a subsequence, we may assume
$$
 \lambda_n\to\lambda,\qquad u_n\rightharpoonup u\quad\hbox{in }X,\qquad u_n\to u\quad\hbox{a.e. in }\mathbb{R}^N.
$$
The common exponential envelope implies, by dominated convergence,
\begin{equation}\label{E:weighted-input-L2-compact}
 u_n\to u\qquad\hbox{strongly in }L^q(\mathbb{R}^N)
 \quad\hbox{for every }2\le q<+\infty.
\end{equation}
All $u_n$ and $u$ take values in one bounded interval.  Since $g_+$ is $C^1$, it is Lipschitz on that interval; therefore \eqref{E:weighted-input-L2-compact} gives
$$
 g_+(u_n)\to g_+(u)\qquad\hbox{strongly in }L^2(\mathbb{R}^N),
$$
and hence in $X'$.  If
$$
 v_n:=T_{\lambda_n}^+(u_n),\qquad v:=(H+\lambda)^{-1}g_+(u),
$$
then Lemma~\ref{L:form-resolvent} yields
$$
 (H+\lambda_n)^{-1}
 \longrightarrow
 (H+\lambda)^{-1}
 \qquad\hbox{in }\mathcal L(X',X).
$$
Together with $g_+(u_n)\to g_+(u)$ in $X'$, this gives
\begin{equation}\label{E:weighted-output-X-compact}
 v_n\to v\qquad\hbox{strongly in }X.
\end{equation}
The equations for $v_n$, the uniform bound
\eqref{E:weighted-resolvent-gain}, and interior $W^{2,p}$ estimates with
$p>N$ give $v_n\to v$ locally uniformly.  The same bound supplies the
stronger common envelope with exponents $(\theta_1/2,\tau_1)$.  Hence
Lemma~\ref{L:envelope-global-compactness}, together with
\eqref{E:weighted-output-X-compact}, gives
$$
 \|e^A(v_n-v)\|_\infty\to0.
$$
Together with \eqref{E:weighted-output-X-compact}, this proves relative compactness in $\mathcal{X}_{\theta,\tau}$.  The same argument, without extracting subsequences, proves continuity with respect to both $u$ and $\lambda$ on bounded sets.  Thus $(\lambda,u)\mapsto T_\lambda^+(u)$ is completely continuous on
$[\lambda_0,\Lambda]\times\mathcal X_{\theta,\tau}$.

\smallskip
\noindent\emph{Step 3: $C^1$-regularity.}
We prove the slightly stronger statement that the Nemytskii map
$$
 \mathcal{N}(u):=g_+(u)
$$
is $C^1$ from $\mathcal{X}_{\theta,\tau}$ into
$\mathcal{F}_{\theta_1,\tau_1}$ on every bounded subset for which
\eqref{E:weighted-intermediate-parameters} is fixed.  Its candidate derivative is
$$
 D\mathcal{N}(u)h=g_+'(u)h.
$$
Indeed, by \ref{gcon2}, for each $M>0$ there is $C_M>0$ such that
\begin{equation}\label{E:gprime-local-power-bound-weighted}
 |g_+'(s)|\leq C_M|s|^{\alpha-2},
 \qquad |s|\leq M,
\end{equation}
where $g_+'(s)=0$ for $s<0$ and $g_+'(0)=0$. Thus, if $u$ ranges in a bounded subset of $\mathcal X_{\theta,\tau}$ and $\|h\|_{\mathcal X_{\theta,\tau}}\leq 1$, then
$$
 |g_+'(u)h| \leq C e^{-(\alpha-1)A},
$$
so $D\mathcal{N}(u)$ is a bounded operator from $\mathcal{X}_{\theta,\tau}$ to $\mathcal{F}_{\theta_1,\tau_1}$.

It remains to prove continuity of the derivative in operator norm.  Suppose $u_n\to u$ in $\mathcal{X}_{\theta,\tau}$ and set
$$
 \delta_n:=\|D\mathcal N(u_n)-D\mathcal N(u)\|_
 {\mathcal L(\mathcal{X}_{\theta,\tau},\mathcal{F}_{\theta_1,\tau_1})}.
$$
Because convergence in $\mathcal{X}_{\theta,\tau}$ implies uniform convergence and the
functions $u_n,u$ remain in one bounded interval, uniform continuity of $g_+'$ gives
$$
 \|g_+'(u_n)-g_+'(u)\|_\infty\to0.
$$
Consequently, for $\|h\|_{\mathcal{X}_{\theta,\tau}}\le1$,
$$
 \|(g_+'(u_n)-g_+'(u))h\|_2
 \le \|g_+'(u_n)-g_+'(u)\|_\infty\|h\|_2\to0
$$
uniformly in $h$.

For the weighted $L^\infty$ part, fix $R>0$ and split the space into
$\{A\le R\}$ and $\{A>R\}$.  On $\{A\le R\}$ the factor $e^{A_1-A}$ is bounded, and
uniform convergence of $g_+'(u_n)$ to $g_+'(u)$ gives
$$
 \sup_{\|h\|_{\mathcal{X}_{\theta,\tau}}\le1}
 \sup_{\{A\le R\}}
 e^{A_1}|(g_+'(u_n)-g_+'(u))h|\to0.
$$
On $\{A>R\}$, the common weighted bound for $u_n$ and $u$, together with
\eqref{E:gprime-local-power-bound-weighted}, gives
$$
 \begin{aligned}
 e^{A_1}|(g_+'(u_n)-g_+'(u))h|
 &\le
 C e^{A_1-(\alpha-1)A}\\
 &=
 C\exp\left(
 -\frac{(\alpha-1)\theta-\theta_1}{2}|y|^2
 -\bigl((\alpha-1)\tau-\tau_1\bigr)|z|
 \right).
 \end{aligned}
$$
By \eqref{E:weighted-intermediate-parameters}, the last expression tends uniformly to
zero as $A(y,z)\to+\infty$, independently of $n$ and $h$.  First choosing $R$ large and
then $n$ large proves $\delta_n\to0$.  Hence $D\mathcal N$ is continuous in operator
norm.  The integral remainder formula
$$
 \mathcal N(u+k)-\mathcal N(u)-D\mathcal N(u)k
 =
 \int_0^1\bigl[D\mathcal N(u+tk)-D\mathcal N(u)\bigr]k\,\mathrm dt
$$
then proves Fr\'echet differentiability and shows that
$\mathcal N\in C^1$.

Combining this with the bounded linear estimate
\eqref{E:weighted-linear-resolvent}, we conclude that, for each fixed
$\lambda\in[\lambda_0,\Lambda]$,
$$
 T_\lambda^+\in C^1(\mathcal{X}_{\theta,\tau},\mathcal{X}_{\theta,\tau}),
 \qquad
 DT_\lambda^+(u)h
 =
 (H+\lambda)^{-1}\bigl(g_+'(u)h\bigr).
$$
No sign condition on $g'$ is used here; sign-changing
linearized sources are handled through the distributional Kato inequality used in
the proof of \eqref{E:weighted-linear-resolvent}.

\smallskip
\noindent\emph{Step 4: identification of the fixed points.}
If $u=T_\lambda^+(u)$, testing the equation by $-u^-$ gives $u\ge0$.  Hence either
$u\equiv0$ or, by the strong maximum principle, $u>0$.  In the latter case
$g_+(u)=g(u)$, so $u$ is a positive solution of \eqref{E:fixed-lambda}.  Since
$u\in\mathcal{X}_{\theta,\tau}\subset X_r\cap C_0(\mathbb{R}^N)$, elliptic regularity and \textbf{Theorem~A} apply.  The function is already centered at $z=0$; a nonconstant decaying function cannot be radial about two distinct centers.  Therefore its free-variable
symmetry center is $0$, and $u\in\mathcal K$.

If \ref{gcon4} also holds, Proposition
\ref{prop:uniform-z-exponential-decay} and \eqref{E:no-loss-weighted-space} place
every centered positive solution with $\lambda\in[\lambda_0,\Lambda]$ in
$\mathcal{X}_{\theta,\tau}$, where it satisfies $u=T_\lambda^+(u)$.  Thus the
nonzero fixed points are exactly the centered positive solutions in the strip.
The proof is complete.
\end{proof}

\subsection{Compactness in the energy space}
\begin{corollary}[Uniform energy bounds]
\label{cor:20260616-2054}\label{cor:20260616-2057}
Assume \ref{gcon1}, \ref{gcon2}, and \ref{gcon4}.  For every
$-d<\Lambda_1\le\Lambda_2<+\infty$,
\begin{equation}\label{eq:20260616-2052}
 \sup_{u\in\mathcal U_{\Lambda_1}^{\Lambda_2}}\|u\|_2<+\infty,
\end{equation}
and
\begin{equation}\label{eq:20260616-2057}
 \sup_{u\in\mathcal U_{\Lambda_1}^{\Lambda_2}}\|u\|_X<+\infty.
\end{equation}
\end{corollary}

\begin{proof}
The product bound \eqref{eq:20260616-1731-global} is integrable and immediately
gives \eqref{eq:20260616-2052}.  Moreover,
$0\le g(u)\le K_{\Lambda_2}u$, so
$$
 \|g(u)\|_{X'}\le\|g(u)\|_2\le K_{\Lambda_2}\|u\|_2.
$$
Since $u=(H+\lambda)^{-1}g(u)$, the uniform form-resolvent bound on
$[\Lambda_1,\Lambda_2]$ and \eqref{eq:20260616-2052} give
\eqref{eq:20260616-2057}.
\end{proof}
\begin{proposition}\label{prop:20260622-0855}
Let $-d<\Lambda_1\leq\Lambda_2<+\infty$, and assume \ref{gcon1} and
\ref{gcon2}.  Then the set $\mathcal U_{\Lambda_1}^{\Lambda_2}$ is bounded
away from $0$ in $X$.
\end{proposition}
\begin{proof}
Let
$$C_{\Lambda_1}:=\max\left\{1, \frac{d+1}{d+\Lambda_1}\right\}.$$
Since $g(s)=o(s)$ as $s\to0^+$ and $g$ has subcritical
growth, for every $0<\varepsilon<\frac{1}{2C_{\Lambda_1}}$, there exists $C_\varepsilon>0$ such that
$$g(s)s\leq \varepsilon s^2+C_\varepsilon s^\beta,\qquad s\geq0.$$
For any $u\in \mathcal{U}_{\Lambda_1}^{\Lambda_2}$, there exists $\lambda\in [\Lambda_1,\Lambda_2]$ such that
$$-\Delta u+|y|^2u+\lambda u=g(u)~\hbox{in}~\mathbb{R}^N.$$
Testing by $u$ and using \eqref{E:coercive2}, we obtain
$$\begin{aligned}
\frac{1}{C_{\Lambda_1}}\|u\|_X^2\leq &\int_{\mathbb{R}^N}\bigl(|\nabla u|^2+|y|^2u^2+\lambda u^2\bigr)\mathrm{d}x\\
=&\int_{\mathbb{R}^N}g(u)u\mathrm{d}x
\leq \varepsilon \|u\|_X^2+ C_\varepsilon\|u\|_X^\beta.
\end{aligned}$$
Since $\beta>2$, there exists $\delta_{\Lambda_1}>0$ such that
$$\|u\|_X\geq \delta_{\Lambda_1}>0,\qquad \forall u\in \mathcal{U}_{\Lambda_1}^{\Lambda_2}.$$
\end{proof}

\begin{proposition}[Compactness in $X$]\label{prop:X-estimate}
Let $-d<\Lambda_1\leq \Lambda_2<+\infty$.  Assume \ref{gcon1}, \ref{gcon2}, and \ref{gcon4}.  Then the set  $\mathcal{U}_{\Lambda_1}^{\Lambda_2}$ is compact in $X_r$.
\end{proposition}

\begin{proof}
Let $u_n\in\mathcal U_{\Lambda_1}^{\Lambda_2}$, with corresponding
$\lambda_n$.  Corollary~\ref{cor:20260616-2057} and extraction give
$$
 \lambda_n\to\lambda\in[\Lambda_1,\Lambda_2],
 \qquad u_n\rightharpoonup u\quad\hbox{in }X.
$$
Lemma~\ref{lemma:20260622-0829} gives $g(u_n)\to g(u)$ in $X'$.
Writing $R_\rho=(H+\rho)^{-1}$, the equation and the first resolvent
identity yield
$$
 \|u_n-R_\lambda g(u)\|_X
 \le C\|g(u_n)-g(u)\|_{X'}
 +C|\lambda_n-\lambda|\,\|g(u)\|_{X'}\longrightarrow0.
$$
The weak limit is therefore $u=R_\lambda g(u)$, and $u_n\to u$ strongly
in $X$.  Proposition~\ref{prop:20260622-0855} excludes $u=0$; the strong
maximum principle gives $u>0$.  Thus the limit remains in
$\mathcal U_{\Lambda_1}^{\Lambda_2}$, proving compactness.
\end{proof}

\begin{proof}[Proof of Theorem~\ref{T:compactness}]
Under assumptions \ref{gcon1} and \ref{gcon2}, the uniform $L^\infty$
estimate follows from Proposition~\ref{prop:L-infty}.  Under the same assumptions,
Lemma~\ref{L:weighted-compact-extension} shows that $T_\lambda^+$ is completely
continuous on bounded subsets of $\mathcal{X}_{\theta,\tau}$, uniformly for
$\lambda\in[\Lambda_1,\Lambda_2]$.  The same lemma also shows that every nonzero
fixed point of $T_\lambda^+$ is a centered positive solution of
\eqref{E:fixed-lambda}.

If, in addition, \ref{gcon4} holds, Proposition~\ref{prop:uniform-z-tail} and
Proposition~\ref{prop:uniform-z-exponential-decay} give the uniform decay in the free
variables and the product-type exponential localization.  Consequently,
Proposition~\ref{prop:X-estimate} yields the compactness of
$\mathcal U_{\Lambda_1}^{\Lambda_2}$ in $X_r$.  Finally,
\eqref{E:no-loss-weighted-space} and Lemma~\ref{L:weighted-compact-extension} show that
every centered positive solution in the strip belongs to
$\mathcal{X}_{\theta,\tau}$.  Hence the nonzero fixed points of $T_\lambda^+$ are
precisely the centered positive solutions in the strip.
\end{proof}

\section{Dimensional reduction at the lower spectral edge}\label{sec:Asymptotics-low}

\subsection{Vanishing and preliminary compactness}
Throughout this section, $\mu=\lambda+d$.  Let $\lambda_n>-d$,
$\mu_n:=\lambda_n+d\to0^+$, and let $(\lambda_n,u_n)\in\mathcal{S}$.  Without loss of generality, we may assume that $\lambda_n\leq1-d$ (that is, $\mu_n\leq1$) for all $n\in\mathbb{N}$.  Proposition~\ref{prop:L-infty} then gives $M>0$ such that
\begin{equation}\label{eq:20260617-0941}
0<u_n\leq M \qquad\hbox{in }\mathbb{R}^N,\quad n\in\mathbb{N}.
\end{equation}
By a standard elliptic estimate, after passing to a subsequence, we may assume that
\begin{equation}\label{eq:20260617-0942}
u_n\to u \qquad\hbox{in}~C^2_{\mathrm{loc}}(\mathbb{R}^N).
\end{equation}
Then $u$ is a bounded nonnegative classical solution of the zero-mass equation
\eqref{E:zero-mass}.

\begin{lemma}[Polynomial growth for zero-mass limits]
\label{lem:polynomial-growth-zero-mass-limit}
For every fixed $z\in\mathbb{R}^m$, the functions
$$u(\cdot,z),\qquad |\nabla_x u(\cdot,z)|,\qquad |\nabla_x^2u(\cdot,z)|$$
have at most polynomial growth in $|y|$. More precisely, there exists a constant $C>0$, independent of $z$, such that
$$
\begin{cases}
|u(y,z)|\leq M,\\
|\nabla_x u(y,z)|\leq C(1+|y|^2),\\
|\nabla_x^2u(y,z)|\leq C(1+|y|^4),
\end{cases}
\qquad x=(y,z)\in\mathbb{R}^d\times\mathbb{R}^m.
$$
In particular, there exists $C>0$ such that
\begin{equation}\label{eq:20260617-0954}
|u(y,z)|+|\nabla_x u(y,z)|+|\nabla_x^2u(y,z)|
\leq C(1+|y|^4),\qquad y\in\mathbb{R}^d, z\in\mathbb{R}^m.
\end{equation}
\end{lemma}

\begin{proof}
The local convergence and the uniform amplitude bound give $0\le u\le M$.
Writing
$$
 -\Delta u=F:=g(u)-(|y|^2-d)u,
$$
we have, on every ball $B_2(x_0)$, $x_0=(y_0,z_0)$,
$$
 \|F\|_{L^\infty(B_2(x_0))}\le C(1+|y_0|^2).
$$
The interior $W^{2,p}$ estimate with $p>N$, followed by Morrey's
embedding, yields
$$
 \|u\|_{C^{1,\gamma}(B_1(x_0))}
 \le C(1+|y_0|^2).
$$
Because $g'$ is bounded on $[0,M]$, the product rule and the last estimate
give
$$
 \|F\|_{C^{0,\gamma}(B_1(x_0))}
 \le C(1+|y_0|^4).
$$
The interior Schauder estimate therefore implies
$$
 \|u\|_{C^{2,\gamma}(B_{1/2}(x_0))}
 \le C(1+|y_0|^4).
$$
All constants are independent of $z_0$.  Since $x_0$ was arbitrary, these
bounds prove \eqref{eq:20260617-0954}.
\end{proof}

\begin{lemma}[Vanishing amplitude]\label{L:lower-vanishing}
Assume \ref{gcon1}, \ref{gcon2}, and \ref{gcon3}.  If $(\lambda_n,u_n)\in\cS$ and $\mu_n:=\lambda_n+d\to0^+$, then
\begin{equation}\label{E:amp-vanish}
        \norm{u_n}_{\infty}\to0.
\end{equation}
\end{lemma}

\begin{proof}
If \eqref{E:amp-vanish} fails, then $\norm{u_n}_\infty\geq c_0>0$.  By \textbf{Theorem~A} and the centering convention, $\norm{u_n}_\infty=u_n(0,0)$; hence,
after taking a subsequence, the local estimates give convergence to a nontrivial
nonnegative solution of
\begin{equation}\label{eq:20260617-0854}
-\Delta u+|y|^2 u-du=g(u)~\hbox{in}~\mathbb{R}^N.
\end{equation}
In particular, $u(0,0)=\lim_nu_n(0,0)\geq c_0>0$.

However, by Lemma \ref{lem:polynomial-growth-zero-mass-limit}, for any fixed $z\in \mathbb{R}^m$, the functions $u(\cdot,z)$, $|\nabla_x u(\cdot,z)|$ and
$|\nabla_x^2 u(\cdot,z)|$ have at most polynomial growth in $|y|$.
Assumption \ref{gcon3} therefore gives $u\equiv0$, contradicting
$u(0)\geq c_0>0$.
\end{proof}

\begin{proposition}[Sharp amplitude scale at the lower spectral edge]
\label{prop:sharp-amplitude-scale-lower-edge}
Assume \ref{gcon1}, \ref{gcon2}, and \ref{gcon3}. Let $(\lambda_n,u_n)\in \mathcal{S}$ and $\mu_n:=\lambda_n+d\to0^+.$
Set
$$M_n:=\|u_n\|_{L^\infty(\mathbb{R}^N)}=u_n(0,0).$$
Then there exist constants $c,C>0$, independent of $n$, such that
$$c\mu_n\leq M_n^{\alpha-2}\leq C\mu_n.$$
Equivalently,
$$0<c\leq \mu_n^{-\frac1{\alpha-2}}M_n\leq C<+\infty.$$
\end{proposition}

\begin{proof}
By Lemma~\ref{L:lower-vanishing}, we already know that $M_n\to0$.
We divide the proof into two steps.

\medskip
\noindent
\textbf{Step 1. Lower bound: $M_n^{\alpha-2}\geq c\mu_n$.}

Since $M_n\to0$ and $g(s)\sim \mu_0s^{\alpha-1}$ as $s\to0^+$, there exists $C>0$
such that, for all large $n$,
$$0\leq g(s)\leq C s^{\alpha-1},\qquad 0\leq s\leq M_n.$$
Since $u_n$ satisfies
$$-\Delta u_n+|y|^2u_n+\lambda_nu_n=g(u_n)~\hbox{in}~\mathbb{R}^N,$$
testing it by $u_n$ gives
$$\int_{\mathbb{R}^N}\left(|\nabla u_n|^2+|y|^2u_n^2+\lambda_nu_n^2\right)\,\mathrm{d}x
=\int_{\mathbb{R}^N}g(u_n)u_n\,\mathrm{d}x,$$
and hence
$$\mu_n\int_{\mathbb{R}^N}u_n^2\,\mathrm{d}x\leq\int_{\mathbb{R}^N}g(u_n)u_n\,\mathrm{d}x\leq CM_n^{\alpha-2}\int_{\mathbb{R}^N}u_n^2\,\mathrm{d}x.$$
Since $u_n\not\equiv0$, we conclude that
$$\mu_n\le CM_n^{\alpha-2},$$
and thus $M_n^{\alpha-2}\ge c\mu_n$.

\medskip
\noindent
\textbf{Step 2. Upper bound: $M_n^{\alpha-2}\leq C\mu_n$.}

We argue by contradiction. Set
$$\kappa_n:=M_n^{\alpha-2}.$$
Suppose that the upper bound is false. Then, up to a subsequence,
\begin{equation}\label{eq:20260617-1347}
\frac{\kappa_n}{\mu_n}\to+\infty.
\end{equation}
We shall derive a contradiction by a blow-up argument in the free variables.

We shall use the Gaussian estimate at the natural amplitude scale.  Set
$$
 K_n:=\sup_{0<s\le M_n}\frac{g(s)}s.
$$
Then $K_n\to0$ by \ref{gcon1} and $M_n\to 0^+$.  Applying Lemma~\ref{L:common-exterior-barrier}(i) to
$$
 (H+\lambda_n-K_n)u_n\le0,\qquad 0<u_n\le M_n,
$$
gives, for each $0<\theta<1$, a constant $A_\theta>0$, independent of
$n$, such that
\begin{equation}\label{eq:20260617-1313}
 0\le\frac{u_n(y,z)}{M_n}
 \le A_\theta e^{-\theta|y|^2/2}.
\end{equation}
Now define the blow-up functions
$$w_n(y,\zeta):=M_n^{-1}u_n\left(y,\frac{\zeta}{\sqrt{\kappa_n}}\right),
\qquad (y,\zeta)\in\mathbb{R}^d\times\mathbb{R}^m.$$
Then
$$0<w_n\le1,\qquad w_n(0,0)=1,$$
and by \eqref{eq:20260617-1313},
\begin{equation}\label{eq:20260617-1349}
0\leq w_n(y,\zeta)\leq A_\theta e^{-\frac{\theta}{2}|y|^2}.
\end{equation}
Moreover $w_n$ satisfies
\begin{equation}\label{eq:20260617-1350}
(H_y-d)w_n+\kappa_n(-\Delta_\zeta w_n)+\mu_nw_n=\kappa_n F_{M_n}(w_n),
\end{equation}
where
$$F_{M_n}(s):=\frac{g(M_ns)}{M_n^{\alpha-1}}.$$
By Lemma \ref{lemma:20260618-1129}, since $M_n\to 0^+$, \eqref{eq:20260622-0819} gives
\begin{equation}\label{eq:20260617-1352}
F_{M_n}(s)\to \mu_0s^{\alpha-1}\qquad\hbox{uniformly for}~s\in[0,1],
\end{equation}
and
\begin{equation}\label{eq:20260617-1353}
0\leq F_{M_n}(s)\leq C s^{\alpha-1},\qquad 0\leq s\leq 1.
\end{equation}

Let $P$ be the orthogonal projection in $L^2(\mathbb{R}^d_y)$ onto
$\mathrm{span}\{\phi_0\}$:
$$(Pw)(y,\zeta):=\phi_0(y)\int_{\mathbb{R}^d}w(\eta,\zeta)\phi_0(\eta)\,\mathrm{d}\eta.$$
Set
$$Q:=I-P,$$
and define
$$v_n(\zeta):=\int_{\mathbb{R}^d}w_n(y,\zeta)\phi_0(y)\,\mathrm{d}y,\qquad r_n:=Qw_n.$$
Thus $\|v_n\|_\infty\leq \|\phi_0\|_{L^1(\mathbb{R}^d)}$ and
$$w_n(y,\zeta)=\phi_0(y)v_n(\zeta)+r_n(y,\zeta),
\qquad\int_{\mathbb{R}^d}r_n(y,\zeta)\phi_0(y)\,\mathrm{d}y=0.$$

Projecting \eqref{eq:20260617-1350} onto $\phi_0$ gives
\begin{equation}\label{eq:20260617-1356}
-\Delta_\zeta v_n+\frac{\mu_n}{\kappa_n}v_n=\int_{\mathbb{R}^d}\phi_0(y)F_{M_n}(w_n(y,\zeta))\,\mathrm{d}y~\hbox{in}~\mathbb{R}^m.
\end{equation}

By \eqref{eq:20260617-1349} and \eqref{eq:20260617-1353}, the right-hand side of
\eqref{eq:20260617-1356} is uniformly bounded in $L^\infty(\mathbb{R}^m)$.
Also $0\leq v_n\leq\|\phi_0\|_{L^1}$.  Since
\eqref{eq:20260617-1347} is equivalent to
$$\frac{\mu_n}{\kappa_n}\to0,$$
standard elliptic estimates applied to \eqref{eq:20260617-1356} imply that,
up to a subsequence,
\begin{equation}\label{eq:20260903-0841}
v_n\to v\qquad\hbox{in }C^{1,\gamma}_{\mathrm{loc}}(\mathbb{R}^m)
\quad\hbox{for every }0<\gamma<1
\end{equation}
for some bounded nonnegative function $v$.

We next prove that $v\not\equiv0$. Since $u_n(0,0)=M_n$ and
$$
-\Delta u_n+
\left(|y|^2+\lambda_n-\frac{g(u_n)}{u_n}\right)u_n=0,
$$
set
$$
V_n(y,z):=
|y|^2+\lambda_n-\frac{g(u_n(y,z))}{u_n(y,z)}.
$$
Since $\lambda_n\to-d$, $M_n=\|u_n\|_\infty\to0$, and
$g(s)/s\to0$ as $s\to0^+$ by \ref{gcon1}, there exists
$K>0$, independent of $n$, such that
$$
\|V_n\|_{L^\infty(B_1(0))}\leq K.
$$
Applying the interior Harnack inequality
\cite[Theorem~8.20, p.~199]{GilbargTrudinger} to
$$
\Delta u_n-V_n(y,z)u_n=0
        \qquad\hbox{in }B_1(0),
$$
we obtain a constant $C_H>0$, independent of $n$, such that
$$
\sup_{B_{1/4}(0)}u_n
    \leq C_H\inf_{B_{1/4}(0)}u_n.
$$
Since $u_n(0,0)=M_n=\|u_n\|_{L^\infty(\mathbb{R}^N)}$, it follows that
$$
u_n(y,z)\geq C_H^{-1}M_n
        \qquad\hbox{for } |(y,z)|\leq\frac{1}{4}.
$$
Therefore
$$w_n(y,0)=\frac{u_n(y,0)}{M_n}\geq C_{H}^{-1}\qquad\hbox{for}~|y|\leq \frac{1}{4}$$
and thus
$$v_n(0)=\int_{\mathbb{R}^d}w_n(y,0)\phi_0(y)\,\mathrm{d}y\geq C_{H}^{-1}\int_{|y|\leq \frac{1}{4}}\phi_0(y)\,\mathrm{d}y=:c_1>0,\quad \forall n\in \mathbb{N}.$$
Passing to the limit gives
$$v(0)\geq c_1>0.$$
Thus $v$ is nontrivial.

It remains to identify the limiting equation. We claim first that
\begin{equation}\label{eq:20260617-1409}
r_n=Qw_n\to0
\qquad\hbox{in}~L_{\mathrm{loc}}^{2}(\mathbb{R}^d\times\mathbb{R}^m).
\end{equation}
Apply Lemma~\ref{L:projected-Q-compactness} to
\eqref{eq:20260617-1350}, with
$$
 \varepsilon_n=\kappa_n,\qquad B_n=0,\qquad
 c_n=\mu_n/\kappa_n,\qquad h_n=F_{M_n}(w_n).
$$
The rescaling of $u_n\in X$ gives the required local form regularity.
The coefficient $c_n$ is bounded by \eqref{eq:20260617-1347}, while
\eqref{eq:20260617-1349} and \eqref{eq:20260617-1353} give the uniform
Gaussian bounds for $w_n$ and $h_n$.  Thus
$$
 \int_K\int_{\mathbb R^d}|Qw_n|^2\,\mathrm dy\,\mathrm d\zeta
 \le C_K\kappa_n\longrightarrow0,
$$
which proves \eqref{eq:20260617-1409}.

Therefore,
$$
w_n-\phi_0v_n=r_n\to0
\qquad\hbox{in }
L_{\mathrm{loc}}^2(\mathbb{R}^d\times\mathbb{R}^m).
$$
Since \eqref{eq:20260903-0841} gives that $v_n\to v$ locally uniformly in $\mathbb{R}^m$, after passing to a subsequence we have
$$
w_n(y,\zeta)\to\phi_0(y)v(\zeta)
\qquad\hbox{for a.e. }(y,\zeta)
\in\mathbb{R}^d\times\mathbb{R}^m.
$$
Set
$$
h_n(\zeta):=
\int_{\mathbb{R}^d}
\phi_0(y)F_{M_n}(w_n(y,\zeta))\,\mathrm dy.
$$
By \eqref{eq:20260617-1349},
\eqref{eq:20260617-1352}, and
\eqref{eq:20260617-1353}, the dominated convergence theorem
on $\mathbb{R}^d\times K$, for every
$K\subset\subset\mathbb{R}^m$, gives
$$
h_n\to
\mu_0\int_{\mathbb{R}^d}
\phi_0(y)\bigl(\phi_0(y)v(\zeta)\bigr)^{\alpha-1}
\,\mathrm dy
=c_\alpha v^{\alpha-1}
$$
in $L^1(K)$.  Since $\{h_n\}$ is uniformly bounded in
$L^\infty(\mathbb{R}^m)$, the convergence also holds in
$L^q_{\mathrm{loc}}(\mathbb{R}^m)$ for every $1\le q<\infty$.
Passing to the limit in the weak formulation of
\eqref{eq:20260617-1356}, and recalling that
$\mu_n/\kappa_n\to0$, we obtain
$$
-\Delta v=c_\alpha v^{\alpha-1}
\qquad\hbox{in }\mathcal D'(\mathbb{R}^m).
$$
Since $0\leq v\in L^\infty(\mathbb{R}^m)\cap C^{1,\gamma}_{\mathrm{loc}}(\mathbb{R}^m)$,
standard elliptic regularity then shows that $v$ is a nontrivial bounded nonnegative classical solution to
\begin{equation}\label{eq:20260617-1441}
-\Delta v=c_\alpha v^{\alpha-1}\qquad\hbox{in}~\mathbb{R}^m.
\end{equation}

However, such a solution is excluded by
\cite[Theorem~8.4]{QuittnerSouplet} when $m=1,2$ and by
\cite[Theorem~1.1]{GidasSpruck} when $m\ge3$, since
$1<\alpha-1<2_m^*-1$ in the latter case.
This contradiction shows that
$$\frac{M_n^{\alpha-2}}{\mu_n}=\frac{\kappa_n}{\mu_n}$$
is bounded from above. Therefore
$$M_n^{\alpha-2}\leq C\mu_n.$$

Combining \textbf{Step 1} and \textbf{Step 2}, we conclude that
$$c\mu_n\le M_n^{\alpha-2}\le C\mu_n.$$
The proof is complete.
\end{proof}

\subsection{Convergence to the effective lower-dimensional profile}
Let $\mu_n:=\lambda_n+d\to0^+$.  Proposition
\ref{prop:sharp-amplitude-scale-lower-edge} shows that
\begin{equation}\label{eq:20260617-1823}
\widetilde{u}_n(y,\zeta)=\mu_n^{-\frac1{\alpha-2}}u_n\left(y,\frac{\zeta}{\sqrt{\mu_n}}\right)
\end{equation}
satisfy
\begin{equation}\label{eq:20260617-1846}
0<c\leq \|\widetilde{u}_n\|_\infty\leq C
\end{equation}
and
\begin{equation}\label{eq:20260617-1054}
(H_y-d)\widetilde{u}_n+\mu_n(-\Delta_\zeta\widetilde{u}_n+\widetilde{u}_n)
=\mu_n F_{\mu_n^{\frac{1}{\alpha-2}}}(\widetilde{u}_n),
\end{equation}
where
\begin{equation}\label{eq:20260617-1903}
F_{\mu_n^{1/(\alpha-2)}}(s)\to\mu_0s^{\alpha-1}
\quad\text{uniformly for bounded $s$}.
\end{equation}

Define
\begin{equation}\label{eq:20260617-1826}
v_n(\zeta):=\int_{\mathbb{R}^d}\widetilde{u}_n(y,\zeta)\phi_0(y)\,\mathrm{d}y.
\end{equation}
Then $0\le v_n\le C\|\phi_0\|_1$.  The Harnack argument (below \eqref{eq:20260903-0841}) in the proof of Proposition~\ref{prop:sharp-amplitude-scale-lower-edge}, together with the lower
sharp-amplitude bound, gives $v_n(0)\ge c>0$.  Hence every locally uniform limit
of $(v_n)$ is nontrivial.

\begin{lemma}[Local vanishing of the orthogonal component]
\label{lem:local-vanishing-orthogonal-component}
Let $\widetilde{u}_n(y,\zeta)$ be given by \eqref{eq:20260617-1823} and $\mu_n=\lambda_n+d\to0^+$.
Define
\begin{equation}\label{eq:20260617-1825}
P w=\phi_0(y)\int_{\mathbb{R}^d}w(\eta,\zeta)\phi_0(\eta)\,\mathrm{d}\eta,\qquad Q:=I-P.
\end{equation}
Then, for every compact set $K\subset\subset\mathbb{R}^m$,
\begin{equation}\label{eq:20260617-1828}
r_n:=Q\widetilde{u}_n\to 0
\qquad\hbox{in}~L^2(\mathbb{R}^d\times K).
\end{equation}
\end{lemma}

\begin{proof}
Apply Lemma~\ref{L:projected-Q-compactness} to
\eqref{eq:20260617-1054} with
$$
 \varepsilon_n=\mu_n,\qquad B_n=0,\qquad c_n=1,\qquad
 h_n=F_{\mu_n^{1/(\alpha-2)}}(\widetilde u_n).
$$
Indeed, \eqref{eq:20260617-1313} and
$M_n=O(\mu_n^{1/(\alpha-2)})$ give, for each $0<\theta<1$,
$$
 0\le\widetilde u_n(y,\zeta)\le Ce^{-\theta|y|^2/2}.
$$
Lemma~\ref{lemma:20260618-1129} gives the same type of Gaussian bound for
$h_n$.  These estimates verify \eqref{E:abstract-Q-local-bounds}, and
\eqref{E:abstract-Q-conclusion} yields
$$
 \int_K\int_{\mathbb{R}^d}|r_n|^2\,\mathrm dy\,\mathrm d\zeta
 \le C_K\mu_n\longrightarrow0.
$$
\end{proof}

\begin{lemma}[Identification of the local limit]
\label{lem:identification-local-limit}
Up to a subsequence,
$$
v_n\to v
\qquad\hbox{in }W^{2,q}_{\mathrm{loc}}(\mathbb{R}^m)
\quad\hbox{for every }1<q<+\infty,
$$
where $v$ is a nontrivial bounded nonnegative radial solution of
\begin{equation}\label{eq:20260617-1837}
-\Delta_\zeta v+v=c_\alpha v^{\alpha-1}\qquad\hbox{in}~\mathbb{R}^m,
\end{equation}
with
$$c_\alpha=\mu_0\int_{\mathbb{R}^d}\phi_0(y)^\alpha\,\mathrm{d}y.$$
\end{lemma}

\begin{proof}
Projecting \eqref{eq:20260617-1054} onto $\phi_0$ gives
\begin{equation}\label{eq:20260617-1841}
-\Delta_\zeta v_n+v_n
=\int_{\mathbb{R}^d}\phi_0(y)F_{\mu_n^{\frac{1}{\alpha-2}}}(\widetilde{u}_n(y,\zeta))\,\mathrm{d}y .
\end{equation}
By Lemma \ref{lemma:20260618-1129} and \eqref{eq:20260617-1846}, there exists $C>0$ such that
\begin{equation}\label{eq:20260617-1847}
F_{\mu_n^{\frac{1}{\alpha-2}}}(\widetilde{u}_n(y,\zeta))\leq C \widetilde{u}_{n}^{\alpha-1}(y,\zeta), \quad (y,\zeta)\in \mathbb{R}^d\times \mathbb{R}^m.
\end{equation}
Fix $\theta\in (0,1)$, by \eqref{eq:20260617-1313} and $M_n=O(\mu_n^{\frac1{\alpha-2}})$, we further have the domination
\begin{equation}\label{eq:20260617-1854}
F_{\mu_n^{\frac{1}{\alpha-2}}}(\widetilde{u}_n(y,\zeta))\leq  CA_{\theta}^{\alpha-1} e^{-\frac{\theta(\alpha-1)}{2}|y|^2}, \quad (y,\zeta)\in \mathbb{R}^d\times \mathbb{R}^m.
\end{equation}
Thus the right-hand side of \eqref{eq:20260617-1841} is uniformly bounded
in $L^\infty(\mathbb{R}^m)$.  Interior $W^{2,q}$ estimates and compact
embedding therefore give, up to a subsequence,
$$
 v_n\to v\qquad\hbox{in }C^{1,\gamma}_{\mathrm{loc}}(\mathbb{R}^m)
 \quad\hbox{for every }0<\gamma<1.
$$
The lower bound $v_n(0)\ge c$ proved above gives $v\not\equiv0$.

Finally, by Lemma~\ref{lem:local-vanishing-orthogonal-component},
$$\widetilde{u}_n-\phi_0v_n=r_n\to0\qquad\hbox{in }L^2_{\mathrm{loc}}(\mathbb{R}^N).$$
Thus, after passing to a further subsequence,
\begin{equation}\label{eq:20260617-1901}
\widetilde{u}_n(y,\zeta)\to\phi_0(y)v(\zeta)
\qquad\hbox{a.e. in}~\mathbb{R}^d\times\mathbb{R}^m.
\end{equation}
Using \eqref{eq:20260617-1903}, \eqref{eq:20260617-1901}, and the dominated
convergence theorem first in $y$ and then on compact subsets of the
$\zeta$-space, we obtain, for every $1<q<+\infty$,
\begin{equation}\label{eq:20260617-1940}
\int_{\mathbb{R}^d}\phi_0(y)F_{\mu_n^{\frac{1}{\alpha-2}}}(\widetilde{u}_n(y,\zeta))\,\mathrm{d}y
\to
\mu_0\int_{\mathbb{R}^d}\phi_0(y)\bigl(\phi_0(y)v(\zeta)\bigr)^{\alpha-1}\,\mathrm{d}y
=c_\alpha v(\zeta)^{\alpha-1}
\quad\hbox{in }L^q_{\mathrm{loc}}(\mathbb{R}^m).
\end{equation}
The interior estimate for the difference of the equations now improves the
convergence to $W^{2,q}_{\mathrm{loc}}$ for every finite $q>1$, and $v$
satisfies \eqref{eq:20260617-1837}.
\end{proof}

\begin{lemma}[Uniform smallness and exponential decay of the projected profiles]
\label{lem:uniform-tail-vn}
The sequence $\{v_n\}$ satisfies
\begin{equation}\label{eq:20260617-1911}
\lim_{R\to+\infty}\sup_n\sup_{|\zeta|\geq R}v_n(\zeta)=0.
\end{equation}
Consequently, for every $\sigma\in(0,1)$, there exists $C>0$ such that
\begin{equation}\label{eq:20260819-1126}
v_n(\zeta)\le C e^{-\sigma|\zeta|}\qquad\hbox{for all}~\zeta\in\mathbb{R}^m
\end{equation}
and all sufficiently large $n$.
\end{lemma}

\begin{proof}
By \textbf{Theorem~A} and the centering convention, each $u_n$ is
radially symmetric and nonincreasing with respect to $|z|$.
The rescaling \eqref{eq:20260617-1823} preserves these
properties in the free variable $\zeta$. Since $\phi_0>0$,
the projected function $v_n$ (defined by \eqref{eq:20260617-1826})
is therefore radial and nonincreasing in $|\zeta|$.

We argue by contradiction. Suppose that \eqref{eq:20260617-1911} fails. Then there exist $\varepsilon_0>0$,
a subsequence still denoted by $v_n$, and $R_n\to+\infty$ such that
$$v_n(R_n)\geq\varepsilon_0.$$

Moreover, $\phi_0 v_n=P\widetilde{u}_n\in L^2(\mathbb{R}^N)$, which implies that $v_n\in L^2(\mathbb{R}^m)$.
Hence,
$$v_n(r)\to0\qquad\hbox{as }r\to+\infty.$$
 Choose $0<\varepsilon<\varepsilon_0$ so small that
$$c_\alpha \varepsilon^{\alpha-2}<1.$$
Choose $\rho_n\ge R_n$ such that
$$
 v_n(\rho_n)=\varepsilon,\qquad
 v_n(r)\le\varepsilon\quad(r\ge\rho_n),
$$
and define
$$
 V_n(\tau):=v_n(\rho_n+\tau),\qquad
 \tau\in(-\rho_n,+\infty).
$$
The radial form of \eqref{eq:20260617-1841} is
\begin{equation}\label{eq:20260617-1921}
-v_n''(r)-\frac{m-1}{r}v_n'(r)+v_n(r)
=\int_{\mathbb{R}^d}\phi_0(y)F_{\mu_n^{\frac{1}{\alpha-2}}}(\widetilde{u}_n(y,r))\,\mathrm{d}y .
\end{equation}
Fix $\omega\in\mathbb S^{m-1}$ and set
$$
 \mathcal U_n(y,\tau):=\widetilde u_n(y,(\rho_n+\tau)\omega),
 \qquad F_n:=F_{\mu_n^{1/(\alpha-2)}}.
$$
Substituting $r=\rho_n+\tau$ in \eqref{eq:20260617-1921} gives
$$
 -V_n''-\frac{m-1}{\rho_n+\tau}V_n'+V_n
 =\int_{\mathbb R^d}\phi_0(y)F_n(\mathcal U_n(y,\tau))\,\mathrm dy.
$$
Both $V_n$ and the right-hand side are uniformly bounded, by
\eqref{eq:20260617-1846} and Lemma~\ref{lemma:20260618-1129}.
Since $\rho_n\to\infty$, interior estimates give local $W^{2,q}$ bounds
for every finite $q>1$.  Compact embedding therefore yields, after extraction,
$$
 V_n\to V\quad\hbox{in }C^1_{\rm loc}(\mathbb R).
$$

To identify the nonlinear limit, we next control the orthogonal component
on these translated intervals.  Write
$$
 \mathcal R_n:=Q\mathcal U_n,\qquad
 \mathcal U_n=\phi_0(y)V_n(\tau)+\mathcal R_n,
 \qquad P\mathcal R_n=0.
$$
The translated form of \eqref{eq:20260617-1054} is
$$
 (H_y-d)\mathcal U_n+\mu_n\left(
 -\partial_{\tau\tau}\mathcal U_n
 -\frac{m-1}{\rho_n+\tau}\partial_\tau\mathcal U_n+\mathcal U_n\right)
 =\mu_nF_n(\mathcal U_n).
$$
Apply Lemma~\ref{L:projected-Q-compactness} with
$$
 \varepsilon_n=\mu_n,\qquad
 B_n(\tau)=-\frac{m-1}{\rho_n+\tau},\qquad
 c_n=1,\qquad h_n=F_n(\mathcal U_n).
$$
Indeed, $B_n\to0$ in $W^{1,\infty}_{\rm loc}(\mathbb R)$, while the
Gaussian bound $0\le\mathcal U_n\le Ce^{-\theta|y|^2/2}$ and
Lemma~\ref{lemma:20260618-1129} verify
\eqref{E:abstract-Q-local-bounds}.  Thus
\begin{equation}\label{E:translated-Q-tail-estimate}
 \int_{\mathbb R^d\times I}|\mathcal R_n|^2\,\mathrm dy\,\mathrm d\tau
 \le C_I\mu_n\longrightarrow0
 \qquad (I\subset\subset\mathbb R).
\end{equation}
The uniform Lipschitz bound for $F_n$ on the bounded range of
$\mathcal U_n$ and $\phi_0V_n$, together with $\|\phi_0\|_2=1$, gives
\begin{align*}
 &\left\|\int_{\mathbb R^d}\phi_0
       \bigl[F_n(\mathcal U_n)-F_n(\phi_0V_n)\bigr]\,\mathrm dy
       \right\|_{L^2(I)}\\
 &\qquad\le C\|\mathcal R_n\|_{L^2(\mathbb R^d\times I)}
       \le C_I\sqrt{\mu_n}\longrightarrow0.
\end{align*}
Moreover, the uniform convergence of $F_n(s)$ to $\mu_0s^{\alpha-1}$
on bounded intervals, $V_n\to V$ locally uniformly, and
$\phi_0\in L^1(\mathbb R^d)$ imply
$$
 \int_{\mathbb R^d}\phi_0F_n(\phi_0V_n)\,\mathrm dy
 \longrightarrow
 \mu_0\left(\int_{\mathbb R^d}\phi_0^\alpha\,\mathrm dy\right)
 V^{\alpha-1}=c_\alpha V^{\alpha-1}
$$
uniformly on $I$.  Hence the right-hand side of the equation for $V_n$
converges to $c_\alpha V^{\alpha-1}$ in $L^2(I)$ and, by its uniform
boundedness, in every finite $L^q(I)$.  Returning to that equation gives
$V_n\to V$ in $W^{2,q}_{\rm loc}$ for every finite $q>1$ and
\begin{equation}\label{eq:20260617-1922}
 -V''+V=c_\alpha V^{\alpha-1}\qquad\hbox{in }\mathbb R,
\end{equation}
with
$$
 V(0)=\varepsilon,\qquad
 0\le V(\tau)\le\varepsilon\quad(\tau\ge0),\qquad V'\le0.
$$

The endpoint limits $V_\pm=\lim_{\tau\to\pm\infty}V(\tau)$ exist.
Equation \eqref{eq:20260617-1922} makes $V''$ bounded; hence $V'$ is
uniformly continuous.  Since $-V'\in L^1(\mathbb{R})$, one has
$V'(\tau)\to0$ at both endpoints, and translation in $\tau$ yields
$$
 V_\pm=c_\alpha V_\pm^{\alpha-1}.
$$
Our choice of $\varepsilon$ gives $V_+=0$, whereas
$V_-\ge V(0)=\varepsilon>0$.  For the scalar stationary functional
$$
 J(s)=\frac{1}{2}s^2-\frac{c_\alpha}{\alpha}s^\alpha
$$
the equilibrium identity gives
$$
 J(V_-)=\left(\frac{1}{2}-\frac1\alpha\right)V_-^2>0.
$$
Lemma~\ref{L:common-cylinder-Hamiltonian}, with
$k=0$, $\mathcal V=1$, and $f(s)=c_\alpha s^{\alpha-1}$, now gives a
contradiction.  Thus \eqref{eq:20260617-1911} holds.
We also compare the value at $y=0$ with its $\phi_0$-average.  Apply the local Harnack inequality in balls centered at $(0,z)$ with bounded $y$-radius and use $u_n\leq M_n=O(\mu_n^{1/(\alpha-2)})$.  There exist $C,\rho>0$, independent of $n$ and $z$ (due to the translation invariance of the equation in $z$), such that
\begin{equation}\label{eq:20260617-1942-local}
        u_n(y,z)\ge C u_n(0,z)
        \qquad\hbox{for } |y|\le\rho .
\end{equation}
After the lower-edge scaling this gives
\begin{equation}\label{eq:20260617-1942}
\widetilde{u}_n(y,\zeta)\geq C \widetilde{u}_n(0,\zeta)~\hbox{for}~|y|\leq \rho.
\end{equation}
Thus,
\begin{align*}
v_n(\zeta):=&\int_{\mathbb{R}^d}\widetilde{u}_n(y,\zeta)\phi_0(y)\,\mathrm{d}y\\
\geq&\int_{B_\rho(0)}\widetilde{u}_n(y,\zeta)\phi_0(y)\,\mathrm{d}y\\
\geq& C_1 \widetilde{u}_n(0,\zeta)\int_{B_\rho(0)}\phi_0(y)\,\mathrm{d}y\\
\geq&C_2 \widetilde{u}_n(0,\zeta),
\end{align*}
that is,
\begin{equation}\label{eq:20260617-1942-average}
\widetilde{u}_n(0,\zeta)\leq C v_n(\zeta).
\end{equation}
Then by the monotonicity in $|y|$ given by \textbf{Theorem~A}, we have
\begin{equation}\label{eq:20260617-2003}
0\leq \widetilde{u}_n(y,\zeta)\leq C v_n(\zeta).
\end{equation}
By \eqref{eq:20260617-1847}, \eqref{eq:20260617-2003} and \eqref{eq:20260617-1911},
for any $\eta>0$ small, there exists $R_0$ large enough, such that
\begin{equation}\label{eq:20260617-2003-small}
F_{\mu_n^{\frac{1}{\alpha-2}}}(\widetilde{u}_n)\leq \eta \widetilde{u}_n\qquad\hbox{for}~|\zeta|\geq R_0.
\end{equation}
Recalling \eqref{eq:20260617-1841}, we arrive at
$$-\Delta v_n+v_n\le \eta v_n\qquad\hbox{for}~|\zeta|\geq R_0.$$
Equivalently,
$$-\Delta v_n+(1-\eta)v_n\leq 0\qquad\hbox{for}~|\zeta|\geq R_0.$$
Choosing $0<\sigma<\sqrt{1-\eta}$ and comparing with $e^{-\sigma(|\zeta|-R_0)}$, we obtain
$$v_n(\zeta)\le Ce^{-\sigma(|\zeta|-R_0)}\qquad\hbox{for}~|\zeta|\geq R_0.$$
After changing the constant, and by the arbitrariness of $\eta>0$ small, we obtain the desired exponential decay.
\end{proof}

\begin{proposition}[Global $L^2$ convergence and strong convergence of the projected profiles]
\label{prop:global-L2-convergence-rescaled}
Up to a subsequence,
$$\widetilde{u}_n\to \phi_0(y)v(\zeta)\qquad\hbox{strongly in}~L^2(\mathbb{R}^N).$$
Consequently,
$$r_n=Q\widetilde{u}_n\to0\qquad\hbox{strongly in}~L^2(\mathbb{R}^N),$$
and
$$v_n\to v\qquad\hbox{strongly in}~H^1(\mathbb{R}^m).$$
\end{proposition}

\begin{proof}
Recalling Lemma \ref{lem:local-vanishing-orthogonal-component}, we already know that
$$
\widetilde{u}_n\to \phi_0v
\qquad\hbox{in }L^2(\mathbb{R}^d\times K)~\hbox{for every compact set $K\subset\subset\mathbb{R}^m$}.
$$
The Gaussian estimate and \eqref{eq:20260617-2003} give
\begin{equation}\label{eq:20260617-2037}
 \widetilde u_n(y,\zeta)
 \le\min\{Cv_n(\zeta),C_\theta e^{-\theta|y|^2/2}\}.
\end{equation}
By Lemma~\ref{lem:uniform-tail-vn},
\begin{equation}\label{eq:20260617-2111}
 v_n(\zeta)\le Ce^{-\sigma|\zeta|}.
\end{equation}
Passing to the local limit gives
\begin{equation}\label{eq:20260617-2048}
 v(\zeta)\le Ce^{-\sigma|\zeta|}.
\end{equation}
Using $\min\{a,b\}\le\sqrt{ab}$, we obtain the common product envelope
$$
 \widetilde u_n(y,\zeta)
 \le C e^{-\theta|y|^2/4-\sigma|\zeta|/2}.
$$
Lemma~\ref{L:envelope-global-compactness}, applied to the local $L^2$
convergence above, now gives
$$
 \widetilde u_n\to\phi_0v\qquad\hbox{strongly in }L^2(\mathbb{R}^N).
$$

Applying $P$ to the $L^2$ convergence gives $\phi_0v_n\to\phi_0v$ in $L^2(\mathbb{R}^N)$, equivalently,
$$v_n\rightarrow v~\hbox{in}~L^2(\mathbb{R}^m).$$
Similarly, applying $Q$ gives
$$r_n\rightarrow 0~\hbox{in}~L^2(\mathbb{R}^N).$$

Recalling \eqref{eq:20260617-1854}, \eqref{eq:20260617-2111}, and
\eqref{eq:20260617-1903}, the dominated convergence theorem yields
\begin{equation}\label{eq:20260617-2114}
\int_{\mathbb{R}^m}\left(\int_{\mathbb{R}^d}\phi_0(y)F_{\mu_n^{\frac{1}{\alpha-2}}}(\widetilde{u}_n(y,\zeta))\,\mathrm{d}y\right) v_n(\zeta)\mathrm{d}\zeta \rightarrow c_\alpha \|v\|_{L^\alpha(\mathbb{R}^m)}^{\alpha}.
\end{equation}
By \eqref{eq:20260617-2048} and \eqref{eq:20260617-1837}, we have
\begin{equation}\label{eq:20260617-2118}
c_\alpha \|v\|_{L^\alpha(\mathbb{R}^m)}^{\alpha}=\|v\|_{H^1(\mathbb{R}^m)}^{2}.
\end{equation}
Testing \eqref{eq:20260617-1841} by $v_n$ gives
\begin{equation}\label{eq:20260617-2120}
\|v_n\|_{H^1(\mathbb{R}^m)}^{2}=\int_{\mathbb{R}^m}\left(\int_{\mathbb{R}^d}\phi_0(y)F_{\mu_n^{\frac{1}{\alpha-2}}}(\widetilde{u}_n(y,\zeta))\,\mathrm{d}y\right) v_n(\zeta)\mathrm{d}\zeta.
\end{equation}
After passing to a subsequence, $v_n\to v$ almost everywhere in
$\mathbb{R}^m$.  Equations \eqref{eq:20260617-2114},
\eqref{eq:20260617-2118}, and \eqref{eq:20260617-2120}, together with the
$L^2$ convergence, therefore give
$$v_n\rightarrow v~\hbox{strongly in}~H^1(\mathbb{R}^m).$$

The proof is complete.
\end{proof}

\begin{proposition}[Strong $H^1$ convergence of the rescaled profiles]
\label{prop:strong-H1-convergence-rescaled}
Under the assumptions of Proposition~\ref{prop:global-L2-convergence-rescaled}, we have
$$
\widetilde{u}_n\to \phi_0(y)v(\zeta)
\qquad\hbox{strongly in }H^1(\mathbb{R}^d\times\mathbb{R}^m).
$$
More precisely,
$$
\nabla_y\widetilde{u}_n\to \nabla_y\bigl(\phi_0v\bigr)
\qquad\hbox{in }L^2(\mathbb{R}^d\times\mathbb{R}^m),
$$
and
$$
\nabla_\zeta\widetilde{u}_n\to \nabla_\zeta\bigl(\phi_0v\bigr)
\qquad\hbox{in }L^2(\mathbb{R}^d\times\mathbb{R}^m).
$$
\end{proposition}

\begin{proof}
Write $U_n=\widetilde u_n$ and $U=\phi_0v$.  By
Proposition~\ref{prop:global-L2-convergence-rescaled}, $U_n\to U$ in $L^2$.
Testing the rescaled equation \eqref{eq:20260617-1054} by $U_n$ gives
\begin{equation}\label{eq:20260617-2140}
 \frac{A_n}{\mu_n}+B_n
 =\int_{\mathbb R^N}F_{\mu_n^{1/(\alpha-2)}}(U_n)U_n,
\end{equation}
where $A_n:=\langle(H_y-d)U_n,U_n\rangle\ge0$ and
$B_n:=\|\nabla_\zeta U_n\|_2^2+\|U_n\|_2^2$.
The common Gaussian--exponential envelope and the small-amplitude limit
of the nonlinearity imply, by dominated convergence,
$$
 \int F_{\mu_n^{1/(\alpha-2)}}(U_n)U_n
 \longrightarrow\mu_0\int U^\alpha
 =c_\alpha\int v^\alpha=\|v\|_{H^1(\mathbb R^m)}^2.
$$
Thus $(\nabla_\zeta U_n)$ is bounded in $L^2$, with weak limit
$\nabla_\zeta U$.  Weak lower semicontinuity and $A_n\ge0$ give
$$
 \|v\|_{H^1}^2\le\liminf B_n\le\limsup B_n\le\|v\|_{H^1}^2.
$$
Consequently $\nabla_\zeta U_n\to\nabla_\zeta U$ strongly in $L^2$
and $A_n=o(\mu_n)$.

To control the confined variables, decompose $U_n=\phi_0v_n+r_n$ with
$r_n=QU_n$.  Orthogonality and the oscillator gap yield
$$
 \|\nabla_y r_n\|_2^2+\||y|r_n\|_2^2
 \le\frac{d+2}{2}\langle(H_y-d)r_n,r_n\rangle
 =\frac{d+2}{2}A_n\longrightarrow0.
$$
Since $v_n\to v$ in $L^2$ and $\nabla_y\phi_0\in L^2$, it follows that
$\nabla_yU_n\to\nabla_y(\phi_0v)$ in $L^2$.  Together with the free-variable
convergence, this proves the assertion.
\end{proof}

\begin{proof}[Proof of Theorem~\ref{T:asymptotics}(i)]
Proposition~\ref{prop:strong-H1-convergence-rescaled} gives \eqref{eq:20260618-1027}.  Lemma~\ref{lem:identification-local-limit} identifies every local limit of $v_n$ as a positive radial solution of \eqref{E:Valpha}.  By the uniqueness of the scalar ground state, this limit is $V_\alpha$, and therefore the whole sequence satisfies \eqref{eq:20260618-1030}.  Proposition~\ref{prop:global-L2-convergence-rescaled} gives the corresponding global convergence.
Indeed, if the whole sequence failed to converge, a subsequence staying a
fixed positive distance from $\phi_0V_\alpha$ would possess a further subsequence
converging to that same profile, a contradiction.
Thus,
$$\phi_0 v_n\rightarrow \phi_0V_\alpha~\hbox{in}~H^1(\mathbb{R}^N),$$
and therefore
$$r_n=\widetilde{u}_n-\phi_0 v_n\rightarrow 0~\hbox{in}~H^1(\mathbb{R}^N),$$
which proves \eqref{eq:20260618-1031}.
The mass identity is a change of variables:
$$
        \norm{u_n}_2^2
        =\mu_n^{\frac{2}{\alpha-2}}\mu_n^{-m/2}
        \int_{\R^d\times\R^m}\widetilde{u}_n(y,\zeta)^2\mathrm{d}y\mathrm{d}\zeta.
$$
Using \eqref{eq:20260618-1030} and $\norm{\phi_0}_2=1$, we obtain \eqref{E:lower-mass} and \eqref{E:lower-limit}.
\end{proof}

\section{The full-dimensional high-frequency limit}\label{sec:Asymptotics-high}
\subsection{The sharp amplitude scale}

We first determine the sharp blow-up rate of positive solutions as $\lambda\to+\infty$.
\begin{proposition}[Sharp amplitude scale as $\lambda\to+\infty$]
\label{prop:sharp-amplitude-scale-infinity}
Assume \ref{gcon1} and \ref{gcon2}. Let $\lambda_n\to+\infty$ and let
$u_n\in \cK$ be a positive solution of
$$
-\Delta u_n+|y|^2u_n+\lambda_nu_n=g(u_n)
\qquad\hbox{in }\mathbb{R}^N.
$$
Set
$$
M_n:=\|u_n\|_{L^\infty(\mathbb{R}^N)}=u_n(0,0).
$$
Then there exist constants $c,C>0$, independent of $n$, such that
$$
c\lambda_n\le M_n^{\beta-2}\le C\lambda_n.
$$
Equivalently,
$$
0<c\le \lambda_n^{-\frac1{\beta-2}}M_n\le C<+\infty.
$$
\end{proposition}

\begin{proof}
We divide the proof into two steps.

\medskip
\noindent
\textbf{Step 1. Lower bound.}

By \textbf{Theorem~A} and the centering convention, the maximum is attained at the origin.
At this point $-\Delta u_n(0,0)\geq0$ and $|y|^2=0$; hence
\begin{equation}\label{eq:20260618-1106}
\lambda_n M_n\leq g(M_n).
\end{equation}
Since $\lambda_n\to+\infty$, this implies $g(M_n)/M_n\to+\infty$ and hence
$M_n\to+\infty$.  By \ref{gcon2}, $g(s)\leq Cs^{\beta-1}$ for all sufficiently
large $s$.  Thus \eqref{eq:20260618-1106} gives
$$M_n^{\beta-2}\geq c\lambda_n.$$

\medskip
\noindent
\textbf{Step 2. Upper bound.}

We argue by contradiction. Suppose that, up to a subsequence,
\begin{equation}\label{eq:20260618-1107}
\frac{M_n^{\beta-2}}{\lambda_n}\to+\infty.
\end{equation}
Set
$$
\rho_n:=M_n^{-\frac{\beta-2}{2}},\qquad
w_n(\xi):=M_n^{-1}u_n(\rho_n\xi),\qquad
\xi=(Y,Z)\in\mathbb{R}^d\times\mathbb{R}^m.
$$
Then $0<w_n\leq1$ and $w_n(0)=1$.
Moreover, $w_n$ satisfies
\begin{equation}\label{eq:20260618-1108}
-\Delta w_n+\rho_n^4|Y|^2w_n+\lambda_n\rho_n^2 w_n
=\frac{g(M_nw_n)}{M_n^{\beta-1}}\qquad\hbox{in}~\mathbb{R}^N.
\end{equation}
Since
$$\rho_n^4=M_n^{-2(\beta-2)}\to0$$
and, by \eqref{eq:20260618-1107},
$$\lambda_n\rho_n^2=\frac{\lambda_n}{M_n^{\beta-2}}\to0,$$
the lower order terms in \eqref{eq:20260618-1108} vanish locally uniformly.

By \ref{gcon2} and the global boundedness $0\leq w_n\le1$, the right-hand side of \eqref{eq:20260618-1108} is uniformly bounded in $L^\infty(\mathbb{R}^N)$. Thus interior $W^{2,p}$ estimates give compactness in
$C^{1,\gamma}_{\rm loc}$ for $p>N$.  Lemma~\ref{lemma:20260618-1129} gives
uniform convergence on $[0,1]$ of the rescaled nonlinearities and their first
derivatives, so the composed right-hand sides converge in local H\"older spaces.
The Schauder estimate then yields, up to a subsequence,
$$w_n\to w\qquad\hbox{in }C^{2,\gamma'}_{\mathrm{loc}}(\mathbb{R}^N)
\quad\hbox{for every }0<\gamma'<\gamma,$$
where
$$0\leq w\leq 1,~\hbox{and}~w(0)=1.$$

Passing to the limit in \eqref{eq:20260618-1108} and using Lemma \ref{lemma:20260618-1129}, we obtain
\begin{equation}\label{eq:20260618-1125}
-\Delta w=\mu_\infty w^{\beta-1}\qquad\hbox{in}~\mathbb{R}^N.
\end{equation}
Thus $w$ is a nontrivial bounded nonnegative entire solution of
\eqref{eq:20260618-1125}.  This is impossible by
\cite[Theorem~8.4]{QuittnerSouplet} when $N=1,2$ and by
\cite[Theorem~1.1]{GidasSpruck} when $N\ge3$, because
$2<\beta<2_N^*$.
Therefore \eqref{eq:20260618-1107} is impossible, and thus
$$M_n^{\beta-2}\leq C\lambda_n.$$
Combining the two estimates gives
$$c\lambda_n\leq M_n^{\beta-2}\leq C\lambda_n.$$
\end{proof}

Let $\widehat{u}_n$ be defined by \eqref{E:high-rescale}.  It satisfies
\begin{equation}\label{E:rescaled-high-eq}
        -\Delta \widehat{u}_n+\widehat{u}_n+\lambda_n^{-2}|Y|^2\widehat{u}_n
        =\frac{g(\lambda_n^{1/(\beta-2)}\widehat{u}_n)}
        {\lambda_n^{(\beta-1)/(\beta-2)}}=F_{\lambda_n^{1/(\beta-2)}}(\widehat{u}_n).
\end{equation}
By Proposition \ref{prop:sharp-amplitude-scale-infinity},
\begin{equation}\label{eq:20260618-1235}
0<c\leq \|\widehat{u}_n\|_{L^\infty(\mathbb{R}^N)}=\widehat{u}_n(0)\leq C<+\infty, \forall n\in \mathbb{N}.
\end{equation}
 By Lemma \ref{lemma:20260618-1129},
 \begin{equation}\label{eq:20260618-1150}
 F_{\lambda_n^{1/(\beta-2)}}(\widehat{u}_n)-\mu_\infty\widehat{u}_n^{\beta-1}\to 0~\hbox{uniformly in}~\mathbb{R}^{N}.
 \end{equation}

Elliptic estimates yield
$$\widehat{u}_n\rightarrow W~\hbox{in}~C_{\mathrm{loc}}^{2}(\mathbb{R}^N),$$
where $W$ is a nonnegative classical bounded solution of
$$-\Delta W+W=\mu_\infty W^{\beta-1}\quad\text{in}~\R^N.$$
Equation~\eqref{eq:20260618-1235} also gives $W\not\equiv 0$ and
$$W(0)=\max_{x\in \mathbb{R}^N}W(x)\geq c>0.$$
In particular, Lemma~\ref{lem:uniform-smallness-high-frequency} below gives
$$W(x)\to 0\qquad\hbox{as }|x|\to\infty.$$
Then standard comparison principle gives that $W$ decays exponentially to zero at infinity and thus $W\in H^1(\mathbb{R}^N)$. The moving-plane theorem of Gidas, Ni, and Nirenberg \cite{GidasNiNirenberg} then shows that $W$ is radial and strictly decreasing about some point.  Since $W(0)=\max W$, that point is the origin.  Kwong's uniqueness theorem \cite{Kwong} therefore gives
\begin{equation}\label{eq:W=Wbeta}
W=W_\beta.
\end{equation}

\begin{lemma}[Positive localized energy of scalar radial limits]\label{lem:positive-localized-energy-scalar}
Let $a,b>0$, $2<p<2_k^*$, and suppose that $W\geq0$ is a nontrivial bounded radial nonincreasing classical solution of
$$-\Delta W+aW=bW^{p-1}\quad\hbox{in}~\mathbb{R}^k.$$
For
$$J_R^a(W):=\frac{1}{2}\int_{B_R^k}\bigl(|\nabla W|^2+aW^2\bigr)\,\mathrm{d}x
        -\frac{b}{p}\int_{B_R^k}W^p\,\mathrm{d}x,$$
there exists $R_0>0$ such that $J_R^a(W)>0$ for all $R\geq R_0$.
\end{lemma}
\begin{proof}
Since $W$ is bounded and radial nonincreasing, the limit
$L:=\lim_{r\to\infty}W(r)$ exists.  Interior elliptic estimates give
bounded first and second derivatives.  Thus $W'$ is uniformly continuous;
as $-W'\in L^1(1,+\infty)$, it follows that $W'(r)\to 0$.
Since $r^{-1}W'(r)\to0$ as $r\to \infty$, the radial equation shows that $W''(r)$ has the limit $aL-bL^{p-1}$.  A nonzero limit would contradict $W'(r)\to0$, and therefore $W''(r)\to0$ and $aL=bL^{p-1}$.  If $L>0$, the energy density at infinity equals
$$\frac{1}{2} aL^2-\frac{b}{p} L^p=\left(\frac{1}{2}-\frac{1}{p}\right)aL^2>0.$$
Thus $J_R^a(W)\to+\infty$ as $R\to+\infty$.
If $L=0$, choose $R$ so large that $bW^{p-2}\le a/2$ for $r\ge R$.  Then
$-\Delta W+(a/2)W\le0$ there.  For any $0<\gamma<\sqrt{a/2}$, comparison on
exterior annuli with a suitable multiple of $e^{-\gamma(r-R)}$ and passage to
the limit give $W(r)\le Ce^{-\gamma r}$.  Interior gradient estimates applied on
unit balls give $|\nabla W(r)|\le Ce^{-\gamma r}$ as well.  Thus
$W\in H^1(\mathbb{R}^k)$, so testing its equation by $W$ gives
$$
\int_{\mathbb{R}^k}\bigl(|\nabla W|^2+aW^2\bigr)\,\mathrm{d}x
=b\int_{\mathbb{R}^k}W^p\,\mathrm{d}x,
$$
and consequently
$$\lim_{R\to\infty}J_R^a(W)=\left(\frac{1}{2}-\frac{1}{p}\right)b\int_{\mathbb{R}^k}W^p\,\mathrm{d}x>0.$$
This proves the lemma.
\end{proof}

\begin{lemma}[Localized Hamiltonian exclusion]\label{L:localized-Hamiltonian-exclusion}
Let $a,b>0$, $2<p<2_k^*$, and let
$V\in C^2(\mathbb{R}^k\times\mathbb{R})$ be bounded and nonnegative.  Suppose
$$
 -\Delta_YV-V_{\tau\tau}+aV=bV^{p-1},
$$
$V$ is radial nonincreasing in $|Y|$ and nonincreasing in $\tau$, and its
pointwise endpoint limits satisfy
$$
 V_+(Y):=\lim_{\tau\to+\infty}V(Y,\tau)\equiv0,
 \qquad
 V_-(Y):=\lim_{\tau\to-\infty}V(Y,\tau)\not\equiv0.
$$
Then no such $V$ exists.
\end{lemma}

\begin{proof}
Translations in $\tau$ and interior elliptic compactness show that $V_\pm$
are bounded radial nonincreasing classical solutions of
$$
 -\Delta_YW+aW=bW^{p-1}\quad\hbox{in }\mathbb{R}^k.
$$
For $R>0$ set
$$
 J_R^a(w)=\frac{1}{2}\int_{B_R^k}(|\nabla w|^2+aw^2)\mathrm{d}Y
 -\frac{b}{p}\int_{B_R^k}w^p\mathrm{d}Y,
 \qquad
 H_R(\tau)=\frac{1}{2}\int_{B_R^k}|V_\tau|^2\mathrm{d}Y-J_R^a(V(\cdot,\tau)).
$$
An integration by parts on the fixed ball gives
$$
 H_R'(\tau)=-\int_{\partial B_R^k}V_\tau\partial_\nu V\,\mathrm dS\le0,
$$
because both factors in the integrand are nonpositive.  Interior estimates,
uniform under translations in $\tau$, give
$\sup_{\mathbb{R}\times B_R}|V_\tau|\le C_R$.  By monotonicity and Tonelli's theorem,
$$
 \int_{\mathbb{R}}\int_{B_R}|V_\tau|^2
 \le C_R\int_{B_R}\int_{\mathbb{R}}(-V_\tau)
 =C_R\int_{B_R}(V_--V_+)<\infty.
$$
Thus there are sequences $s_j\to+\infty$ and $t_j\to-\infty$ for which the
$L^2(B_R)$-norms of $V_\tau(\cdot,s_j)$ and $V_\tau(\cdot,t_j)$ tend to zero.
Local convergence to $V_\pm$ then gives
$$
 \lim_{\tau\to+\infty}H_R(\tau)=0,
 \qquad
 \lim_{\tau\to-\infty}H_R(\tau)=-J_R^a(V_-).
$$
The limits are limits of the entire monotone function $H_R$, not merely of the chosen subsequences.  Hence $-J_R^a(V_-)\ge0$ for every $R$.  This contradicts Lemma~\ref{lem:positive-localized-energy-scalar}, which gives
$J_R^a(V_-)>0$ for all sufficiently large $R$.
\end{proof}

\subsection{Uniform localization after isotropic rescaling}

\begin{lemma}[Uniform smallness at infinity for the high-frequency scaling]
\label{lem:uniform-smallness-high-frequency}
We have
$$\lim_{R\to+\infty}\sup_n\sup_{|\xi|\geq R}\widehat{u}_n(\xi)=0.$$
\end{lemma}

\begin{proof}
We argue by contradiction.  Suppose that the conclusion is false.  Then there are
$\varepsilon_0>0$, a subsequence still denoted by $\widehat{u}_n$, and points
$$\xi_n=(Y_n,Z_n)\in\mathbb{R}^d\times\mathbb{R}^m,\qquad |\xi_n|\to+\infty,$$
such that
\begin{equation}\label{eq:20260618-1212-revised}
        \widehat{u}_n(\xi_n)\geq\varepsilon_0 .
\end{equation}
Discarding finitely many indices causes no loss: each fixed
$\widehat{u}_n\in C_0(\mathbb{R}^N)$ has a vanishing tail, so those indices can be
absorbed by enlarging the final radius.  Thus every subsequence used below may be
assumed to lie in the asymptotic regime of the preceding uniform estimates. Write $\widehat{u}_n(Y,Z)=\widehat{U}_n(|Y|,|Z|)$.
Put $\widehat F_n:=F_{\lambda_n^{1/(\beta-2)}}$.

\smallskip
\noindent\textbf{Case 1: $|Z_n|\to\infty$.}
By monotonicity,
$$\widehat{U}_n(0,R_n)\geq \varepsilon_0,\qquad R_n:=|Z_n|\to+\infty.$$
Choose $0<\varepsilon<\varepsilon_0$ such that
\begin{equation}\label{eq:20260618-1231-revised}
        \mu_\infty\varepsilon^{\beta-2}<1 .
\end{equation}
Since each fixed $\widehat{u}_n$ belongs to $C_0(\mathbb{R}^N)$ (the
rescaling preserves the decay in Lemma~\ref{L:individual-decay}), there is $r_n\geq R_n$ such that
$$\widehat{u}_n(0,r_n)=\varepsilon,\qquad
        \widehat{U}_n(0,r)\leq\varepsilon\quad\hbox{for}~r\geq r_n.$$
Set
$$V_n(Y,\tau):=\widehat{u}_n(|Y|,r_n+\tau),
        \qquad (Y,\tau)\in\mathbb{R}^d\times(-r_n,+\infty).$$
Then
$$
-\Delta_YV_n-\partial_{\tau\tau}V_n
-\frac{m-1}{r_n+\tau}\partial_\tau V_n
+V_n+\lambda_n^{-2}|Y|^2V_n=\widehat{F}_n(V_n).
$$
By Lemma~\ref{lemma:20260618-1129} and standard interior estimates,
after passing to a subsequence,
$$V_n\to V\quad\hbox{in}~C^2_{\mathrm{loc}}(\mathbb{R}^d\times\mathbb{R}),$$
where
\begin{equation}\label{E:high-tail-limit-Z-revised}
        -\Delta_YV-\partial_{\tau\tau}V+V=\mu_\infty V^{\beta-1}
        \quad\hbox{in}~\mathbb{R}^d\times\mathbb{R} .
\end{equation}
Moreover
$$ V(0,0)=\varepsilon,\qquad 0\leq V(Y,\tau)\leq\varepsilon \quad\hbox{for}~\tau\geq0,$$
and $V$ is nonincreasing in $\tau$ and in $|Y|$.  Hence
$V_+(Y):=\lim_{\tau\to+\infty}V(Y,\tau)$ exists and $0\le V_+\le\varepsilon$.  To see that
$V_+$ satisfies the limiting elliptic equation, take any sequence $\tau_j\to+\infty$ and
consider $V_j(Y,\tau):=V(Y,\tau+\tau_j)$.  By local elliptic estimates, a subsequence
converges in $C^2_{\rm loc}$.  The monotonicity in $\tau$ forces the limit to be
independent of $\tau$ and equal to $V_+(Y)$.  Passing to the limit in
\eqref{E:high-tail-limit-Z-revised} gives
$$
        -\Delta_YV_+ +V_+=\mu_\infty V_+^{\beta-1}\quad\hbox{in }\mathbb{R}^d,
        \qquad 0\leq V_+\leq\varepsilon .
$$
Since $V_+$ is radial nonincreasing, it attains its maximum at $Y=0$.  Evaluating at
$Y=0$ and using \eqref{eq:20260618-1231-revised}, we obtain $V_+\equiv0$.
The backward limit $V_-(Y):=\lim_{\tau\to-\infty}V(Y,\tau)$ exists by
monotonicity and is nontrivial, since $V_-(0)\ge V(0,0)=\varepsilon$.  Translation
compactness shows that it solves the corresponding stationary equation.

The endpoint limits satisfy the hypotheses of Lemma~\ref{L:localized-Hamiltonian-exclusion} with $k=d$, $a=1$, $b=\mu_\infty$ and $p=\beta$.  That lemma rules out Case~1.

\smallskip
\noindent\textbf{Case 2: $|Y_n|\to\infty$.}
By monotonicity,
$$
 \widehat U_n(R_n,0)\ge\varepsilon_0,\qquad
 R_n:=|Y_n|\to\infty.
$$
Choose $\varepsilon$ as in \eqref{eq:20260618-1231-revised}.
The fixed-$n$ decay and monotonicity give $r_n\ge R_n$ such that
$$
 \widehat U_n(r_n,0)=\varepsilon,\qquad
 \widehat U_n(r,0)\le\varepsilon\quad(r\ge r_n).
$$
We first show that $r_n/\lambda_n$ is bounded.  Set
$f_n(Y):=\widehat u_n(Y,0)$.  Since $Z=0$ is a maximum in the free
variables, $-\Delta_Z\widehat u_n(Y,0)\ge0$, and hence
$$
 -\Delta_Y f_n+\left(1+\frac{|Y|^2}{\lambda_n^2}\right)f_n
 \le\widehat F_n(f_n).
$$
By \eqref{eq:20260618-1235} and Lemma~\ref{lemma:20260618-1129},
$0\le f_n\le M$ and $0\le\widehat F_n(f_n)\le C$ uniformly in $n$.
Moreover, $f_n\ge\varepsilon$ on $B_{r_n}^d$ by monotonicity.
Choose $0\le\eta\in C_c^\infty(\{1/2<|Y|<1\})$ with
$\int_{\mathbb R^d}\eta \mathrm{d}Y=1$, and set
$\eta_n(Y)=r_n^{-d}\eta(Y/r_n)$. Then $\int_{\mathbb{R}^d}\eta_n \mathrm{d}Y=1$. Noting that
$$1+\frac{|Y|^2}{\lambda_n^2}\geq \frac{r_n^2}{4\lambda_n^2}, \quad f_n\geq \varepsilon, \qquad (\frac{1}{2}r_n\leq |Y|\leq r_n),$$
testing against $\eta_n$ gives
\begin{align*}
 \frac{\varepsilon r_n^2}{4\lambda_n^2}
 &\le\int_{\mathbb R^d}
       \left(1+\frac{|Y|^2}{\lambda_n^2}\right)f_n\eta_n\\
 &\le\int_{\mathbb R^d}\widehat F_n(f_n)\eta_n
       +\int_{\mathbb R^d}f_n\Delta\eta_n\\
 &\le C+\frac{M}{r_n^2}\|\Delta\eta\|_{L^1}.
\end{align*}
So there exists some $C_\varepsilon>0$ such that
$$r_n/\lambda_n\leq C_\varepsilon,\qquad \forall n\in \mathbb{N}.$$
Up to a subsequence, $r_n/\lambda_n\to\ell\in[0,\infty)$.

Define
$$
 V_n(\tau,Z):=\widehat U_n(r_n+\tau,|Z|),
 \qquad (\tau,Z)\in(-r_n,\infty)\times\mathbb R^m.
$$
Then
\begin{equation}\label{E:high-tail-translated-Y-equation}
 -\partial_{\tau\tau}V_n
 -\frac{d-1}{r_n+\tau}\partial_\tau V_n
 -\Delta_ZV_n
 +\left(1+\lambda_n^{-2}(r_n+\tau)^2\right)V_n
 =\widehat F_n(V_n).
\end{equation}
The coefficients converge locally uniformly, and interior elliptic
estimates yield, after extraction, $V_n\to V$ in $C^2_{\rm loc}$, where
\begin{equation}\label{E:high-tail-limit-Y-revised}
 -\partial_{\tau\tau}V-\Delta_ZV+(1+\ell^2)V
 =\mu_\infty V^{\beta-1}
 \qquad\hbox{in }\mathbb R\times\mathbb R^m.
\end{equation}
Moreover, $V(0,0)=\varepsilon$, $0\le V\le\varepsilon$ for $\tau\ge0$,
and $V$ is nonincreasing in $\tau$ and in $|Z|$.
As in Case~1, the limits $V_\pm(Z)=\lim_{\tau\to\pm\infty}V(\tau,Z)$
solve
$$
 -\Delta_ZV_\pm+(1+\ell^2)V_\pm=\mu_\infty V_\pm^{\beta-1}.
$$
Since $0\le V_+\le\varepsilon$ and
$\mu_\infty\varepsilon^{\beta-2}<1\le1+\ell^2$, evaluation at its
maximum point $Z=0$ gives $V_+\equiv0$, whereas $V_-(0)\ge\varepsilon$.
Lemma~\ref{L:localized-Hamiltonian-exclusion}, with $k=m$,
$a=1+\ell^2$, $b=\mu_\infty$, and $p=\beta$, now rules out Case~2.

The contradiction proves
$$\lim_{R\to+\infty}\sup_n\sup_{|\xi|\geq R}\widehat{u}_n(\xi)=0.$$
\end{proof}

\subsection{Identification of the full-dimensional profile}

\begin{proof}[Proof of Theorem \ref{T:asymptotics}(ii)]
We first record the relative small-argument estimate that is needed for the exterior comparison.  Put $a_n:=\lambda_n^{1/(\beta-2)}$.  For every $\varepsilon>0$, by \eqref{eq:20260618-1130}, there are $\delta>0$ and $n_\varepsilon$ such that
\begin{equation}\label{E:high-rescaled-relative-smallness}
 0\leq s\leq\delta,\quad n\geq n_\varepsilon
 \quad\Longrightarrow\quad
 \frac{g(a_ns)}{a_n^{\beta-1}}\leq\varepsilon s.
\end{equation}

Combining \eqref{E:rescaled-high-eq},
\eqref{E:high-rescaled-relative-smallness} with $\varepsilon=1/2$, and
Lemma~\ref{lem:uniform-smallness-high-frequency}, we find $R_0>0$ such that
$$
 -\Delta\widehat{u}_n+\frac{1}{2}\widehat{u}_n\le0
 \qquad\hbox{on }\{|x|>R_0\}
$$
for all sufficiently large $n$.  The boundary values at $|x|=R_0$ are
uniformly bounded by \eqref{eq:20260618-1235}.  Applying
Lemma~\ref{L:common-exterior-barrier}(iii) gives constants $C_1,C_2>0$,
independent of $n$, such that
\begin{equation}\label{eq:20260618-1344}
 \widehat{u}_n(x)\le C_1e^{-C_2|x|},\qquad |x|\ge R_0 .
\end{equation}
Recalling \eqref{eq:W=Wbeta}, the uniform exponential decay
\eqref{eq:20260618-1344} gives $\widehat{u}_n\to W_\beta$ in
$L^q(\mathbb{R}^N)$ for every $1\leq q<+\infty$.

By \eqref{eq:20260618-1344} and $\lambda_n\to +\infty$, the same domination gives
$$\int_{\mathbb{R}^N}\lambda_{n}^{-2}|Y|^2\widehat{u}_n^2\mathrm{d}x=o(1).$$
The dominated convergence theorem gives
$$\int_{\mathbb{R}^N} \frac{g(\lambda_n^{1/(\beta-2)}\widehat{u}_n)}
        {\lambda_n^{(\beta-1)/(\beta-2)}} \widehat{u}_n \mathrm{d}x\rightarrow \mu_\infty \int_{\mathbb{R}^N} W_{\beta}^{\beta} \mathrm{d}x.$$
Testing \eqref{E:rescaled-high-eq} by $\widehat{u}_n$ gives
\begin{equation}\label{eq:20260618-1353}
\|\widehat{u}_n\|_{H^1(\mathbb{R}^N)}^{2}\rightarrow \mu_\infty \int_{\mathbb{R}^N} W_{\beta}^{\beta} \mathrm{d}x.
\end{equation}
Thus the weak limit of $\{\widehat{u}_n\}$ in $H^1(\mathbb{R}^N)$ must be $W_\beta$. On the other hand, testing \eqref{E:Wbeta} by $W_\beta$ gives
$$\|W_\beta\|_{H^1(\mathbb{R}^N)}^{2}=\mu_\infty \int_{\mathbb{R}^N} W_{\beta}^{\beta} \mathrm{d}x.$$
Hence,
$$\widehat{u}_n\rightharpoonup W_\beta ~\hbox{in}~ H^1(\mathbb{R}^N), ~\hbox{and}~ \|\widehat{u}_n\|_{H^1(\mathbb{R}^N)}^{2}\rightarrow \|W_\beta\|_{H^1(\mathbb{R}^N)}^{2},$$
and therefore $\widehat{u}_n\to W_\beta$ strongly in $H^1(\mathbb{R}^N)$.

The preceding argument applies to every subsequence.  Since the only possible
limit is the unique centered ground state $W_\beta$, the subsequence principle shows
that the original sequence, rather than merely a selected subsequence, converges in
all the asserted topologies.

Finally,
$$\norm{u_n}_2^2
        =\lambda_n^{\frac{2}{\beta-2}}\lambda_n^{-N/2}
        \int_{\R^N}\widehat{u}_n^2\mathrm{d}\xi,$$
which proves \eqref{E:high-mass} and \eqref{E:high-limit}.
\end{proof}

\section{The local branch, nondegeneracy, and degree at the spectral edge}\label{sec:LUTD}

This section proves Theorem~\ref{T:local-degree}.  We now use the
Lyapunov--Schmidt reduction not only for the index computation, but also for the
local construction of the small solution.

Throughout this section we write
$$
        \mu:=\lambda+d>0 .
$$
For a solution $u$ of \eqref{E:fixed-lambda} with $\mu>0$ small, we use the lower-edge rescaling
\begin{equation}\label{E:sec7-rescale}
        U(y,\zeta):=\mu^{-\frac1{\alpha-2}}
        u\left(y,\frac{\zeta}{\sqrt\mu}\right).
\end{equation}
Then $U$ satisfies
\begin{equation}\label{E:sec7-rescaled-eq}
        (H_y-d)U+\mu(-\Delta_\zeta U+U)=\mu \mathfrak{g}_\mu(U),
\end{equation}
where, in the notation of Lemma~\ref{lemma:20260618-1129},
\begin{equation}\label{E:sec7-gmu}
        \mathfrak{g}_\mu(s):=F_{\mu^{1/(\alpha-2)}}^{(0)}(s)=
        \frac{g\bigl(\mu^{\frac1{\alpha-2}}s\bigr)}
        {\mu^{\frac{\alpha-1}{\alpha-2}}},
        \qquad s\geq0 .
\end{equation}
Applying Lemma~\ref{lemma:20260618-1129} with
$a=\mu^{1/(\alpha-2)}$, for every bounded interval of $s\ge0$,
\begin{equation}\label{E:sec7-gmu-conv}
        \mathfrak{g}_\mu(s)\to \mu_0s^{\alpha-1},
        \qquad
        \mathfrak{g}_\mu'(s)\to \mu_0(\alpha-1)s^{\alpha-2}
\end{equation}
uniformly as $\mu\to0^+$.

For the construction we use the positive-part extension $g_+$ introduced in
\eqref{E:positive-part-extension}, and define
\begin{equation}\label{E:sec7-gmu-plus}
        \mathfrak{g}_\mu^+(s):=
        \frac{g_+\bigl(\mu^{\frac1{\alpha-2}}s\bigr)}
        {\mu^{\frac{\alpha-1}{\alpha-2}}},
        \qquad s\in\mathbb{R} .
\end{equation}
On $[0,+\infty)$, $\mathfrak{g}_\mu^+=\mathfrak{g}_\mu$.  The corresponding limiting nonlinearity is
$$
        \mathfrak{g}_0^+(s):=\mu_0(s^+)^{\alpha-1}.
$$
Near the positive profile $\phi_0V_\alpha$, the derivative of this limiting map is the
same as the derivative of $\mu_0s^{\alpha-1}$, because $\phi_0V_\alpha>0$ pointwise.

Let $P$ be the orthogonal projection in $L^2(\mathbb{R}^d_y)$ onto
$\operatorname{span}\{\phi_0\}$,
\begin{equation}\label{E:sec7-PQ}
        Pw=\phi_0(y)\int_{\mathbb{R}^d}w(\eta,\zeta)\phi_0(\eta)\,\mathrm{d}\eta,
        \qquad Q:=I-P.
\end{equation}
We shall repeatedly use the spectral gap
\begin{equation}\label{E:sec7-gap}
        \int_{\mathbb{R}^d}
        \bigl(|\nabla_y r|^2+(|y|^2-d)r^2\bigr)\,\mathrm{d} y
        \ge 2\int_{\mathbb{R}^d}r^2\,\mathrm{d} y,
        \qquad r\perp\phi_0 .
\end{equation}

For the Lyapunov--Schmidt construction below, we work in the Banach spaces
$$
\begin{aligned}
        \mathcal Y
        &:={\bigl\{}v\in H^1_{\rm rad}(\mathbb{R}^m)\cap L^\infty(\mathbb{R}^m):
        \|v\|_{\mathcal Y}:=\|v\|_{H^1(\mathbb{R}^m)}
        +\|v\|_{L^\infty(\mathbb{R}^m)}<\infty{\bigr\}},\\
        \mathcal Z
        &:={\bigl\{}r\in X_r\cap L^\infty(\mathbb{R}^N):\ Pr=0,
        \|r\|_{\mathcal Z}:=\|r\|_X
        +\|r\|_{L^\infty(\mathbb{R}^N)}<\infty{\bigr\}}.
\end{aligned}
$$

\subsection{Local existence near the lower edge}

\begin{theorem}[Lyapunov--Schmidt existence near the lower edge]\label{T:sec7-LS-existence}
Assume \ref{gcon1} and \ref{gcon2}.
There exists $\mu_e>0$ such that, for every $0<\lambda+d<\mu_e$,
problem \eqref{E:fixed-lambda} has at least one positive solution
$u_\lambda\in\mathcal{K}$.  Moreover, after the rescaling \eqref{E:sec7-rescale}, this
solution satisfies
\begin{equation}\label{E:sec7-LS-profile-conv}
        U_\mu\to \phi_0V_\alpha
        \qquad\hbox{in }H^1(\mathbb{R}^d\times\mathbb{R}^m)
        \quad\hbox{as }\mu\to0^+ .
\end{equation}
The constructed family can be chosen so that
$\mu\mapsto(v_\mu,R_\mu(v_\mu))$ is $C^1$ into
$\mathcal Y\oplus\mathcal Z$ for $\mu>0$, and
$U_\mu:=\phi_0v_\mu+R_\mu(v_\mu)$ extends continuously to $\mu=0$
after setting $U_0=\phi_0V_\alpha$.
\end{theorem}

\begin{proof}
We first solve the extended rescaled problem
\begin{equation}\label{E:sec7-extended-rescaled}
        (H_y-d)U+\mu(-\Delta_\zeta U+U)=\mu \mathfrak{g}_\mu^+(U)
        \quad\hbox{in }\mathbb{R}^d\times\mathbb{R}^m .
\end{equation}
Write
\begin{equation}\label{E:sec7-LS-decomp}
        U=\phi_0(y)v(\zeta)+r(y,\zeta),
        \qquad Pr=0 .
\end{equation}
After dividing \eqref{E:sec7-extended-rescaled} by $\mu$, the projected system is
\begin{equation}\label{E:sec7-P-equation-existence}
        -\Delta_\zeta v+v
        =
        \int_{\mathbb{R}^d}\phi_0(y)\mathfrak{g}_\mu^+(\phi_0(y)v+r(y,\cdot))\,\mathrm{d} y,
\end{equation}
and
\begin{equation}\label{E:sec7-Q-equation-existence}
        \left\{\frac1\mu(H_y-d)-\Delta_\zeta+1\right\}r
        =Q\mathfrak{g}_\mu^+(\phi_0v+r).
\end{equation}
On the $Q$-component, set
$$
        \mathcal{A}_\mu:=
        Q\left\{\frac1\mu(H_y-d)-\Delta_\zeta+1\right\}Q .
$$
By \eqref{E:sec7-gap},
\begin{equation}\label{E:sec7-Acoercive-existence}
\begin{aligned}
        \langle\mathcal{A}_\mu r,r\rangle
        &\ge
        \frac2\mu\|r\|_{L^2(\mathbb{R}^d\times\mathbb{R}^m)}^2
        +\int_{\mathbb{R}^d\times\mathbb{R}^m}
        \bigl(|\nabla_\zeta r|^2+r^2\bigr)\,\mathrm{d} y\mathrm{d}\zeta .
\end{aligned}
\end{equation}
Thus $\mathcal{A}_\mu$ is invertible on the $Q$-component and its inverse is small
in the estimates relevant to the $Q$-equation as $\mu\to0^+$.

Let $\mathcal V$ be a small neighborhood of $V_\alpha$ in $\mathcal Y$, and
let $\mathcal{R}$ be a small ball centered at $0$ in $\mathcal Z$.
Lemma~\ref{lemma:20260622-1443} below shows that, for every sufficiently small
$\mu>0$ and each $v\in\mathcal V$, the $Q$-equation has a unique solution
\begin{equation}\label{E:sec7-Rmu-existence}
        r=R_\mu(v)\in\mathcal{R} .
\end{equation}
Moreover, $R_\mu$ is $C^1$ and
\begin{equation}\label{E:sec7-Rmu-small-existence}
        \|R_\mu(v)\|_{\mathcal Z}=o(1),
        \qquad
        \|DR_\mu(v)\|_{\mathcal L(\mathcal Y,\mathcal Z)}=o(1)
\end{equation}
uniformly for $v\in\mathcal V$ as $\mu\to0^+$.

Substituting $r=R_\mu(v)$ into \eqref{E:sec7-P-equation-existence}, we get the
reduced equation
\begin{equation}\label{E:sec7-reduced-equation-existence}
        \mathcal G_\mu(v):=
        v-(-\Delta_\zeta+1)^{-1}
        \int_{\mathbb{R}^d}\phi_0(y)
        \mathfrak{g}_\mu^+(\phi_0(y)v+R_\mu(v))\,\mathrm{d} y=0 .
\end{equation}
The strengthened form of Lemma~\ref{lemma:20260618-1129}, with
$a=\mu^{1/(\alpha-2)}$, gives numbers $\varepsilon_\mu\to0$ such that, on
every fixed bounded interval,
\begin{equation}\label{E:sec7-relative-C1-remainder}
 \left|\mathfrak{g}_\mu^+(s)-\mathfrak{g}_0^+(s)\right|
 \le\varepsilon_\mu|s|^{\alpha-1},
 \qquad
 \left|(\mathfrak{g}_\mu^+)'(s)-(\mathfrak{g}_0^+)'(s)\right|
 \le\varepsilon_\mu|s|^{\alpha-2}.
\end{equation}
Let $E:=L^2(\mathbb R^N)\cap L^\infty(\mathbb R^N)$,
endowed with the norm $\|\cdot\|_E:=\|\cdot\|_2+\|\cdot\|_\infty$.
On the bounded neighborhood under consideration, $w=\phi_0v+r$ is uniformly
bounded in $E$, i.e., there exists $M>0$ such that
$$\|w\|_E\leq M,\quad w=\phi_0v+r,\qquad (v,r)\in \mathcal{V}\times \mathcal{R}.$$
The associated Nemytskii operator  is defined by
$$\mathcal{N}_\mu(w)(x):=\mathfrak{g}_\mu^+(w(x)).$$
By
$$\|(\mathfrak{g}_\mu^+-\mathfrak{g}_0^+)(w)\|_2\leq \varepsilon_\mu\|w\|_\infty^{\alpha-2}\|w\|_2~\hbox{and}~
\|(\mathfrak{g}_\mu^+-\mathfrak{g}_0^+)(w)\|_\infty\leq \varepsilon_\mu\|w\|_\infty^{\alpha-2}\|w\|_\infty,
$$
$$
 \sup_{\|w\|_E\le M}\|\mathcal N_\mu(w)-\mathcal N_0(w)\|_E
 \le\varepsilon_\mu M^{\alpha-1}\longrightarrow0.
$$
By Lemma~\ref{L:sec7-parameter-regularity}(ii) below,
the Fr\'echet derivative of $\mathcal N_\mu$, for $\mu>0$, is given by
$$D \mathcal{N}_\mu(w)[h]=(\mathfrak{g}_\mu^+)'(w)h.$$
So the second estimate in \eqref{E:sec7-relative-C1-remainder} gives
\begin{align*}
\|(D\mathcal{N}_\mu(w)-D\mathcal{N}_0(w)) [h]\|_E\leq& \|(\mathfrak{g}_0^+)'(w)-(\mathfrak{g}_\mu^+)'(w)\|_\infty \|h\|_E\\
\leq&\varepsilon_\mu M^{\alpha-2} \|h\|_E,
\end{align*}
and thus
$$\|D\mathcal{N}_\mu(w)-D\mathcal{N}_0(w)\|_{\mathcal{L}(E,E)}\leq \varepsilon_\mu M^{\alpha-2}.$$
The same derivative formula holds for
$\mathcal N_0(w)=\mu_0(w^+)^{\alpha-1}$.
Thus $\mathcal N_\mu\to\mathcal N_0$ in $C^1$ uniformly on bounded
subsets of $E$.  Moreover, the uniform continuity of
$(\mathfrak g_0^+)'$ on bounded intervals and
\eqref{E:sec7-Rmu-small-existence} give
$$
 \sup_{v\in\mathcal V}
 \|D\mathcal N_0(\phi_0v+R_\mu(v))-D\mathcal N_0(\phi_0v)\|
 _{\mathcal L(E,E)}\longrightarrow0.
$$
The map
$$
 h(y,\zeta)\longmapsto
 \int_{\mathbb{R}^d}\phi_0(y)h(y,\zeta)\,\mathrm dy
$$
is bounded, on block-radial inputs, from $E$ to
$L^2_{\rm rad}(\mathbb R^m)\cap L^\infty(\mathbb R^m)$, and
$(-\Delta_\zeta+1)^{-1}$ is bounded from the latter space to $\mathcal Y$.
Together with \eqref{E:sec7-Rmu-small-existence}, these observations give
\begin{equation}\label{E:sec7-Gmu-C1-conv}
        \mathcal G_\mu\to\mathcal G_0
        \quad\hbox{in }C^1\hbox{ on a neighborhood of }V_\alpha\hbox{ in }\mathcal{Y},
\end{equation}
where
$$\mathcal{G}_0(v)=v-(-\Delta_\zeta+1)^{-1}\bigl(c_\alpha(v^+)^{\alpha-1}\bigr).$$
Since $V_\alpha>0$ and solves
$-\Delta V_\alpha+V_\alpha=c_\alpha V_\alpha^{\alpha-1}$, we have
$\mathcal G_0(V_\alpha)=0$.  The derivative at $V_\alpha$ is
$$
        D\mathcal G_0(V_\alpha)
        =I-(-\Delta+1)^{-1}
        \bigl(c_\alpha(\alpha-1)V_\alpha^{\alpha-2}\,\cdot\bigr).
$$
This operator is invertible on $H^1_{\mathrm{rad}}(\mathbb{R}^m)$, because it is
equivalent to the radial linearized operator
$$
        L_\alpha=-\Delta+1-c_\alpha(\alpha-1)V_\alpha^{\alpha-2},
$$
whose kernel in the radial space is trivial.  We shall use the implicit function theorem in the space $\mathcal{Y}$.  The derivative
above is an isomorphism on $\mathcal{Y}$ as well.  Indeed, write
$$
        K\xi:=(-\Delta+1)^{-1}
        \bigl(c_\alpha(\alpha-1)V_\alpha^{\alpha-2}\xi\bigr).
$$
Since $V_\alpha$ decays exponentially, multiplication by
$V_\alpha^{\alpha-2}$ localizes the function.  More precisely, if
$\{\xi_n\}$ is bounded in $\mathcal Y$, then, after passing to a subsequence,
$\xi_n$ converges in $L^2_{\rm loc}$ and almost everywhere.  The common
$L^\infty$ bound and the exponential decay of
$V_\alpha^{\alpha-2}$ show that
$V_\alpha^{\alpha-2}\xi_n$ converges strongly in $L^2$ and, by interpolation,
in $L^p$ for every finite $p\geq2$.  The resolvent
$(-\Delta+1)^{-1}$ maps $L^2$ boundedly into $H^2$ and maps $L^p$ into
$W^{2,p}$; choosing $p>m/2$ and using Sobolev embedding gives convergence
both in $H^1$ and in $L^\infty$.  Thus $K$ is compact on $\mathcal{Y}$.  Hence
$I-K=D\mathcal{G}_0(V_\alpha)$ is Fredholm of index zero on $\mathcal{Y}$.  Its kernel
is contained in $H^1_{\mathrm{rad}}(\mathbb{R}^m)$ and is therefore trivial by the radial nondegeneracy of $V_\alpha$.  Consequently
$D\mathcal{G}_0(V_\alpha):\mathcal{Y}\to\mathcal{Y}$ is an isomorphism.  The implicit function theorem therefore gives, for all sufficiently small $\mu>0$, a unique solution $v_\mu$ of \eqref{E:sec7-reduced-equation-existence} in $\mathcal{V}$, with
\begin{equation}\label{E:sec7-vmu-conv}
        v_\mu\to V_\alpha
        \qquad\hbox{in }\mathcal Y.
\end{equation}
Let
\begin{equation}\label{E:sec7-Umu-def}
        U_\mu:=\phi_0v_\mu+R_\mu(v_\mu).
\end{equation}
Then $U_\mu$ solves the extended rescaled equation \eqref{E:sec7-extended-rescaled},
and by \eqref{E:sec7-Rmu-small-existence} and \eqref{E:sec7-vmu-conv},
$$
        U_\mu\to\phi_0V_\alpha
        \qquad\hbox{in }H^1(\mathbb{R}^d\times\mathbb{R}^m).
$$
By Lemmas~\ref{L:sec7-parameter-regularity} and
\ref{lemma:20260622-1443}, $(\mu,v)\mapsto R_\mu(v)$ is jointly $C^1$
for $0<\mu<\mu_R$, and so is the reduced map $(\mu,v)\mapsto\mathcal G_\mu(v)$.
In view of \eqref{E:sec7-Gmu-C1-conv} and \eqref{E:sec7-vmu-conv},
$D_v\mathcal G_\mu(v_\mu)$ remains invertible for all sufficiently small
$\mu>0$.  The parameter-dependent implicit function theorem therefore
makes $\mu\mapsto v_\mu$ a $C^1$ map from $(0,\mu_e)$ to $\mathcal Y$,
after reducing $\mu_e$ if necessary.  Composition with $R_\mu$ proves the asserted
$C^1$ dependence of $(v_\mu,R_\mu(v_\mu))$.  The convergences
\eqref{E:sec7-Rmu-small-existence} and \eqref{E:sec7-vmu-conv} give its
continuous extension to $(V_\alpha,0)$ at $\mu=0$; no differentiability
at this singular endpoint is asserted.

Returning to the original variables, define
\begin{equation}\label{E:sec7-ulambda-from-Umu}
        \widetilde{u}_\lambda(y,z)
        :=\mu^{\frac1{\alpha-2}}U_\mu(y,\sqrt\mu z),
        \qquad \mu=\lambda+d .
\end{equation}
Then $\widetilde{u}_\lambda\in X_r$ solves the extended equation
\begin{equation}\label{E:sec7-extended-original}
        -\Delta u+|y|^2u+\lambda u=g_+(u)
        \quad\hbox{in }\mathbb{R}^N .
\end{equation}
Let
$\widetilde{u}_\lambda^-:=\max\{-\widetilde{u}_\lambda,0\}$.  Testing
\eqref{E:sec7-extended-original} by $-\widetilde{u}_\lambda^-$, and using the fact
that $g_+(s)=0$ for $s\le0$, yields
$$
        \int_{\mathbb{R}^N}
        \bigl(|\nabla \widetilde{u}_\lambda^-|^2
        +|y|^2(\widetilde{u}_\lambda^-)^2
        +\lambda(\widetilde{u}_\lambda^-)^2\bigr)\,\mathrm{d} x=0 .
$$
By the spectral inequality $H+\lambda\ge d+
\lambda=\mu>0$, we obtain
$\widetilde{u}_\lambda^-=0$.  Hence $\widetilde{u}_\lambda\ge0$.  Since
$U_\mu\to\phi_0V_\alpha\not\equiv0$, the function $\widetilde{u}_\lambda$ is not
identically zero.  The strong maximum principle then gives
$\widetilde{u}_\lambda>0$ in $\mathbb{R}^N$.  Consequently
$g_+(\widetilde{u}_\lambda)=g(\widetilde{u}_\lambda)$, and
$\widetilde{u}_\lambda$ solves the original equation \eqref{E:fixed-lambda}.

The solution $\widetilde u_\lambda\in X_r$ was constructed in the centered
radial subspace.  If $m\ge2$, positivity immediately gives
$\widetilde u_\lambda\in\mathcal K$.  If $m=1$, Theorem~A supplies
monotonicity about a free-variable symmetry center.  Since
$\widetilde u_\lambda$ is even about $z=0$ and decays at infinity, that center
must be $0$; otherwise the function would be periodic.  Hence
$\widetilde u_\lambda$ is nonincreasing in both $|y|$ and $|z|$, and again
$\widetilde u_\lambda\in\mathcal K$.

We set $u_\lambda:=\widetilde{u}_\lambda$.
This proves the existence of a positive solution in $\mathcal{K}$.  The convergence
\eqref{E:sec7-LS-profile-conv} follows from the construction.
\end{proof}

\begin{lemma}[Parameter regularity for the lower-edge reduction]
\label{L:sec7-parameter-regularity}
Assume \ref{gcon1} and \ref{gcon2}.  Write
$$
 E:=L^2(\mathbb R^N)\cap L^\infty(\mathbb R^N),
 \qquad \|h\|_E:=\|h\|_2+\|h\|_\infty,
 \qquad T_\mu:=\mathcal A_\mu^{-1}Q.
$$
Then the following assertions hold.
\begin{enumerate}[label=(\roman*)]
\item The map
$$
 (0,\infty)\ni\mu\longmapsto T_\mu
 \in\mathcal L(E,X\cap L^\infty)
$$
is $C^1$, where the target space has norm $\|\cdot\|_X+\|\cdot\|_\infty$.
It is also $C^1$ as a map into $\mathcal L(L^2,X)$.
Both $T_\mu$ and its parameter derivative preserve block-radiality.
\item The Nemytskii map
$$
 \mathcal N:(0,\infty)\times E\longrightarrow E,
 \qquad \mathcal N(\mu,w):=\mathfrak g_\mu^+(w),
$$
is jointly $C^1$.  Consequently,
$$
 \mathfrak F_\mu(v,r):=r-T_\mu\mathfrak g_\mu^+(\phi_0v+r)
$$
defines a jointly $C^1$ map from
$(0,\infty)\times\mathcal Y\times\mathcal Z$ to $\mathcal Z$.
\end{enumerate}
\end{lemma}

\begin{proof}
Fix $I=[\underline\mu,\overline\mu]\subset(0,\infty)$ and put
$L_y=H_y-d$.  To prove (i), choose $p>\max\{2,N/2\}$ and set
$$
 F_p=L^2(\mathbb R^N)\cap L^p(\mathbb R^N),\qquad V(y)=1+|y|^2,
$$
$$
 D_q=\{u\in W^{2,q}(\mathbb R^N):Vu\in L^q(\mathbb R^N)\},
 \qquad q\in\{2,p\},\qquad D=D_2\cap D_p,
$$
with their sum norms.  Interpolation and Sobolev embedding give
$E\hookrightarrow F_p$ and $D\hookrightarrow X\cap L^\infty$.
Also, for any $u\in D$, we have that
\begin{align*}
 \|L_yu\|_{F_p}
 &=\sum_{q\in\{2,p\}}\|L_yu\|_q\\
 &\le\sum_{q\in\{2,p\}}
   \bigl(\|\Delta_yu\|_q+\||y|^2u\|_q+d\|u\|_q\bigr)\\
 &\le C_d\sum_{q\in\{2,p\}}
   \bigl(\|u\|_{W^{2,q}}+\|Vu\|_q\bigr)\\
 &=C_d\sum_{q\in\{2,p\}}\|u\|_{D_q}=C_d\|u\|_D,
\end{align*}

so $L_y:D\to F_p$ is bounded.  We first establish
\begin{equation}\label{E:sec7-parameter-elliptic-bound}
 \|T_\mu f\|_{D_2}\le C_I\|f\|_2,\qquad
 \|T_\mu f\|_D\le C_I\|f\|_{F_p},\qquad \mu\in I,
\end{equation}
where the second estimate is asserted for $f\in F_p$.

For $f\in L^2$, the coercivity of $\mathcal A_\mu$ and Lax--Milgram
give $u=T_\mu f\in QX$ with $\|u\|_X\le C_I\|f\|_2$.
Since $P$ commutes with the operator and $L_yP=0$, the weak equation
holds against every test function in $X$, not only those in $QX$:
$$
 -\mu^{-1}\Delta_yu-\Delta_\zeta u
 +\bigl(\mu^{-1}|y|^2+1-d/\mu\bigr)u=Qf.
$$
Suppose first that $f\in F_p$.  The Gaussian formula for $P$ shows that
$Q$ is bounded on $L^p$.  Kato's inequality gives
$$
 -\mu^{-1}\Delta_y|u|-\Delta_\zeta|u|-C_I|u|\leq |Qf|.
$$
The local boundedness estimate for subsolutions
\cite[Theorem~8.17]{GilbargTrudinger}, applied on unit balls, therefore gives
$$
 \|u\|_\infty\leq C_I(\|u\|_2+\|Qf\|_p)
 \leq C_I\|f\|_{F_p}.
$$
The constant is independent of the ball center: the principal part is
uniformly elliptic for $\mu\in I$, and the nonnegative potential has
been discarded.  Interpolation also gives $\|u\|_p\leq C_I\|f\|_{F_p}$.

Rewrite the equation as
$$
 -\mu^{-1}\Delta_yu-\Delta_\zeta u+\mu^{-1}Vu=F,
 \qquad F=Qf+\bigl((d+1)/\mu-1\bigr)u.
$$
For $q\in\{2,p\}$ and $M\ge1$, put $V_M=\min\{V,M\}$.
Since $u\in X\cap L^\infty$ and $V_M$ is bounded and Lipschitz, a direct computation verifies that $V_M^{q-1}|u|^{q-2}u$ belongs to $X$, and thus it is an admissible test. Testing and absorbing the mixed gradient term by Young's inequality yield
\begin{align*}
 \mu^{-1}\int V V_M^{q-1}|u|^q
 &\le \int |F|V_M^{q-1}|u|^{q-1}
       +C_I\int V_M^{q-3}|\nabla_yV_M|^2|u|^q\\
 &\le \|F\|_q\|V_Mu\|_q^{q-1}
       +C_I\int V_M^{q-2}|u|^q.
\end{align*}
Here all integrals are over $\mathbb R^N$, and
$|\nabla_yV_M|^2\le4V_M$ was used in the second line.
By H\"older's inequality,
$$
 \int V_M^{q-2}|u|^q
 \leq \|V_Mu\|_q^{q-2}\|u\|_q^2,\qquad q\geq2.
$$
Since $V\ge V_M$, we obtain that
$$\mu^{-1}\|V_Mu\|_{q}^{q}\leq \|F\|_q\|V_Mu\|_{q}^{q-1}+C_I \|V_Mu\|_{q}^{q-2} \|u\|_{q}^{2}$$
and then
$$\mu^{-1}\|V_Mu\|_{q}^{2}\leq \|F\|_q \|V_Mu\|_{q}+C_I\|u\|_{q}^{2}.$$
Another application of Young's inequality gives
$$\|V_Mu\|_{q}^{2}\leq C_I(\|F\|_q^2+\|u\|_{q}^{2}).$$
Hence,
$$
 \|V_Mu\|_q\leq C_I(\|F\|_q+\|u\|_q),
$$
uniformly in $M$.  Monotone convergence yields the same estimate for
$Vu$.  Now the constant-coefficient equation
$-\mu^{-1}\Delta_yu-\Delta_\zeta u=F-\mu^{-1}Vu$ and the interior
$W^{2,q}$ estimate \cite[Theorem~9.11]{GilbargTrudinger}, summed over a
covering by unit balls with bounded overlap, give
$$
 \|u\|_{W^{2,q}}+\|Vu\|_q
 \le C_I(\|F\|_q+\|u\|_q),\qquad q\in\{2,p\}.
$$
For $q=2$, the test is simply $V_Mu\in X$, so this argument applies
to every $f\in L^2$ without an $L^\infty$ bound on $u$.
Together with the preceding bounds on $u$, this proves
\eqref{E:sec7-parameter-elliptic-bound}.

We can now take difference quotients directly in the equation.
Subtracting the equations for $T_\nu f$ and $T_\mu f$ gives
\begin{equation}\label{E:sec7-parameter-resolvent-identity}
 T_\nu-T_\mu
 =\left(\frac1\mu-\frac1\nu\right)T_\nu L_yT_\mu,
 \qquad \mu,\nu\in I.
\end{equation}
Every composition is well defined: $T_\mu:F_p\to D$,
$L_y:D\to F_p$, and $PL_yT_\mu=0$.
The bounds \eqref{E:sec7-parameter-elliptic-bound} imply
$$
 \|T_\nu-T_\mu\|_{\mathcal L(F_p,D)}\le C_I|\nu-\mu|.
$$
Set $B_\mu=\mu^{-2}T_\mu L_yT_\mu$.  For $h\ne0$ with
$\mu,\mu+h\in I$, \eqref{E:sec7-parameter-resolvent-identity} gives
\begin{align*}
 \frac{T_{\mu+h}-T_\mu}{h}-B_\mu
 &=\frac{(T_{\mu+h}-T_\mu)L_yT_\mu}{\mu(\mu+h)}
   -\frac{hT_\mu L_yT_\mu}{\mu^2(\mu+h)}.
\end{align*}
Consequently,
$$
 \left\|\frac{T_{\mu+h}-T_\mu}{h}-B_\mu\right\|_{\mathcal L(F_p,D)}
 \le C_I|h|\longrightarrow0.
$$
The same bounds give
$\|B_\nu-B_\mu\|_{\mathcal L(F_p,D)}\le C_I|\nu-\mu|$.
Thus $\partial_\mu T_\mu=B_\mu$ is continuous in operator norm.
The embeddings $E\hookrightarrow F_p$ and
$D\hookrightarrow X\cap L^\infty$ prove the claimed $C^1$ dependence.
Using $L_y:D_2\to L^2$ and the first estimate in
\eqref{E:sec7-parameter-elliptic-bound}, the same difference-quotient
argument proves $C^1$ dependence in $\mathcal L(L^2,D_2)$, hence in
$\mathcal L(L^2,X)$.
Finally, rotation invariance of the equation and uniqueness show that
$T_\mu$, and therefore $B_\mu$, preserve block-radiality.

To prove (ii), put $a=\mu^{1/(\alpha-2)}$.  Scalar differentiation gives
$$
 (\mathfrak g_\mu^+)'(s)=\mu^{-1}g_+'(as),
 \qquad
 \partial_\mu\mathfrak g_\mu^+(s)
 =\frac{as\,g_+'(as)-(\alpha-1)g_+(as)}
        {(\alpha-2)\mu a^{\alpha-1}}.
$$
Fix $M>0$.  The first expression is uniformly continuous on
$I\times[-M-1,M+1]$.  Define its modulus in the scalar argument by
$$
 \omega_{I,M}(\delta):=
 \sup_{\substack{\mu\in I,\ |s|,|t|\le M+1\\|s-t|\le\delta}}
 |(\mathfrak g_\mu^+)'(s)-(\mathfrak g_\mu^+)'(t)|.
$$
Then $\omega_{I,M}(\delta)\to0$ as $\delta\downarrow0$.  For
$\|w\|_\infty\le M$ and $\|h\|_\infty\le1$, the integral remainder satisfies
$$
 \|\mathcal N(\mu,w+h)-\mathcal N(\mu,w)
       -(\mathfrak g_\mu^+)'(w)h\|_E
 \le\omega_{I,M}(\|h\|_\infty)\|h\|_E.
$$
Hence $D_w\mathcal N(\mu,w)h=(\mathfrak g_\mu^+)'(w)h$, and this
derivative is jointly continuous in operator norm.

For the parameter derivative, define
$$
 b(\mu,s):=
 \begin{cases}
 s^{-1}\partial_\mu\mathfrak g_\mu^+(s),&s\ne0,\\
 0,&s=0.
 \end{cases}
$$
The displayed formula and $g_+(0)=g_+'(0)=0$ show that $b$ is continuous
on $I\times[-M,M]$.  In particular,
$|\partial_\mu\mathfrak g_\mu^+(s)|\le C_{I,M}|s|$, so
$\partial_\mu\mathfrak g_\mu^+(w)\in E$.
For $\mu,\mu+h\in I$, the integral formula in $\mu$ gives
\begin{align*}
 &\|\mathcal N(\mu+h,w)-\mathcal N(\mu,w)
       -h\,\partial_\mu\mathfrak g_\mu^+(w)\|_E\\
 &\qquad\le |h|\|w\|_E
 \sup_{\substack{0\le\theta\le1\\|s|\le M}}
 |b(\mu+\theta h,s)-b(\mu,s)|=o(|h|).
\end{align*}
Moreover, $(\mu,w)\mapsto b(\mu,w)w$ is continuous into $E$, by
uniform continuity of $b$ on bounded sets.  This proves joint $C^1$
regularity of $\mathcal N$ using only $g\in C^1$.
Finally, $(v,r)\mapsto\phi_0v+r$ is bounded from
$\mathcal Y\times\mathcal Z$ to $E$ and preserves block-radiality,
so the last assertion follows by composition with (i).
\end{proof}

\begin{lemma}[Reduction to smooth bounded data]
\label{L:sec7-smooth-data-reduction}
Let $L_y:=H_y-d$, let $0<\mu\leq1$, and suppose that
$$
 F\in L^2(\mathbb{R}^d\times\mathbb{R}^m)
 \cap L^\infty(\mathbb{R}^d\times\mathbb{R}^m),
 \qquad QF=F.
$$
Denote by $R\in X$, $QR=R$, the Lax--Milgram solution of
\begin{equation}\label{E:sec7-smooth-data-model-equation}
 (L_y-\Delta_\xi+\mu)R=F
 \quad\hbox{in }\mathbb{R}^d\times\mathbb{R}^m.
\end{equation}
Then there are functions
$$
 F_j\in C^\infty(\mathbb{R}^{d+m})\cap L^2(\mathbb{R}^d\times\mathbb{R}^m)
 \cap L^\infty(\mathbb{R}^d\times\mathbb{R}^m),
 \qquad QF_j=F_j,
$$
such that
\begin{equation}\label{E:sec7-smooth-data-approximation}
 F_j\to F\quad\hbox{in }L^2(\mathbb{R}^d\times\mathbb{R}^m),
 \qquad
 \sup_j\|F_j\|_\infty
 \leq (1+2^{d/2})\|F\|_\infty.
\end{equation}
The corresponding weak solutions $R_j\in X$, $QR_j=R_j$, of
$$
 (L_y-\Delta_\xi+\mu)R_j=F_j
$$
belong to $C^\infty(\mathbb{R}^{d+m})\cap L^\infty(\mathbb{R}^{d+m})$ and satisfy
\begin{equation}\label{E:sec7-smooth-solution-convergence}
 R_j\to R\qquad\hbox{strongly in }X.
\end{equation}
Consequently, any $L^\infty$ estimate for the above equation that is uniform
in $\mu\in(0,1]$ and valid for smooth bounded data and their corresponding
solutions extends, after changing only the dimensional constant, to every
$F\in L^2(\mathbb{R}^d\times\mathbb{R}^m)\cap L^\infty(\mathbb{R}^d\times\mathbb{R}^m)$ with $QF=F$.
\end{lemma}

\begin{proof}
Let $\chi_j\in C_c^\infty(\mathbb{R}^{d+m})$ satisfy
$0\leq\chi_j\leq1$ and $\chi_j\to1$ pointwise.  By first truncating and then
mollifying, we can choose a standard mollifier $\rho_{\varepsilon_j}$ so that
$$
 G_j:=\rho_{\varepsilon_j}*(\chi_jF)\in C_c^\infty(\mathbb{R}^{d+m}),
 \qquad
 G_j\to F\quad\hbox{in }L^2,
 \qquad
 \|G_j\|_\infty\leq\|F\|_\infty.
$$
Set $F_j:=QG_j$.  Since $Q$ is the orthogonal projection on $L^2$ and
$QF=F$, we have $F_j\to F$ in $L^2$ and $QF_j=F_j$.  Moreover,
$$
 (PG_j)(y,\xi)
 =\phi_0(y)\int_{\mathbb{R}^d}G_j(\eta,\xi)\phi_0(\eta)\,\mathrm d\eta .
$$
The function on the right is smooth; its $\xi$-support is compact, and it is
bounded in the $y$-variables by the Gaussian factor $\phi_0(y)$.  Hence
$$F_j=G_j-PG_j \in C^\infty(\mathbb{R}^d\times\mathbb{R}^m)\cap L^2(\mathbb{R}^d\times\mathbb{R}^m)\cap L^\infty(\mathbb{R}^d\times\mathbb{R}^m).$$  The normalization
of $\phi_0$ gives
$$
 \|P h\|_\infty
 \leq\|\phi_0\|_\infty\|\phi_0\|_1\|h\|_\infty
 =2^{d/2}\|h\|_\infty,
$$
and therefore \eqref{E:sec7-smooth-data-approximation} follows.

Let $R_j$ be the Lax--Milgram solution associated with $F_j$.  Standard
interior elliptic regularity for the operator
$$
 -\Delta_{y,\xi}+|y|^2-d+\mu
$$
shows that $R_j\in C^\infty(\mathbb{R}^{d+m})$.  We next verify separately
that $R_j$ is globally bounded.  Choose $A>0$ so large that
$$
 |y|^2-d+\mu\geq1\qquad\hbox{for }|y|\geq A,\quad 0<\mu\leq1.
$$
For every point $x_0=(y_0,\xi_0)$ with $|y_0|\leq A+1$, the coefficients of
the equation are uniformly bounded on $B_1(x_0)$.  The local boundedness
estimate for linear elliptic equations therefore yields
\begin{equation}\label{eq:20260820-2247}
 \|R_j\|_{L^\infty(B_{1/2}(x_0))}
 \leq C_A\bigl(\|R_j\|_{L^2(B_1(x_0))}
 +\|F_j\|_{L^\infty(B_1(x_0))}\bigr)
 \leq C_A\bigl(\|R_j\|_2+\|F_j\|_\infty\bigr).
\end{equation}
The constant does not depend on $\xi_0$.  Thus $R_j$ is bounded on the whole
cylinder $\{|y|\leq A+1\}$.

Put
$$
 M_j:=\max\left\{
 \sup_{|y|=A,\,\xi\in\mathbb{R}^m}|R_j(y,\xi)|,
 \|F_j\|_\infty\right\}<+\infty.
$$
In the exterior domain $\Omega_A:=\{|y|>A\}$, the function
$w_j:=(R_j-M_j)^+$ has zero trace on $\partial\Omega_A$, belongs to
$H_0^1(\Omega_A)$, and is an admissible test function.  Indeed,
$0\leq w_j\leq R_j^+\in L^2$ and
$\nabla w_j=\boldsymbol1_{\{R_j>M_j\}}\nabla R_j$.  Since
$$
 (L_y-\Delta_\xi+\mu)(R_j-M_j)
 =F_j-(|y|^2-d+\mu)M_j\leq0
 \qquad\hbox{in }\Omega_A,
$$
testing by $w_j$ gives
$$
 \int_{\Omega_A}\left(|\nabla w_j|^2
 +( |y|^2-d+\mu)w_j^2\right)\,\mathrm dx\leq0.
$$
Hence $w_j=0$.  Applying the same argument to $-R_j-M_j$ gives
$|R_j|\leq M_j$ in $\Omega_A$.  Together with the cylinder estimate, this
proves $R_j\in L^\infty(\mathbb{R}^{d+m})$.

Finally, put $W_j:=R_j-R$.  Subtracting the two equations and testing by
$W_j$ gives
$$
 \langle L_yW_j,W_j\rangle
 +\|\nabla_\xi W_j\|_2^2+\mu\|W_j\|_2^2
 =\int_{\mathbb{R}^{d+m}}(F_j-F)W_j\,\mathrm dx .
$$
The spectral gap \eqref{E:sec7-gap} implies
$\|W_j\|_2^2\leq\frac{1}{2}\langle L_yW_j,W_j\rangle$.  It also gives
$$
 \|W_j\|_X^2
 \leq C_d\left(
 \langle L_yW_j,W_j\rangle+\|\nabla_\xi W_j\|_2^2
 +\mu\|W_j\|_2^2\right).
$$
Consequently, Cauchy's inequality yields
$$
 \|R_j-R\|_X=\|W_j\|_X
 \leq C_d\|F_j-F\|_2\longrightarrow0.
$$
This proves \eqref{E:sec7-smooth-solution-convergence}.
If a uniform
$L^\infty$ estimate holds for the smooth pairs $(F_j,R_j)$, that is,
\begin{equation}\label{eq:20260820-2256}
\|R_j\|_{\infty}\leq C \|F_j\|_\infty, \forall j.
\end{equation}
then
\eqref{E:sec7-smooth-data-approximation} gives a uniform $L^\infty$ bound for
$R_j$ that
  \begin{equation}\label{eq:20260820-2257}
\|R_j\|_{\infty}\leq C (1+2^{d/2})\|F\|_\infty, \forall j.
\end{equation}
By \eqref{E:sec7-smooth-solution-convergence}, a subsequence converges
to $R$ almost everywhere, so the same bound passes to $R$.  This proves the
last assertion.
\end{proof}

\begin{lemma}[Uniform $Q$-resolvent and nonlinear reduction]
\label{lemma:20260622-1443}
Let $\mathscr{H}_Q:=QL^2(\mathbb{R}^{d+m})$ and let $\mathscr{V}_Q$ be the
closed $Q$-subspace of the form domain of
$$
 \mathcal{A}_\mu=Q\left\{\mu^{-1}(H_y-d)-\Delta_\zeta+1\right\}Q.
$$
The operator $\mathcal{A}_\mu$ is the positive self-adjoint operator associated
with the closed form
$$
 \mathfrak{a}_\mu[r]=\mu^{-1}\langle(H_y-d)r,r\rangle
       +\|\nabla_\zeta r\|_2^2+\|r\|_2^2,
 \qquad r\in\mathscr{V}_Q.
$$
For $0<\mu\le1$ its inverse is well defined on $\mathscr{H}_Q$; moreover it
maps $Q(L^2\cap L^\infty)$ into $QX\cap L^\infty$ (and into $\mathcal Z$ for block-radial data) and satisfies
\begin{equation}\label{E:sec7-XLinfty-resolvent-estimate}
        \|\mathcal{A}_\mu^{-1}Qh\|_X
        \le C\mu^{1/2}\|h\|_2,
        \qquad
        \|\mathcal{A}_\mu^{-1}Qh\|_\infty
        \le C\mu\|h\|_\infty .
\end{equation}
There exists $\mu_R>0$ such that, for every $0<\mu<\mu_R$ and every
$v\in\mathcal V$, the $Q$-equation \eqref{E:sec7-Q-equation-existence} has a
unique solution $r=R_\mu(v)$ in the fixed ball $\mathcal{R}\subset\mathcal Z$.
The map $(\mu,v)\mapsto R_\mu(v)$ is jointly $C^1$
from $(0,\mu_R)\times\mathcal V$ to $\mathcal Z$, and
\begin{equation}\label{E:sec7-Rmu-Z-small}
 \sup_{v\in\mathcal V}\left(
 \|R_\mu(v)\|_{\mathcal Z}
 +\|DR_\mu(v)\|_{\mathcal L(\mathcal Y,\mathcal Z)}\right)
 \le C\mu^{1/2}.
\end{equation}
In particular, both $\|R_\mu(v)\|_X=o(1)$ and
$\|DR_\mu(v)\|_{\mathcal L(\mathcal Y,X)}=o(1)$ uniformly for
$v\in\mathcal V$ as $\mu\to0^+$.
\end{lemma}

\begin{proof}
We use the space $E$ from
Lemma~\ref{L:sec7-parameter-regularity} and first prove the resolvent estimates.
The first estimate of \eqref{E:sec7-XLinfty-resolvent-estimate} follows directly from the energy identity.  If
$r=\mathcal{A}_\mu^{-1}Qh$, then testing $\mathcal{A}_\mu r=Qh$ by $r$ and using the
spectral gap \eqref{E:sec7-gap} gives
$$
        \|r\|_{L^2}\le C\mu\|h\|_2,
        \qquad
        \|\nabla_\zeta r\|_2\le C\mu^{1/2}\|h\|_2,
        \qquad
        \|(H_y-d)^{1/2}r\|_2\le C\mu\|h\|_2 .
$$
Therefore $\|r\|_X\le C\mu^{1/2}\|h\|_2$.

We next establish the $L^\infty$ estimate in \eqref{E:sec7-XLinfty-resolvent-estimate}.  Put
$L_y:=H_y-d$.  We first claim that there exists $C>0$, independent of
$0<\mu\le1$, such that, for every
$F\in L^2(\mathbb{R}^N)\cap L^\infty(\mathbb{R}^N)$ with $QF=F$, the
Lax--Milgram solution $R\in X$, $QR=R$, of
\begin{equation}\label{E:sec7-rescaled-Q-resolvent}
 (L_y-\Delta_\xi+\mu)R=F
 \quad\hbox{in }\mathbb{R}^d\times\mathbb{R}^m,
\end{equation}
belongs to $L^\infty(\mathbb{R}^N)$ and satisfies
\begin{equation}\label{E:sec7-uniform-elliptic-Linfty}
 \|R\|_\infty\le C\|F\|_\infty.
\end{equation}
By Lemma~\ref{L:sec7-smooth-data-reduction}, it is enough to prove
\eqref{E:sec7-uniform-elliptic-Linfty} when $F$ and its corresponding
solution $R$ are smooth and bounded.

Suppose, by contradiction, that \eqref{E:sec7-uniform-elliptic-Linfty} fails.
Then there exist $\mu_n\in(0,1]$ and smooth bounded pairs $(F_n,R_n)$ with
$QF_n=F_n$ and $QR_n=R_n$ such that
\begin{equation}\label{E:sec7-elliptic-contradiction-normalization}
 (L_y-\Delta_\xi+\mu_n)R_n=F_n,\qquad
 \|R_n\|_\infty=1,\qquad \|F_n\|_\infty\to0.
\end{equation}
After changing the sign of $R_n$ if necessary, the approximate maximum principle
provides points $(y_n,\xi_n)$ such that
\begin{equation}\label{E:sec7-approximate-max}
 R_n(y_n,\xi_n)\ge1-\frac1n,\qquad
 \Delta_{y,\xi}R_n(y_n,\xi_n)\le\frac1n.
\end{equation}
Indeed, choose $x_n'$ with $R_n(x_n')>1-(2n)^{-1}$ and maximize
$R_n(x)-\varepsilon_n|x|^2$, where
$\varepsilon_n=[2Nn(1+|x_n'|^2)]^{-1}$.  At a maximizer $x_n$ the penalized
function exceeds $1-1/n$ and $\Delta R_n(x_n)\le2N\varepsilon_n\le1/n$.
Evaluating
\eqref{E:sec7-elliptic-contradiction-normalization} at $(y_n,\xi_n)$ gives
$$
 (|y_n|^2-d+\mu_n)R_n(y_n,\xi_n)
 \le \|F_n\|_\infty+\frac1n.
$$
Consequently, $\{y_n\}$ is bounded.  Passing to a subsequence, we may assume
that
$$
 \mu_n\to\mu_*\in[0,1],\qquad y_n\to y_*\in\mathbb{R}^d.
$$

Translate only in the free variables and set
$$
 \widetilde R_n(y,\xi):=R_n(y,\xi+\xi_n),\qquad
 \widetilde F_n(y,\xi):=F_n(y,\xi+\xi_n).
$$
On every bounded subset of $\mathbb{R}^{d+m}$, the coefficients in
\eqref{E:sec7-elliptic-contradiction-normalization} are uniformly bounded.
More precisely, for every $R>0$ and every $p>N$, the interior estimate gives
$$
 \|\widetilde R_n\|_{W^{2,p}(B_R)}
 \le C_R\left(
 \|\widetilde R_n\|_{L^p(B_{2R})}
 +\|\widetilde F_n\|_{L^p(B_{2R})}
 \right)\le C_R.
$$
The compact embedding $W^{2,p}(B_R)\hookrightarrow C^{1,\gamma}(B_R)$,
where $0<\gamma<1-N/p$, and a diagonal argument therefore yield, along a
subsequence,
\begin{equation}\label{E:sec7-elliptic-local-convergence}
 \widetilde R_n\to R
 \quad\hbox{in }C^1_{\rm loc}(\mathbb{R}^{d+m}).
\end{equation}
The limit is bounded, satisfies
\begin{equation}\label{E:sec7-elliptic-limit-equation}
 (L_y-\Delta_\xi+\mu_*)R=0,
\end{equation}
and is nonzero because $R(y_*,0)=1$.

The condition $Q\widetilde R_n=\widetilde R_n$ also passes to the limit.  Indeed,
$\phi_0\in L^1(\mathbb{R}^d)$, $\|\widetilde R_n\|_\infty\le1$, and the local
convergence in \eqref{E:sec7-elliptic-local-convergence}, followed by a split
into $|y|\le A$ and $|y|>A$, gives
$$
 \int_{\mathbb{R}^d}\phi_0(y)R(y,\xi)\,\mathrm dy=0
 \qquad\hbox{for every }\xi\in\mathbb{R}^m.
$$
Thus $QR=R$.

Let $\{\phi_\nu\}_{\nu\in\mathbb N_0^d}$ be the Hermite basis, normalized so
that
$$
 H_y\phi_\nu=(d+2|\nu|)\phi_\nu.
$$
For each multi-index $\nu$, define the bounded function
$$
 a_\nu(\xi):=\int_{\mathbb{R}^d}R(y,\xi)\phi_\nu(y)\,\mathrm dy.
$$
Projecting \eqref{E:sec7-elliptic-limit-equation} onto $\phi_\nu$ gives, in the
distributional sense,
\begin{equation}\label{E:sec7-Hermite-Liouville}
 -\Delta_\xi a_\nu+(2|\nu|+\mu_*)a_\nu=0
 \quad\hbox{in }\mathbb{R}^m.
\end{equation}
The coefficient $a_0$ vanishes because $QR=R$.  If $|\nu|\ge1$, then
$2|\nu|+\mu_*\ge2$, and the only bounded entire solution of
\eqref{E:sec7-Hermite-Liouville} is zero.  This follows, for example, by taking
the Fourier transform in the sense of tempered distributions:
$$
 \bigl(|\eta|^2+2|\nu|+\mu_*\bigr)\widehat a_\nu(\eta)=0~~~\hbox{in}~~~\mathcal{S}'(\mathbb{R}^m)
$$
and the multiplier is strictly positive.  Hence every
Hermite coefficient of $R$ vanishes.  The Hermite functions form a Schauder
basis of the Schwartz space $\mathcal{S}(\mathbb{R}^d)$; in particular,
\begin{equation}\label{eq:Hermite-dense}
\overline{\mathrm{span}\{\phi_\nu: \nu\in \mathbb{N}_0^d\}}^{\mathcal{S}}=\mathcal{S}(\mathbb{R}^d).
\end{equation}
Fix $\xi\in\mathbb{R}^m$ and regard the bounded function
$R(\cdot,\xi)$ as an element of $\mathcal{S}'(\mathbb{R}^d)$.  Its pairing
with every $\phi_\nu$ equals $a_\nu(\xi)=0$.  By
\eqref{eq:Hermite-dense}, continuity of a tempered distribution therefore
implies $R(\cdot,\xi)=0$ in $\mathcal{S}'(\mathbb{R}^d)$.  Since $\xi$ was
arbitrary and $R$ is continuous, $R\equiv0$. This contradicts
$R(y_*,0)=1$ and proves
\eqref{E:sec7-uniform-elliptic-Linfty}.

We now apply the claim to the original resolvent equation.  If
$r=\mathcal{A}_\mu^{-1}Qh$, multiply its equation by $\mu$ and introduce the
nondegenerate scaling
$$
 \xi:=\frac{\zeta}{\sqrt\mu},\qquad
 R(y,\xi):=r(y,\sqrt\mu\,\xi),\qquad
 H(y,\xi):=h(y,\sqrt\mu\,\xi).
$$
Then
$$
 (L_y-\Delta_\xi+\mu)R=\mu QH,\qquad QR=R.
$$
Since $\|P H\|_\infty \leq \|\phi_0\|_1 \|\phi_0\|_\infty \|H\|_\infty=2^{\frac{d}{2}} \|H\|_\infty$, and hence
$$\|QH\|_\infty=\|H-PH\|_\infty \leq (1+2^{\frac{d}{2}})\|H\|_\infty.$$
 Thus $Q$ is bounded on $L^\infty(\mathbb{R}^d)$, and \eqref{E:sec7-uniform-elliptic-Linfty} yields
$$
 \|\mathcal{A}_\mu^{-1}Qh\|_\infty
 =\|R\|_\infty
 \le C\mu\|QH\|_\infty
 \le C\mu\|h\|_\infty.
$$
This proves the second estimate in \eqref{E:sec7-XLinfty-resolvent-estimate}.

The resolvent preserves block-radiality.  For block-radial $h$, combining
the two estimates in \eqref{E:sec7-XLinfty-resolvent-estimate} therefore gives
\begin{equation}\label{E:sec7-Z-resolvent-estimate}
        \|\mathcal{A}_\mu^{-1}Qh\|_{\mathcal Z}
        \le C\mu^{1/2}\|h\|_{E},
        \qquad 0<\mu\le1.
\end{equation}

We now turn to the nonlinear term.  Since $\overline{\mathcal V}$ and
$\overline{\mathcal{R}}$ are bounded in $\mathcal Y$ and $\mathcal Z$, respectively,
there exists $M>0$ such that
$$
        \|\phi_0v+r\|_\infty\le M,
        \qquad v\in\overline{\mathcal V},\quad r\in\overline{\mathcal{R}}.
$$
By \ref{gcon2} and the definition of $\mathfrak{g}_\mu^+$, for this fixed $M$ there are
$C_M>0$ and $\mu_M>0$ such that, whenever $0<\mu<\mu_M$,
\begin{equation}\label{E:sec7-gmu-local-C1-bound}
        |\mathfrak{g}_\mu^+(s)|\le C_M|s|,
        \qquad
        |(\mathfrak{g}_\mu^+)'(s)|\le C_M,
        \qquad |s|\le M .
\end{equation}
It follows that, uniformly for $v\in\overline{\mathcal{V}}$ and
$r,r_1,r_2\in\overline{\mathcal{R}}$,
\begin{equation}\label{E:sec7-Nemytskii-XLinfty-bound}
        \|\mathfrak{g}_\mu^+(\phi_0v+r)\|_{E}\le C,
\end{equation}
and
\begin{equation}\label{E:sec7-Nemytskii-XLinfty-Lipschitz}
        \|\mathfrak{g}_\mu^+(\phi_0v+r_1)-\mathfrak{g}_\mu^+(\phi_0v+r_2)\|_{E}
        \le C\|r_1-r_2\|_{\mathcal Z} .
\end{equation}
Combining \eqref{E:sec7-Z-resolvent-estimate} with
\eqref{E:sec7-Nemytskii-XLinfty-bound}--\eqref{E:sec7-Nemytskii-XLinfty-Lipschitz},
we see that
$$
        r\mapsto \mathcal{A}_\mu^{-1}Q\mathfrak{g}_\mu^+(\phi_0v+r)
$$
maps the small ball $\mathcal{R}\subset\mathcal Z$ into itself and is a contraction
there, uniformly for $v\in\overline{\mathcal V}$, provided $\mu>0$ is sufficiently
small.  Hence the $Q$-equation has a unique solution $r=R_\mu(v)\in\mathcal{R}$.
Moreover,
\begin{equation}\label{E:sec7-Rmu-X-small-revised}
        \|R_\mu(v)\|_{\mathcal Z}\le C\mu^{1/2}
\end{equation}
uniformly for $v\in\mathcal V$.

Finally, we use the jointly $C^1$ map $\mathfrak F_\mu$
from Lemma~\ref{L:sec7-parameter-regularity}(ii).
Moreover,
$$
        D_r\mathfrak{F}_\mu(v,r)[\rho]
        =\rho-\mathcal{A}_\mu^{-1}Q\bigl((\mathfrak{g}_\mu^+)'(\phi_0v+r)\rho\bigr),
$$
and the second term has operator norm $O(\mu^{1/2})$ on $\mathcal Z$ by
\eqref{E:sec7-Z-resolvent-estimate}.  Thus $D_r\mathfrak F_\mu(v,r)$ is invertible
for small $\mu$.  The parameter-dependent implicit
function theorem gives joint $C^1$ dependence of $R_\mu(v)$ on $(\mu,v)$;
the local maps agree by the uniqueness of the fixed point.
Differentiating $\mathfrak{F}_\mu(v,R_\mu(v))=0$ in the direction
$\psi\in\mathcal Y$ yields
$$
\begin{aligned}
        DR_\mu(v)[\psi]
        &=\left[I-\mathcal{A}_\mu^{-1}Q(\mathfrak{g}_\mu^+)'(\phi_0v+R_\mu(v))\right]^{-1} \\
        &\quad\circ
        \mathcal{A}_\mu^{-1}Q\bigl((\mathfrak{g}_\mu^+)'(\phi_0v+R_\mu(v))\phi_0\psi\bigr).
\end{aligned}
$$
Since
$$
        \|(\mathfrak{g}_\mu^+)'(\phi_0v+R_\mu(v))\phi_0\psi\|_{E}
        \le C\|\psi\|_{\mathcal Y},
$$
we get from \eqref{E:sec7-Z-resolvent-estimate} that
$$
        \|DR_\mu(v)[\psi]\|_{\mathcal Z}
        \le C\mu^{1/2}\|\psi\|_{\mathcal Y}.
$$
Together with \eqref{E:sec7-Rmu-X-small-revised}, this proves
\eqref{E:sec7-Rmu-Z-small} and completes the proof.
\end{proof}

\subsection{Local uniqueness near the lower edge}

\begin{theorem}[Local uniqueness near the lower edge]\label{T:sec7-uniqueness-contradiction}
Assume \ref{gcon1}, \ref{gcon2}, and \ref{gcon3}.
There exists $\mu_u>0$ such that, for every $0<\lambda+d<\mu_u$,
problem \eqref{E:fixed-lambda} has at most one positive solution in the cone
$\mathcal{K}$.
\end{theorem}

\begin{proof}
We argue by contradiction.  Suppose that there exist $\mu_n:=\lambda_n+d\to0^+$ and,
for each $n$, two distinct solutions
$$
        u_{n,1},u_{n,2}\in\mathcal{K},
        \qquad
        u_{n,1}\ne u_{n,2},
$$
of \eqref{E:fixed-lambda} with the same parameter $\lambda=\lambda_n$.  Define the
lower-edge rescaled profiles
\begin{equation}\label{E:two-rescaled-solutions}
        U_{n,i}(y,\zeta):=\mu_n^{-\frac1{\alpha-2}}
        u_{n,i}\left(y,\frac{\zeta}{\sqrt{\mu_n}}\right),
        \qquad i=1,2.
\end{equation}
By Theorem~\ref{T:asymptotics}(i), after passing to a subsequence,
\begin{equation}\label{E:two-solutions-limit}
        U_{n,i}\to U_0:=\phi_0(y)V_\alpha(\zeta)
        \qquad\hbox{strongly in }H^1(\mathbb{R}^d\times\mathbb{R}^m),
\end{equation}
and hence also $U_{n,i}\to U_0$ a.e. in $\mathbb{R}^d\times\mathbb{R}^m$.  In
addition, by the lower-edge estimates obtained in Section~\ref{sec:Asymptotics-low} (see \eqref{eq:20260617-2037} and \eqref{eq:20260819-1126}), there exist constants
$C>0$, $\theta_0\in(0,1)$, and $\sigma>0$ such that, for all large $n$ and
$i=1,2$,
\begin{equation}\label{E:sec7-uniform-bounds-corrected}
        0\leq U_{n,i}\leq C,
        \qquad
        U_{n,i}(y,\zeta)\leq C e^{-\frac{\theta_0}{2}|y|^2},
        \qquad
        U_{n,i}(y,\zeta)\leq C e^{-\sigma|\zeta|}.
\end{equation}
Set
$$
        D_n:=\|U_{n,1}-U_{n,2}\|_{L^2(\mathbb{R}^d\times\mathbb{R}^m)}>0,
        \qquad
        \Phi_n:=\frac{U_{n,1}-U_{n,2}}{D_n}.
$$
Then $\|\Phi_n\|_2=1$.  By subtracting the two rescaled equations and using the mean
value theorem, we get
\begin{equation}\label{E:Phi-equation-sec7}
        (H_y-d)\Phi_n+\mu_n(-\Delta_\zeta\Phi_n+\Phi_n)
        =\mu_n a_n(y,\zeta)\Phi_n,
\end{equation}
where
\begin{equation}\label{E:an-definition-sec7}
        a_n(y,\zeta):=\int_0^1
        \mathfrak{g}_{\mu_n}'\bigl(tU_{n,1}(y,\zeta)+(1-t)U_{n,2}(y,\zeta)\bigr)\,\mathrm{d} t .
\end{equation}
Applying Lemma~\ref{lemma:20260618-1129} with
$a=\mu_n^{1/(\alpha-2)}$, for every fixed $M>0$,
$$
        \mathfrak{g}_{\mu_n}'(s)\to \mu_0(\alpha-1)s^{\alpha-2}
        \qquad\hbox{uniformly for }0\leq s\leq M.
$$
Combining this convergence with \eqref{E:two-solutions-limit} and
\eqref{E:sec7-uniform-bounds-corrected}, we obtain
\begin{equation}\label{E:an-limit-sec7}
        a_n(y,\zeta)\to a_0(y,\zeta)
        :=\mu_0(\alpha-1)\bigl(\phi_0(y)V_\alpha(\zeta)\bigr)^{\alpha-2}
        \quad\hbox{a.e. in }\mathbb{R}^d\times\mathbb{R}^m.
\end{equation}
Moreover, since the arguments of $g'$ are small for all large $n$, \ref{gcon2} also
gives
\begin{equation}\label{E:an-decay-sec7}
        |a_n(y,\zeta)|
        \leq C\bigl(U_{n,1}(y,\zeta)+U_{n,2}(y,\zeta)\bigr)^{\alpha-2}
        \leq C e^{-\sigma_1|\zeta|}
\end{equation}
for some $\sigma_1>0$.  Dividing \eqref{E:Phi-equation-sec7} by $\mu_n$ gives
\begin{equation}\label{E:Phi-divided-sec7}
        \left\{\frac1{\mu_n}(H_y-d)-\Delta_\zeta+1\right\}\Phi_n
        =a_n\Phi_n .
\end{equation}

Decompose
\begin{equation}\label{E:Phi-PQ-decomp-sec7}
        \Phi_n(y,\zeta)=\phi_0(y)\xi_n(\zeta)+\eta_n(y,\zeta),
        \qquad P\eta_n=0 .
\end{equation}
Testing \eqref{E:Phi-divided-sec7} by $\Phi_n$, using $\|\Phi_n\|_2=1$ and \eqref{E:an-decay-sec7} gives
\begin{equation}\label{E:Phi-Hzeta-bound-sec7}
        \|\Phi_n\|_{H^1_\zeta(L^2_y)}\le C.
\end{equation}
Applying $Q$ to \eqref{E:Phi-divided-sec7} and testing by $\eta_n$, the spectral gap
\eqref{E:sec7-gap} yields
\begin{equation}\label{E:eta-small-sec7}
        \frac2{\mu_n}\|\eta_n\|_2^2
        +\|\eta_n\|_{H^1_\zeta(L^2_y)}^2
        \le C\|\Phi_n\|_2\|\eta_n\|_2
        =C\|\eta_n\|_2 .
\end{equation}
Consequently,
\begin{equation}\label{E:eta-small-conclusion-sec7}
        \|\eta_n\|_2\le C\mu_n,
        \qquad
        \|\eta_n\|_{H^1_\zeta(L^2_y)}\le C\mu_n^{1/2}.
\end{equation}
Since $\|\Phi_n\|_2=1$, $\|\phi_0\|_2=1$, and
$\|\eta_n\|_2\to0$, it follows that
\begin{equation}\label{E:xi-nonzero-sec7}
        \|\xi_n\|_{L^2(\mathbb{R}^m)}\to1 .
\end{equation}
By \eqref{E:Phi-Hzeta-bound-sec7}, the sequence $\{\xi_n\}$ is bounded in
$H^1_{\mathrm{rad}}(\mathbb{R}^m)$.

Projecting \eqref{E:Phi-divided-sec7} onto $\phi_0$ gives
\begin{equation}\label{E:xi-equation-sec7}
        -\Delta_\zeta\xi_n+\xi_n=b_n(\zeta)\xi_n+e_n(\zeta),
\end{equation}
where
\begin{equation}\label{E:bn-en-definition-sec7}
        b_n(\zeta):=\int_{\mathbb{R}^d}\phi_0(y)^2a_n(y,\zeta)\,\mathrm{d} y,
        \qquad
        e_n(\zeta):=\int_{\mathbb{R}^d}\phi_0(y)a_n(y,\zeta)\eta_n(y,\zeta)\,\mathrm{d} y .
\end{equation}
By \eqref{E:eta-small-conclusion-sec7} and the boundedness of $a_n$,
\begin{equation}\label{E:En-small-sec7}
        e_n\to0\qquad\hbox{strongly in }L^2(\mathbb{R}^m).
\end{equation}
Moreover, \eqref{E:an-limit-sec7}, \eqref{E:an-decay-sec7}, and dominated convergence
imply
\begin{equation}\label{E:bn-limit-sec7}
        b_n\to b_0:=c_\alpha(\alpha-1)V_\alpha^{\alpha-2}
        \qquad\hbox{strongly in }L^q(\mathbb{R}^m)
        \quad\hbox{for every }1\le q<+\infty,
\end{equation}
and the family $\{b_n\}$ is uniformly exponentially localized:
\begin{equation}\label{E:bn-tail-small-sec7}
        |b_n(\zeta)|\le C e^{-\sigma_1|\zeta|}.
\end{equation}
We record explicitly why this implies operator-norm convergence of the corresponding
Birman--Schwinger operators.  Choose $q>\max\{m/2,2\}$ and define $r$ by
$1/r=1/2+1/q$.  Then $1<r<2$ and
$1/r-1/2=1/q<2/m$.  H\"older's inequality and the Bessel-potential estimate
$(-\Delta+1)^{-1}:L^r(\mathbb{R}^m)\to L^2(\mathbb{R}^m)$ give
$$
 \left\|(-\Delta+1)^{-1}\bigl((b_n-b_0)\xi\bigr)\right\|_2
 \le C\|b_n-b_0\|_q\|\xi\|_2.
$$
Thus the condition $q>m/2$ is exactly the strict
Sobolev/Bessel-potential range needed for this mapping, while $q>2$ ensures $r>1$. Consequently the compact operators
$$
        \mathcal{R}_n\xi:=(-\Delta+1)^{-1}(b_n\xi)
$$
converge in operator norm on $L^2_{\mathrm{rad}}(\mathbb{R}^m)$ to the compact operator
$$
        \mathcal{R}_0\xi:=(-\Delta+1)^{-1}(b_0\xi).
$$
Indeed, compactness of $\mathcal{R}_0$ follows by truncating the exponentially localized
coefficient $b_0$ to a large ball and using Rellich's theorem there.  Equation
\eqref{E:xi-equation-sec7} can now be written as
$$
        \xi_n=\mathcal{R}_n\xi_n+(-\Delta+1)^{-1}e_n .
$$
Since $\{\xi_n\}$ is bounded in $L^2_{\rm rad}(\mathbb{R}^m)$,
\eqref{E:En-small-sec7} and the operator-norm convergence above imply, after passing
to a subsequence,
\begin{equation}\label{E:xi-strong-limit-sec7}
        \xi_n\to\xi
        \qquad\hbox{strongly in }L^2(\mathbb{R}^m)
\end{equation}
for some radial function $\xi$.  By \eqref{E:xi-nonzero-sec7}, $\|\xi\|_2=1$, so
$\xi\ne0$.  Passing to the limit in \eqref{E:xi-equation-sec7}, using
\eqref{E:En-small-sec7} and \eqref{E:bn-limit-sec7}, we obtain
\begin{equation}\label{E:limit-linearized-sec7}
        -\Delta\xi+\xi
        =c_\alpha(\alpha-1)V_\alpha^{\alpha-2}\xi
        \qquad\hbox{in }\mathbb{R}^m, \quad \xi\in L^2(\mathbb{R}^m).
\end{equation}
By elliptic regularity, $\xi\in H^1_{\mathrm{rad}}(\mathbb{R}^m)$.  Equivalently,
$$
        L_\alpha\xi=0,
        \qquad
        L_\alpha:=-\Delta+1-c_\alpha(\alpha-1)V_\alpha^{\alpha-2}.
$$
This contradicts the pure-power nondegeneracy theorem
\cite[Lemma~2.1]{ChangGustafsonNakanishiTsai}, namely
$$
        \ker L_\alpha\cap H^1_{\mathrm{rad}}(\mathbb{R}^m)=\{0\}.
$$
The contradiction proves the local uniqueness.
\end{proof}

\subsection{Linearized nondegeneracy, Morse index, and local degree}

The Lyapunov--Schmidt construction also gives the spectral information needed for the
local degree.  We formulate it separately because this information is intrinsic to the
centered equation and does not depend on the Banach space in which the degree is computed.
The word ``centered'' is essential here: translation invariance in the free variables
gives the zero modes $\partial_{z_j}u_\lambda$ in the full space, whereas these modes do
not belong to the block-radial space $X_r$.

\begin{lemma}[Centered nondegeneracy and Morse index]\label{L:lower-edge-Morse}
There exists $\mu_M>0$ such that, for every $0<\mu=\lambda+d<\mu_M$, the centered
linearized operator
\begin{equation}\label{E:lower-edge-linearized-full}
 \mathcal L_\lambda:=H+\lambda-g'(u_\lambda)
\end{equation}
realized as a self-adjoint operator in the centered radial $L^2$-space with
form domain $X_r$, has
\begin{equation}\label{E:lower-edge-Morse-data}
 \ker(\mathcal L_\lambda|_{X_r})=\{0\},\qquad
 \operatorname{ind}_{X_r}(\mathcal L_\lambda)=1.
\end{equation}
\end{lemma}

\begin{proof}
We work first in the lower-edge variables used in the Lyapunov--Schmidt
construction.  Recalling \eqref{E:sec7-Umu-def} and
\eqref{E:sec7-ulambda-from-Umu}, we have
$$
 u_\lambda(y,z)=\widetilde{u}_\lambda(y,z)=\mu^{\frac1{\alpha-2}}U_\mu(y,\sqrt\mu z),
 \qquad
 U_\mu=\phi_0v_\mu+R_\mu(v_\mu).
$$
Since
$$
 (\mathfrak{g}_\mu^+)'(s)
 =\frac1\mu (g_+)'\bigl(\mu^{\frac1{\alpha-2}}s\bigr),
$$
the original linearized potential satisfies
$$
 g'(u_\lambda(y,z))
 =\mu a_\mu(y,\sqrt\mu z),
 \qquad
 a_\mu(y,\zeta):=(\mathfrak{g}_\mu^+)'(U_\mu(y,\zeta)).
$$
Thus, if $h(y,z)=\Phi(y,\sqrt\mu z)$, then
$$
 \mathcal L_\lambda h=0
 \quad\Longleftrightarrow\quad
 \mathbb L_\mu\Phi=0,
$$
where
\begin{equation}\label{E:rescaled-linearized-Morse}
 \mathbb L_\mu
 :=\frac1\mu(H_y-d)-\Delta_\zeta+1-a_\mu(y,\zeta).
\end{equation}
The operator is understood through its self-adjoint quadratic form on the
centered radial energy space.  By \eqref{E:sec7-gmu-conv},
\eqref{E:sec7-vmu-conv}, and \eqref{E:sec7-Rmu-Z-small},
$$
 U_\mu\longrightarrow\phi_0V_\alpha
 \qquad\hbox{uniformly on }\mathbb{R}^N.
$$
Consequently, the family $\{a_\mu\}$ is uniformly bounded and
\begin{equation}\label{E:amu-uniform-limit-Morse}
 a_\mu\longrightarrow a_0
 :=\mu_0(\alpha-1)(\phi_0V_\alpha)^{\alpha-2}
 \qquad\hbox{in }L^\infty(\mathbb{R}^N).
\end{equation}

We now decompose the operator explicitly.  Every centered radial function in
the form domain has a unique decomposition
$$
 \Phi(y,\zeta)=\phi_0(y)\xi(\zeta)+r(y,\zeta),
 \qquad Pr=0,
$$
where $\xi\in H^1_{\rm rad}(\mathbb{R}^m)$.  We identify the $P$-space with
$L^2_{\rm rad}(\mathbb{R}^m)$ through the unitary map
$\xi\mapsto\phi_0\xi$.  Since $P$ is the spectral projection of $H_y$
associated with $\phi_0$, it commutes with $H_y-d$.  It also commutes with
$-\Delta_\zeta+1$.  Hence the only terms coupling the $P$- and $Q$-spaces
come from multiplication by $-a_\mu$.

Let
$$
 b_\mu(\zeta)
 :=\int_{\mathbb{R}^d}\phi_0(y)^2a_\mu(y,\zeta)\,\mathrm dy.
$$
The square of $\phi_0$ appears here because one factor comes from the
$P$-projection and the other from the input $\phi_0\xi$.  Indeed,
$$
\begin{aligned}
 P\mathbb L_\mu P(\phi_0\xi)
 &=P\left\{\phi_0(-\Delta_\zeta\xi+\xi)
       -a_\mu\phi_0\xi\right\}\\
 &=\phi_0\left(-\Delta_\zeta\xi+\xi-b_\mu\xi\right).
\end{aligned}
$$
Therefore, after the above identification, the $P$-block is
$$
 A_\mu:=-\Delta_\zeta+1-b_\mu.
$$
On the $Q$-space set
$$
 D_\mu:=Q\mathbb L_\mu Q,
 \qquad
 B_\mu:=Qa_\mu P,
 \qquad B_\mu^*=Pa_\mu Q.
$$
The four blocks of \eqref{E:rescaled-linearized-Morse} are consequently
\begin{equation}\label{E:block-linearized-Morse}
 \mathbb L_\mu
 =\begin{pmatrix}
       A_\mu&-B_\mu^*\\
       -B_\mu&D_\mu
   \end{pmatrix}
 \quad\hbox{with respect to }P\oplus Q.
\end{equation}

We next show that the $Q$-block is uniformly separated from zero.  For
$r\in QX_r$, the spectral gap \eqref{E:sec7-gap} and the uniform boundedness
of $a_\mu$ give
\begin{equation}\label{E:Dmu-positive-Morse}
\begin{aligned}
 \langle D_\mu r,r\rangle
 &=\frac1\mu\langle(H_y-d)r,r\rangle
   +\|\nabla_\zeta r\|_2^2+\|r\|_2^2
   -\int_{\mathbb{R}^N}a_\mu r^2\\
 &\ge \|\nabla_\zeta r\|_2^2
   +\left(\frac2\mu+1-C\right)\|r\|_2^2.
\end{aligned}
\end{equation}
It follows that, for all sufficiently small $\mu$, $D_\mu$ is strictly
positive and invertible on the $Q$-space, and
\begin{equation}\label{E:Dmu-inverse-Morse}
 \|D_\mu^{-1}\|_{L^2\to L^2}\le C\mu.
\end{equation}

Here and below the block calculation is an identity of closed quadratic forms, rather than a finite-dimensional matrix calculation.  More precisely, $A_\mu$ and $D_\mu$ have their natural closed form domains, multiplication by $a_\mu$ defines the bounded off-diagonal form $B_\mu$, and $D_\mu^{-1}B_\mu\xi$ belongs to the form domain of $D_\mu$.  The map
$$
 \mathscr U_\mu(\xi,r)
 :=\bigl(\xi,r-D_\mu^{-1}B_\mu\xi\bigr)
$$
is a bounded invertible map between the corresponding form domains.  This follows directly from the coercivity in \eqref{E:Dmu-positive-Morse} and the form version of Lax--Milgram.

We now eliminate the $Q$-variable.
  The equation $\mathbb{L}_\mu(\phi_0\xi+r)=0$ is, by
\eqref{E:block-linearized-Morse}, equivalent to
\begin{equation}\label{E:block-system-Morse}
 \begin{cases}
   A_\mu\xi-B_\mu^*r=0,\\
   -B_\mu\xi+D_\mu r=0.
 \end{cases}
\end{equation}
The second equation and the invertibility of $D_\mu$ yield
$$
 r=D_\mu^{-1}B_\mu\xi.
$$
Substitution into the first equation gives
$$
 S_\mu\xi=0,
$$
where the Schur complement is
\begin{equation}\label{E:Schur-complement-Morse}
\begin{aligned}
 S_\mu
 &:=A_\mu-B_\mu^*D_\mu^{-1}B_\mu\\
 &=(-\Delta_\zeta+1-b_\mu)
   -Pa_\mu QD_\mu^{-1}Qa_\mu P.
\end{aligned}
\end{equation}
The last term is understood under the identification
$P(\phi_0\xi)\leftrightarrow\xi$.  We have in particular obtained the
one-to-one correspondence
\begin{equation}\label{E:kernel-correspondence-Morse}
 \ker S_\mu\ni\xi
 \longmapsto
 \phi_0\xi+D_\mu^{-1}B_\mu\xi\in\ker\mathbb L_\mu.
\end{equation}

For the Morse index we use the same elimination at the level of quadratic
forms.  From \eqref{E:block-linearized-Morse},
$$
 q_\mu(\xi,r)
 :=\langle\mathbb L_\mu(\phi_0\xi+r),\phi_0\xi+r\rangle
 =\langle A_\mu\xi,\xi\rangle
  -2\langle B_\mu\xi,r\rangle+\langle D_\mu r,r\rangle.
$$
Completing the square gives the exact identity
\begin{equation}\label{E:Schur-square-Morse}
\begin{aligned}
 q_\mu(\xi,r)
 &=\langle S_\mu\xi,\xi\rangle\\
 &\quad+
 \left\langle
 D_\mu\bigl(r-D_\mu^{-1}B_\mu\xi\bigr),
 r-D_\mu^{-1}B_\mu\xi
 \right\rangle.
\end{aligned}
\end{equation}
Since $D_\mu$ is strictly positive, the second term in
\eqref{E:Schur-square-Morse} is nonnegative and vanishes precisely when
$r=D_\mu^{-1}B_\mu\xi$.  It follows directly that
\begin{equation}\label{E:inertia-correspondence-Morse}
 \dim\ker\mathbb L_\mu=\dim\ker S_\mu,
 \qquad
 \operatorname{ind}(\mathbb L_\mu)=\operatorname{ind}(S_\mu).
\end{equation}
Thus no negative or zero direction is lost in passing to the Schur
complement. Equivalently, \eqref{E:Schur-square-Morse} is a congruence of closed forms under $\mathscr U_\mu$; the two equalities in
\eqref{E:inertia-correspondence-Morse} follow from the variational definition of the negative index and the kernel correspondence, with no appeal to a finite-dimensional Sylvester law.

It remains to determine the spectrum of $S_\mu$.  From
\eqref{E:amu-uniform-limit-Morse} and $\|\phi_0\|_2=1$,
$$
 \int_{\mathbb{R}^d}\phi_0(y)^2a_0(y,\zeta)\,\mathrm dy
 =\mu_0(\alpha-1)V_\alpha(\zeta)^{\alpha-2}
   \int_{\mathbb{R}^d}\phi_0(y)^\alpha\,\mathrm dy
 =c_\alpha(\alpha-1)V_\alpha(\zeta)^{\alpha-2}.
$$
It follows that
\begin{equation}\label{E:bmu-uniform-limit-Morse}
 \left\|b_\mu-c_\alpha(\alpha-1)V_\alpha^{\alpha-2}\right\|_\infty
 \longrightarrow0.
\end{equation}
Moreover, \eqref{E:Dmu-inverse-Morse} and the uniform boundedness of
$a_\mu$ show that
$$
 \|B_\mu^*D_\mu^{-1}B_\mu\|_{L^2\to L^2}\le C\mu.
$$
Consequently
\begin{equation}\label{E:Smu-operator-convergence-Morse}
 S_\mu=L_\alpha+E_\mu,
 \qquad
 \|E_\mu\|_{L^2\to L^2}\longrightarrow0,
\end{equation}
where
$$
 L_\alpha
 :=-\Delta_\zeta+1-c_\alpha(\alpha-1)V_\alpha^{\alpha-2}
 \quad\hbox{on }H^1_{\rm rad}(\mathbb{R}^m).
$$

The scalar ground state $V_\alpha$ is radially nondegenerate by the
pure-power spectral theorem \cite[Lemma~2.1]{ChangGustafsonNakanishiTsai}:
$$
 \ker L_\alpha\cap H^1_{\rm rad}(\mathbb{R}^m)=\{0\}.
$$
Its radial Morse index is exactly one.  Indeed, $V_\alpha$ minimizes the
scalar energy
$$
 J_\alpha(v):=\frac{1}{2}\|v\|_{H^1}^2
 -\frac{c_\alpha}{\alpha}\|v\|_\alpha^\alpha
$$
on its Nehari manifold
$$
 \mathcal N_\alpha
 :=\left\{v\ne0:\|v\|_{H^1}^2
 =c_\alpha\|v\|_\alpha^\alpha\right\}.
$$
This standard ground-state characterization follows by minimizing the Sobolev quotient (equivalently, the energy on the Nehari manifold); see \cite[Chapter~4]{Willem}.
Using the equation for $V_\alpha$, its tangent space at $V_\alpha$ is
$$
 T_{V_\alpha}\mathcal{N}_\alpha
 =\left\{\psi\in H^1_{\rm rad}(\mathbb{R}^m):
   \int_{\mathbb{R}^m}V_\alpha^{\alpha-1}\psi\,\mathrm{d}\zeta=0\right\},
$$
which has codimension one.  Since $V_\alpha$ is a constrained minimizer and
$J_\alpha'(V_\alpha)=0$, the quadratic form of $L_\alpha=J_\alpha''(V_\alpha)$
is nonnegative on $T_{V_\alpha}\mathcal N_\alpha$.  Any subspace of dimension
at least two intersects this codimension-one tangent space nontrivially;
hence a negative definite subspace can have dimension at most one.  On the
other hand,
$$
 \langle L_\alpha V_\alpha,V_\alpha\rangle
 =-(\alpha-2)c_\alpha
   \int_{\mathbb{R}^m}V_\alpha^\alpha\,\mathrm{d}\zeta<0,
$$
so its Morse index is at least one.

Since $V_\alpha^{\alpha-2}$ decays exponentially, the essential spectrum of
$L_\alpha$ is $[1,+\infty)$.  Thus radial nondegeneracy implies
\begin{equation}\label{E:Lalpha-zero-gap-Morse}
 \delta_0:=\operatorname{dist}
 \bigl(0,\sigma(L_\alpha|_{L^2_{\rm rad}})\bigr)>0.
\end{equation}
For small $\mu$, \eqref{E:Smu-operator-convergence-Morse} gives
$\|E_\mu\|_{L^2\to L^2}<\delta_0/2$.  The spectral inclusion for bounded
self-adjoint perturbations then shows that zero remains outside
$\sigma(S_\mu)$. More precisely,
$$
\left(-\frac{\delta_0}{2},\frac{\delta_0}{2}\right)
\subset
\rho(L_\alpha+tE_\mu)
\qquad\text{for every }t\in[0,1].
$$
Thus the negative spectral projections
$$
P_t^-:=
\mathbf{1}_{(-\infty,0)}(L_\alpha+tE_\mu)
$$
depend continuously on $t$ in operator norm. Since the rank of a
norm-continuous family of finite-rank projections is locally constant,
and $[0,1]$ is connected, it follows that
$$
\dim\operatorname{Ran}P_t^-
=
\dim\operatorname{Ran}P_0^-
=
\operatorname{ind}_{H^1_{\rm rad}}(L_\alpha)
=
1
\qquad\text{for all }t\in[0,1].
$$
In particular,
$$
\ker S_\mu=\{0\},
\qquad
\operatorname{ind}_{H^1_{\rm rad}}(S_\mu)=1.
$$
Together with \eqref{E:inertia-correspondence-Morse}, this yields
$$
 \ker\mathbb L_\mu=\{0\},
 \qquad
 \operatorname{ind}(\mathbb L_\mu)=1.
$$

Finally, the change of variables relating the two quadratic forms is also
explicit.  If $h(y,z)=\Phi(y,\sqrt\mu z)$, then
\begin{equation}\label{E:quadratic-rescaling-Morse}
 \langle\mathcal L_\lambda h,h\rangle
 =\mu^{1-\frac m2}
   \langle\mathbb L_\mu\Phi,\Phi\rangle.
\end{equation}
The rescaling is invertible and the factor in
\eqref{E:quadratic-rescaling-Morse} is strictly positive.  It therefore
preserves both the kernel dimension and the negative index.  We conclude
that \eqref{E:lower-edge-Morse-data} holds for all sufficiently small
$\mu>0$.
\end{proof}

We next justify that passing to the weighted Banach space does not change the spectral
information that enters the Leray--Schauder degree.

\begin{lemma}[Compatibility of the nonzero spectrum with the weighted space]
\label{L:weighted-spectrum}
Assume \ref{gcon1}--\ref{gcon3} and let $u_{\lambda_0}$ denote the locally
unique solution furnished by Theorems~\ref{T:sec7-LS-existence} and
\ref{T:sec7-uniqueness-contradiction}.  Fix
$0<\lambda_0+d<\min\{\mu_e,\mu_u,\mu_M\}$, choose $\theta\in(0,1)$ with
$\lambda_0+\theta d>0$, and let $0<\tau<\sqrt{\lambda_0+\theta d}$.  Set
\begin{equation}\label{E:Klinearized-weighted}
 K_{\lambda_0}h:=(H+\lambda_0)^{-1}\bigl(g'(u_{\lambda_0})h\bigr).
\end{equation}
Then $K_{\lambda_0}$ is compact both on the centered energy space and on
$\mathcal{X}_{\theta,\tau}$.  Its nonzero eigenvalues and their algebraic multiplicities
are the same in the two spaces.
\end{lemma}

\begin{proof}
Put
$$
 A_0:=H+\lambda_0,
 \qquad
 M_0(x):=g'(u_{\lambda_0}(x)),
$$
so that $K_{\lambda_0}=A_0^{-1}M_0$.  Since
$\lambda_0>-d$, the spectral lower bound $H\ge d$ gives
\begin{equation}\label{E:A0-coercive-weighted-spectrum}
 \langle A_0h,h\rangle\ge(d+\lambda_0)\|h\|_2^2,
 \qquad h\in X_r.
\end{equation}
We divide the proof into compactness, equality of the nonzero eigenspaces,
and equality of algebraic multiplicities.

\smallskip
\noindent\emph{Step 1: compactness in the two spaces.}
For this fixed solution no compact-strip Hamiltonian assumption is needed:
the comparison used in Proposition~\ref{prop:uniform-z-exponential-decay}, now
started from the fixed-solution decay $u_{\lambda_0}(x)\to0$, gives
$u_{\lambda_0}(y,z)\le C e^{-\theta|y|^2/2-\tau|z|}$.  Together with
$g'(0)=0$, this fixed-solution estimate gives
\begin{equation}\label{E:M0-decay-weighted-spectrum}
 M_0\in L^\infty(\mathbb{R}^N),
 \qquad
 \lim_{R\to\infty}
 \|M_0\|_{L^\infty(\mathbb{R}^N\setminus B_R)}=0.
\end{equation}
We claim that multiplication by $M_0$ is compact from $X_r$ into $L^2$.
Indeed, let $\{h_n\}$ be bounded in $X_r$.  By local Rellich compactness
and a diagonal argument, after passing to a subsequence there exists
$h\in X_r$ such that
$$
 h_n\longrightarrow h
 \qquad\hbox{strongly in }L^2(B_R)
 \quad\hbox{for every }R>0.
$$
For a fixed $R$, the boundedness of $M_0$ gives
$$
 \|M_0(h_n-h)\|_{L^2(B_R)}\longrightarrow0.
$$
On the complement of $B_R$, using the common $L^2$ bound, we have
$$
 \|M_0(h_n-h)\|_{L^2(\mathbb{R}^N\setminus B_R)}
 \le C\|M_0\|_{L^\infty(\mathbb{R}^N\setminus B_R)}.
$$
Taking the upper limit as $n\to\infty$ and then letting
$R\to\infty$ proves that $M_0h_n\to M_0h$ in $L^2$.  Thus
$$
 M_0:X_r\longrightarrow L^2(\mathbb{R}^N)
$$
is compact.  The form resolvent $A_0^{-1}:L^2\to X$ is bounded
and preserves block-radial symmetry.  Hence
$K_{\lambda_0}=A_0^{-1}M_0$ is compact on $X_r$.

On $\mathcal{X}_{\theta,\tau}$, Lemma
\ref{L:weighted-compact-extension} shows that $T_{\lambda_0}^+$ is a
$C^1$ completely continuous map.  Since $u_{\lambda_0}>0$, its derivative
at $u_{\lambda_0}$ is precisely
$$
 DT_{\lambda_0}^+(u_{\lambda_0})h
 =A_0^{-1}(M_0h)=K_{\lambda_0}h.
$$
The derivative of a $C^1$ completely continuous map is a compact linear
operator.  Indeed, on the unit ball the difference quotients
$$
 h\longmapsto n\left[T_{\lambda_0}^+
 \left(u_{\lambda_0}+\frac hn\right)-T_{\lambda_0}^+(u_{\lambda_0})\right]
$$
have relatively compact images for each fixed $n$ and converge uniformly
to $DT_{\lambda_0}^+(u_{\lambda_0})h$.  Hence $K_{\lambda_0}$ is compact
on $\mathcal{X}_{\theta,\tau}$ as well.

\smallskip
\noindent\emph{Step 2: equality of the nonzero eigenspaces.}
Suppose that
$$
 K_{\lambda_0}h=\gamma h
 \qquad\hbox{in }X_r,
 \qquad \gamma\ne0.
$$
Applying $A_0$ gives
\begin{equation}\label{E:weighted-eigenfunction-equation}
 (H+\lambda_0)h=\gamma^{-1}M_0h.
\end{equation}
Since $M_0$ is bounded, local boundedness for linear equations with bounded
zeroth-order coefficient first gives $h\in L^\infty_{\rm loc}$.  Interior
$W^{2,p}$ estimates applied to \eqref{E:weighted-eigenfunction-equation}
then yield $h\in W^{2,p}_{\rm loc}(\mathbb{R}^N)$ for every finite $p>1$;
choosing $p>N$ gives $h\in C^1_{\rm loc}(\mathbb{R}^N)$.

Choose $\varepsilon>0$ so small that
\begin{equation}\label{E:epsilon-choice-weighted-spectrum}
 0<\varepsilon<\min\left\{
 d+\lambda_0,\ \lambda_0+\theta d-\tau^2
 \right\}.
\end{equation}
This is possible by the assumptions on $\theta$ and $\tau$.  In view of
\eqref{E:M0-decay-weighted-spectrum}, there exists $R>0$ such that
$$
 |\gamma|^{-1}|M_0(x)|\le\varepsilon
 \qquad\hbox{for }|x|\ge R.
$$
Kato's distributional inequality \cite[Lemma~A, p.~138]{Kato1972}, applied to
\eqref{E:weighted-eigenfunction-equation}, gives
$$
 (H+\lambda_0)|h|
 \le |\gamma|^{-1}|M_0|\,|h|
 \qquad\hbox{in }\mathcal D'(\mathbb{R}^N).
$$
The choice of $R$ therefore yields
\begin{equation}\label{E:Kato-weighted-spectrum}
 (H+\lambda_0-\varepsilon)|h|\le0
 \qquad\hbox{in }\Omega_R:=\mathbb{R}^N\setminus\overline{B_R}.
\end{equation}

Put
$$
 \Phi(y,z):=e^{-\theta|y|^2/2-\tau\langle z\rangle}.
$$
By \eqref{E:epsilon-choice-weighted-spectrum} and
Lemma~\ref{L:common-exterior-barrier}(ii),
$$
 (H+\lambda_0-\varepsilon)\Phi\ge0.
$$
Since $h$ is continuous on the compact boundary $\partial B_R$, a suitable
multiple of $\Phi$ dominates $|h|$ there.  Applying the comparison part of
Lemma~\ref{L:common-exterior-barrier} to
\eqref{E:Kato-weighted-spectrum} gives
$$
 |h(y,z)|\le C e^{-\theta|y|^2/2-\tau\langle z\rangle}
 \le C e^{-\theta|y|^2/2-\tau|z|}
$$
outside $B_R$; increasing $C$ gives the estimate globally.  Consequently
$h\in\mathcal X_{\theta,\tau}$.
We have proved that every energy-space eigenfunction associated with a
nonzero eigenvalue belongs to the weighted space.  The reverse inclusion is
immediate from $\mathcal{X}_{\theta,\tau}\hookrightarrow X_r$.  Hence the
nonzero eigenvalues and their geometric eigenspaces coincide in the two
spaces.

\smallskip
\noindent\emph{Step 3: equality of algebraic multiplicities.}
Define on $X_r$ the energy inner product
$$
 (h,k)_{A_0}:=\langle A_0h,k\rangle.
$$
It is an equivalent Hilbert inner product by
\eqref{E:A0-coercive-weighted-spectrum}.  Since $M_0$ is a real-valued
multiplication operator,
\begin{equation}\label{E:K-selfadjoint-weighted-spectrum}
\begin{aligned}
 (K_{\lambda_0}h,k)_{A_0}
 &=\langle M_0h,k\rangle_{L^2}\\
 &=\langle h,M_0k\rangle_{L^2}
 =(h,K_{\lambda_0}k)_{A_0}.
\end{aligned}
\end{equation}
Thus $K_{\lambda_0}$ is compact and self-adjoint on the energy space.  Its
nonzero eigenvalues are therefore real and semisimple, so their algebraic
and geometric multiplicities agree.

A Jordan chain in the weighted space would also be a Jordan chain in
the energy space, where self-adjointness excludes it.  The coincidence of
geometric eigenspaces therefore gives equality of algebraic
multiplicities as well.
\end{proof}

\begin{proof}[Proof of Theorem~\ref{T:local-degree}]
The existence and profile assertions under \ref{gcon1}--\ref{gcon2} follow directly
from Theorem~\ref{T:sec7-LS-existence}.  Assume \ref{gcon3} for the remainder of the
proof and put
$$
 \mu_*:=\min\{\mu_e,\mu_u,\mu_M\}.
$$
Theorem \ref{T:sec7-LS-existence} gives a centered positive solution $u_\lambda$ for
$0<\lambda+d<\mu_*$, and Theorem \ref{T:sec7-uniqueness-contradiction} shows that it is
the only centered positive solution.  Lemma \ref{L:lower-edge-Morse} gives
\eqref{E:lower-edge-Morse-data}.  It also yields $C^1$ dependence in the
original energy space: the map
$(\lambda,u)\mapsto Hu+\lambda u-g_+(u)$ is $C^1$ from
$(-d,\infty)\times X_r$ to $X_r'$, and its derivative in $u$ is an
isomorphism at each point of this branch.  The implicit-function curves
agree with the constructed branch by positivity and uniqueness.

Fix $\lambda_0\in(-d,-d+\mu_*)$ and admissible $\theta,\tau$.  For this fixed positive
solution, $u_{\lambda_0}\in C_0$ and $g(u_{\lambda_0})/u_{\lambda_0}\to0$ at infinity.
The Gaussian and exponential supersolution comparisons used earlier therefore give
$u_{\lambda_0}\in\mathcal{X}_{\theta,\tau}$; no uniform compact-strip estimate is needed
at this point.  The fixed-point identification in
Lemma~\ref{L:weighted-compact-extension}, together with uniqueness in
$\mathcal K$, makes $u_{\lambda_0}$ the only nonzero fixed point and hence an
isolated one in any sufficiently small neighborhood not containing the origin.
Lemma~\ref{L:weighted-compact-extension} shows that
$T_{\lambda_0}^+$ is $C^1$ and compact near this fixed point, with derivative
$K_{\lambda_0}$ from \eqref{E:Klinearized-weighted}.  Moreover,
for the energy inner product,
\begin{equation}\label{E:degree-Morse-identity}
 \bigl((I-K_{\lambda_0})h,h\bigr)_{H+\lambda_0}
 =\langle\mathcal L_{\lambda_0}h,h\rangle.
\end{equation}
Hence $1$ is not an eigenvalue of $K_{\lambda_0}$, and the total multiplicity of its real
eigenvalues larger than $1$ equals the Morse index of $\mathcal L_{\lambda_0}$, namely
one.  By Lemma \ref{L:weighted-spectrum}, this spectral count is unchanged when
$K_{\lambda_0}$ is viewed on $\mathcal{X}_{\theta,\tau}$.  The local degree formula for a
$C^1$ compact map at a nondegenerate fixed point (see, for example, \cite[Chapter~2]{Nirenberg}) therefore gives
$$
 \deg_{LS}(I-T_{\lambda_0}^+,\mathcal O_{\lambda_0},0)
 =(-1)^1=-1
$$
for every sufficiently small isolating neighborhood $\mathcal O_{\lambda_0}$.
This proves Theorem \ref{T:local-degree} for every $m\ge1$ in one and the same functional
framework.
\end{proof}

\section{Global continuation across the spectral interval}\label{sec:global-conti}

\begin{lemma}[Strip-crossing continuation]\label{L:strip-crossing-continuation}
Let $E$ be a Banach space, $J=[a,b]$, and let $T:J\times\overline\Omega\to E$ be jointly completely continuous, where $\Omega\subset E$ is bounded and open.  Assume
$$
 u\ne T(\lambda,u)\quad\hbox{on }J\times\partial\Omega,
 \qquad
 \deg_{LS}(I-T(a,\cdot),\Omega,0)\ne0.
$$
Then the fixed-point set in $J\times\Omega$ contains a compact connected subset meeting both $\{a\}\times\Omega$ and $\{b\}\times\Omega$.  If the fixed point on the initial slice is unique in $\Omega$, this connected set contains that point.
\end{lemma}

\begin{proof}
This is the Leray--Schauder continuation principle
\cite[Theorem~4.3.4]{AmbrosettiArcoya}.  Applied to the fixed-point set in
$J\times\overline\Omega$, it yields a connected set meeting the initial slice and
either the terminal slice or the lateral boundary.  The boundary alternative is
excluded by hypothesis.  The fixed-point set is closed, and complete continuity of
$T$ makes every bounded fixed-point sequence relatively compact; hence the resulting
connected set may be taken compact.  Uniqueness on the initial slice identifies its
point of attachment.
\end{proof}

Let
$$
 \mathcal B_*:=\{(\lambda,u_\lambda):-d<\lambda<-d+\mu_*\}
$$
be the local branch.  The implicit-function construction in Section \ref{sec:LUTD}
shows that $\lambda\mapsto u_\lambda$ is continuous; hence $\mathcal B_*$ is connected.
Fix
$$
 \lambda_0\in(-d,-d+\mu_*),\qquad \Lambda>\lambda_0.
$$
Choose $\theta\in(0,1)$ with $\lambda_0+\theta d>0$ and
$0<\tau<\sqrt{\lambda_0+\theta d}$.  We work on the whole strip
$[\lambda_0,\Lambda]$ in the single Banach space $\mathcal{X}_{\theta,\tau}$.

By Theorem \ref{T:local-degree}, the local Leray--Schauder degree of
$I-T_{\lambda_0}^+$ at $u_{\lambda_0}$ is $-1$.  Lemma
\ref{L:weighted-compact-extension} shows that
$$
 (\lambda,u)\longmapsto T_\lambda^+(u)
$$
is completely continuous on bounded subsets of
$[\lambda_0,\Lambda]\times\mathcal{X}_{\theta,\tau}$ and that all of its nonzero fixed
points are precisely the centered positive solutions in the strip.  The compact-strip
bounds from Section \ref{sec:priori-estimates}, together with
\eqref{E:no-loss-weighted-space}, show that the full nonzero fixed-point set in this strip
is bounded in $\mathcal{X}_{\theta,\tau}$.  Complete continuity therefore makes it
relatively compact.

The nonzero fixed points are also uniformly separated from the zero solution.  Indeed, if
$u_n\ne0$, $\lambda_n\in[\lambda_0,\Lambda]$, and $\|u_n\|_\infty\to0$, then testing the
equation by $u_n$ gives
$$
 d+\lambda_n\le
 \sup_{0<s\le\|u_n\|_\infty}\frac{g(s)}s\longrightarrow0,
$$
contrary to $d+\lambda_n\ge d+\lambda_0>0$.  Hence the nonzero fixed-point set is
closed as well as relatively compact, and is therefore compact in
$[\lambda_0,\Lambda]\times\mathcal{X}_{\theta,\tau}$.

We apply Lemma~\ref{L:strip-crossing-continuation}.  By the compactness and uniform separation just
proved, there exist $0<r<R<+\infty$ such that every nonzero fixed point in the strip lies
in the bounded open annulus
$$
 \Omega:=\{u\in\mathcal{X}_{\theta,\tau}:
 r<\|u\|_{\mathcal{X}_{\theta,\tau}}<R\},
$$
and no fixed point lies on $\partial\Omega$.  At $\lambda=\lambda_0$, local uniqueness,
Theorem~\ref{T:local-degree}, and the excision property of the degree give
$$
 \deg_{LS}(I-T_{\lambda_0}^+,\Omega,0)
 =\deg_{LS}(I-T_{\lambda_0}^+,\mathcal O_{\lambda_0},0)=-1.
$$
Thus all the hypotheses of Lemma~\ref{L:strip-crossing-continuation} hold for
$\Phi(\lambda,u)=u-T_\lambda^+(u)$ on
$[\lambda_0,\Lambda]\times\Omega$.  It yields a compact connected set
\begin{equation}\label{E:crossing-continuum}
 \mathcal{C}_{\lambda_0,\Lambda}
 \subset\{(\lambda,u)\in[\lambda_0,\Lambda]\times\mathcal K:
 u>0\ \hbox{solves }\eqref{E:fixed-lambda}\}
\end{equation}
which meets both $\{\lambda_0\}\times\mathcal K$ and
$\{\Lambda\}\times\mathcal K$.  Its image under the continuous inclusion
$\mathcal{X}_{\theta,\tau}\hookrightarrow X_r$ is connected and lies in $\mathcal{S}$.

Let $\mathcal{C}$ be the connected component of $\mathcal{S}$ containing the local branch
$\mathcal{B}_*$.  For every $\Lambda>\lambda_0$, the crossing continuum
$\mathcal{C}_{\lambda_0,\Lambda}$ intersects $\mathcal{B}_*$ at the unique point
$(\lambda_0,u_{\lambda_0})$ and hence is contained in $\mathcal{C}$.  Thus
$\Lambda\in P_1(\mathcal{C})$ for every $\Lambda>\lambda_0$.  Since
$\mathcal{B}_*\subset\mathcal{C}$, we conclude that
$$
 P_1(\mathcal{C})=(-d,+\infty).
$$
No nesting of the continua $\mathcal{C}_{\lambda_0,\Lambda}$ for different $\Lambda$ is asserted or needed.  Their common intersection with the locally unique point $(\lambda_0,u_{\lambda_0})$ places every one of them in the same connected component $\mathcal{C}$ of $\mathcal{S}$.
This proves Theorem~\ref{T:global}.

\section{Endpoint uniqueness, mass slopes, and orbital stability}
\label{sec:endpoint-stability}

The mass asymptotics in Theorem~\ref{T:asymptotics} by themselves cannot be
differentiated: an $o(1)$ remainder need not have a derivative of the expected
order.  The purpose of this section is to prove differentiability of the two
endpoint branches and to control the differentiated remainders.  Near
$\lambda=-d$ this uses the Lyapunov--Schmidt reduction already constructed in
Section~\ref{sec:LUTD}.  At high frequency, isotropic rescaling leads to
the full-dimensional scalar ground state.  The only zero modes of its
linearization are translations, all absent from the centered block-radial
space.  Together with localization, this gives high-frequency uniqueness,
a uniform inverse in the rescaled form spaces, and a $C^1$ branch.

We then pass from the centered spectral statement to the full linearized
operators.  The $z$-translations reappear as genuine zero modes, whereas the
confining potential moves the would-be $y$-translation modes to the positive
side.  The resulting spectral picture and the differentiated mass laws fit the
Grillakis--Shatah--Strauss slope criterion
\cite{GrillakisShatahStraussI,GrillakisShatahStraussII}; see also the coercivity
method of \cite{WeinsteinStability} and the external-potential implementation in
\cite{Stuart}.

\subsection{The unique high-frequency branch}

For $\lambda>0$ set
\begin{equation}\label{E:stability-high-rescaling}
 a_\lambda:=\lambda^{1/(\beta-2)},\qquad
 U_\lambda(\xi):=a_\lambda^{-1}u_\lambda(\xi/\sqrt\lambda),
 \qquad
 F_\lambda(s):=\frac{g(a_\lambda s)}{a_\lambda^{\beta-1}}.
\end{equation}
The rescaled equation is
\begin{equation}\label{E:stability-high-equation}
 -\Delta U+U+\lambda^{-2}|Y|^2U=F_\lambda(U)
 \qquad\hbox{in }\mathbb R^N,
\end{equation}
where $\xi=(Y,Z)\in\mathbb R^d\times\mathbb R^m$.  Its centered real
linearization is
\begin{equation}\label{E:stability-high-linearization}
 \mathcal A_\lambda
 :=-\Delta+1+\lambda^{-2}|Y|^2-F_\lambda'(U_\lambda).
\end{equation}

The natural form space for this operator is
\begin{equation}\label{E:high-form-space}
 \mathfrak X_\lambda
 :=\{h\in H^1(\mathbb R^N): |Y|h\in L^2(\mathbb R^N)\},
 \qquad
 \|h\|_{\mathfrak X_\lambda}^2
 :=\|h\|_{H^1}^2+\lambda^{-2}\||Y|h\|_2^2.
\end{equation}
Then $h\in \mathfrak X_\lambda$ if and only if $h\in X$, but
\begin{equation}\label{eq:20260903-2004}
\|h\|_{\mathfrak{X}_\lambda}\leq \|h\|_X\leq \lambda \|h\|_{\mathfrak{X}_\lambda},\qquad  \forall h\in X,\quad \forall \lambda\geq 1.
\end{equation}
Thus, for each fixed $\lambda$, the two norms are equivalent; however, the equivalence constants depend on $\lambda$ and are not uniform as $\lambda\to+\infty$.

We use the same notation for its centered block-radial subspace.  Thus
$\mathcal A_\lambda$ is a bounded form operator from
$\mathfrak X_\lambda$ to $\mathfrak X_\lambda'$.  This is the appropriate
dual pair for uniform inverse estimates; for arbitrary
$h\in\mathfrak X_\lambda$, the term $|Y|^2h$ need not belong to the
unweighted space $H^{-1}$.

\begin{lemma}[Compact spectral convergence in the high-frequency scaling]
\label{L:high-spectral-compactness}
Let $\lambda_n\to+\infty$ and let $b_n\in L^\infty(\mathbb R^N)$ be real
functions such that
\begin{equation}\label{E:abstract-high-potential-convergence}
 b_n\to b\quad\hbox{locally uniformly},
 \qquad
 \lim_{R\to\infty}\sup_n
 \|b_n\|_{L^\infty(\mathbb R^N\setminus B_R)}=0.
\end{equation}
Set
\begin{equation*}
 A_n:=-\Delta+1+\lambda_n^{-2}|Y|^2-b_n,
 \qquad A_\infty:=-\Delta+1-b.
\end{equation*}
Then the following assertions hold.
\begin{enumerate}[label=(\roman*)]
\item If $\nu_n$ remains in a compact subset of $(-\infty,1)$ and
$A_n\Phi_n=\nu_n\Phi_n$, $\|\Phi_n\|_2=1$, then, after extraction,
$\nu_n\to\nu<1$ and $\Phi_n\to\Phi$ strongly in $L^2$ and weakly in
$H^1$, where $A_\infty\Phi=\nu\Phi$ and $\|\Phi\|_2=1$.
\item Let $J\subset\subset(-\infty,1)$ be an interval whose endpoints do not
belong to $\sigma(A_\infty)$.  If $J$ contains only isolated eigenvalues of
$A_\infty$, then
\begin{equation}\label{E:abstract-high-spectral-rank}
 \operatorname{rank}\mathbf 1_J(A_n)
 =\operatorname{rank}\mathbf 1_J(A_\infty)
\end{equation}
for all sufficiently large $n$.
\end{enumerate}
The conclusions remain valid after restriction to a fixed closed symmetry or
parity subspace invariant under all the operators.
\end{lemma}

\begin{proof}
Multiplication by $b_n$ is relatively compact with respect to
$-\Delta+1+\lambda_n^{-2}|Y|^2$.  The essential spectrum of $A_n$ therefore
starts at $1+d/\lambda_n$, so a fixed compact spectral window below $1$
contains only discrete spectrum for large $n$.

After extraction, write $\nu_n\to\nu\le1-\delta$ for some $\delta>0$.
Testing the eigenvalue equation by $\Phi_n$ gives
\begin{equation}\label{E:abstract-high-eigenfunction-bound}
 \|\nabla\Phi_n\|_2^2+\|\Phi_n\|_2^2
 +\lambda_n^{-2}\||Y|\Phi_n\|_2^2\le C.
\end{equation}
Choose $\chi_R\in C^\infty(\mathbb R^N)$ with
\begin{equation*}
 0\le\chi_R\le1,\qquad \chi_R=0\ \hbox{on }B_R,\qquad
 \chi_R=1\ \hbox{on }\mathbb R^N\setminus B_{2R},\qquad
 |\nabla\chi_R|\le C/R.
\end{equation*}
Testing by $\chi_R^2\Phi_n$, yields
\begin{align}
 &\|\nabla(\chi_R\Phi_n)\|_2^2
 +(1-\nu_n)\|\chi_R\Phi_n\|_2^2
 +\lambda_n^{-2}\||Y|\chi_R\Phi_n\|_2^2 \notag\\
 &\qquad\leq \int_{\mathbb{R}^N} b_n\chi_R^2\Phi_n^2\mathrm{d}x
       +\frac{C}{R^2}\|\Phi_n\|_2^2.
 \label{E:abstract-high-IMS-tail}
\end{align}
For large $R$, \eqref{E:abstract-high-potential-convergence} bounds the first
term on the right by $(\delta/2)\|\chi_R\Phi_n\|_2^2$.  Hence
\begin{equation}\label{E:abstract-high-L2-tightness}
 \lim_{R\to\infty}\sup_n
 \int_{\mathbb R^N\setminus B_{2R}}|\Phi_n|^2=0.
\end{equation}
Local Rellich compactness gives strong convergence in $L^2(\mathbb R^N)$ to
a function $\Phi$ of norm one.  Passing to compactly supported test functions
gives $A_\infty\Phi=\nu\Phi$, proving (i).

For (ii), write
\begin{equation*}
 N_n(t):=\operatorname{rank}\mathbf1_{(-\infty,t)}(A_n),
 \qquad
 N_\infty(t):=\operatorname{rank}\mathbf1_{(-\infty,t)}(A_\infty),
 \qquad t<1.
\end{equation*}
Whenever $t\notin\sigma(A_\infty)$, compactly supported approximations of
the eigenfunctions below $t$ and the min--max principle give
$N_n(t)\ge N_\infty(t)$ for large $n$.  If the reverse inequality failed,
one could choose $N_\infty(t)+1$ orthonormal eigenfunctions of $A_n$ with
eigenvalues below $t$.  The operators are uniformly bounded from below by
$1-\sup_n\|b_n\|_\infty$, so part (i), applied successively, gives the same
number of orthonormal strong limits with limiting eigenvalues at most $t$.
Equality with $t$ is impossible because $t$ lies in the resolvent set.
This contradicts the definition of $N_\infty(t)$.  Hence
$N_n(t)=N_\infty(t)$ at every such $t$.  Applying this equality to the two
endpoints of $J$ and subtracting proves
\eqref{E:abstract-high-spectral-rank}.  The proof is unchanged in an
invariant closed subspace.
\end{proof}

\begin{proposition}[High-frequency uniqueness and uniform nondegeneracy]
\label{P:high-unique-branch}
Assume \ref{gcon1}--\ref{gcon4}.  There exists $\Lambda_{\mathrm h}>0$
such that, for every $\lambda>\Lambda_{\mathrm h}$,
\eqref{E:fixed-lambda} has exactly one centered positive solution.  The
solutions form a $C^1$ branch $\lambda\mapsto u_\lambda$.  In the rescaled
variables,
\begin{equation}\label{E:high-branch-convergence-strong}
 U_\lambda\longrightarrow W_\beta
 \quad\hbox{in }H^1(\mathbb R^N)\cap L^\infty(\mathbb R^N),
 \qquad
 U_\lambda(\xi)\le Ce^{-c|\xi|},
\end{equation}
with constants in the last estimate independent of large $\lambda$.
Moreover,
\begin{equation}\label{E:high-centered-spectral-data}
 \ker(\mathcal A_\lambda|_{X_r})=\{0\},
 \qquad
 \operatorname{ind}_{X_r}(\mathcal A_\lambda)=1,
\end{equation}
and the inverses are uniformly bounded, in the sense that
\begin{equation}\label{E:high-uniform-inverse}
 \|h\|_{\mathfrak X_\lambda}
 \le C\|\mathcal A_\lambda h\|_{\mathfrak X_\lambda'}
 \qquad(h\in\mathfrak X_\lambda\cap X_r,
        \ \lambda>\Lambda_{\mathrm h}).
\end{equation}
\end{proposition}

\begin{proof}
Existence at every large frequency follows from Theorem~\ref{T:global}: its
component $\mathcal{C}$ meets each slice $\{\lambda\}\times X_r$.  The
convergence and exponential bound in
\eqref{E:high-branch-convergence-strong} hold for every sequence of centered positive solutions by Section~\ref{sec:Asymptotics-high}, in particular by \eqref{E:high-convergence} and \eqref{eq:20260618-1344}.  We prove that two different solutions cannot occur at a sequence $\lambda_n\to\infty$.

Suppose otherwise, and let $U_{n,1}$ and $U_{n,2}$ be the two profiles obtained
from \eqref{E:stability-high-rescaling}.  Put
\begin{equation}\label{E:high-normalized-difference}
 D_n:=\|U_{n,1}-U_{n,2}\|_2,
 \qquad
 \Phi_n:=D_n^{-1}(U_{n,1}-U_{n,2}).
\end{equation}
Then $\|\Phi_n\|_2=1$, and subtraction of the two equations gives
\begin{equation}\label{E:high-difference-equation}
 \bigl(-\Delta+1+\lambda_n^{-2}|Y|^2-b_n(\xi)\bigr)\Phi_n=0,
\end{equation}
where
\begin{equation}\label{E:high-difference-coefficient}
 b_n(\xi):=\int_0^1F_{\lambda_n}'
 \bigl(tU_{n,1}(\xi)+(1-t)U_{n,2}(\xi)\bigr)\,\mathrm dt.
\end{equation}
The uniform $C^1$ limit in Lemma~\ref{lemma:20260618-1129}, the global
$L^\infty$ convergence, and the common exponential tail imply
\begin{equation}\label{E:high-coefficient-compact-convergence}
 b_n\longrightarrow b_\infty
 :=\mu_\infty(\beta-1)W_\beta^{\beta-2}
 \quad\hbox{locally uniformly and in }L^q(\mathbb R^N)
\end{equation}
for every finite $q\ge1$.  They also imply the uniform tail estimate
\begin{equation}\label{E:high-coefficient-tail}
 \lim_{R\to\infty}\sup_n
 \|b_n\|_{L^\infty(\mathbb R^N\setminus B_R)}=0.
\end{equation}
For completeness, put $\delta=\min\{\alpha-2,\beta-2\}>0$.
The two asymptotics of $g'$ give, for every $M>0$,
\begin{equation}\label{E:high-derivative-power-envelope}
 a^{2-\beta}|g'(as)|\le C_Ms^\delta,
 \qquad a\ge1,\quad 0\le s\le M.
\end{equation}
Indeed, split at $as=1$: below this value use
$|g'(as)|\le C(as)^{\alpha-2}$, and above it use
$|g'(as)|\le C(as)^{\beta-2}$.  If $\alpha>\beta$ in the first region,
write $a^{\alpha-\beta}s^{\alpha-2}
=(as)^{\alpha-\beta}s^{\beta-2}$.
The common exponential bound for $U_{n,i}$ therefore gives an exponential
majorant for $b_n$.  Dominated convergence proves both
\eqref{E:high-coefficient-compact-convergence} and
\eqref{E:high-coefficient-tail}.

Testing \eqref{E:high-difference-equation} by $\Phi_n$ shows that
$\{\Phi_n\}$ is bounded in $H^1$.  After extraction,
$\Phi_n\rightharpoonup\Phi$ in $H^1$ and $\Phi_n\to\Phi$ in
$L^2_{\rm loc}$.  Passing to the limit gives
\begin{equation}\label{E:high-limit-kernel-equation}
 \bigl(-\Delta+1-\mu_\infty(\beta-1)W_\beta^{\beta-2}\bigr)\Phi=0.
\end{equation}
The pure-power nondegeneracy theorem
\cite[Lemma~2.1]{ChangGustafsonNakanishiTsai} says that the kernel in the
full space is
$\operatorname{span}\{\partial_{\xi_1}W_\beta,\ldots,
\partial_{\xi_N}W_\beta\}$.  Each of these functions is odd in one
coordinate, while $\Phi$ is radial in $Y$ and in $Z$.  Hence $\Phi=0$.

It remains to use the information at infinity; this is the step that turns
weak convergence into a contradiction.  By local Rellich compactness and
\eqref{E:high-coefficient-tail}, multiplication by $b_n$ is collectively
compact from bounded subsets of $H^1$ to $H^{-1}$.  Therefore
\begin{equation}\label{E:high-potential-difference-vanishes}
 \int_{\mathbb R^N}b_n\Phi_n^2\,\mathrm{d}\xi\longrightarrow0.
\end{equation}
Testing \eqref{E:high-difference-equation} once more gives
\begin{equation*}
 \|\nabla\Phi_n\|_2^2+\|\Phi_n\|_2^2
 +\lambda_n^{-2}\||Y|\Phi_n\|_2^2
 =\int_{\mathbb R^N}b_n\Phi_n^2\,\mathrm{d}\xi\longrightarrow0,
\end{equation*}
contradicting $\|\Phi_n\|_2=1$.  This proves high-frequency uniqueness.

We next prove the uniform inverse estimate in the correct form-dual topology.
If \eqref{E:high-uniform-inverse} failed, there would be
$\lambda_n\to\infty$ and centered
$h_n\in\mathfrak X_{\lambda_n}$ such that
\begin{equation}\label{E:high-inverse-contradiction-sequence}
 \|h_n\|_{\mathfrak X_{\lambda_n}}=1,
 \qquad
 \|\mathcal A_{\lambda_n}h_n\|_{\mathfrak X_{\lambda_n}'}\to0.
\end{equation}
Put $\widetilde b_n=F_{\lambda_n}'(U_{\lambda_n})$.  The same argument as
in \eqref{E:high-coefficient-compact-convergence}--
\eqref{E:high-coefficient-tail} gives local convergence to
$\mu_\infty(\beta-1)W_\beta^{\beta-2}$ and a uniform $L^\infty$ tail.
After extraction, $h_n\rightharpoonup h$ in $H^1$ and
$h_n\to h$ locally in $L^2$.  Passing to the weak equation gives
$L_\beta h=0$.  Since $h$ is centered block-radial and every element of
$\ker L_\beta$ is a linear combination of the odd functions
$\partial_{\xi_j}W_\beta$, one has $h=0$.  Collective compactness of
multiplication by $\widetilde b_n$ then gives
\begin{equation*}
 \int_{\mathbb R^N}\widetilde b_nh_n^2\mathrm{d}\xi\longrightarrow0.
\end{equation*}
Pairing the residual in
\eqref{E:high-inverse-contradiction-sequence} with $h_n$ yields
\begin{equation*}
 \|h_n\|_{\mathfrak X_{\lambda_n}}^2
 =\int_{\mathbb R^N}\widetilde b_nh_n^2\mathrm{d}\xi
  +\langle\mathcal A_{\lambda_n}h_n,h_n\rangle
 \longrightarrow0,
\end{equation*}
contrary to its normalization.  This proves
\eqref{E:high-uniform-inverse} and centered nondegeneracy.

The limiting operator has exactly one negative eigenvalue and no kernel in
the centered subspace.  Lemma~\ref{L:high-spectral-compactness}, applied in
that invariant subspace, shows that $\mathcal A_\lambda$ has at most one
negative eigenvalue for large $\lambda$.  It has at least one, since its
quadratic form evaluated at $W_\beta$ converges to
\begin{equation*}
 \left\langle
 \bigl(-\Delta+1-\mu_\infty(\beta-1)W_\beta^{\beta-2}\bigr)W_\beta,
 W_\beta\right\rangle
 =(2-\beta)\mu_\infty\int W_\beta^\beta<0.
\end{equation*}
Thus \eqref{E:high-centered-spectral-data} holds.

For finite $\lambda$, use the positive-part extension $g_+$ and consider
$$
 \mathscr F(\lambda,u):=Hu+\lambda u-g_+(u),
 \qquad \mathscr F:(-d,\infty)\times X_r\longrightarrow X_r'.
$$
This map is $C^1$ by the subcritical growth and $g_+\in C^1(\mathbb R)$.
Its derivative at a positive solution is $L_{+,\lambda}$, a compact
perturbation of the isomorphism $H+\lambda:X_r\to X_r'$.
Centered nondegeneracy therefore makes it an isomorphism, and the implicit
function theorem gives a local $C^1$ curve.  Every nearby nonzero solution
of the extended equation is positive by testing its negative part and
using the strong maximum principle.  The curves agree by uniqueness.

We also need differentiability of the rescaled profiles.  This follows
directly, for each fixed $\lambda>\Lambda_{\mathrm h}$, by applying the
same theorem to
$$
 (\lambda,U)\longmapsto
 -\Delta U+U+\lambda^{-2}|Y|^2U-F_\lambda^+(U)
 \quad\hbox{from }(0,\infty)\times X_r\hbox{ to }X_r',
$$
where $F_\lambda^+(s)=g_+(a_\lambda s)/a_\lambda^{\beta-1}$.
Its derivative is $\mathcal A_\lambda$, and its parameter derivative has
subcritical growth on compact frequency intervals.  Thus
$\lambda\mapsto U_\lambda$ is $C^1$ in $X_r$ as well.  This argument
requires no differentiability of the dilation group on arbitrary $X$-data.
\end{proof}

\subsection{Differentiating the endpoint mass laws}

We now control the derivatives of the rescaled profiles, using the
Lyapunov--Schmidt equations at the lower edge and the uniform centered
inverse at high frequency.  The mass slopes then follow by differentiating
the exact scaling identities.

\begin{proposition}[Differentiated endpoint profiles]
\label{P:differentiated-profiles}
Along the two branches of Theorem~\ref{T:endpoint-stability},
\begin{align}
 &\mu\|\partial_\mu U_\mu\|_{L^2(\mathbb{R}^N)}\longrightarrow0
 &&(\mu\to0^+), \label{E:lower-profile-derivative-small}\\
 &\lambda\|\partial_\lambda U_\lambda\|_{H^1(\mathbb{R}^N)}
 \longrightarrow0
 &&(\lambda\to\infty), \label{E:high-profile-derivative-small}
\end{align}
where $U_\mu$ is defined by \eqref{E:sec7-rescale} and $U_\lambda$ by
\eqref{E:stability-high-rescaling}.  Consequently,
\eqref{E:main-lower-mass-slope} and
\eqref{E:main-high-mass-slope} hold.
\end{proposition}

\begin{proof}
\smallskip
\noindent\emph{The lower spectral edge.}
Theorem~\ref{T:sec7-LS-existence} and
Lemma~\ref{lemma:20260622-1443} give the required $C^1$ dependence for
$\mu>0$.  We start from
$$U_\mu=\phi_0v_\mu+R_\mu(v_\mu).$$
Differentiating along the branch gives
\begin{align}
 \mu\partial_\mu U_\mu
 &=\phi_0\,\mu\partial_\mu v_\mu
   +(\mu\partial_\mu R_\mu)(v_\mu)
   +DR_\mu(v_\mu)[\mu\partial_\mu v_\mu].
 \label{E:P93-total-Umu-derivative}
\end{align}
Here $\partial_\mu R_\mu(v)$ keeps $v$ fixed, whereas $DR_\mu(v)$
denotes differentiation with respect to $v$.  Thus we need an estimate
for the reduced derivative $\mu\partial_\mu v_\mu$ and estimates for
these two derivatives of the complementary map.  We first control the
reduced derivative; the complementary estimates needed for this step
will also handle the last two terms in
\eqref{E:P93-total-Umu-derivative}.

\smallskip
\noindent\emph{Step 1: the reduced derivative $\mu\partial_\mu v_\mu$.}
Differentiating $\mathcal{G}_\mu(v_\mu)=0$ along the branch yields
\begin{equation}\label{E:P93-total-Gmu-derivative}
 D_v\mathcal{G}_\mu(v_\mu)[\mu\partial_\mu v_\mu]
 +\mu\partial_\mu\mathcal{G}_\mu(v_\mu)=0.
\end{equation}
It therefore suffices to prove that $D_v\mathcal{G}_\mu(v_\mu)$ has a
uniformly bounded inverse on $H^1_{\rm rad}(\mathbb{R}^m)$ and that
$\mu\partial_\mu\mathcal{G}_\mu(v_\mu)$ tends to zero in that space.

\smallskip
\noindent\emph{(a) Uniform invertibility of $D_v\mathcal{G}_\mu(v_\mu)$.}
For $v\in\mathcal V$, write
\begin{equation}\label{E:P93-wmu-v}
 w_{\mu,v}:=\phi_0v+R_\mu(v),
\end{equation}
and set
\begin{equation}\label{E:P93-Bmu-v}
 B_{\mu,v}:=(\mathfrak g_\mu^+)'(w_{\mu,v}).
\end{equation}
The bounded Lyapunov--Schmidt neighborhood and
\eqref{E:sec7-gmu-local-C1-bound} give
\begin{equation}\label{E:P93-Bmu-uniform}
 \sup_{0<\mu<\mu_{\mathrm d}}\sup_{v\in\mathcal V}
 \|B_{\mu,v}\|_\infty<+\infty
\end{equation}
for some sufficiently small $\mu_{\mathrm d}>0$.  Introduce
\begin{equation}\label{E:P93-Pi}
 (\Pi h)(\zeta):=\int_{\mathbb{R}^d}\phi_0(y)h(y,\zeta)\,\mathrm dy,
 \qquad
 \|\Pi h\|_{L^2(\mathbb{R}^m)}\le\|h\|_{L^2(\mathbb{R}^N)}.
\end{equation}
Then $Ph=\phi_0\Pi h$, and the reduced map is
$$
 \mathcal{G}_\mu(v)
 =v-(-\Delta_\zeta+1)^{-1}\Pi\mathfrak g_\mu^+(w_{\mu,v}).
$$
Its derivative, initially acting on $\psi\in\mathcal Y$, is
$$
 \begin{aligned}
 D_v\mathcal{G}_\mu(v_\mu)\psi
 &=\psi-(-\Delta_\zeta+1)^{-1}
   \bigl[b_\mu\psi+
   \Pi\bigl(B_{\mu,v_\mu}DR_\mu(v_\mu)\psi\bigr)\bigr],\\
 b_\mu&:=\Pi(B_{\mu,v_\mu}\phi_0).
 \end{aligned}
$$
We therefore need to control $DR_\mu(v_\mu)$ in $L^2$.

For this purpose, recall the complementary equation
\begin{equation}\label{E:lower-Q-fixed-equation-differentiated}
 \mathfrak{F}_\mu(v,R_\mu(v))=0,
 \qquad
 \mathfrak{F}_\mu(v,r):=
 r-\mathcal A_\mu^{-1}Q\mathfrak{g}_\mu^+(\phi_0v+r),
\end{equation}
where, on $\mathscr{H}_Q=QL^2(\mathbb{R}^N)$,
\begin{equation}\label{E:P93-Amu}
 \mathcal{A}_\mu
 =Q\{\mu^{-1}(H_y-d)-\Delta_\zeta+1\}Q.
\end{equation}
Differentiation with respect to $v$, with $\mu$ fixed, gives
$$
 (I-\mathcal{K}_{\mu,v})DR_\mu(v)\psi
 =\mathcal{A}_\mu^{-1}Q(B_{\mu,v}\phi_0\psi),
 \qquad
 \mathcal{K}_{\mu,v}:=\mathcal{A}_\mu^{-1}QB_{\mu,v}.
$$
The oscillator gap gives
$\|\mathcal{A}_\mu^{-1}Q\|_{L^2\to L^2}\le\mu/2$.
Together with \eqref{E:P93-Bmu-uniform}, this yields
\begin{equation}\label{E:P93-Kmu-small}
 \sup_{v\in\mathcal V}
 \|\mathcal{K}_{\mu,v}\|_{\mathcal L(L^2)}\le C\mu.
\end{equation}
After reducing $\mu_{\mathrm d}$, the Neumann series gives
\begin{equation}\label{E:P93-Kmu-inverse}
 \sup_{0<\mu<\mu_{\mathrm d}}\sup_{v\in\mathcal V}
 \|(I-\mathcal{K}_{\mu,v})^{-1}\|_{\mathcal L(L^2)}\le2.
\end{equation}
Hence
$$
 DR_\mu(v)\psi
 =(I-\mathcal{K}_{\mu,v})^{-1}\mathcal{A}_\mu^{-1}Q
       (B_{\mu,v}\phi_0\psi).
$$
Since $\|\phi_0\psi\|_2=\|\psi\|_2$, this formula extends uniquely
from $\mathcal Y$ to $L^2_{\rm rad}(\mathbb{R}^m)$, with values in
$QL^2(\mathbb{R}^N)$, and gives, uniformly for $v\in\mathcal V$,
\begin{equation}\label{E:lower-DR-L2-small}
 \|DR_\mu(v)[\psi]\|_2\le C\mu\|\psi\|_2,
 \qquad \psi\in L^2_{\rm rad}(\mathbb{R}^m).
\end{equation}
In particular, the displayed formula for $D_v\mathcal{G}_\mu(v_\mu)$
extends to $H^1_{\rm rad}$.

Now $B_{\mu,v_\mu}=a_\mu$, so
\eqref{E:amu-uniform-limit-Morse} gives
$b_\mu\to b_0:=c_\alpha(\alpha-1)V_\alpha^{\alpha-2}$ uniformly.
Using \eqref{E:lower-DR-L2-small} and
$(-\Delta_\zeta+1)^{-1}:L^2\to H^1$, we obtain
$$
 \|D_v\mathcal{G}_\mu(v_\mu)-D\mathcal{G}_0(V_\alpha)\|
 _{\mathcal L(H^1_{\rm rad})}
 \le C\bigl(\|b_\mu-b_0\|_\infty+\mu\bigr)\longrightarrow0.
$$
Since $D\mathcal{G}_0(V_\alpha)$ is invertible on $H^1_{\rm rad}$,
we conclude that
\begin{equation}\label{E:P93-uniform-reduced-inverse}
 \sup_{0<\mu<\mu_{\mathrm d}}
 \|D_v\mathcal{G}_\mu(v_\mu)^{-1}\|
 _{\mathcal L(H^1_{\rm rad},H^1_{\rm rad})}<+\infty,
\end{equation}
after decreasing $\mu_{\mathrm d}$ once more.

\smallskip
\noindent\emph{(b) Smallness of $\mu\partial_\mu\mathcal{G}_\mu(v_\mu)$.}
Keeping $v$ fixed in the reduced map gives
\begin{align}
 \mu\partial_\mu\mathcal{G}_\mu(v)
 =-(-\Delta_\zeta+1)^{-1}\Pi
 \Bigl[& (\mu\partial_\mu\mathfrak{g}_\mu^+)(w_{\mu,v}) \notag\\
       &+B_{\mu,v}\,\mu\partial_\mu R_\mu(v)\Bigr].
 \label{E:P93-partial-Gmu}
\end{align}
Thus we need an $L^2$ estimate for each term in brackets.  The second
term involves the same partial derivative that occurs in
\eqref{E:P93-total-Umu-derivative}; we estimate it here with $v$ fixed,
without using any bound for $\partial_\mu v_\mu$.

For the first term, put $a=\mu^{1/(\alpha-2)}$.  Direct scalar
differentiation gives
\begin{equation}\label{E:lower-scaled-parameter-derivative}
 \mu\partial_\mu\mathfrak{g}_\mu(s)
 =\frac{as g'(as)-(\alpha-1)g(as)}
        {(\alpha-2)a^{\alpha-1}}.
\end{equation}
The two limits at zero in
\eqref{E:gsmalllarge}--\eqref{E:gprimelimits} cancel the leading term
in the numerator.  Consequently, for every fixed $M>0$,
\begin{equation}\label{E:lower-parameter-cancellation}
 \sup_{0\le s\le M}
 |\mu\partial_\mu\mathfrak{g}_\mu(s)|\longrightarrow0,
\end{equation}
and, more precisely, there are $\varepsilon_\mu\to0$ such that
\begin{equation}\label{E:lower-parameter-relative-bound}
 |\mu\partial_\mu\mathfrak{g}_\mu(s)|
 \le\varepsilon_\mu |s|^{\alpha-1},\qquad 0\le s\le M.
\end{equation}
The same estimate holds for the positive-part extension on $[-M,M]$,
since both $\mathfrak{g}_\mu^+(s)$ and its partial $\mu$-derivative
vanish for $s\le0$.  The Lyapunov--Schmidt neighborhood satisfies
\begin{equation}\label{E:P93-w-bounds}
 \sup_{0<\mu<\mu_{\mathrm d}}\sup_{v\in\mathcal V}
 \bigl(\|w_{\mu,v}\|_2+\|w_{\mu,v}\|_\infty\bigr)<+\infty.
\end{equation}
Choosing $M$ to dominate this $L^\infty$ bound, we obtain
\begin{align}
 \sup_{v\in\mathcal V}
 \|(\mu\partial_\mu\mathfrak{g}_\mu^+)(w_{\mu,v})\|_2
 &\le\varepsilon_\mu\sup_{v\in\mathcal V}
   \|w_{\mu,v}\|_\infty^{\alpha-2}\|w_{\mu,v}\|_2 \notag\\
 &\le C\varepsilon_\mu=o(1).
 \label{E:P93-partial-gmu-L2}
\end{align}

To estimate the second term in \eqref{E:P93-partial-Gmu},
differentiate \eqref{E:lower-Q-fixed-equation-differentiated} in
$\mu$, again with $v$ fixed.  This yields
\begin{align}
 (I-\mathcal{K}_{\mu,v})\,\mu\partial_\mu R_\mu(v)
 &=(\mu\partial_\mu\mathcal{A}_\mu^{-1})Q
       \mathfrak{g}_\mu^+(w_{\mu,v}) \notag\\
 &\quad+\mathcal{A}_\mu^{-1}Q
       (\mu\partial_\mu\mathfrak{g}_\mu^+)(w_{\mu,v}).
 \label{E:P93-differentiated-Q-equation}
\end{align}
The only new operator estimate needed here is for
$\mu\partial_\mu\mathcal{A}_\mu^{-1}$.  On $\mathscr{H}_Q$, the joint
Hermite--Fourier representation is
\begin{equation}\label{E:P93-Amu-inverse-multiplier}
 \widehat{(\mathcal{A}_\mu^{-1}Qf)}_\nu(\eta)
 =\frac{\widehat f_\nu(\eta)}{\rho_\nu/\mu+|\eta|^2+1},
 \qquad \rho_\nu=2|\nu|\ge2.
\end{equation}
Lemma~\ref{L:sec7-parameter-regularity}(i) justifies parameter
differentiation in operator norm on compact subintervals of $\mu>0$.
Differentiating the multiplier gives
\begin{equation}\label{E:lower-resolvent-parameter-derivative}
 \mu\partial_\mu\mathcal{A}_\mu^{-1}
 =\mathcal{A}_\mu^{-1}\mu^{-1}(H_y-d)\mathcal{A}_\mu^{-1}.
\end{equation}
The multiplier on the right is
$(\rho_\nu/\mu)(\rho_\nu/\mu+|\eta|^2+1)^{-2}$;
both it and the original resolvent multiplier are bounded by
$\mu/\rho_\nu\le\mu/2$.  Thus Plancherel gives
\begin{equation}\label{E:P93-resolvent-and-derivative-L2}
 \|\mathcal{A}_\mu^{-1}Qf\|_2\le\frac\mu2\|f\|_2,
 \qquad
 \|\mu\partial_\mu\mathcal{A}_\mu^{-1}Qf\|_2
 \le\frac\mu2\|f\|_2.
\end{equation}
In particular, the composition in
\eqref{E:lower-resolvent-parameter-derivative} is bounded despite the
unbounded factor $H_y-d$.

By \eqref{E:sec7-gmu-local-C1-bound} and \eqref{E:P93-w-bounds},
$\|\mathfrak{g}_\mu^+(w_{\mu,v})\|_2\le C$ uniformly in $\mu,v$.
Hence the two terms on the right of
\eqref{E:P93-differentiated-Q-equation} are respectively $O(\mu)$
and $o(\mu)$ in $L^2$, uniformly for $v\in\mathcal V$.
Using \eqref{E:P93-Kmu-inverse}, we obtain
\begin{equation}\label{E:lower-partial-R-derivative}
 \sup_{v\in\mathcal V}\|\mu\partial_\mu R_\mu(v)\|_2
 \le C\mu+o(\mu)=O(\mu).
\end{equation}
Returning to \eqref{E:P93-partial-Gmu}, the bounds for $\Pi$,
$B_{\mu,v}$, and $(-\Delta_\zeta+1)^{-1}:L^2\to H^1$ now give
\begin{equation}\label{E:P93-partial-Gmu-small}
 \sup_{v\in\mathcal V}
 \|\mu\partial_\mu\mathcal{G}_\mu(v)\|_{H^1(\mathbb{R}^m)}
 \le C(\varepsilon_\mu+\mu)=o(1).
\end{equation}

Both estimates required by \eqref{E:P93-total-Gmu-derivative} are
now established.  Combining that identity with
\eqref{E:P93-uniform-reduced-inverse} and
\eqref{E:P93-partial-Gmu-small}, we obtain
\begin{equation}\label{E:lower-v-derivative-small}
 \mu\partial_\mu v_\mu
 =-[D_v\mathcal{G}_\mu(v_\mu)]^{-1}
       \mu\partial_\mu\mathcal{G}_\mu(v_\mu)
 \longrightarrow0
 \quad\hbox{in }H^1_{\rm rad}(\mathbb{R}^m).
\end{equation}
Since $\|\phi_0\|_2=1$, this controls the first term of
\eqref{E:P93-total-Umu-derivative}.

\smallskip
\noindent\emph{Step 2: the remaining two terms.}
The estimate \eqref{E:lower-partial-R-derivative}, evaluated at $v_\mu$,
already controls the second term.  For the third, use
\eqref{E:lower-DR-L2-small} and
\eqref{E:lower-v-derivative-small}.  Explicitly,
\begin{align*}
 \|(\mu\partial_\mu R_\mu)(v_\mu)\|_2&\le C\mu\longrightarrow0,\\
 \|DR_\mu(v_\mu)[\mu\partial_\mu v_\mu]\|_2
 &\le C\mu\|\mu\partial_\mu v_\mu\|_2\longrightarrow0.
\end{align*}
All three terms in \eqref{E:P93-total-Umu-derivative} therefore tend
to zero in $L^2$, proving \eqref{E:lower-profile-derivative-small}.

\smallskip
\noindent\emph{The high-frequency end.}
By Proposition~\ref{P:high-unique-branch}, $U_\lambda$ is $C^1$ in
$X_r$.  Differentiating \eqref{E:stability-high-equation} gives
\begin{equation}\label{E:high-differentiated-equation}
 \mathcal{A}_\lambda(\partial_\lambda U_\lambda)
 =2\lambda^{-3}|Y|^2U_\lambda
  +(\partial_\lambda F_\lambda)(U_\lambda),
\end{equation}
where the last derivative keeps the scalar argument fixed.
The uniform inverse estimate \eqref{E:high-uniform-inverse} and
$L^2\hookrightarrow\mathfrak{X}_\lambda'$ give
\begin{align*}
 \|\lambda\partial_\lambda U_\lambda\|_{\mathfrak{X}_\lambda}
 &\le C\bigl\|
 2\lambda^{-2}|Y|^2U_\lambda
 +(\lambda\partial_\lambda F_\lambda)(U_\lambda)
 \bigr\|_{\mathfrak{X}_\lambda'}\\
 &\le C\bigl\|
 2\lambda^{-2}|Y|^2U_\lambda
 +(\lambda\partial_\lambda F_\lambda)(U_\lambda)
 \bigr\|_2.
\end{align*}
The embedding constant is independent of $\lambda$.  We now verify
that each source term tends to zero in $L^2$.

For the potential term, the common exponential bound in
\eqref{E:high-branch-convergence-strong} gives
\begin{equation}\label{E:P93-high-moment}
 \sup_{\lambda\ge\Lambda_{\mathrm h}}
 \||Y|^2U_\lambda\|_2<+\infty,
\end{equation}
so $\|2\lambda^{-2}|Y|^2U_\lambda\|_2\to0$.

For the nonlinear term, put $a=a_\lambda$ as in
\eqref{E:stability-high-rescaling}.  Scalar differentiation gives
\begin{equation}\label{E:high-scaled-parameter-derivative}
 \lambda\partial_\lambda F_\lambda(s)
 =\frac{as g'(as)-(\beta-1)g(as)}
        {(\beta-2)a^{\beta-1}}.
\end{equation}
The paired limits at infinity in
\eqref{E:gsmalllarge}--\eqref{E:gprimelimits} cancel the leading term.
Splitting $as$ at a fixed large number, as in the proof of
Lemma~\ref{lemma:20260618-1129}, yields, for every $M>0$,
\begin{equation}\label{E:high-parameter-cancellation}
 \sup_{0\le s\le M}
 |\lambda\partial_\lambda F_\lambda(s)|\longrightarrow0.
\end{equation}
This splitting avoids imposing a large-argument estimate when $s$ is
near zero.  Moreover, the global bounds
$|g(t)|\le C(t+t^{\beta-1})$ and
$|g'(t)|\le C(1+t^{\beta-2})$ give
\begin{equation}\label{E:high-parameter-domination}
 |\lambda\partial_\lambda F_\lambda(s)|
 \le C_M(s+s^{\beta-1}),\qquad 0\le s\le M,\quad\lambda\ge1,
\end{equation}
since $a_\lambda^{2-\beta}\le1$.
Choose $M$ to dominate the uniform $L^\infty$ bound for $U_\lambda$.
Then \eqref{E:high-parameter-cancellation} gives pointwise convergence
of $(\lambda\partial_\lambda F_\lambda)(U_\lambda)$ to zero, while
\eqref{E:high-parameter-domination} and the common exponential tail give
$$
 |(\lambda\partial_\lambda F_\lambda)(U_\lambda(\xi))|
 \le C\bigl(e^{-c|\xi|}+e^{-c(\beta-1)|\xi|}\bigr).
$$
Dominated convergence, together with the potential estimate, proves
\begin{equation}\label{E:P93-high-right-small}
 \left\|2\lambda^{-2}|Y|^2U_\lambda
       +(\lambda\partial_\lambda F_\lambda)(U_\lambda)\right\|_2
 \longrightarrow0.
\end{equation}
The inverse estimate displayed above now gives
$\|\lambda\partial_\lambda U_\lambda\|_{\mathfrak{X}_\lambda}\to0$.
Since this norm dominates the $H^1$-norm,
\eqref{E:high-profile-derivative-small} follows.

\smallskip
\noindent\emph{Differentiation of the mass.}
For the lower rescaling,
$$
 u_\lambda(y,z)
 =\mu^{1/(\alpha-2)}U_\mu(y,\sqrt\mu z),
 \qquad \mu=\lambda+d.
$$
The substitution $\zeta=\sqrt\mu z$ gives the exact identity
\begin{equation}\label{E:mass-exact-scaled-identities}
 \|u_\lambda\|_2^2
 =\mu^{\kappa_-}\|U_\mu\|_2^2,
 \qquad
 \kappa_-:=\frac2{\alpha-2}-\frac m2.
\end{equation}
Since $\mathrm d/\mathrm d\lambda=\mathrm d/\mathrm d\mu$, differentiation
of \eqref{E:mass-exact-scaled-identities} gives
\begin{align}
 \frac{\mathrm d}{\mathrm d\lambda}\|u_\lambda\|_2^2
 &=\mu^{\kappa_--1}
 \left\{\kappa_-\|U_\mu\|_2^2
 +2\langle U_\mu,\mu\partial_\mu U_\mu\rangle_{L^2}\right\}.
 \label{E:P93-lower-mass-derivative}
\end{align}
Now $U_\mu\to\phi_0V_\alpha$ in $L^2$, $\|\phi_0\|_2=1$, and
$\mu\partial_\mu U_\mu\to0$ in $L^2$.  The expression in braces is
therefore $\kappa_-\|V_\alpha\|_2^2+o(1)$, proving
\eqref{E:main-lower-mass-slope}.

Likewise, the high-frequency scaling
$u_\lambda(x)=\lambda^{1/(\beta-2)}U_\lambda(\sqrt\lambda x)$ gives
\begin{equation}\label{E:P93-high-exact-mass}
 \|u_\lambda\|_2^2
 =\lambda^{\kappa_+}\|U_\lambda\|_2^2,
 \qquad
 \kappa_+:=\frac2{\beta-2}-\frac N2.
\end{equation}
Differentiating this equality yields
\begin{align}
 \frac{\mathrm d}{\mathrm d\lambda}\|u_\lambda\|_2^2
 &=\lambda^{\kappa_+-1}
 \left\{\kappa_+\|U_\lambda\|_2^2
 +2\langle U_\lambda,
          \lambda\partial_\lambda U_\lambda\rangle_{L^2}\right\}.
 \label{E:P93-high-mass-derivative}
\end{align}
Using $U_\lambda\to W_\beta$ in $L^2$ and
$\lambda\partial_\lambda U_\lambda\to0$ in $H^1$, the expression in
braces is $\kappa_+\|W_\beta\|_2^2+o(1)$.  This proves
\eqref{E:main-high-mass-slope} and completes the proof.
\end{proof}

\subsection{The full linearized spectrum}

For an endpoint solution write
\begin{equation}\label{E:Lplus-Lminus}
 L_{+,\lambda}:=H+\lambda-g'(u_\lambda),
 \qquad
 L_{-,\lambda}:=H+\lambda-\frac{g(u_\lambda)}{u_\lambda}.
\end{equation}
The quotient in $L_{-,\lambda}$ is defined pointwise because
$u_\lambda>0$; moreover it tends to zero as $|x|\to\infty$ because
$g(s)/s\to0$ as $s\to0^+$.

\begin{proposition}[Endpoint spectral structure in the full space]
\label{P:full-endpoint-spectrum}
For $\lambda+d>0$ sufficiently small and for $\lambda$ sufficiently large,
\begin{align}
 &\operatorname{ind}_{X}(L_{+,\lambda})=1,
 \qquad
 \ker L_{+,\lambda}
 =\operatorname{span}\{\partial_{z_j}u_\lambda:1\le j\le m\},
 \label{E:full-Lplus-data}\\
 &L_{-,\lambda}\ge0,
 \qquad
 \ker L_{-,\lambda}=\operatorname{span}\{u_\lambda\}.
 \label{E:full-Lminus-data}
\end{align}
All kernels are taken in the real energy space $X$.
\end{proposition}

\begin{proof}
We first establish \eqref{E:full-Lminus-data} at both endpoints.
We then prove \eqref{E:full-Lplus-data} separately in the lower-edge
and high-frequency scalings.

\smallskip
\noindent\emph{The imaginary linearization.}
For
$h\in C_c^\infty(\mathbb{R}^N)$, the ground-state transform and the
equation for $u_\lambda$ give
\begin{equation}\label{E:ground-state-transform-Lminus}
 \langle L_{-,\lambda}h,h\rangle
 =\int_{\mathbb{R}^N}u_\lambda^2
 \left|\nabla\left(\frac h{u_\lambda}\right)\right|^2\,\mathrm dx.
\end{equation}
Since $g(u_\lambda)/u_\lambda\in L^\infty(\mathbb{R}^N)$ and
$C_c^\infty(\mathbb{R}^N)$ is dense in $X$, the quadratic form of
$L_{-,\lambda}$ is continuous on $X$.  Approximating an arbitrary
$h\in X$ in \eqref{E:ground-state-transform-Lminus} therefore gives
\begin{equation}\label{E:Lminus-nonnegative-form}
 \langle L_{-,\lambda}h,h\rangle\ge0,
 \qquad h\in X.
\end{equation}

It remains to identify the zero space.  The stationary equation gives
\begin{equation}\label{E:Lminus-phase-mode}
 L_{-,\lambda}u_\lambda=0.
\end{equation}
Conversely, let $h\in\ker L_{-,\lambda}$ and choose
$h_n\in C_c^\infty(\mathbb{R}^N)$ such that $h_n\to h$ in $X$.  Then
\begin{equation}\label{E:Lminus-form-approximation-zero}
 \langle L_{-,\lambda}h_n,h_n\rangle
 \longrightarrow
 \langle L_{-,\lambda}h,h\rangle=0.
\end{equation}
Fix $R>0$.  Since $u_\lambda>0$ and $u_\lambda$ is continuous, it is
bounded away from zero on $\overline{B_R}$.  Consequently,
\begin{equation}\label{E:Lminus-local-quotient-convergence}
 \frac{h_n}{u_\lambda}\longrightarrow\frac h{u_\lambda}
 \qquad\hbox{in }H^1(B_R).
\end{equation}
Using \eqref{E:ground-state-transform-Lminus} for $h_n$, we obtain
\begin{align*}
 \int_{B_R}u_\lambda^2
 \left|\nabla\left(\frac h{u_\lambda}\right)\right|^2\,\mathrm dx
 &=\lim_{n\to\infty}
 \int_{B_R}u_\lambda^2
 \left|\nabla\left(\frac{h_n}{u_\lambda}\right)\right|^2\,\mathrm dx\\
 &\le \lim_{n\to\infty}
 \int_{\mathbb{R}^N}u_\lambda^2
 \left|\nabla\left(\frac{h_n}{u_\lambda}\right)\right|^2\,\mathrm dx\\
 &=\lim_{n\to\infty}
 \langle L_{-,\lambda}h_n,h_n\rangle=0.
\end{align*}
Since $R>0$ is arbitrary and $\mathbb{R}^N$ is connected,
$h/u_\lambda$ is constant.  Hence $h\in\operatorname{span}\{u_\lambda\}$.
Together with \eqref{E:Lminus-phase-mode}, this proves
\eqref{E:full-Lminus-data}.

We turn to $L_{+,\lambda}$.  The spatial derivatives used below belong
to $X$ and to the operator domain of $L_{+,\lambda}$.  To see this, use
the fixed-solution Gaussian--exponential bound and interior estimates
on unit balls: $u_\lambda$, $\nabla u_\lambda$, and $D^2u_\lambda$ are
bounded by a polynomial in $|y|$ times a weaker product envelope.
Thus $\partial_{x_j}u_\lambda\in X$.  Differentiation of the equation
in distributions gives $L_{+,\lambda}\partial_{z_j}u_\lambda=0$ and
$L_{+,\lambda}\partial_{y_j}u_\lambda=-2y_ju_\lambda\in L^2$,
which establishes operator-domain membership through the form realization.
The same argument applies after either rescaling.  At high frequency
the needed bounds are uniform, since $F_\lambda'$ is uniformly bounded
on the profile range and $\lambda^{-2}|Y|^2\le |Y|^2$ for $\lambda\ge1$.

The first part of the kernel is visible
without taking any limit.  Since the potential $|y|^2$ is independent of
$z$, differentiating the stationary equation with respect to $z_j$ gives
\begin{equation}\label{E:z-exact-zero-modes}
 L_{+,\lambda}\partial_{z_j}u_\lambda=0,
 \qquad 1\le j\le m.
\end{equation}
Thus the $m$ free translations always produce zero modes.  The issue is to
show that they exhaust the kernel and that admitting nonradial
perturbations does not create another negative direction.  We do this
separately at the two ends.

\smallskip
\noindent\emph{The lower spectral edge.}
Here the oscillator complement is positive.  We eliminate it by a
Schur reduction and then use the exact free translations to exhaust
the limiting zero spectral cluster.
Put $\mu=\lambda+d$ and $\zeta=\sqrt\mu z$.  If
$h(y,z)=\Phi(y,\sqrt\mu z)$, then the identities used in
Lemma~\ref{L:lower-edge-Morse} give
\begin{equation}\label{E:full-lower-quadratic-rescaling}
 \langle L_{+,\lambda}h,h\rangle
 =\mu^{1-m/2}\langle\mathbb L_\mu\Phi,\Phi\rangle,
\end{equation}
where
\begin{equation}\label{E:full-lower-rescaled-Lplus}
 \mathbb L_\mu
 =\frac1\mu(H_y-d)-\Delta_\zeta+1-a_\mu(y,\zeta),
 \qquad
 a_\mu=(\mathfrak{g}_\mu^+)'(U_\mu).
\end{equation}
The factor in \eqref{E:full-lower-quadratic-rescaling} is positive, so the
rescaling preserves both the kernel and the number of negative
directions.

We now use the lowest-mode projection $P$ from
\eqref{E:sec7-PQ}, but we no longer impose radial symmetry in $\zeta$.
Write
\begin{equation}\label{E:full-lower-PQ-decomposition}
 \Phi(y,\zeta)=\phi_0(y)\xi(\zeta)+r(y,\zeta),
 \qquad Pr=0,
\end{equation}
with $\xi\in H^1(\mathbb{R}^m)$.  Relative to $P\oplus Q$, set
\begin{equation*}
 A_\mu=P\mathbb L_\mu P,\qquad
 B_\mu=P\mathbb L_\mu Q,\qquad
 C_\mu=Q\mathbb L_\mu P,\qquad
 D_\mu=Q\mathbb L_\mu Q.
\end{equation*}
Here $B_\mu=C_\mu^*$.  The oscillator gap
\eqref{E:sec7-gap} and the uniform boundedness of $a_\mu$ yield, for
$r\in QX$,
\begin{align}
 \langle D_\mu r,r\rangle
 &\ge \frac2\mu\|r\|_2^2
       +\|\nabla_\zeta r\|_2^2+\|r\|_2^2
       -\|a_\mu\|_\infty\|r\|_2^2 \notag\\
 &\ge \left(\frac2\mu-C\right)\|r\|_2^2
       +\|\nabla_\zeta r\|_2^2.
 \label{E:full-lower-D-positive}
\end{align}
Consequently $D_\mu$ is strictly positive for small $\mu$, and
$\|D_\mu^{-1}\|_{L^2\to L^2}=O(\mu)$.

This positivity allows us to eliminate $r$ without losing spectral
information.  Define the Schur complement on the $P$-space, identified
with $L^2(\mathbb{R}^m)$, by
\begin{equation}\label{E:full-lower-Schur-definition}
 S_\mu^{\rm full}:=A_\mu-B_\mu D_\mu^{-1}C_\mu.
\end{equation}
For $\Phi=\phi_0\xi+r$, a direct completion of the square gives
\begin{align}
 \langle\mathbb L_\mu\Phi,\Phi\rangle
 &=\left\langle D_\mu
   \bigl(r+D_\mu^{-1}C_\mu(\phi_0\xi)\bigr),
   r+D_\mu^{-1}C_\mu(\phi_0\xi)\right\rangle \notag\\
 &\quad+\langle S_\mu^{\rm full}\xi,\xi\rangle.
 \label{E:full-lower-square-completion}
\end{align}
Eliminating the $Q$-row of $\mathbb L_\mu(\phi_0\xi+r)=0$ gives the
exact kernel correspondence
\begin{equation}\label{E:full-lower-kernel-correspondence}
 \xi\in\ker S_\mu^{\rm full}
 \quad\Longleftrightarrow\quad
 \phi_0\xi-D_\mu^{-1}C_\mu(\phi_0\xi)
 \in\ker\mathbb L_\mu.
\end{equation}
Moreover, the first term in \eqref{E:full-lower-square-completion}
is nonnegative, and the triangular change of variables used there is
invertible on the form domain.  Thus the Morse indices of
$\mathbb L_\mu$ and $S_\mu^{\rm full}$ are equal, not merely bounded
in one direction.

It remains to identify the limiting Schur complement and, in particular, to
keep track of the multiplicity of the eigenvalue at the origin.  Set
\begin{equation}\label{E:full-lower-bmu-definition}
 b_\mu(\zeta):=
 \int_{\mathbb{R}^d}\phi_0(y)^2a_\mu(y,\zeta)\,\mathrm dy,
 \qquad
 b_0(\zeta):=c_\alpha(\alpha-1)V_\alpha(\zeta)^{\alpha-2}.
\end{equation}
Under the unitary identification $\phi_0\xi\leftrightarrow\xi$ of the
$P$-space with $L^2(\mathbb{R}^m)$, one has
\begin{equation*}
 A_\mu=-\Delta_\zeta+1-b_\mu.
\end{equation*}
Indeed, $U_\mu\to\phi_0V_\alpha$ and the $C^1$ convergence of the rescaled
nonlinearities give, by \eqref{E:amu-uniform-limit-Morse},
\begin{equation}\label{E:full-lower-amu-convergence}
 a_\mu\longrightarrow
 a_0:=\mu_0(\alpha-1)(\phi_0V_\alpha)^{\alpha-2}
 \qquad\hbox{in }L^\infty(\mathbb{R}^N).
\end{equation}
Consequently,
\begin{align*}
 \int_{\mathbb{R}^d}\phi_0(y)^2a_0(y,\zeta)\,\mathrm dy
 &=\mu_0(\alpha-1)V_\alpha(\zeta)^{\alpha-2}
   \int_{\mathbb{R}^d}\phi_0(y)^\alpha\,\mathrm dy\\
 &=b_0(\zeta),
\end{align*}
and hence
\begin{equation}\label{E:full-lower-bmu-convergence}
 \|b_\mu-b_0\|_{L^\infty(\mathbb{R}^m)}
 \le \|a_\mu-a_0\|_{L^\infty(\mathbb{R}^N)}
 \longrightarrow0.
\end{equation}
Equivalently, one may obtain \eqref{E:full-lower-bmu-convergence} directly
from local uniform convergence by first fixing a ball in the $\zeta$-space
and then using the common exponential tail of $a_\mu$ and $a_0$ on its
complement.  This explains precisely where the localization argument enters.

Let $M_{b_\mu-b_0}$ denote multiplication by $b_\mu-b_0$.  Since the
differential part of $\mathbb L_\mu$ preserves the decomposition $P\oplus Q$,
the off-diagonal blocks $B_\mu$ and $C_\mu$ come only from multiplication by
$-a_\mu$.  The uniform $L^\infty$ bound in
\eqref{E:full-lower-amu-convergence} therefore implies
\begin{equation}\label{E:full-lower-offdiagonal-bound}
 \|B_\mu\|_{L^2\to L^2}
 +\|C_\mu\|_{L^2\to L^2}\le C.
\end{equation}
Combining this with \eqref{E:full-lower-D-positive}, we obtain
\begin{equation}\label{E:full-lower-Schur-error-bound}
 \|B_\mu D_\mu^{-1}C_\mu\|_{L^2\to L^2}
 \le
 \|B_\mu\|\,\|D_\mu^{-1}\|\,\|C_\mu\|
 \le C\mu.
\end{equation}
Thus, if
\begin{equation}\label{E:full-lower-limit-operator}
 L_\alpha:=-\Delta_\zeta+1-b_0
 =-\Delta_\zeta+1-c_\alpha(\alpha-1)V_\alpha^{\alpha-2},
\end{equation}
then the definition of the Schur complement gives the exact decomposition
\begin{align}
 S_\mu^{\rm full}
 &=L_\alpha+E_\mu,
 \qquad
 E_\mu=-M_{b_\mu-b_0}-B_\mu D_\mu^{-1}C_\mu,
 \label{E:full-lower-Schur-convergence}\\
 \|E_\mu\|_{L^2\to L^2}
 &\le \|b_\mu-b_0\|_\infty+C\mu
 \longrightarrow0.                                      \label{E:full-lower-error-L2}
\end{align}
In particular, the convergence holds also in
$\mathcal L(H^1,H^{-1})$, but the stronger $L^2$-operator norm convergence
in \eqref{E:full-lower-error-L2} is useful for preserving spectral
multiplicities.

The spectral structure of $L_\alpha$ is classical.  It has exactly one
negative eigenvalue and
\begin{equation}\label{E:full-lower-limit-kernel}
 \ker L_\alpha
 =\operatorname{span}\{\partial_{\zeta_1}V_\alpha,\ldots,
                         \partial_{\zeta_m}V_\alpha\}
\end{equation}
\cite{ChangGustafsonNakanishiTsai}.  Moreover,
$\sigma_{\rm ess}(L_\alpha)=[1,+\infty)$ because
$V_\alpha^{\alpha-2}$ decays exponentially.  Hence zero is isolated with
multiplicity $m$.  We may choose $\rho>0$ so small that
\begin{equation}\label{E:full-lower-limit-spectral-gap}
 \sigma(L_\alpha)\cap[-2\rho,2\rho]=\{0\},
\end{equation}
and the unique negative eigenvalue of $L_\alpha$ lies to the left of
$-2\rho$.

For completeness, we now justify the stability of both relevant spectral
ranks.  Let $\Gamma_0$ be a positively oriented circle enclosing the origin
and contained in $\{z\in\mathbb C:|z|<2\rho\}$, and let $\Gamma_-$ be a
disjoint contour enclosing only the negative eigenvalue of $L_\alpha$.
For $z$ on either contour, \eqref{E:full-lower-error-L2} and the factorization
\begin{equation*}
 S_\mu^{\rm full}-z
 =\bigl[I+E_\mu(L_\alpha-z)^{-1}\bigr](L_\alpha-z)
\end{equation*}
show that $z\in\rho(S_\mu^{\rm full})$ for all sufficiently small $\mu$.
The resolvent identity then gives, uniformly on $\Gamma_0\cup\Gamma_-$,
\begin{equation}\label{E:full-lower-resolvent-convergence}
 (S_\mu^{\rm full}-z)^{-1}-(L_\alpha-z)^{-1}
 =-(S_\mu^{\rm full}-z)^{-1}E_\mu(L_\alpha-z)^{-1}
 \longrightarrow0
\end{equation}
in $\mathcal L(L^2,L^2)$.  Integrating
\eqref{E:full-lower-resolvent-convergence} over the two contours shows that
the corresponding Riesz projections converge in operator norm.  Two
projections at operator distance strictly less than one have the same rank;
therefore, for all sufficiently small $\mu$,
\begin{equation}\label{E:full-lower-spectral-ranks}
 \operatorname{rank}\mathbf 1_{(-\infty,-\rho)}
       (S_\mu^{\rm full})=1,
 \qquad
 \operatorname{rank}\mathbf 1_{(-\rho,\rho)}
       (S_\mu^{\rm full})=m.
\end{equation}
Here the first equality also uses the spectral inclusion for bounded
self-adjoint perturbations, which excludes any additional spectrum away from
the two contours when $\|E_\mu\|<\rho$.

It remains to determine whether the $m$ eigenvalues in the second cluster of
\eqref{E:full-lower-spectral-ranks} stay at zero or split to either side.
Translation invariance answers this question exactly.  From the scaling
formula for $u_\lambda$,
\begin{equation}\label{E:full-lower-translation-scaling}
 \partial_{z_j}u_\lambda(y,z)
 =\mu^{\frac1{\alpha-2}+\frac12}
   \partial_{\zeta_j}U_\mu(y,\sqrt\mu z),
\end{equation}
and \eqref{E:z-exact-zero-modes} therefore yields
\begin{equation}\label{E:full-lower-rescaled-exact-modes}
 \mathbb L_\mu\partial_{\zeta_j}U_\mu=0,
 \qquad 1\le j\le m.
\end{equation}
Write the $P$-component of this null vector as
\begin{equation}\label{E:full-lower-exact-Schur-modes}
 \xi_{\mu,j}(\zeta)
 :=\int_{\mathbb{R}^d}\phi_0(y)
       \partial_{\zeta_j}U_\mu(y,\zeta)\,\mathrm dy.
\end{equation}
The kernel correspondence \eqref{E:full-lower-kernel-correspondence} implies
$\xi_{\mu,j}\in\ker S_\mu^{\rm full}$.  Moreover, the strong $H^1$
convergence $U_\mu\to\phi_0V_\alpha$ gives
\begin{equation}\label{E:full-lower-exact-Schur-mode-limit}
 \xi_{\mu,j}\longrightarrow\partial_{\zeta_j}V_\alpha
 \qquad\hbox{in }L^2(\mathbb{R}^m).
\end{equation}
Since $V_\alpha$ is radial,
\begin{equation}\label{E:full-lower-translation-Gram-limit}
 \left\langle\partial_{\zeta_i}V_\alpha,
                    \partial_{\zeta_j}V_\alpha\right\rangle
 =\frac{\delta_{ij}}{m}\|\nabla V_\alpha\|_2^2.
\end{equation}
Thus the Gram matrix of
$\{\xi_{\mu,1},\ldots,\xi_{\mu,m}\}$ converges to a positive definite
matrix, and these $m$ vectors are linearly independent for all sufficiently
small $\mu$.  It follows that
\begin{equation*}
 m\le\dim\ker S_\mu^{\rm full}
 \le\operatorname{rank}\mathbf 1_{(-\rho,\rho)}
       (S_\mu^{\rm full})=m.
\end{equation*}
Consequently the entire near-zero spectral subspace is precisely the kernel:
\begin{equation}\label{E:full-lower-Schur-kernel-index}
 \ker S_\mu^{\rm full}
 =\operatorname{span}\{\xi_{\mu,1},\ldots,\xi_{\mu,m}\},
 \qquad
 \operatorname{ind}(S_\mu^{\rm full})=1.
\end{equation}
In particular, no member of the limiting translation cluster can drift to
the negative side: its total spectral multiplicity is $m$, and it is already
exhausted by the $m$ exact zero modes forced by translation invariance.

Finally, \eqref{E:full-lower-kernel-correspondence} and
\eqref{E:full-lower-square-completion} transfer both statements in
\eqref{E:full-lower-Schur-kernel-index} to $\mathbb L_\mu$; the corresponding
null vectors are exactly those in
\eqref{E:full-lower-rescaled-exact-modes}.  The positive rescaling factor in
\eqref{E:full-lower-quadratic-rescaling}, together with
\eqref{E:full-lower-translation-scaling}, then gives
\begin{equation*}
 \operatorname{ind}_{X}(L_{+,\lambda})=1,
 \qquad
 \ker L_{+,\lambda}
 =\operatorname{span}\{\partial_{z_j}u_\lambda:1\le j\le m\}
\end{equation*}
whenever $\lambda+d=\mu$ is sufficiently small.  This proves
\eqref{E:full-Lplus-data} at the lower spectral edge.

\smallskip
\noindent\emph{The high-frequency end.}
We work with the isotropic variables $\xi=\sqrt\lambda x$.  If
$h(x)=\Phi(\sqrt\lambda x)$, then
\begin{equation}\label{E:full-high-quadratic-rescaling}
 \langle L_{+,\lambda}h,h\rangle
 =\lambda^{1-N/2}\langle\mathcal{A}_\lambda\Phi,\Phi\rangle.
\end{equation}
Here $\mathcal{A}_\lambda$ denotes the full-space realization of
\eqref{E:stability-high-linearization}.  The prefactor is positive,
so our goal is
$$
 \operatorname{ind}(\mathcal{A}_\lambda)=1,\qquad
 \ker\mathcal{A}_\lambda
 =\operatorname{span}\{\partial_{Z_j}U_\lambda:1\le j\le m\}.
$$
The main issue is the effect of the confined translations: their span
is not invariant, so positivity on that span alone does not determine
the spectrum.  We first set up the Schur reduction and identify the
estimates it requires, then prove the complementary-block estimates,
and finally determine the reduced form on the translation space.

\smallskip
\noindent\emph{Step 1: the translation-space reduction.}
Put $\varepsilon=\lambda^{-2}$.  Translation invariance in $Z$ gives
\begin{equation}\label{E:rescaled-z-exact-zero-modes}
 \mathcal{A}_\lambda\partial_{Z_j}U_\lambda=0,
 \qquad 1\le j\le m.
\end{equation}
In contrast, differentiating the original equation in $y_j$ gives
\begin{equation}\label{E:y-translation-identity}
 L_{+,\lambda}\partial_{y_j}u_\lambda=-2y_ju_\lambda,
\end{equation}
and integration by parts yields
\begin{align}
 \left\langle L_{+,\lambda}\partial_{y_i}u_\lambda,
 \partial_{y_j}u_\lambda\right\rangle
 &=-2\int_{\mathbb{R}^N}y_i u_\lambda
       \partial_{y_j}u_\lambda\,\mathrm dx \notag\\
 &=-\int_{\mathbb{R}^N}y_i\partial_{y_j}(u_\lambda^2)\,\mathrm dx
 =\delta_{ij}\|u_\lambda\|_2^2.
 \label{E:y-translation-positive-matrix}
\end{align}
In the rescaled variables these identities become
\begin{equation}\label{E:rescaled-y-translation-residual}
 \mathcal{A}_\lambda\partial_{Y_j}U_\lambda
 =-2\varepsilon Y_jU_\lambda,
 \qquad
 \left\langle\mathcal{A}_\lambda\partial_{Y_i}U_\lambda,
 \partial_{Y_j}U_\lambda\right\rangle
 =\varepsilon\delta_{ij}\|U_\lambda\|_2^2.
\end{equation}
Thus the free translations are exact zero modes, while the confined
translations have an $O(\varepsilon)$ residual and a strictly positive
$O(\varepsilon)$ quadratic form.  To account for their coupling with
the remaining directions, introduce
\begin{align*}
 \mathcal{K}_{\lambda,Y}
 &:=\operatorname{span}\{\partial_{Y_1}U_\lambda,\ldots,
                           \partial_{Y_d}U_\lambda\},\\
 \mathcal{K}_{\lambda,Z}
 &:=\operatorname{span}\{\partial_{Z_1}U_\lambda,\ldots,
                           \partial_{Z_m}U_\lambda\},
 \qquad
 \mathcal{K}_\lambda
 :=\mathcal{K}_{\lambda,Y}\oplus\mathcal{K}_{\lambda,Z}.
\end{align*}
The sum is orthogonal by parity.  Moreover, the strong $H^1$ convergence
$U_\lambda\to W_\beta$ gives convergence of the spatial derivatives in
$L^2$, and hence
\begin{equation}\label{E:full-high-K-Gram}
 \left(
 \left\langle\partial_{\xi_i}U_\lambda,
                    \partial_{\xi_j}U_\lambda\right\rangle
 \right)_{1\le i,j\le N}
 \longrightarrow
 \left(
 \left\langle\partial_{\xi_i}W_\beta,
                    \partial_{\xi_j}W_\beta\right\rangle
 \right)_{1\le i,j\le N}
 =\frac{\|\nabla W_\beta\|_2^2}{N}I_N.
\end{equation}
Consequently $\dim\mathcal{K}_\lambda=N$ for large $\lambda$, and the
Euclidean norm of the coefficients in this derivative basis is
uniformly equivalent to the $L^2$ norm of the resulting vector.
Let $P_\lambda$ be the orthogonal projection onto $\mathcal{K}_\lambda$
and put $Q_\lambda=I-P_\lambda$.
Define the self-adjoint form compression on
$\mathcal{K}_\lambda^\perp\cap\mathfrak{X}_\lambda$ by
\begin{equation}\label{E:full-high-D-definition}
 D_\lambda:=Q_\lambda\mathcal{A}_\lambda Q_\lambda.
\end{equation}

For the moment, suppose that $D_\lambda$ is invertible; this will be
proved in Step~2.  The Feshbach--Schur operator at zero is then
\begin{equation}\label{E:full-high-Feshbach-operator}
 \mathcal{S}_\lambda
 :=P_\lambda\mathcal{A}_\lambda P_\lambda
 -P_\lambda\mathcal{A}_\lambda Q_\lambda
  D_\lambda^{-1}
  Q_\lambda\mathcal{A}_\lambda P_\lambda.
\end{equation}
Writing $\Phi=p+q$, with $p=P_\lambda\Phi$ and $q=Q_\lambda\Phi$,
the $Q_\lambda$-row of $\mathcal{A}_\lambda\Phi=0$, i.e., $Q_\lambda \mathcal{A}_\lambda P_\lambda p+D_\lambda q=0$,  gives
\begin{equation}\label{E:full-high-Q-elimination}
 q=-D_\lambda^{-1}Q_\lambda\mathcal{A}_\lambda P_\lambda p.
\end{equation}
Substitution into the $P_\lambda$-row, i.e., $P_\lambda\mathcal{A}_\lambda P_\lambda p +P_\lambda\mathcal{A}_\lambda Q_\lambda q=0$, gives
$$P_\lambda\mathcal{A}_\lambda P_\lambda p - P_\lambda\mathcal{A}_\lambda Q_\lambda D_\lambda^{-1}Q_\lambda\mathcal{A}_\lambda P_\lambda p=0.$$
That is,
$$\mathcal{S}_\lambda p=0.$$
Conversely,
\eqref{E:full-high-Q-elimination} lifts every element of
$\ker\mathcal{S}_\lambda$ to an element of
$\ker\mathcal{A}_\lambda$.  At the level of quadratic forms, the same
elimination gives
\begin{align}
 \langle\mathcal{A}_\lambda(p+q),p+q\rangle
 &=
 \left\langle
 D_\lambda\bigl(q+D_\lambda^{-1}
 Q_\lambda\mathcal{A}_\lambda P_\lambda p\bigr),
 q+D_\lambda^{-1}
 Q_\lambda\mathcal{A}_\lambda P_\lambda p
 \right\rangle \notag\\
 &\quad+\langle\mathcal{S}_\lambda p,p\rangle.
 \label{E:full-high-square-completion}
\end{align}
These are identities of closed quadratic forms: the translation derivatives belong
to the operator domain, so the off-diagonal blocks are bounded
finite-rank operators, and $D_\lambda^{-1}Q_\lambda\mathcal{A}_\lambda
P_\lambda p$ belongs to the form domain of $D_\lambda$.
The triangular change of variables in
\eqref{E:full-high-square-completion} is invertible on the form domain.
It therefore yields
\begin{equation}\label{E:full-high-inertia-splitting}
 \operatorname{ind}(\mathcal{A}_\lambda)
 =\operatorname{ind}(D_\lambda)
  +\operatorname{ind}(\mathcal{S}_\lambda).
\end{equation}

This reduction identifies the two remaining tasks.  In Step~2 we prove
that $D_\lambda$ has a uniformly bounded inverse and exactly one
negative direction.  In Step~3 we show that $\mathcal{S}_\lambda$ is
positive definite on $\mathcal{K}_{\lambda,Y}$ and vanishes on
$\mathcal{K}_{\lambda,Z}$.  The key comparison there is between the
$O(\varepsilon)$ positive direct term in
\eqref{E:rescaled-y-translation-residual} and an $O(\varepsilon^2)$
Schur correction.  Unlike the lower-edge $Q$-block, $D_\lambda$ is not
assumed positive; its negative direction must be retained in
\eqref{E:full-high-inertia-splitting}.

\smallskip
\noindent\emph{Step 2: the complementary block.}
The limiting operator needed for this step is
\begin{equation}\label{E:full-high-limit-operator}
 L_\beta=-\Delta+1-\mu_\infty(\beta-1)W_\beta^{\beta-2}.
\end{equation}
It has exactly one negative eigenvalue and
\begin{equation}\label{E:full-high-limit-kernel}
 \ker L_\beta
 =\operatorname{span}\{\partial_{\xi_1}W_\beta,\ldots,
                         \partial_{\xi_N}W_\beta\}.
\end{equation}
Write $b_\lambda=F_\lambda'(U_\lambda)$ and
$b=\mu_\infty(\beta-1)W_\beta^{\beta-2}$.  The convergence of the
profiles and rescaled nonlinearities gives local uniform convergence
$b_\lambda\to b$; the bound \eqref{E:high-derivative-power-envelope}
and the common exponential tail of $U_\lambda$ also give
\eqref{E:high-coefficient-tail}.  These are the compactness inputs
used in both parts of this step.

\smallskip
\noindent\emph{(a) Uniform invertibility.}
We claim that $D_\lambda$ is invertible for all sufficiently large
$\lambda$, with
\begin{equation}\label{E:full-high-D-inverse-bound}
 \|D_\lambda^{-1}\|_{L^2\to L^2}\le C.
\end{equation}
To prove this by a limiting-kernel argument, we first need to pass
from the compressed equation to the full one.  For
$q\in\mathcal D(D_\lambda)$, the compression gives
\begin{equation}\label{E:full-high-compression-action}
 D_\lambda q=Q_\lambda\mathcal{A}_\lambda q,
 \qquad
 \langle D_\lambda q,q\rangle
 =\langle\mathcal{A}_\lambda q,q\rangle,
\end{equation}
initially in the form sense.  Thus smallness of $D_\lambda q$ controls
only the $Q_\lambda$-component; the $P_\lambda$-component must also be
estimated.

The translation identities
\eqref{E:rescaled-z-exact-zero-modes} and
\eqref{E:rescaled-y-translation-residual}, the uniform bound
$\|Y_jU_\lambda\|_2\le C$, and \eqref{E:full-high-K-Gram} give
\begin{equation}\label{E:full-high-K-residual-bound}
 \|\mathcal{A}_\lambda p\|_2
 \le C\varepsilon\|p\|_2,
 \qquad p\in\mathcal{K}_\lambda.
\end{equation}
By self-adjointness, this also controls the missing component:
\begin{align}
 \|P_\lambda\mathcal{A}_\lambda q\|_2
 &=\sup_{\substack{p\in\mathcal{K}_\lambda\\ \|p\|_2=1}}
   |\langle\mathcal{A}_\lambda q,p\rangle| \notag\\
 &=\sup_{\substack{p\in\mathcal{K}_\lambda\\ \|p\|_2=1}}
   |\langle q,\mathcal{A}_\lambda p\rangle|
 \le C\varepsilon\|q\|_2.                 \label{E:full-high-missing-component}
\end{align}
For $q\in\mathcal D(D_\lambda)$, the left side is initially defined
by the finitely many form pairings with $p\in\mathcal{K}_\lambda$.
Together with $D_\lambda q\in L^2$, the estimate shows that
$\mathcal{A}_\lambda q$ is represented by an $L^2$ function.  Hence
$q\in\mathcal D(\mathcal{A}_\lambda)$ and
$$
 \mathcal{A}_\lambda q
 =D_\lambda q+P_\lambda\mathcal{A}_\lambda q.
$$
The off-diagonal blocks are bounded finite-rank operators.  Thus
$D_\lambda\oplus(P_\lambda\mathcal{A}_\lambda P_\lambda)$ is a
bounded finite-rank perturbation of $\mathcal{A}_\lambda$; in
particular, the negative spectrum of $D_\lambda$ is discrete.

Suppose now that the claimed uniform invertibility fails.  By the
spectral theorem for the self-adjoint compressions, there are
$\lambda_n\to\infty$ and $q_n\in\mathcal D(D_{\lambda_n})$ such that
\begin{equation}\label{E:full-high-D-approximate-kernel}
 q_n\in\mathcal{K}_{\lambda_n}^{\perp},
 \qquad
 \|q_n\|_2=1,
 \qquad
 \|D_{\lambda_n}q_n\|_2\longrightarrow0.
\end{equation}
Combining the two components just estimated gives
\begin{equation}\label{E:full-high-full-approximate-kernel}
 \|\mathcal{A}_{\lambda_n}q_n\|_2
 \le \|D_{\lambda_n}q_n\|_2+C\lambda_n^{-2}
 \longrightarrow0.
\end{equation}
Taking the scalar product with $q_n$ and using the uniform
$L^\infty$ bound for $b_{\lambda_n}$, we obtain
\begin{equation}\label{E:full-high-D-energy-bound}
 \|\nabla q_n\|_2^2+\|q_n\|_2^2
 +\lambda_n^{-2}\||Y|q_n\|_2^2
 =\int_{\mathbb{R}^N}b_{\lambda_n}q_n^2\,\mathrm d\xi+o(1)
 \le C.
\end{equation}
To retain a nonzero limit, we also need $L^2$ tightness.
Let $\chi_R$ be the exterior cutoff used in
Lemma~\ref{L:high-spectral-compactness}.  Testing the approximate
kernel equation by $\chi_R^2q_n$ and using the IMS identity yields
\begin{align}
 &\|\nabla(\chi_Rq_n)\|_2^2+\|\chi_Rq_n\|_2^2
 +\lambda_n^{-2}\||Y|\chi_Rq_n\|_2^2 \notag\\
 &\qquad\le
 \int_{\mathbb{R}^N}b_{\lambda_n}\chi_R^2q_n^2\,\mathrm d\xi
 +\frac{C}{R^2}+o_n(1).                    \label{E:full-high-D-tail-estimate}
\end{align}
For $R$ sufficiently large, the common tail estimate absorbs the first
term on the right into the left.  Consequently
\begin{equation}\label{E:full-high-D-tightness}
 \lim_{R\to\infty}\limsup_{n\to\infty}
 \int_{\mathbb{R}^N\setminus B_{2R}}|q_n|^2\,\mathrm d\xi=0.
\end{equation}
Local Rellich compactness and this tightness give, after extraction,
\begin{equation}\label{E:full-high-D-strong-limit}
 q_n\rightharpoonup q\quad\hbox{in }H^1,
 \qquad
 q_n\longrightarrow q\quad\hbox{in }L^2,
 \qquad
 \|q\|_2=1.
\end{equation}
Passing to the limit in
\eqref{E:full-high-full-approximate-kernel} against compactly
supported tests gives $L_\beta q=0$.  On the other hand,
$q_n\perp\mathcal{K}_{\lambda_n}$ and the spatial derivatives
converge strongly in $L^2$, so
$$
 \langle q,\partial_{\xi_j}W_\beta\rangle
 =\lim_{n\to\infty}
   \langle q_n,\partial_{\xi_j}U_{\lambda_n}\rangle=0,
 \qquad 1\le j\le N.
$$
Thus $q\in\ker L_\beta\cap(\ker L_\beta)^\perp=\{0\}$,
contradicting $\|q\|_2=1$.  This proves
\eqref{E:full-high-D-inverse-bound} and makes the Schur reduction in
Step~1 available for all sufficiently large $\lambda$.

\smallskip
\noindent\emph{(b) Exactly one negative direction.}
The inverse bound excludes spectrum near zero, but does not determine
the negative index of $D_\lambda$.  To identify a negative direction,
we will project the persistent negative eigenfunction of
$\mathcal{A}_\lambda$ onto $\mathcal{K}_\lambda^\perp$.  We first
locate that eigenfunction spectrally and compare the translation space
with the true near-zero spectral space.

Choose $\rho>0$ so small that
\begin{equation}\label{E:full-high-limit-spectral-gap}
 \sigma(L_\beta)\cap[-2\rho,2\rho]=\{0\},
\end{equation}
and the unique negative eigenvalue of $L_\beta$ lies to the left of
$-2\rho$.  Choose $M>0$ so large that all the negative spectra under
consideration lie in $(-M,0)$.  Let $\Pi_{\lambda,-}$ and
$\Pi_{\lambda,0}$ be the spectral projections of $\mathcal{A}_\lambda$
onto $(-M,-2\rho)$ and $(-\rho,\rho)$, respectively.
Lemma~\ref{L:high-spectral-compactness}, with the coefficients
$b_\lambda$ and $b$ above, gives
\begin{equation}\label{E:full-high-Riesz-ranks}
 \operatorname{rank}\Pi_{\lambda,-}=1,
 \qquad
 \operatorname{rank}\Pi_{\lambda,0}=N
\end{equation}
for all sufficiently large $\lambda$.  Applying the same lemma to
closed intervals contained in the limiting gaps rules out any other
spectrum between these clusters.  These ranks locate one persistent
negative eigenvalue, but do not yet determine the signs within the
near-zero cluster.

Set $\mathcal E_\lambda:=\operatorname{Ran}\Pi_{\lambda,0}$.
For any sequence $\lambda_n\to\infty$, choose an orthonormal
eigenbasis $\{e_{n,1},\ldots,e_{n,N}\}$ of
$\mathcal E_{\lambda_n}$.  Lemma~\ref{L:high-spectral-compactness}
and a finite diagonal extraction give
$$
 e_{n,j}\longrightarrow e_j\quad\hbox{strongly in }L^2,
 \qquad 1\le j\le N.
$$
By \eqref{E:full-high-limit-spectral-gap}, the limiting eigenvalues
are all zero.  Strong convergence preserves orthonormality, so the
$e_j$ form an orthonormal basis of $\ker L_\beta$.  The spatial
derivatives spanning $\mathcal{K}_{\lambda_n}$ converge to a basis
of the same kernel, with nonsingular limiting Gram matrix by
\eqref{E:full-high-K-Gram}.  Thus, first along the extracted
subsequence and then for the full family,
\begin{equation}\label{E:full-high-projection-gap}
 \|P_\lambda-\Pi_{\lambda,0}\|_{L^2\to L^2}\longrightarrow0.
\end{equation}

Let $e_{\lambda,-}$ be a normalized eigenfunction for the persistent
negative eigenvalue $\nu_{\lambda,-}\le-c_0<0$ of
$\mathcal{A}_\lambda$.  Since $e_{\lambda,-}\perp\mathcal E_\lambda$,
\eqref{E:full-high-projection-gap} gives
$\|P_\lambda e_{\lambda,-}\|_2\to0$.  Hence
$q_{\lambda,-}:=Q_\lambda e_{\lambda,-}$ satisfies
$\|q_{\lambda,-}\|_2\to1$, and
\eqref{E:full-high-K-residual-bound} yields
\begin{align}
 \langle D_\lambda q_{\lambda,-},q_{\lambda,-}\rangle
 &=\langle\mathcal{A}_\lambda q_{\lambda,-},
                         q_{\lambda,-}\rangle \notag\\
 &=\nu_{\lambda,-}\|q_{\lambda,-}\|_2^2
   -\langle\mathcal{A}_\lambda P_\lambda e_{\lambda,-},
                         q_{\lambda,-}\rangle
 <0.                                             \label{E:full-high-D-negative-direction}
\end{align}
Thus $\operatorname{ind}(D_\lambda)\ge1$.  In particular, the negative
direction survives the projection; no exact inclusion of
$e_{\lambda,-}$ in $\mathcal{K}_\lambda^\perp$ is needed.

For the reverse inequality, suppose that
$\operatorname{ind}(D_{\lambda_n})\ge2$ along some
$\lambda_n\to\infty$.  The negative spectrum is discrete, and the
uniform gap \eqref{E:full-high-D-inverse-bound} allows us to choose
two orthonormal eigenfunctions satisfying
\begin{equation}\label{E:full-high-D-two-negative-modes}
 D_{\lambda_n}q_{n,j}=\nu_{n,j}q_{n,j},
 \qquad
 -C\le\nu_{n,j}\le-c<0,
 \qquad j=1,2.
\end{equation}
Estimate \eqref{E:full-high-missing-component} gives the full
approximate eigenvalue equations
\begin{equation}\label{E:full-high-D-approximate-eigenpairs}
 \|\mathcal{A}_{\lambda_n}q_{n,j}
      -\nu_{n,j}q_{n,j}\|_2
 =\|P_{\lambda_n}\mathcal{A}_{\lambda_n}q_{n,j}\|_2
 \le C\lambda_n^{-2}.
\end{equation}
The energy and cutoff argument in part~(a) applies to these equations
as well: the coefficient $1$ in the tail estimate is replaced by
$1-\nu_{n,j}\ge1+c$.  After a finite diagonal extraction,
$$
 \nu_{n,j}\to\nu_j<0,
 \qquad
 q_{n,j}\to q_j\quad\hbox{strongly in }L^2,
 \qquad j=1,2.
$$
Passing to the limit gives $L_\beta q_j=\nu_jq_j$, and strong
$L^2$ convergence preserves orthogonality:
$$
 \langle q_i,q_j\rangle
 =\lim_{n\to\infty}\langle q_{n,i},q_{n,j}\rangle
 =\delta_{ij}.
$$
This contradicts $\operatorname{ind}(L_\beta)=1$.  We have proved
\begin{equation}\label{E:full-high-complement-index}
 \operatorname{ind}(D_\lambda)
 =\operatorname{ind}\bigl(
 L_\beta\big|_{(\ker L_\beta)^\perp}\bigr)=1.
\end{equation}
The two complementary-block conclusions are now established:
$D_\lambda$ is uniformly invertible and has one negative direction.
These are compatible because that negative spectrum stays uniformly
away from zero.

\smallskip
\noindent\emph{Step 3: the reduced translation block and the conclusion.}
We now return to \eqref{E:full-high-Feshbach-operator}.  Step~2 makes
its definition valid and supplies the uniform inverse bound needed
to estimate the correction.  On $\mathcal{K}_{\lambda,Y}$, the direct
term has, in the derivative basis, the matrix
\begin{equation}\label{E:full-high-direct-Y-block}
 \varepsilon\|U_\lambda\|_2^2 I_d
 =\varepsilon\bigl(\|W_\beta\|_2^2+o(1)\bigr)I_d.
\end{equation}
For $p\in\mathcal{K}_{\lambda,Y}$, the correction satisfies
\begin{align*}
 &\left|\left\langle
 D_\lambda^{-1}Q_\lambda\mathcal{A}_\lambda p,
 Q_\lambda\mathcal{A}_\lambda p\right\rangle\right|\\
 &\qquad\le
 \|D_\lambda^{-1}\|_{L^2\to L^2}
 \|Q_\lambda\mathcal{A}_\lambda p\|_2^2
 \le C\varepsilon^2\|p\|_2^2,
\end{align*}
by \eqref{E:full-high-D-inverse-bound} and
\eqref{E:full-high-K-residual-bound}.  No sign of $D_\lambda^{-1}$
is required.  Using the uniformly nonsingular Gram matrix
\eqref{E:full-high-K-Gram} to express this estimate in the derivative
basis, we obtain
\begin{equation}\label{E:full-high-positive-Y-Feshbach}
 \bigl(\langle\mathcal{S}_\lambda\partial_{Y_i}U_\lambda,
        \partial_{Y_j}U_\lambda\rangle\bigr)_{i,j=1}^d
 =\varepsilon\|U_\lambda\|_2^2 I_d+O(\varepsilon^2)>0.
\end{equation}
Indeed, $\|U_\lambda\|_2^2\to\|W_\beta\|_2^2>0$, so the positive
$O(\varepsilon)$ term dominates the $O(\varepsilon^2)$ correction.
This proves positivity of the reduced form on the confined
translation space, including the effect of coupling to its complement.

On $\mathcal{K}_{\lambda,Z}$, the exact identities
\eqref{E:rescaled-z-exact-zero-modes} make the corresponding columns
of both terms in \eqref{E:full-high-Feshbach-operator} vanish.
Self-adjointness makes the corresponding rows, including all
$Y$--$Z$ cross terms, vanish as well.  Hence
\begin{equation}\label{E:full-high-Feshbach-kernel-index}
 \ker\mathcal{S}_\lambda=\mathcal{K}_{\lambda,Z},
 \qquad
 \operatorname{ind}(\mathcal{S}_\lambda)=0.
\end{equation}
Thus the Schur complement is positive on the $Y$-block and zero on the
$Z$-block, rather than positive definite on all of $\mathcal{K}_\lambda$.

Finally, for $p\in\mathcal{K}_{\lambda,Z}$ we have
$\mathcal{A}_\lambda p=0$, so the complementary correction in
\eqref{E:full-high-Q-elimination} also vanishes.  The kernel
correspondence from Step~1 therefore gives
\begin{equation}\label{E:full-high-exact-kernel}
 \ker\mathcal{A}_\lambda
 =\mathcal{K}_{\lambda,Z}
 =\operatorname{span}\{\partial_{Z_j}U_\lambda:1\le j\le m\}.
\end{equation}
The index follows from \eqref{E:full-high-inertia-splitting},
\eqref{E:full-high-complement-index}, and
\eqref{E:full-high-Feshbach-kernel-index}:
$$
 \operatorname{ind}(\mathcal{A}_\lambda)=1+0=1.
$$
Together with \eqref{E:full-high-Riesz-ranks}, these conclusions show
that the near-zero cluster consists of the $m$ exact zero modes and
$d$ positive eigenvalues.  Returning to the original variables through
\eqref{E:full-high-quadratic-rescaling} proves
\eqref{E:full-Lplus-data} at high frequency and completes the proof.
\end{proof}

\subsection{Slope criterion and orbital stability}

Assume \ref{condition-D} throughout this subsection.  By definition, the flow
\eqref{E:main-dynamical-NLS} is locally well posed on $X_{\mathbb{C}}$, has
continuous dependence and the continuation property, and conserves the mass and energy in \eqref{E:mass-energy-dynamics}.
The separation of this dynamical input from the elliptic assumptions is
explained in Remark~\ref{R:dynamical-assumption}.  The orbit of a positive
state is
\begin{equation}\label{E:standing-wave-orbit}
 \mathcal O(u_\lambda)
 :=\{e^{i\vartheta}u_\lambda(y,z-z_0):
       \vartheta\in\mathbb{R},\ z_0\in\mathbb{R}^m\}.
\end{equation}
Orbital stability means that for every $\varepsilon>0$ there exists
$\delta>0$ such that
\begin{equation*}
 \operatorname{dist}_X(\psi(0),\mathcal O(u_\lambda))<\delta
 \quad\Longrightarrow\quad
 \psi\text{ is global and }\sup_{t\in\mathbb{R}}
 \operatorname{dist}_X(\psi(t),\mathcal O(u_\lambda))<\varepsilon.
\end{equation*}

\begin{lemma}[Slope criterion for the partially confined flow]
\label{L:GSS-Weinstein-translations}
Assume \ref{gcon1}--\ref{gcon2} and \ref{condition-D}.  Let
$I\ni\lambda\mapsto v_\lambda\in X_r$ be a $C^1$ family of positive
standing-wave profiles for \eqref{E:main-dynamical-NLS}, and set
$$
 S_\lambda:=E+\frac{\lambda}{2}M,
 \qquad d(\lambda):=S_\lambda(v_\lambda).
$$
At the frequency under consideration, assume
\begin{equation}\label{E:abstract-GSS-spectral-assumptions}
 \operatorname{ind}\bigl(S_\lambda''(v_\lambda)\bigr)=1,
 \qquad
 \ker S_\lambda''(v_\lambda)
 =\operatorname{span}_{\mathbb{R}}
 \{iv_\lambda,\partial_{z_1}v_\lambda,\ldots,
                    \partial_{z_m}v_\lambda\},
\end{equation}
on the realification of $X_{\mathbb C}$, with zero isolated from the
remaining spectrum.  If $d''(\lambda)>0$, the standing wave
$e^{i\lambda t}v_\lambda$ is orbitally stable modulo phase and
$z$-translations.  If $d''(\lambda)<0$, it is orbitally unstable modulo
the same symmetries.
\end{lemma}

\begin{proof}
Write $v=v_\lambda$ and $(f,h)_{\mathbb{R}}=\operatorname{Re}\langle f,h\rangle$.
The action Hessian splits into the real and imaginary blocks $L_+$ and
$L_-$.  Since $v>0$ and $L_-v=0$, the ground-state transform gives
$L_-\ge0$ with kernel $\operatorname{span}\{v\}$.  Thus $L_+$ has exactly
one negative eigenvalue.  Differentiating the profile equation gives
$$
 L_+\partial_\lambda v_\lambda=-v,
 \qquad
 d''(\lambda)=-\langle L_+^{-1}v,v\rangle,
$$
where the inverse is taken on $(\ker L_+)^\perp$.

Suppose first that $d''(\lambda)>0$.  The one-constraint inertia formula
and the spectral gap give
$$
 \langle S_\lambda''(v)h,h\rangle\ge c\|h\|_X^2
 \quad\text{if}\quad
 (h,v)_{\mathbb{R}}=(h,iv)_{\mathbb{R}}
 =(h,\partial_{z_j}v)_{\mathbb{R}}=0\quad(1\le j\le m).
$$
Here coercivity in the form norm follows from the $L^2$ spectral gap and
the boundedness of the nonlinear potentials.  Local modulation in phase
and translation is available because the orbit tangent vectors have a
nonsingular Gram matrix.  The usual mass-constrained Lyapunov argument,
using conservation of $E$ and $M$, now proves stability in the full energy
space; nearby masses are accommodated by $\frac{\mathrm d}{\mathrm d\lambda}M(v_\lambda)=2d''(\lambda)\ne0$.
This is the constrained-coercivity argument of
\cite{GrillakisShatahStraussI,WeinsteinStability}.

For $d''(\lambda)<0$, we apply the instability theorem on the complex
block-radial subspace $(X_r)_{\mathbb C}$.  This subspace is invariant
under the flow, by rotational equivariance and uniqueness of the Cauchy
problem.  The unique negative eigenfunction of $L_+$ can be chosen
strictly positive and is simple, so invariance of the coefficients makes
it block-radial.  The translation kernels are odd in a free coordinate
and do not belong to $X_r$.  Consequently the restricted action Hessian
has Morse index one and kernel $\operatorname{span}_{\mathbb{R}}\{iv\}$.
The one-symmetry GSS instability theorem therefore applies, with odd
index $1$ when $d''(\lambda)<0$; the required flow regularity is part of
\ref{condition-D} \cite{GrillakisShatahStraussI,GrillakisShatahStraussII}.

It remains to compare the two notions of orbital distance.  For
block-radial $\psi$ sufficiently close to $\mathcal O(v)$,
\begin{equation}\label{E:GSS-radial-orbit-distance}
 \inf_{\vartheta\in\mathbb{R}}\|\psi-e^{i\vartheta}v\|_X
 \le C\operatorname{dist}_X(\psi,\mathcal O(v)).
\end{equation}
Indeed, if $\psi$ is close to $e^{i\vartheta}v(y,z-a)$, reflection in $z$
shows that $v(y,z-a)$ is close in $L^2$ to $v(y,z+a)$.  This forces $a$
to be small: equality is possible only for $a=0$, and the two translates
become asymptotically orthogonal as $|a|\to\infty$.  Local phase and
translation modulation is then unique, and reflection sends its
translation parameter to its negative while leaving $\psi$ unchanged.
That parameter is therefore zero, proving
\eqref{E:GSS-radial-orbit-distance}.  Stability modulo the full orbit
would thus imply phase-orbital stability within the invariant
block-radial subspace, contradicting the preceding GSS conclusion.
\end{proof}

\begin{proof}[Proof of the stability assertions in
Theorem~\ref{T:endpoint-stability}]
Introduce the action
\begin{equation}\label{E:action-functional}
 S_\lambda(\psi):=E(\psi)+\frac\lambda2M(\psi),
 \qquad
 d(\lambda):=S_\lambda(u_\lambda).
\end{equation}
The standing-wave equation is $S_\lambda'(u_\lambda)=0$.  Since the branch
is $C^1$,
\begin{equation}\label{E:dprime-dsecond}
 d'(\lambda)=\frac{1}{2}M(u_\lambda),
 \qquad
 d''(\lambda)=\frac{1}{2}\frac{\mathrm d}{\mathrm d\lambda}
                         \|u_\lambda\|_2^2.
\end{equation}
Writing a perturbation as $a+ib$, the Hessian separates as
\begin{equation}\label{E:action-hessian-splitting}
 \langle S_\lambda''(u_\lambda)(a+ib),a+ib\rangle
 =\langle L_{+,\lambda}a,a\rangle
  +\langle L_{-,\lambda}b,b\rangle.
\end{equation}
Assumption~\ref{condition-D} gives the flow hypotheses in
Lemma~\ref{L:GSS-Weinstein-translations}.  The endpoint branches are
$C^1$ by the lower-edge construction and
Propositions~\ref{P:high-unique-branch} and
\ref{P:differentiated-profiles}.  We verify the spectral hypothesis.
By
Proposition~\ref{P:full-endpoint-spectrum} and
\eqref{E:action-hessian-splitting},
\begin{equation}\label{E:full-action-Hessian-data}
 \operatorname{ind}\bigl(S_\lambda''(u_\lambda)\bigr)=1,
 \qquad
 \ker S_\lambda''(u_\lambda)
 =\operatorname{span}_{\mathbb{R}}
 \{iu_\lambda,\partial_{z_1}u_\lambda,\ldots,
                    \partial_{z_m}u_\lambda\}.
\end{equation}
The right-hand side is precisely the tangent space at $u_\lambda$ to
the orbit \eqref{E:standing-wave-orbit}.  Moreover,
$g'(u_\lambda)$ and $g(u_\lambda)/u_\lambda$ tend to zero at infinity.
Hence $L_{+,\lambda}$ and $L_{-,\lambda}$ are relatively compact
perturbations of $H+\lambda$, whose essential spectrum starts at
$d+\lambda>0$.  The zero eigenspaces in
\eqref{E:full-action-Hessian-data} are therefore isolated, and the
remaining positive spectrum is bounded away from zero.  All the spectral
assumptions in \eqref{E:abstract-GSS-spectral-assumptions} are verified.

By \eqref{E:main-lower-mass-slope}, the lower-edge slope has the sign of
\begin{equation*}
 \kappa_-=\frac2{\alpha-2}-\frac m2,
\end{equation*}
which is positive for $\alpha<2+4/m$ and negative for
$\alpha>2+4/m$.  By \eqref{E:main-high-mass-slope}, the high-frequency
slope has the sign of
\begin{equation*}
 \kappa_+=\frac2{\beta-2}-\frac N2,
\end{equation*}
which is positive for $\beta<2+4/N$ and negative for
$\beta>2+4/N$.  If either exponent is critical, the displayed leading
coefficient is zero and the sign depends on the next term; this is why the
critical cases are excluded from the assertion.  Applying
Lemma~\ref{L:GSS-Weinstein-translations} now gives orbital stability in
the two subcritical ranges and orbital instability in the corresponding
supercritical ranges, as claimed.
\end{proof}

Combining the preceding argument with Propositions
\ref{P:high-unique-branch} and \ref{P:differentiated-profiles} proves all
remaining statements of Theorem~\ref{T:endpoint-stability}.

\begin{corollary}[Cubic cigar-shaped trap]
\label{C:cubic-endpoint-stability}
Let $(N,d,m)=(3,2,1)$ and $g(s)=s^3$.  Under the standard pure-power Cauchy
theory, so that \ref{condition-D} holds, the endpoint positive profile is unique
up to translation in the free variable, and its standing-wave orbit is
orbitally stable for
$0<\lambda+2\ll1$ and orbitally unstable for
$\lambda\gg1$.  More generally, for the pure powers in the BBJV range
$10/3<p<6$, the lower-edge branch is stable and the high-frequency branch
is unstable.
\end{corollary}

\begin{proof}
For a pure power, $\alpha=\beta=p$.  Since $m=1$ and $N=3$, the two
mass-critical powers are $2+4/m=6$ and $2+4/N=10/3$.  The conclusion is
therefore exactly Theorem~\ref{T:endpoint-stability}.
\end{proof}
\section{Prescribed mass as an application of the global continuum}
\label{sec:normalized-solution}

We now return to the normalized problem \eqref{E:normalized}.  The fixed-frequency
analysis of the preceding sections enters through the mass map
\begin{equation}\label{E:rho}
        \rho(\lambda,u):=\norm{u}_{2}^{2}.
\end{equation}
Existence and multiplicity will be read from the connected component $\mathcal{C}$,
whereas nonexistence requires the endpoint control on the full positive-solution set
$\mathcal{S}$.  This distinction is built into the following statement.

\begin{corollary}[Existence, nonexistence, and multiplicity for prescribed mass]
\label{C:normalized}
Assume \ref{gcon1}, \ref{gcon2}, \ref{gcon3}, and \ref{gcon4}.  Let
$\mathcal{C}$ be the global continuum from Theorem \ref{T:global}.  Then
$\rho(\mathcal{C})$ is an interval whose closure in the extended half-line
$[0,+\infty]$ contains the two endpoint limits
\begin{align}\label{E:endpoint-masses}
        L_-&:=\lim_{\substack{(\lambda,u)\in\cC\\ \lambda+d\to0^+}}\norm{u}_2^2,
        &
        L_+&:=\lim_{\substack{(\lambda,u)\in\cC\\ \lambda\to+\infty}}\norm{u}_2^2,
\end{align}
where $L_-,L_+\in[0,+\infty]$ are extended real numbers given by \eqref{E:lower-limit} and
\eqref{E:high-limit}.  Therefore, for every
\begin{equation}\label{E:a-in-interval}
        a\in \bigl(\min\{L_-,L_+\},\max\{L_-,L_+\}\bigr),
\end{equation}
with the usual convention when one endpoint is $+\infty$, the normalized problem
\eqref{E:normalized} has at least one positive solution.

The endpoint limits also control the masses of all positive solutions with
$\lambda>-d$, not only those on $\mathcal{C}$.  For the nonexistence assertions below, assume in addition the full zero-mass Liouville condition
$$
 \textup{(Z)}\qquad
 (H-d)u=g(u),\quad 0\le u\in C^2(\mathbb{R}^N)\cap X
 \quad\Longrightarrow\quad u\equiv0.
$$
Condition \textup{(Z)} is automatic when $m=N-d\le4$ by
\cite[Theorem~1.1(iii-1)]{CWZZ}; further sufficient conditions are given in
\cite[Theorem~1.1(ii)--(iii)]{CWZZ}.  Together with the exclusion of
$\lambda<-d$ in \cite[Theorem~1.1(i)]{CWZZ}, it reduces every positive
finite-energy solution to the coercive range $\lambda>-d$.

If $L_-<+\infty$ and $L_+<+\infty$, then
\begin{equation}\label{E:finite-global-upper-mass}
 \overline a:=\sup_{(\lambda,u)\in\mathcal{S}}\rho(\lambda,u)<+\infty,
\end{equation}
and \eqref{E:normalized} has no positive solution for $a>\overline a$.
If $L_->0$ and $L_+>0$, allowing either limit to be $+\infty$, then
\begin{equation}\label{E:positive-global-lower-mass}
 \underline a:=\inf_{(\lambda,u)\in\mathcal{S}}\rho(\lambda,u)>0,
\end{equation}
and \eqref{E:normalized} has no positive solution for
$0<a<\underline a$.

For a compact statement of all nine endpoint configurations, write
$$
 c_-:=\|V_\alpha\|_2^2\quad\hbox{when }\alpha=2+\frac4m,
 \qquad
 c_+:=\|W_\beta\|_2^2\quad\hbox{when }\beta=2+\frac4N.
$$
The row records $L_-$, the column records $L_+$, and an attained mass is
one for which \eqref{E:normalized} has a positive solution.  The existence and
multiplicity entries require only \ref{gcon1}--\ref{gcon4}; every
no-solution entry additionally uses \textup{(Z)}.

\begin{table}[ht]
\centering
\normalfont\small
\renewcommand{\arraystretch}{1.28}
\begin{tabular}{c|>{\raggedright\arraybackslash}p{0.245\textwidth}|>{\raggedright\arraybackslash}p{0.245\textwidth}|>{\raggedright\arraybackslash}p{0.245\textwidth}}
 $L_-\backslash L_+$
 & \centering $0$\arraybackslash
 & \centering $c_+\in(0,+\infty)$\arraybackslash
 & \centering $+\infty$\arraybackslash\\ \hline
 $0$
 & At least two solutions for every sufficiently small $a>0$;
   no solution for all sufficiently large $a$.
 & Every $0<a<c_+$ is attained; no solution for all sufficiently large $a$.
 & Every $a>0$ is attained.\\ \hline
 $c_-\in(0,+\infty)$
 & Every $0<a<c_-$ is attained; no solution for all sufficiently large $a$.
 & Every $a$ strictly between $c_-$ and $c_+$ is attained; no solution
   for all sufficiently small or sufficiently large $a$.
 & Every $a>c_-$ is attained; no solution for all sufficiently small $a>0$.\\ \hline
 $+\infty$
 & Every $a>0$ is attained.
 & Every $a>c_+$ is attained; no solution for all sufficiently small $a>0$.
 & At least two solutions for every sufficiently large $a$;
   no solution for all sufficiently small $a>0$.
\end{tabular}
\caption{Prescribed-mass conclusions determined by the two endpoint limits.}
\label{T:mass-endpoint-table}
\end{table}

In the central cell, if $c_-=c_+$, the argument still gives both
nonexistence conclusions but no open interval of attained masses.  In the two
diagonal multiplicity cells $(0,0)$ and $(+\infty,+\infty)$, the two
solutions may be chosen so that, as $a$ tends respectively to $0$ or
$+\infty$, one Lagrange multiplier tends to $-d$ and the other to
$+\infty$.  Remark~\ref{R:exact-multiplicity-diagonal} gives the exact
counts in these two cases and specifies the role of \textup{(Z)}.
\end{corollary}

\begin{proof}[Proof of Corollary~\ref{C:normalized}]
We first extract the existence information carried by the global component.  The mass
map
$$
 \rho(\lambda,u)=\norm{u}_2^2
$$
is continuous on $\mathcal{C}$.  Since $\mathcal{C}$ is connected,
$\rho(\mathcal{C})$ is an interval.  Moreover, Theorem~\ref{T:global} provides
sequences in $\mathcal{C}$ whose parameters tend respectively to $-d$ and to
$+\infty$, while Theorem~\ref{T:asymptotics} gives
$$
 \lim_{\substack{(\lambda,u)\in\mathcal{C}\\ \lambda+d\to0^+}}
 \rho(\lambda,u)=L_-,
 \qquad
 \lim_{\substack{(\lambda,u)\in\mathcal{C}\\ \lambda\to+\infty}}
 \rho(\lambda,u)=L_+.
$$
Thus the closure of $\rho(\mathcal{C})$ contains both endpoint values, and hence
$$
 \bigl(\min\{L_-,L_+\},\max\{L_-,L_+\}\bigr)
 \subset\rho(\mathcal{C}).
$$
This proves the first existence assertion.

For nonexistence we use the full solution set $\mathcal{S}$, as in the
global-branch approach of \cite[Theorem~1.2]{JZZ}.  Since
Theorem~\ref{T:asymptotics} applies to every sequence in $\mathcal{S}$,
its endpoint mass limits are uniform over that set: failure of uniformity
would produce a sequence contradicting the theorem.
More precisely, if $L_-<+\infty$, then
\begin{equation}\label{E:uniform-mass-lower-edge-finite}
 \lim_{\delta\to0^+}
 \sup_{\substack{(\lambda,u)\in\mathcal{S}\\0<\lambda+d<\delta}}
 \bigl|\rho(\lambda,u)-L_-\bigr|=0,
\end{equation}
whereas if $L_-=+\infty$, then
\begin{equation}\label{E:uniform-mass-lower-edge-infinite}
 \lim_{\delta\to0^+}
 \inf_{\substack{(\lambda,u)\in\mathcal{S}\\0<\lambda+d<\delta}}
 \rho(\lambda,u)=+\infty.
\end{equation}
Likewise, if $L_+<+\infty$, then
\begin{equation}\label{E:uniform-mass-upper-edge-finite}
 \lim_{\Lambda\to+\infty}
 \sup_{\substack{(\lambda,u)\in\mathcal{S}\\\lambda>\Lambda}}
 \bigl|\rho(\lambda,u)-L_+\bigr|=0,
\end{equation}
whereas if $L_+=+\infty$, then
\begin{equation}\label{E:uniform-mass-upper-edge-infinite}
 \lim_{\Lambda\to+\infty}
 \inf_{\substack{(\lambda,u)\in\mathcal{S}\\\lambda>\Lambda}}
 \rho(\lambda,u)=+\infty.
\end{equation}

We now combine these uniform endpoint estimates with compactness on an intermediate
frequency strip.  Choose $-d<\lambda_-<\lambda_+<+\infty$.  By
Theorem~\ref{T:compactness}, the previously defined set
$\mathcal U_{\lambda_-}^{\lambda_+}$ is compact in $X_r$.  It is nonempty by
Theorem~\ref{T:global}, and it does not contain $0$.  Since the $L^2$-norm is
continuous on $X_r$, there exist constants
\begin{equation}\label{E:middle-strip-mass-bounds}
 0<m_{\lambda_-,\lambda_+}
 :=\min_{u\in\mathcal U_{\lambda_-}^{\lambda_+}}\norm{u}_2^2
 \le
 \max_{u\in\mathcal U_{\lambda_-}^{\lambda_+}}\norm{u}_2^2
 =:M_{\lambda_-,\lambda_+}<+\infty.
\end{equation}

Suppose first that $L_-<+\infty$ and $L_+<+\infty$.  Fix
$$
 b>\max\{L_-,L_+\}.
$$
By \eqref{E:uniform-mass-lower-edge-finite} and
\eqref{E:uniform-mass-upper-edge-finite}, $\lambda_-$ may be chosen sufficiently
close to $-d$ and $\lambda_+$ sufficiently large so that
$$
 \rho(\lambda,u)<b
 \quad\hbox{whenever}\quad
 (\lambda,u)\in\mathcal{S},
 \quad \lambda\in(-d,\lambda_-)\cup(\lambda_+,+\infty).
$$
Together with \eqref{E:middle-strip-mass-bounds}, this gives
$$
 \sup_{(\lambda,u)\in\mathcal{S}}\rho(\lambda,u)
 \le\max\{b,M_{\lambda_-,\lambda_+}\}<+\infty,
$$
which proves \eqref{E:finite-global-upper-mass}.

Suppose next that $L_->0$ and $L_+>0$, allowing either limit to be
$+\infty$.  Choose $b>0$ smaller than every finite number among $L_-$ and
$L_+$; if both are infinite, take any $b>0$.  The corresponding finite or
infinite versions of the uniform endpoint estimates allow us to choose
$\lambda_-$ and $\lambda_+$ so that
$$
 \rho(\lambda,u)>b
 \quad\hbox{whenever}\quad
 (\lambda,u)\in\mathcal{S},
 \quad \lambda\in(-d,\lambda_-)\cup(\lambda_+,+\infty).
$$
Using \eqref{E:middle-strip-mass-bounds} in the remaining strip, we obtain
$$
 \inf_{(\lambda,u)\in\mathcal{S}}\rho(\lambda,u)
 \ge\min\{b,m_{\lambda_-,\lambda_+}\}>0,
$$
which proves \eqref{E:positive-global-lower-mass}.

These bounds exclude arbitrary positive normalized solutions, not merely already
centered ones.  Indeed, \cite[Theorem~1.1(i)]{CWZZ} rules out $\lambda<-d$, while
condition \textup{(Z)} rules out $\lambda=-d$.  Hence every positive
finite-energy solution lies in the range $\lambda>-d$.  By Theorem~A and
Lemma~\ref{L:individual-decay}, it can be translated in the free variables into
an element of $\mathcal{S}$, and this translation leaves its $L^2$-mass unchanged.

It remains to justify the multiplicity assertions in the two diagonal cells
$(L_-,L_+)=(0,0)$ and $(+\infty,+\infty)$ of
Table~\ref{T:mass-endpoint-table}.  Fix
$\lambda_0\in(-d,-d+\mu_*)$ and put
$$
 a_0:=\norm{u_{\lambda_0}}_2^2.
$$
If $L_-=L_+=0$, then for every $a\in(0,a_0)$ the continuity of the local branch
and its lower-edge mass limit produce a solution of mass $a$ with parameter in
$(-d,\lambda_0)$.  By \eqref{E:uniform-mass-upper-edge-finite}, we may choose
$\Lambda>\lambda_0$ so large that every solution in the slice $\lambda=\Lambda$
has mass smaller than $a$.  The crossing continuum
$\mathcal{C}_{\lambda_0,\Lambda}$ from \eqref{E:crossing-continuum} contains
$(\lambda_0,u_{\lambda_0})$ and meets that terminal slice.  Its continuous image
under $\rho$ is an interval containing $a_0$ and a number smaller than $a$;
hence it contains $a$.  Since $u_{\lambda_0}$ is the unique centered positive
solution at $\lambda_0$ and has mass $a_0\ne a$, the second solution has parameter
in $(\lambda_0,\Lambda]$.  Thus the two solutions are distinct.

If $L_-=L_+=+\infty$, the same argument applies with every $a>a_0$ and all
inequalities reversed.  The local branch produces one solution with parameter in
$(-d,\lambda_0)$, while \eqref{E:uniform-mass-upper-edge-infinite} and the crossing
continuum produce a second one with parameter greater than $\lambda_0$.

The asserted locations of the two Lagrange multipliers also follow from the
compact-strip bounds.  Indeed, if $L_-=L_+=0$ and $a_n\to0^+$, the corresponding parameters cannot remain in a fixed compact subinterval of $(-d,+\infty)$ because of the positive minimum in
\eqref{E:middle-strip-mass-bounds}.  Hence the parameters of the two solutions chosen
on opposite sides of $\lambda_0$ tend respectively to $-d$ and $+\infty$.  If
$L_-=L_+=+\infty$ and $a_n\to+\infty$, the finite maximum in
\eqref{E:middle-strip-mass-bounds} gives the analogous conclusion.

Finally, the endpoint formulae are
$$
 L_-=
 \begin{cases}
 0, & \alpha<2+4/m,\\
 \norm{V_\alpha}_{2}^{2}, & \alpha=2+4/m,\\
 +\infty, & \alpha>2+4/m,
 \end{cases}
 \qquad
 L_+=
 \begin{cases}
 +\infty, & \beta<2+4/N,\\
 \norm{W_\beta}_{2}^{2}, & \beta=2+4/N,\\
 0, & \beta>2+4/N.
 \end{cases}
$$
The nine entries of Table~\ref{T:mass-endpoint-table} now follow directly from
the existence interval, the two global bounds, and the multiplicity argument above.
In particular, in the double critical case both endpoint masses are finite and
strictly positive, so sufficiently small and sufficiently large masses are excluded,
while the open interval between the two endpoint masses is covered whenever they are
different.
\end{proof}

\begin{remark}[Exact multiplicity in the two diagonal cases]
\label{R:exact-multiplicity-diagonal}
Under \ref{gcon1}--\ref{gcon4}, the two diagonal multiplicity conclusions
in Corollary~\ref{C:normalized} can be strengthened: among solutions
with $\lambda>-d$, problem \eqref{E:normalized} has exactly two
centered positive solutions for every sufficiently small mass when
$L_-=L_+=0$, and for every sufficiently large mass when
$L_-=L_+=+\infty$.  Equivalently, positive solutions are counted modulo
translations in the free variables.

Indeed, by Theorem~\ref{T:local-degree},
Proposition~\ref{P:high-unique-branch}, and the slope formulae
\eqref{E:main-lower-mass-slope}--\eqref{E:main-high-mass-slope},
we may fix $-d<\lambda_-<\lambda_+$ such that every centered positive
solution with parameter in $(-d,\lambda_-]$ or $[\lambda_+,+\infty)$
belongs to the corresponding unique endpoint branch, and its mass
is strictly monotone in $\lambda$ on each of these intervals.
On the intermediate strip, \eqref{E:middle-strip-mass-bounds} gives
$$
 0<a_{\min}:=
 \min_{u\in\mathcal U_{\lambda_-}^{\lambda_+}}\|u\|_2^2
 \le
 \max_{u\in\mathcal U_{\lambda_-}^{\lambda_+}}\|u\|_2^2
 =:a_{\max}<+\infty.
$$

If $L_-=L_+=0$, then $\kappa_->0$ and $\kappa_+<0$.
Hence the mass increases strictly from zero along the lower-edge
branch and decreases strictly to zero along the high-frequency
branch.  Since the masses at $\lambda_-$ and $\lambda_+$ are both
at least $a_{\min}$, every $a\in(0,a_{\min})$ is attained exactly once
on each endpoint branch.  No solution of this mass can lie in the
intermediate strip, so there are exactly two centered positive
solutions.

If $L_-=L_+=+\infty$, then $\kappa_-<0$ and $\kappa_+>0$.
The mass decreases strictly from $+\infty$ along the lower-edge
branch and increases strictly to $+\infty$ along the high-frequency
branch.  Since the masses at $\lambda_-$ and $\lambda_+$ are both
at most $a_{\max}$, every $a>a_{\max}$ is attained exactly once on
each endpoint branch and nowhere in the intermediate strip.
This proves the second exact-multiplicity assertion.

If condition \textup{(Z)} in Corollary~\ref{C:normalized} also holds,
then the exclusion of $\lambda\le-d$ recalled there shows that these
are all positive normalized solutions, modulo translations in the
free variables, without any restriction on the Lagrange multiplier.
This refinement uses only the elliptic conclusions and does not
require the dynamical assumption \ref{condition-D}.
\end{remark}

\begin{remark}[The BBJV problem in the power case]\label{R:BBJV-open-problem}
For the model of \cite{BBJV,WeiWu,JeanjeanSong},
$$
 (N,d,m)=(3,2,1),\qquad g(s)=s^{p-1},
 \qquad \frac{10}{3}<p<6,
$$
we have $\alpha=\beta=p$ and $L_-=L_+=0$.  In particular, for the cubic case
$p=4$,
\begin{align*}
 \norm{u}_2^2&=(\lambda+2)^{1/2}
 \bigl(\norm{V_4}_2^2+o(1)\bigr)
 &&\text{as }\lambda+2\to0^+,\\
 \norm{u}_2^2&=\lambda^{-1/2}
 \bigl(\norm{W_4}_2^2+o(1)\bigr)
 &&\text{as }\lambda\to+\infty.
\end{align*}
By Remark~\ref{R:exact-multiplicity-diagonal}, there are exactly two
positive normalized solutions, modulo translations in the free variable,
for every sufficiently small mass, with multipliers tending respectively
to $-2$ and $+\infty$ as the mass tends to zero.
Here condition \textup{(Z)} follows from
\cite[Theorem~1.1(iii-1)]{CWZZ} because $m=1$.
The same condition also justifies the large-mass nonexistence conclusion
in Table~\ref{T:mass-endpoint-table}.

The existence of at least two small-mass positive solutions was obtained
perturbatively in \cite{WeiWu} and nonperturbatively in \cite{JeanjeanSong}.
Here the endpoint uniqueness and mass slopes give the exact count.
The global component covers the full frequency interval, connects the two
endpoint geometries, and is accompanied by endpoint asymptotics for every
centered positive-solution sequence.  The two small-mass states
therefore lie in the opposite geometries of Remark~\ref{R:physical-crossover}.
Under \ref{condition-D}, Theorem~\ref{T:endpoint-stability} also distinguishes
them dynamically:
the state with multiplier tending to $-2$ is orbitally stable, whereas the
state with multiplier tending to $+\infty$ is orbitally unstable.  Thus the
small-mass multiplicity is accompanied by a stability exchange between the
effective one-dimensional and three-dimensional regimes.
\end{remark}

The endpoint analysis does not determine the geometry or stability of
the intermediate part of the global component.
Further multiplicity and a global stability classification require
additional control of degeneracies and turning points of the mass map
along its smooth portions.

\subsection*{Declaration on the use of AI tools}
OpenAI's ChatGPT was used during the preparation of this manuscript as an editorial and typesetting aid, principally for editing the LaTeX source, polishing the English, and implementing author-specified revisions to the exposition and presentation of proof details. The research questions, the formulation of the lemmas and results, and the underlying proof strategies were conceived and determined by the author. Every AI-assisted modification was carefully reviewed and mathematically verified by the author, who takes full responsibility for the final content of the manuscript.


\begin{thebibliography}{99}

\bibitem{AmbrosettiArcoya}
A. Ambrosetti and D. Arcoya,
\emph{An Introduction to Nonlinear Functional Analysis and Elliptic Problems},
Birkh\"auser, Boston, 2011.

\bibitem{AntonelliCarlesDrumondSilva}
P. Antonelli, R. Carles and J. Drumond Silva,
\emph{Scattering for nonlinear Schr\"odinger equation under partial harmonic confinement},
Comm. Math. Phys. \textbf{334} (2015), 367--396.

\bibitem{BartschZhongZou}
T. Bartsch, X. Zhong and W. Zou,
\emph{Normalized solutions for a coupled Schr\"odinger system},
Math. Ann. \textbf{380} (2021), 1713--1740.

\bibitem{BBJV}
J. Bellazzini, N. Boussa\"id, L. Jeanjean and N. Visciglia,
\emph{Existence and stability of standing waves for supercritical NLS with a partial confinement},
Comm. Math. Phys. \textbf{353} (2017), 229--251.

\bibitem{BenAbdallahMehatsSchmeiserWeishaeupl}
N. Ben Abdallah, F. M\'ehats, C. Schmeiser and R. M. Weish\"aupl,
\emph{The nonlinear Schr\"odinger equation with a strongly anisotropic harmonic potential},
SIAM J. Math. Anal. \textbf{37} (2005), 189--199.

\bibitem{BerestyckiLions}
H. Berestycki and P.-L. Lions,
\emph{Nonlinear scalar field equations. I. Existence of a ground state},
Arch. Rational Mech. Anal. \textbf{82} (1983), 313--345.

\bibitem{BrezisLieb}
H. Brezis and E. Lieb,
\emph{Minimum action solutions of some vector field equations},
Comm. Math. Phys. \textbf{96} (1984), 97--113.

\bibitem{CazenaveLions}
T. Cazenave and P.-L. Lions,
\emph{Orbital stability of standing waves for some nonlinear Schr\"odinger equations},
Comm. Math. Phys. \textbf{85} (1982), 549--561.


\bibitem{CWZZ}
Z. Chen, Y. Wang, J. Zhang and X. Zhong,
\emph{Liouville theorems and symmetry of positive solutions for partially confined
nonlinear Schr\"odinger equations}, arXiv:2608.21724, 2026.

\bibitem{ChangGustafsonNakanishiTsai}
S.-M. Chang, S. Gustafson, K. Nakanishi and T.-P. Tsai,
\emph{Spectra of linearized operators for NLS solitary waves},
SIAM J. Math. Anal. \textbf{39} (2007), 1070--1111.

\bibitem{ChenWalshWheeler}
R. M. Chen, S. Walsh and M. H. Wheeler,
\emph{Global bifurcation for monotone fronts of elliptic equations},
J. Eur. Math. Soc. (2024), published online first,
doi:10.4171/JEMS/1561.

\bibitem{GidasNiNirenberg}
B. Gidas, W.-M. Ni and L. Nirenberg,
\emph{Symmetry and related properties via the maximum principle},
Comm. Math. Phys. \textbf{68} (1979), 209--243.

\bibitem{GidasSpruck}
B. Gidas and J. Spruck,
\emph{Global and local behavior of positive solutions of nonlinear elliptic equations},
Comm. Pure Appl. Math. \textbf{34} (1981), 525--598.

\bibitem{GrillakisShatahStraussI}
M. Grillakis, J. Shatah and W. Strauss,
\emph{Stability theory of solitary waves in the presence of symmetry. I},
J. Funct. Anal. \textbf{74} (1987), 160--197,
doi:10.1016/0022-1236(87)90044-9.

\bibitem{GrillakisShatahStraussII}
M. Grillakis, J. Shatah and W. Strauss,
\emph{Stability theory of solitary waves in the presence of symmetry. II},
J. Funct. Anal. \textbf{94} (1990), 308--348,
doi:10.1016/0022-1236(90)90016-E.

\bibitem{GilbargTrudinger}
D. Gilbarg and N. S. Trudinger,
\emph{Elliptic Partial Differential Equations of Second Order},
Springer-Verlag, Berlin, 2001.

\bibitem{GorlitzEtAl}
A. G\"orlitz, J. M. Vogels, A. E. Leanhardt, C. Raman, T. L. Gustavson,
J. R. Abo-Shaeer, A. P. Chikkatur, S. Gupta, S. Inouye, T. P. Rosenband,
D. E. Pritchard and W. Ketterle,
\emph{Realization of Bose--Einstein condensates in lower dimensions},
Phys. Rev. Lett. \textbf{87} (2001), 130402.

\bibitem{HaniThomann}
Z. Hani and L. Thomann,
\emph{Asymptotic behavior of the nonlinear Schr\"odinger equation with harmonic trapping},
Comm. Pure Appl. Math. \textbf{69} (2016), 1727--1776.

\bibitem{JinHong}
Y. Hong and S. Jin,
\emph{Uniqueness and orbital stability of standing waves for the nonlinear Schr\"odinger equation with a partial confinement},
SIAM J. Math. Anal. \textbf{56} (2024), 2714--2737.

\bibitem{JeanjeanSong}
L. Jeanjean and L. Song,
\emph{Multiplicity result for a mass supercritical NLS with a partial confinement},
SIAM J. Math. Anal. \textbf{57} (2025), 6709--6730,
doi:10.1137/25M1740048.

\bibitem{JZZ}
L. Jeanjean, J. Zhang and X. Zhong,
\emph{A global branch approach to normalized solutions for the Schr\"odinger equation},
J. Math. Pures Appl. \textbf{183} (2024), 44--75.

\bibitem{Jia}
H. Jia,
\emph{Stable standing waves for cubic nonlinear Schr\"odinger systems with partial confinement},
Discrete Contin. Dyn. Syst. \textbf{40} (2020), 5285--5309.

\bibitem{Kato1972}
T. Kato,
\emph{Schr\"odinger operators with singular potentials},
Israel J. Math. \textbf{13} (1972), 135--148.

\bibitem{Kwong}
M. K. Kwong,
\emph{Uniqueness of positive solutions of $\Delta u-u+u^p=0$ in $\mathbb{R}^n$},
Arch. Rational Mech. Anal. \textbf{105} (1989), 243--266.

\bibitem{LiebSeiringerYngvason}
E. H. Lieb, R. Seiringer and J. Yngvason,
\emph{One-dimensional behavior of dilute, trapped Bose gases},
Comm. Math. Phys. \textbf{244} (2004), 347--393.

\bibitem{Nirenberg}
L. Nirenberg,
\emph{Topics in Nonlinear Functional Analysis},
Courant Lecture Notes in Mathematics, Vol.~6, American Mathematical Society,
Providence, RI, 2001.

\bibitem{QuittnerSouplet}
P. Quittner and P. Souplet,
\emph{Superlinear Parabolic Problems: Blow-up, Global Existence and Steady States},
2nd ed., Birkh\"auser/Springer, Cham, 2019.

\bibitem{Rabinowitz}
P. H. Rabinowitz,
\emph{Some global results for nonlinear eigenvalue problems},
J. Funct. Anal. \textbf{7} (1971), 487--513.

\bibitem{Reed-Simon-1980}
M. Reed and B. Simon, Methods of Modern Mathematical Physics, Vol. I: Functional Analysis, Academic Press, 1980.


\bibitem{SalasnichParolaReatto}
L. Salasnich, A. Parola and L. Reatto,
\emph{Effective wave equations for the dynamics of cigar-shaped and disk-shaped Bose condensates},
Phys. Rev. A \textbf{65} (2002), 043614.

\bibitem{SchneeYngvason}
K. Schnee and J. Yngvason,
\emph{Bosons in disc-shaped traps: From 3D to 2D},
Comm. Math. Phys. \textbf{269} (2007), 659--691.

\bibitem{ShanShuaiWu}
L. Shan, W. Shuai and L. Wu,
\emph{Multiple solutions to the nonlinear Schr\"odinger equation with a partial confinement},
J. Anal. Math. \textbf{158} (2026), 283--316,
doi:10.1007/s11854-025-0409-9.


\bibitem{SongHajaiej}
L. Song and H. Hajaiej,
\emph{A new method to prove the existence, non-existence, multiplicity,
uniqueness, and orbital stability/instability of standing waves for NLS
with partial confinement},
arXiv:2211.10058v1, 2022.


\bibitem{SongYan}
L. Song and R. Yan,
\emph{Existence of normalized ground states for supercritical/critical NLS with a partial confinement},
Proc. Amer. Math. Soc. \textbf{154} (2026), 2421--2431,
doi:10.1090/proc/17270.

\bibitem{Stuart}
C. A. Stuart,
\emph{Uniqueness and stability of ground states for some nonlinear Schr\"odinger equations},
J. Eur. Math. Soc. \textbf{8} (2006), 399--414.

\bibitem{Weinstein}
M. I. Weinstein,
\emph{Nonlinear Schr\"odinger equations and sharp interpolation estimates},
Comm. Math. Phys. \textbf{87} (1983), 567--576.

\bibitem{WeinsteinStability}
M. I. Weinstein,
\emph{Lyapunov stability of ground states of nonlinear dispersive evolution equations},
Comm. Pure Appl. Math. \textbf{39} (1986), 51--68,
doi:10.1002/cpa.3160390103.

\bibitem{WeiWu}
J. Wei and Y. Wu,
\emph{On some nonlinear Schr\"odinger equations in $\mathbb{R}^N$},
Proc. Roy. Soc. Edinburgh Sect. A \textbf{153} (2023), 1503--1528.

\bibitem{Willem}
M. Willem,
\emph{Minimax Theorems},
Birkh\"auser Boston, MA, 1996.

\end{thebibliography}
\end{document}